\documentclass[11pt,a4paper]{article}

\usepackage{amsthm, amsfonts,amssymb,euscript}
\usepackage{amsmath}
\usepackage{bm}
\usepackage{amssymb}
\usepackage{amsthm}
\usepackage{graphicx}
\usepackage{caption}
\usepackage{subcaption}
\usepackage{amsfonts}
\usepackage{float}
\usepackage{color}
\usepackage{slashed}

\usepackage[usenames,dvipsnames,svgnames,table]{xcolor}
\usepackage[colorlinks=true]{hyperref}
\hypersetup{linkcolor=BrickRed, urlcolor=green, citecolor=blue, linktoc=page}

\numberwithin{equation}{section}

\newtheorem*{proposition*}{Proposition}
\newtheorem*{theorem*}{Theorem}
\newtheorem*{conjecture*}{Conjecture}
\newtheorem*{claim*}{Claim}
\newtheorem*{lemma*}{Lemma}
\newtheorem*{corollary*}{Corollary}
\newtheorem{theorem}{Theorem}[section]
\newtheorem{proposition}[theorem]{Proposition}

\newtheorem{lemma}[theorem]{Lemma}

\newtheorem*{definition*}{Definition}

\newtheorem*{assumption*}{\mathcal{A}ssumption}

\newtheorem*{remark*}{Remark}
\newtheorem{remark}{Remark}[section]

\newcommand{\R}{\mathbb{R}}
\newcommand{\s}{\mathbb{S}}

\newcommand{\N}{\mathbb{N}}

\newcommand{\snabla}{\slashed{\nabla}}

\allowdisplaybreaks

\begin{document}

\title{A vector field approach to almost-sharp decay for the wave equation on spherically symmetric, stationary spacetimes}
\author{Y. Angelopoulos, S. Aretakis, and D. Gajic}
\date{February 15, 2018}
\normalsize

\maketitle

\begin{abstract}

We present a new vector field approach to almost-sharp decay for the wave equation on spherically symmetric, stationary and asymptotically flat spacetimes. Specifically, we derive a new hierarchy of higher-order weighted energy estimates by employing appropriate commutator vector fields. In cases where an integrated local energy decay estimate holds, like in the case of sub-extremal Reissner--Nordstr\"{o}m black holes, this hierarchy leads to almost-sharp global energy and pointwise time-decay estimates with decay rates that go beyond those obtained by the traditional vector field method. Our estimates play a fundamental role in our companion paper where precise late-time asymptotics are obtained for linear scalar fields on such backgrounds.

\end{abstract}

\tableofcontents

\section{Introduction }

\subsection{Decay estimates for the wave equation: Overview} 

The study of the long-time behavior of solutions to the wave equation
\begin{equation}
\label{waveequation}
\square_g\psi=\frac{1}{\sqrt{|\text{det}g|}}\partial_{a}\Big(\sqrt{|\text{det}g|}\cdot g^{ab}\cdot \partial_{b}\psi \Big)=0,
\end{equation}
where $g$ is a Lorentzian metric on a 4-dimensional Lorentzian manifold $\mathcal{M}$, is of fundamental importance throughout mathematical physics. In particular, decay results for the wave equation have applications to the following fundamental problems in general relativity:
A) long-time behavior of non-linear wave equations, including the Einstein equations (see \cite{k1,christab,HV2016}), B) stability of black hole backgrounds (see \cite{lecturesMD,blukerr,Dafermos2016}), C) propagation of gravitational waves (see \cite{gravitation,memorychistodoulou}).

Traditional methods for rigorously establishing decay estimates for the wave equation \eqref{waveequation} include the use of  1) representation formulas and/or properties of the fundamental solution (see \cite{evans,sogge}), 2) the Fourier transform (see \cite{hormander1}), 3) the conformal compactification and local theory (see \cite{fr1,chcon}), and 4) the vector field method (see \cite{mor1,muchT}).

The vector field method is very robust to non-linear perturbations and hence has been successfully used in obtaining various remarkable results in hyperbolic PDE (see, for instance, \cite{christab}). 
A novel energy approach to uniform decay results introduced by Dafermos--Rodnianski  \cite{newmethod}, and generalized by Moschidis \cite{moschidis1}, extended the applicability of the vector field method to a very general class of asymptotically flat spacetimes. One disadvantage, however, of the known vector field techniques is that they provide relatively weak  bounds on the decay rates : $t^{-2}$ decay for the global energy flux, $t^{-3/2}$ decay for the scalar field and $u^{-1/2}$ decay for the radiation field along null infinity (here $t$ is a time coordinate and $u$ is a retarded time coordinate).
On the other hand, an upper bound for the decay rate for solutions $\psi$ to the wave equation 
on Schwarzschild is suggested by Price's  heuristic 
  polynomial law (see \cite{price1972}):  
\begin{equation}
\psi(t,r_{0},\theta,\phi)\sim t^{-3},
\label{pricepolylaw}
\end{equation}
asymptotically as $t\rightarrow \infty$ along constant $r=r_{0}$ hypersurfaces. In fact \cite{price1972} suggested that if $\psi_{\ell}$ denotes the projection of $\psi$ on the $\ell^{\text{th}}$ angular frequency then the following late-time polynomial law along constant  $r=r_{0}$ hypersurfaces holds:
\begin{equation}
\psi_{\ell}(t,r_{0},\theta,\phi)\sim t^{-2\ell-3}.
\label{pricepolylawell}
\end{equation}
Subsequent work by Gundlach, Price and Pullin \cite{CGRPJP94} suggested that the radiation field $r\psi$ obeys the following polynomial law along the null infinity $\mathcal{I}$
\begin{equation}
r\psi\left.\right|_{\mathcal{I}}(u,\cdot)\sim u^{-2}.
\label{rfpullin}
\end{equation}

There have been numerous rigorous works which proved upper bounds on solutions to the wave equation which are consistent with the above heuristics and use in one way or another explicit representation formulas.
 Kronthaler \cite{kro} obtained $t^{-3}$ decay for spherically symmetric solutions to the (decoupled) wave equation on Schwarzschild under the assumptions that the initial data are   compactly supported and  supported away from the event horizon. 
Subsequently Donninger, Schlag and Soffer in \cite{other1} obtained $\ell$-dependent decay rates for fixed angular frequencies. They proved  at least $t^{-2\ell-2}$ decay  for general initial data and  faster $t^{-2\ell-3}$ decay for static initial data. Subsequently,  Metcalfe, Tataru and Tohaneanu in \cite{metal} proved $t^{-3}$ decay  for a general class of nonstationary asymptotically flat spacetimes   based on properties of the fundamental solution for the constant coefficient d'Alembertian (see also \cite{tataru3} for stationary spacetimes). Finally, for upper bounds for non-linear wave equations,  we refer to the work of  Dafermos and Rodnianski \cite{MDIR05} on the spherically symmetric Einstein--Maxwell-scalar field model.

The purpose of this paper is to present a new \textit{vector field approach} to decay estimates  which yields almost-sharp $t^{-5+\epsilon}$ decay for the global energy flux, $t^{-3+\epsilon}$ decay for the scalar field and $u^{-2+\epsilon}$ decay for the radiation field on spherically symmetric, stationary and asymptotically flat spacetimes. 
 One of the key new features of our method, which builds on the Dafermos--Rodnianski method, is that it makes use of the conservation laws associated to the limiting Newman--Penrose scalars on null infinity  introduced in Section \ref{sec:1stnpconstant}. Our method plays a crucial role in our companion paper \cite{paper2} where we derive the precise late-time asymptotics for solutions to the wave equation which, in particular, enable us to  obtain the first rigorous proof of the Price's polynomial law \eqref{pricepolylaw} as an upper and \textit{lower} bound.

We provide a brief introduction to the classical vector field method in Section \ref{introvf} and the Dafermos--Rodnianski method in Section \ref{introdr}. An overview of the new approach
 is presented in Section \ref{sketch} and a summary of the main results is presented in Section \ref{sintro}.  Section \ref{introappli} contains a  detailed discussion on applications of our method and estimates.

\subsubsection{The vector field method}
\label{introvf}

One associates to a scalar field $\psi$ the so-called energy-momentum tensor
\[\textbf{T}_{ab}[\psi ]=\partial_{a}\psi\cdot \partial_{b}\psi-\frac{1}{2}g_{ab}\partial^{c}\psi\cdot \partial_{c}\psi \]
which is a symmetric 2-tensor satisfying the following
\[\textbf{T}[\psi ](N_{1}, N_{2})\sim \sum_{a}|\partial_{a}\psi|^{2}, \]
\[\text{Div}\big(\textbf{T}[\psi]\big)=\Box_{g}\psi\cdot d\psi, \]
where $N_{1},N_{2}$ are timelike vector fields (the constant in $\sim$ depends implicitly on the norms of $N_{1},N_{2}$). To obtain useful estimates one then considers energy currents $J_{a}^{V}[W\psi]$ of the form
\[\Big(J^{V}[W\psi]\Big)_{a}=\textbf{T}_{ab}[W\psi]\cdot V^{b}. \]
Here $V$ and $W$ are vector fields; in fact $W$ can be taken to be a general differential operator. The vector field $V$ is traditionally called the vector field \textit{multiplier}  and the vector field $L$ is called the vector field \textit{commutator}. The use of vector fields as multipliers goes back to Morawetz \cite{mor1}, whereas the use of vector fields as commutators was initiated by Klainerman \cite{muchT, sergiunull}.

An immediate observation is that if the multiplier vector field $V$ is a Killing field then the energy current $J_{a}^{V}[\psi]$ is divergence-free leading to the following divergence identity
\[\int_{\widetilde{\Sigma}_{t}}\textbf{T}[\psi](V,n_{\widetilde{\Sigma}_{t}})=\int_{\widetilde{\Sigma}_{0}}\textbf{T}[\psi](V,n_{\widetilde{\Sigma}_{0}}),  \]
where $\widetilde{\Sigma}_{t},\widetilde{\Sigma}_{0}$ are Cauchy hypersurfaces and $n_{\widetilde{\Sigma}_{t}}, n_{\widetilde{\Sigma}_{0}}$ are their future-directed timelike unit normals, respectively. An appropriate choice of $V$ yields the boundedness of non-negative quadratic expressions of the first-order derivatives of $\psi$ in terms of the corresponding expressions of the initial data. On the other hand, the use of appropriate commutator vector fields allows for the control of \textit{higher-order} quadratic expressions of $\psi$. 

Minkowski spacetime has a wealth of (conformal) symmetries: 1) Translations $T=\partial_{t}, \partial_{x^{i}}$, 2) Scaling $S=t\partial_{t}+x^{i}\partial_{x^{i}}$, 3) Conformal Morawetz $K=(t^2+r^2)\partial_{t}+2tx^{i}\partial_{x^{i}}$, 4) Rotations $\Omega_{ab}=x^{a}\partial_{x^{b}}-x^{b}\partial_{x^{a}}$. Here, $(x^{0}=t,x^1,x^2,x^3)$ is a rectilinear coordinate system and $r=\sqrt{(x^{1})^{2}+(x^{2})^{2}+(x^{3})^{2}}$ is the radius coordinate.

If we apply the time translation vector field $T=\partial_{t}$  as a multiplier, then we obtain
\begin{equation}
\int_{\widetilde{\Sigma}_{t}}\textbf{T}[\psi](T,T)=\int_{\widetilde{\Sigma}_{0}}\textbf{T}[\psi](T,T),
\label{tidentity}
\end{equation}
where $\textbf{T}[\psi](T,T) \sim \sum_{a}|\partial_{a}\psi|^{2}$. 	


\begin{figure}[h!]
\begin{center}
\includegraphics[width=4in]{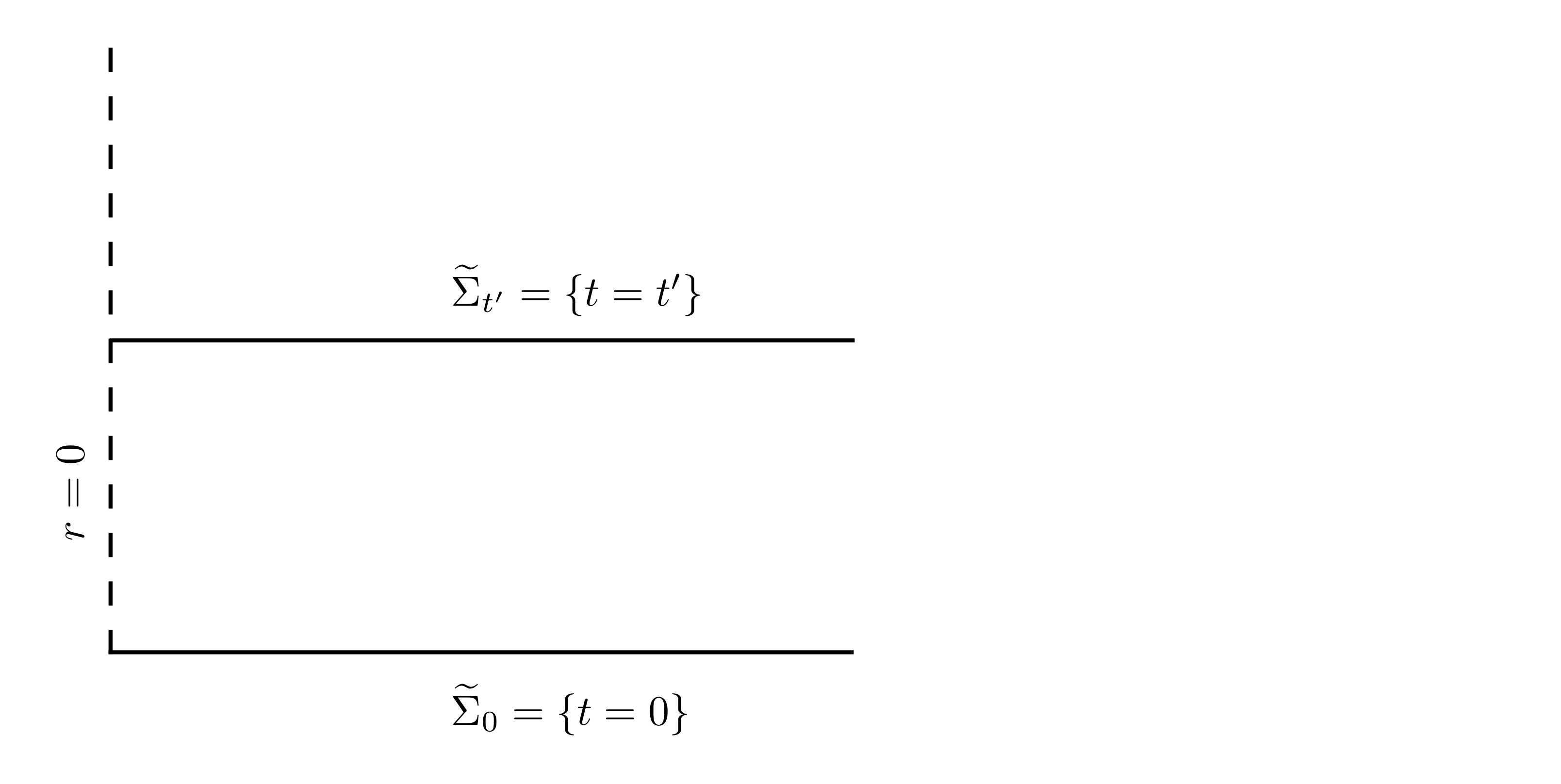}
\caption{\label{fig:sdsewrew}A foliation of Minkowski by hypersurfaces of constant $t$.}
\end{center}
\end{figure}

The above identity provides a bound for the energy flux of $\psi$ through the constant $t$ hypersurfaces $\widetilde{\Sigma}_{t}$; however it also shows that this energy flux does not decay in time. In order to obtain decay estimates, one needs to 1) apply multiplier vector fields with weights in $t$, and 2) restrict to appropriate regions, such as $\left\{r\leq R \right\}$, with $R>0$. Indeed, applying the conformal Morawetz vector field $K$
as a multiplier yields the boundedness of the (appropriately modified) energy current $J^{K}[\psi]$ which in turn satisfies 
\[ J^{K}[\psi]\cdot n_{t} \geq C_{R}\cdot  t^{2}\cdot \textbf{T}[\psi](T,T) \]
in  $\widetilde{\Sigma}_{t}\cap\left\{r\leq R\right\}$ with $R>0$. This implies the boundedness of the flux
\[\int_{\widetilde{\Sigma}_{t}\cap \left\{r\leq R \right\}}t^{2}\cdot \textbf{T}[\psi](T,T)\]
which immediately translates to decay for the flux of $\psi$ through $\widetilde{\Sigma}_{t}\cap \left\{r\leq R \right\}$. This approach yields the non-integrable $\frac{1}{t}$ pointwise decay for $\psi$.

Morawetz \cite{mor1} introduced the radial vector field $\mathbb{M}=\partial_{r}$ (which points towards the direction of spatial dispersion) as an alternative choice of multiplier vector field. The vector field $\mathbb{M}$ is not symmetry-generating and, hence, the spacetime terms (which arise from the divergence of the associated energy current $Q^{\mathbb{M}}[\psi]$) are non-trivial. In fact, the crucial observation is that these terms (modulo appropriate modifications of the energy current) are positive-definite, which leads to an integrated local energy decay estimate of the form
\begin{equation}
\int_{0}^{t}\int_{\widetilde{\Sigma}_{\bar{t}}\cap \left\{r\leq R \right\}}\textbf{T}[\psi](T,T)\,d\bar{t}\leq C\int_{\widetilde{\Sigma}_{0}} \textbf{T}[\psi](T,T).
\label{moraintro}
\end{equation}

Another very important development is due to Klainerman, who initiated the use of commutator vector fields and combined \eqref{tidentity} with the commutator vector fields $\partial_{t}, \partial_{x^{a}}, S, \Omega_{ab}$ to deduce the boundedness of higher-order weighted quadratic expressions of $\psi$. In turn, via appropriate generalizations of the Sobolev inequality, now known as the Klainerman--Sobolev inequality, this yields
\[|\psi|\leq C\sqrt{E}\frac{1}{(|t-r|+1)^{1/2}}\frac{1}{(|t+r|+1)}. \] 
Here $E$ is initial data higher-order weighted norm of $\psi$. The above estimate immediately yields  
\[|\psi|\leq C(R)\sqrt{E}\frac{1}{t^{3/2}}\]
for $r\leq R$. Note that the integrability in time of the upper bound using Klainerman's method plays a fundamental role in the study of non-linear wave equations.

\subsubsection{The Dafermos--Rodnianski method}
\label{introdr}

The above estimates for the flat wave equation do not extend immediately to more general curved spacetimes, including black hole spacetimes; the main reason being the growth in time of the error terms associated with the failure of either the multipliers or the commutators to be Killing (see however \cite{luk,lukkerr} for a modification of the above estimates that does produce decay results in slowly rotating Kerr spacetimes).  Dafermos and Rodnianski \cite{newmethod} introduced in 2008 a new approach which circumvents these difficulties very efficiently. Their method 1) captures the radiative properties of waves in a natural way, 2) can be applied to a vast class of spacetimes, and 3) does not make use of multipliers or commutators with weights growing in $t$ along fixed $r$ hypersurfaces.


\begin{figure}[h!]
\begin{center}
\includegraphics[width=4in]{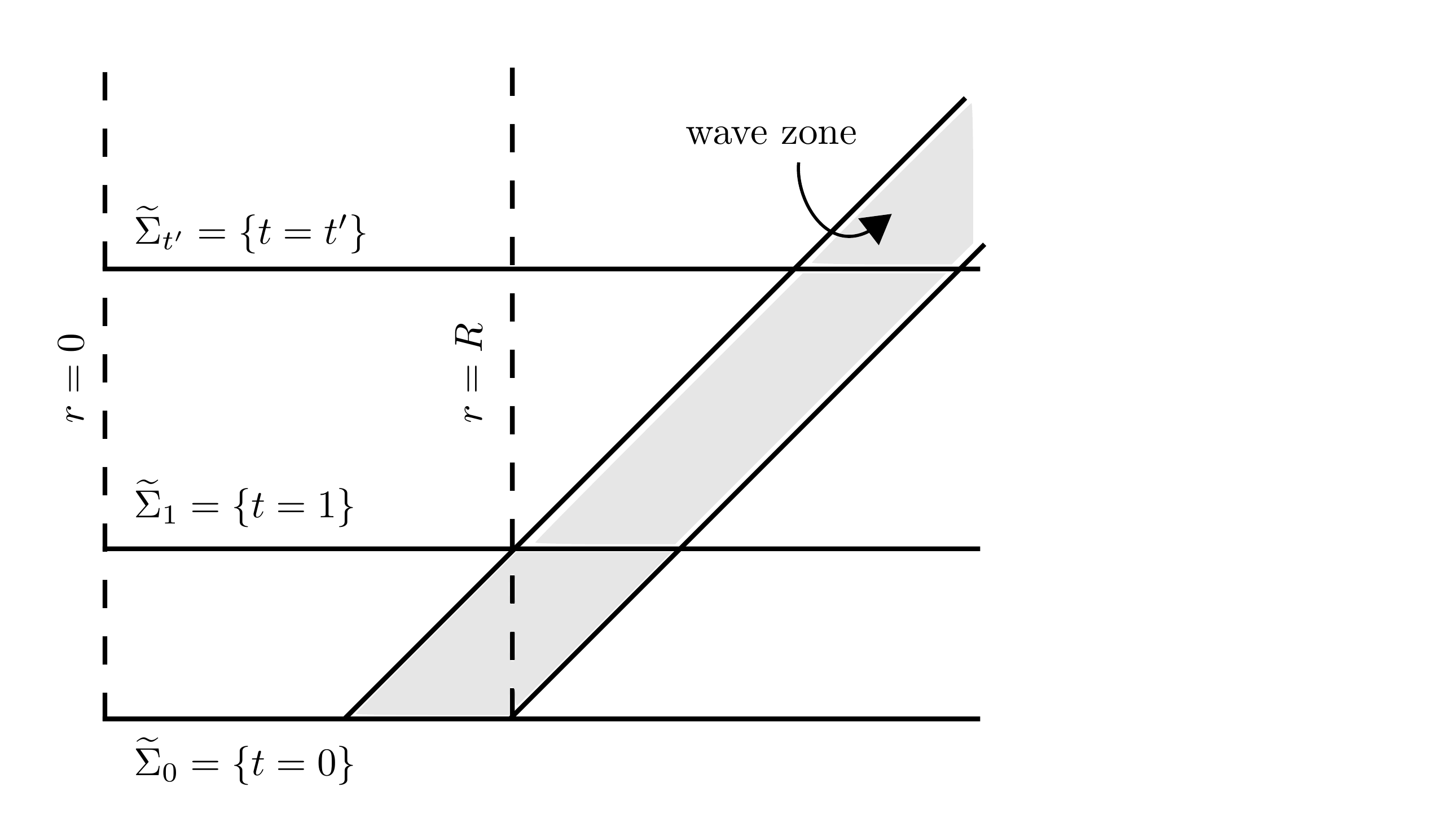}
\caption{\label{fig:sds}Spacelike hypersurfaces and the wave zone}
\end{center}
\end{figure}

A first important observation is that the constant $t$ hypersurfaces $\widetilde{\Sigma}_{t}$ are not suitable for capturing the decay properties of waves. Indeed, for all $t>0$, the spacelike hypersurface $\widetilde{\Sigma}_{t}$ intersects the wave zone bounded by the outgoing null hypersurfaces emanating from $\left\{r=R\right\}\cap \widetilde{\Sigma}_{0}$ and $\left\{r=R\right\}\cap \widetilde{\Sigma}_{1}$, for some $R>0$. Hence, the energy flux of $\widetilde{\Sigma}_{t}$ will always take into account the radiation propagated in this wave zone.

This leads to a new kind of spacelike-null hypersurfaces ${\Sigma}_{\tau}$ given by  
\[{\Sigma}_{\tau}=\left\{\widetilde{\Sigma}_{\tau} \text{ for }r\leq R  \right\} \cup \left\{\mathcal{N}_{\tau} \text{ for }r\geq R  \right\} \]
where $\mathcal{N}_{\tau}$ is the outgoing null hypersurface emanating from $\left\{r= R \right\}\cap \widetilde{\Sigma}_{\tau}$.


\begin{figure}[h!]
\begin{center}
\includegraphics[width=4in]{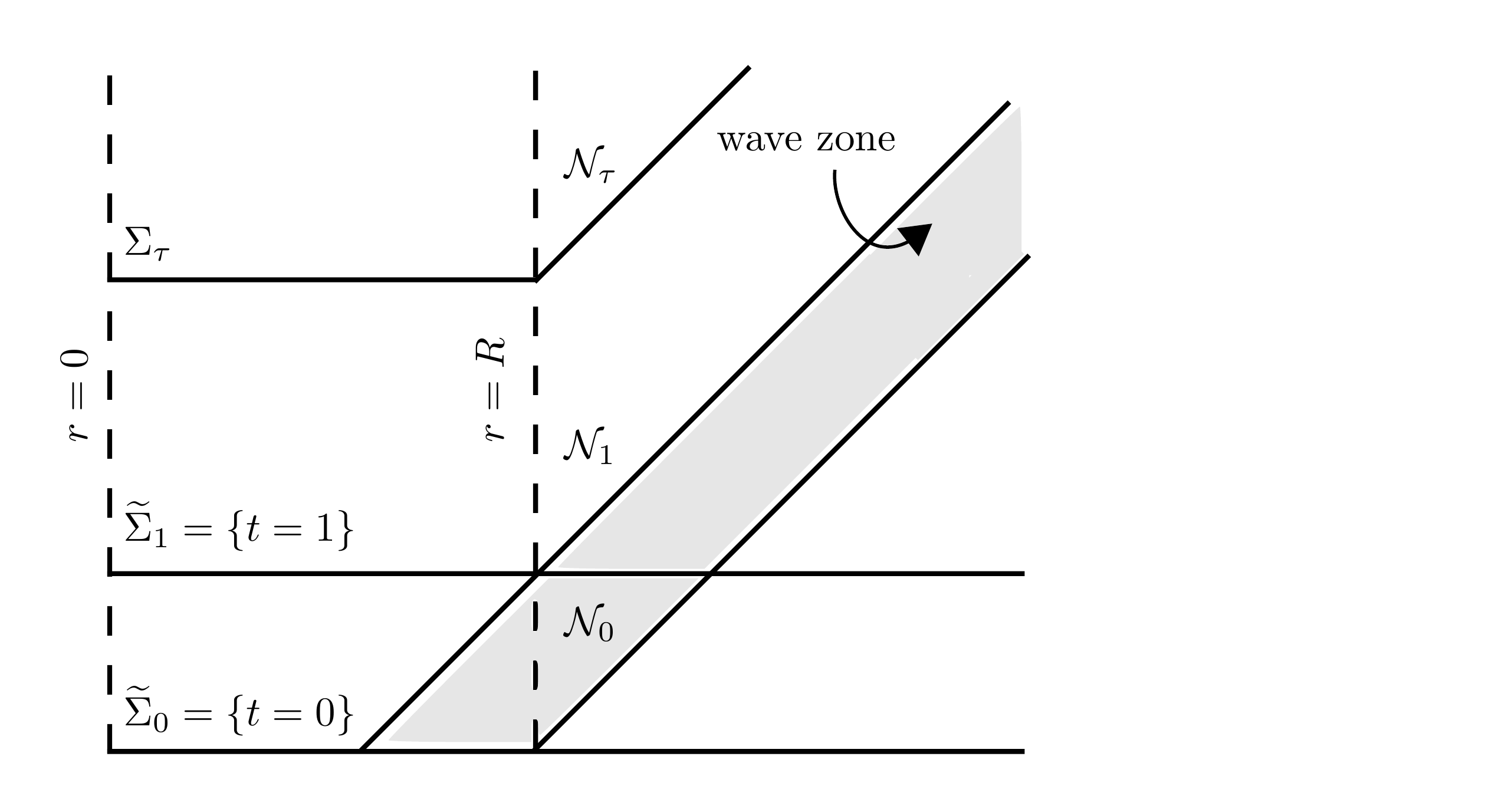}
\caption{\label{fig3}A foliation by spacelike-null hypersurfaces ${\Sigma}_{\tau}$.}
\end{center}
\end{figure}

Applying now the energy identity with $T=\partial_{t}$ as the multiplier vector field in the region bounded by the hypersurfaces ${\Sigma}_{0}$ and ${\Sigma}_{\tau}$ we obtain
\begin{equation}
\int_{{\Sigma}_{\tau}}\textbf{T}[\psi](T, n_{{\Sigma}_{\tau}}) + \int_{\mathcal{I}^{+}_{\tau}}\textbf{T}[\psi](T, n_{\mathcal{I}^{+}_{\tau}})= \int_{{\Sigma}_{0}}\textbf{T}[\psi](T, n_{{\Sigma}_{0}}), 
\label{boundnullflux}
\end{equation} 
where $\int_{\mathcal{I}^{+}_{\tau}}\textbf{T}[\psi](T, n_{\mathcal{I}^{+}_{\tau}})$ denotes the flux through the future null infinity $\mathcal{I}^{+}$. Here, $\mathcal{I}^{+}$ is the limiting hypersurface formed by the limit points of future null geodesics along which $r\rightarrow \infty$.
 

\begin{figure}[h!]
\begin{center}
\includegraphics[width=4in]{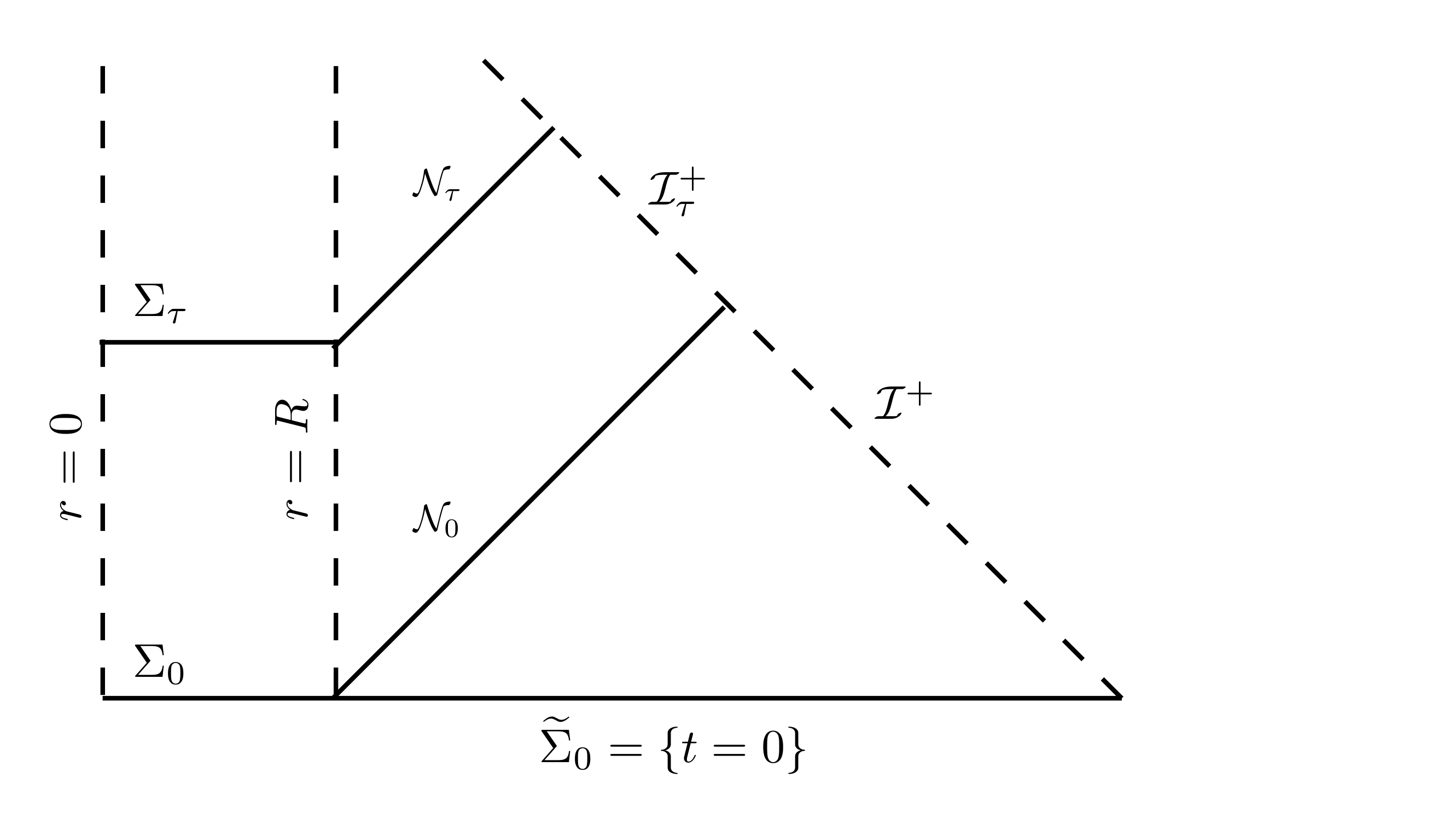}
\caption{\label{fig4}The hypersurfaces $\Sigma_{\tau}$ and the future null-infinity $\mathcal{I}^+$}
\end{center}
\end{figure}

Hence the flux through ${\Sigma_{\tau}}$ measures the part of the energy of the waves that has not been radiated to null infinity up to retarded time $\tau$. Clearly, decay for the energy flux through the hypersurfaces ${\Sigma}_{\tau}$ is equivalent to the statement that the total initial energy is radiated through null infinity:
\[\lim_{\tau\rightarrow\infty}\int_{\mathcal{I}^{+}_{\tau}}\textbf{T}[\psi](T, n_{\mathcal{I}^{+}_{\tau}})= \int_{{\Sigma}_{0}}\textbf{T}[\psi](T, n_{{\Sigma}_{0}}). \] 
The method of Dafermos and Rodnianski establishes a hierarchy of weighted estimates in the ``far away'' region $r\geq R$ (foliated by the null hypersurfaces $\mathcal{N}_{\tau}$) which are used in conjunction with the ILED estimate \eqref{moraintro} in the ``near region'' $r\leq R$. Specifically, their method uses the vector field $r^{p}\cdot L$ as a multiplier, where the null vector field 
\[L=\partial_{v}\]
is normal to the null hypersurface $\mathcal{N}_{\tau}$. Here $\partial_{v}$ is to be considered with respect to the null Eddington--Finkelstein coordinate system $(u,v)$, where $u$ is the retarded time and $v$ is the advanced time coordinate. The associated energy identities for $1\leq p\leq 2$  yield (after removing various distracting terms\footnote{such as terms involving angular derivatives and higher-order fluxes originating from the \textit{trapping effect}.}) the following schematic hierarchy of $r$-weighted estimates
\begin{equation}
\begin{split}
\int_{\tau_{1}}^{\tau_{2}}\int_{\mathcal{N}_{\tau}}\textbf{T}[\psi](T, n_{\mathcal{N}_{\tau}}) & \leq C \int_{N_{\tau_{1}}}r\cdot (L(r\psi))^{2}\ drd\omega,  \\
\int_{\tau_{1}}^{\tau_{2}}\int_{\mathcal{N}_{\tau}}r\cdot (L(r\psi))^{2}\ drd\omega &\leq C \int_{N_{\tau_{1}}}r^{2}\cdot (L(r\psi))^{2}\ drd\omega,  
\end{split}
\label{drhierarchy}
\end{equation}
 The precise estimates include weighted positive-definite fluxes which for simplicity we have also omitted here. Here $d\omega$ denotes the area element over the spherical sections of $\mathcal{N}_{\tau}$. 
 The hierarchy \eqref{drhierarchy} (coupled with \eqref{boundnullflux} and \eqref{moraintro}) yields the energy decay estimate
\[\int_{{\Sigma}_{\tau}}\textbf{T}[\psi](T, n_{{\Sigma}_{\tau}}) \leq C\cdot D_{0}[\psi]\frac{1}{\tau^{2}}  \]
 and the pointwise estimates
\[|\psi|\leq C\sqrt{D_{1}[\psi]}\frac{1}{\tau}, \ \ |r\psi|\leq C\sqrt{D_{2}[\psi]}\frac{1}{\sqrt{\tau}}  \]
where $D_{0},D_{1},D_{2}$ are suitable weighted norms of the initial data of $\psi$. 

The decay rate for the scalar field $\psi$ can be improved if one uses the vector fields $L$ and $rL$ as commutator vector fields. This was first demonstrated in the work of Schlue \cite{volker1} for Schwarzschild spacetimes and then in the work of Moschidis \cite{moschidis1} for a very general class of asymptotically flat spacetimes with non-constant Bondi mass
  (for applications see also \cite{sergiualex,molog,moinsta,shiwu}). This method yields improved decay for the higher-order energy
\[\int_{{\Sigma}_{\tau}}\textbf{T}[T\psi](T, n_{{\Sigma}_{\tau}}) \leq C\cdot D_{3}[\psi]\frac{1}{\tau^{4}}  \]
 which can been turned into improved pointwise decay for the scalar field
\[|\psi|\leq C\sqrt{D_{4}[\psi]}\frac{1}{\tau^{3/2}} \]
where $D_{3},D_{4}$ are weighted norms of the initial data of $\psi$. This method is comparable to commuting with the scaling vector field $S$, since the commutator vector fields have equal weights in $r$. 

\subsection{Sketch of the new vector field approach}
\label{sketch}

In this paper, we introduce a new vector field approach which allows us to obtain almost-sharp decay rates for asymptotically flat backgrounds, going therefore beyond the decay rates of the traditional vector field method. Specifically, for any $\epsilon >0$, we obtain $\tau^{-5+\epsilon}$ decay rate for the global energy flux through the hypersurfaces $\Sigma_{\tau}$ (see Figure \ref{fig4}), $\tau^{-3+\epsilon}$ decay rate for the scalar field and $\tau^{-2+\epsilon}$ decay rate for the radiation field, provided that the initial data decay are sufficiently regular. In fact, we obtain a new hierarchy of estimates which provides a correspondence between decay in \emph{space} of the initial data and  decay in \emph{time} of the scalar field. 

The $\epsilon$ loss (for, say, compactly supported smooth initial data) is removed in our companion paper \cite{paper2} where in fact we derive the precise late-time asymptotics of the scalar field and hence, we obtain \textit{sharp upper and lower} bounds. The lower bounds have not been previously derived by either physical space or Fourier analytic methods. The hierarchy of estimates derived in the present paper plays a fundamental role in \cite{paper2}. 

In this paper we restrict to the class of spherically symmetric, stationary, asymptotically flat spacetimes, a class which includes the Reissner--Nordstr\"{o}m family of black hole backgrounds. This restriction is mostly for convenience and simplicity of the estimates. In a future work we  obtain improved decay estimates for a general class of asymptotically flat backgrounds.

We next present a brief schematic overview of the estimates and techniques. Our approach builds on the Dafermos--Rodnianski hierarchy \eqref{drhierarchy}. The idea is to obtain  improved decay directly for the energy flux\footnote{That is, to obtain energy decay rates faster than $\tau^{-2}$. } $\int_{\mathcal{N}_{\tau}}\textbf{T}[\psi](T, n_{\mathcal{N}_{\tau}})$, instead of deriving improved decay for higher-order fluxes as in the method of Moschidis and Schlue. We will in fact obtain decay for the $r$-weighted flux $\int_{N_{\tau_{1}}}r^{2}\cdot (L(r\psi))^2\ drd\omega$, which can then be translated into decay for the standard energy flux via \eqref{drhierarchy}. Ideally, one would extend the hierarchy \eqref{drhierarchy} for values $p>2$, however this is not possible. Instead, following Klainerman's commutator vector field method, \textit{we obtain an extended hierarchy in the ``far away''region $\left\{r\geq R \right\}$
 by using commutator vector fields of the order of $r^{2}\cdot L$.}

The Hardy inequality on the null pieces $\mathcal{N}_{\tau}$
\begin{equation}
 \int_{\mathcal{N}_{\tau}}f^{2}\, dvd\omega\leq C\cdot \int_{\mathcal{N}_{\tau}}\Big(\partial_{v}(rf)\Big)^{2}\, dvd\omega,
\label{hardy}
\end{equation}
which holds if $f=O(1/r)$ and $f=0$ at $r=R$, applied to 
\[f= r\partial_{v}(r\psi)\]
yields schematically (modulo cut-off terms near $r=R$ that can be bounded by the Morawetz estimate-assuming that an integrated local energy decay estimate holds)
\[ \int_{\mathcal{N}_{\tau}}\Big(r\cdot (\partial_{v}(r\psi) \Big)^{2}\, dvd\omega \leq 
C\int_{\mathcal{N}_{\tau}}\Big(\partial_{v}\big(r^{2}\cdot\partial_{v}(r\psi) \big)  \Big)^{2} \,dvd\omega  .\]
We therefore need to show that the integral 
\[\int_{\mathcal{N}_{\tau}}\Big(\partial_{v}\big(r^{2}\cdot\partial_{v}(r\psi) \big)  \Big)^{2} \,dvd\omega, \]
decays in time. This naturally suggests that we need to commute with the vector field $r^{2}\partial_{v}=r^{2}\cdot L$.

Let us, for convenience,  define the higher-order expressions
\[E^{(1)}_{p}(\tau)=\int_{\mathcal{N}_{\tau}}r^{p-1}\cdot\Big(\partial_{v}\big(r^{2}\cdot\partial_{v}(r\psi) \big)  \Big)^{2} \,drd\omega. \]
We want to prove decay for $E^{(1)}_{1}(\tau)$.  We would ideally want to establish a hierarchy that schematically looks like
\begin{equation}
 \int_{\tau_{1}}^{\tau_{2}} E^{(1)}_{p}[\psi](\tau) \, d\tau \leq C E^{(1)}_{p+1}[\psi](\tau_{1}),
\label{hier}
\end{equation}
for all $0< p \leq 3$. However, this is not possible. First, in Section \ref{sharpnessofhierarchy}, we show that for generic smooth compactly supported initial data we have 
\[ \int_{\tau_{1}}^{\infty} E^{(1)}_{3}[\psi](\tau)=\infty. \] 
Hence, we only hope that \eqref{hier} holds for all $p\in (0,3)$. This is closely related to the $\epsilon$ loss of our decay rates. 
Note that \textbf{if} \eqref{hier} holds for all $p\in (0,3)$ then
\begin{itemize}
\item 
$p=1$ yields decay with rate $\tau^{-3}$ for the energy flux of $\psi$, 
\item
$1\leq p\leq 2$ yields decay with rate $\tau^{-4}$ for the energy flux of $\psi$, 

\item 
$1\leq p\leq 3-\epsilon$ yields decay with rate $\tau^{-5+\epsilon}$ for the energy flux of $\psi$.
\end{itemize}
The main difficulty in proving the hierarchy \eqref{hier} arises from the need to decompose $\psi$ in the spherically symmetric part and the non-spherically symmetric part:
\begin{equation}
\psi =\psi_{0}+\psi_{ 1}= \left(\frac{1}{4\pi} \int_{\mathbb{S}^{2}}\psi\right) + \left(\psi- \frac{1}{\pi}\int_{\mathbb{S}^{2}}\psi\right). 
\label{sphdecintro}
\end{equation}
 The main obstructions appear for the term $\psi_{1}=\psi- \int_{\mathbb{S}^{2}}\psi$. 
 In fact, for this term we can only establish the hierarchy \eqref{hier} for $p\in (0,2]$ (in fact, $p\in (-4,2]$, see Theorem \ref{thm:hierdrpsi1} in Section \ref{sintro}). In order to derive further decay, we use the Hardy inequality \eqref{hardy} for $f=r\partial_{v}(r^{2}\cdot \partial_{v}(r\psi_{1}))$ allowing us to control $E^{(1)}_{3}[\psi_{1}](\tau)$ (which appears on the right hand side of \eqref{hier} for $p=2$) in terms of the following $r$-weighted third-order energy flux:
 \[E^{(1)}_{3}[\psi_{1}](\tau) \leq C \int_{\mathcal{N}_{\tau}}\Big(\partial_{v}\big(r^{2}\cdot\partial_{v}(r^{2}\cdot\partial_{v}(r\psi_{1})) \big)  \Big)^{2} \,dvd\omega. \]
 This suggests that we need to commute with $r^{2} \cdot L$ \emph{twice} and obtain a hierarchy of estimates that schematically looks like 
 \begin{equation}
 \int_{\tau_{1}}^{\tau_{2}} E^{(2)}_{p}[\psi_{1}](\tau) \, d\tau \leq C E^{(2)}_{p+1}[\psi_{1}](\tau_{1}),
\label{hier1}
\end{equation}
for all $-6< p <1$, where 
\[E^{(2)}_{p}[\psi_{1}](\tau)= \int_{\mathcal{N}_{\tau}}r^{p-1}\cdot\Big(\partial_{v}\big(r^{2}\cdot\partial_{v}\big(r^{2}\cdot\partial_{v}(r\psi) \big) \big) \Big)^{2} \,drd\omega.\]
 See Theorem \ref{thm:hierdrpsi2} in Section \ref{sintro}.  In establishing the hierarchies \eqref{hier} and \eqref{hier1} for $\psi_{1}$, one of the most problematic terms in the associated energy identities  is
a non-trivial flux term on null infinity which has the wrong sign. A crucial point is  that the coefficients of the various terms are such that the dangerous term can be absorbed by using appropriate Hardy and Poincar\'{e} inequalities. Note that the Poincar\'{e} inequality on the sphere is applicable since we have  removed the spherically symmetric part in $\psi_{1}$. Here the asymptotically flatness assumption and the structure of the wave equation plays a crucial role.  
 
We next return to the spherically symmetric part $\psi_{0}$. Recall that the Dafermos--Rodnianski hierarchy holds for $0<p\leq 2$ (which yields \eqref{drhierarchy}). However, one observes  (see Theorem \ref{thm:hierpsi01} in Section \ref{sintro}) that for \emph{spherically symmetric} scalar fields $\psi_{0}$ the hierarchy \eqref{drhierarchy} in fact holds for all $0<p<3$.   
 The range of $p$ can be further extended only if the Newman--Penrose constant (see Section \ref{sec:1stnpconstant}) along null infinity
\[{I}_{0}[\psi_{0}]=\lim_{v\rightarrow\infty}r^{2}\cdot\partial_{v}(r\psi_{0})\]
is zero. Indeed, as is shown in Section \ref{sharpnessofhierarchy}, if ${I}_{0}[\psi_{0}]\neq 0$ then 
\[\int_{\tau_{1}}^{\infty}\int_{N_{\bar{\tau}}}r^{2}\cdot \big(\partial_{v}(r\psi_{0})\big)^{2}\, drd\omega d\bar{\tau}=\infty,\]
which shows that the Dafermos--Rodnianski hierarchy does not hold in this case for $p=3$. If, on the other hand,  ${I}_{0}[\psi_{0}]=0$ then the hierarchy  can be extended for $0<p<5$ (Theorem \ref{thm:hierpsi02} in Section \ref{sintro}). We also remark that in the proof of the hierarchy for $4<p<5$ we need to use the hierarchy for $0<p\leq 4$ and the resulting improved energy and pointwise decay for the scalar field in order to extend the range of $p$ to the interval $(4,5)$.  

It is evident from the above that the conservation law associated to the  Newman--Penrose constant $I_{0}$ on null infinity\footnote{We refer the reader to \cite{aretakisglue} for more on conservation laws on characteristic hypersurfaces.} plays a fundamental role in our vector field approach for decay. 

The above extended hierarchies yield almost-sharp energy decay (see Theorems \ref{thm:endecaypsi} in Section \ref{sintro}) for both cases $I_{0}=0$ and $I_{0}\neq 0$. In particular, the decay estimates for the case $I_{0}\neq 0$ play a crucial role in our companion paper \cite{paper2}. The hierarchies and the associated decay estimates of the weighted norms $E^{(1)}_{p}[\psi](\tau),E^{(2)}_{p}[\psi](\tau)$ yield pointwise decay for both the scalar field (see Theorem \ref{thm:decaypsi}) and the radiation field (see Theorem \ref{thm:pdecayradfield}). The aforementioned hierarchies can be extended further for solutions of the form $T^{k}\psi$ (where $T$ is the stationary Killing vector field) which in turn yields further energy and pointwise decay for solutions of this form (See Theorems \ref{thm:edecayTkpsi}, \ref{thm:decayTkradfield} and \ref{thm:decayTkpsi} in Section \ref{sintro}). This improvement is useful for non-linear applications. 
 
\begin{remark}
Our commutator vector field is at the same level in powers of $r$ as the conformal Killing vector field $K$ for Minkowski space which is used as part of the conformal compactification method. Note, however, that the latter approach yields (in Minkowski space) energy identities with vanishing spacetime terms. On the other hand, our approach establishes energy identities where the spacetime terms is positive and control the weighted fluxes $E^{(1)}_{p}, E_{p}^{(2)}$. Furthermore, we do not need to assume that the spacetime admits a regular conformal compactification.  

\label{remark1intro}
\end{remark} 
 
 \begin{remark}
 The method is \textit{optimal  in terms of the regularity required for the initial data}. We only need to commute twice with the vector field $r^{2}\cdot L$ to obtain the full hierarchy in the ``far away'' region (see Theorem \ref{thm:hierdrpsi2} in Section \ref{sintro}). This implies (see Theorem \ref{thm:endecaypsi} in Section \ref{sintro}) that to get the full energy decay $\tau^{-5+\epsilon}$ requires the boundedness of a weighted \underline{fifth-order} norm of the initial data.  Furthermore, the method is \textit{optimal in terms of the decay required  for the initial data}. In particular, the assumed decay is consistent with the bounds
\[|r\psi|\leq C(u), |\partial_{u}(r\psi)|\leq C(u), |\partial_{v}(r\psi)|\leq \frac{C(u)}{r^{2}},  \]
where $(u,v)$ is the Eddington--Finkelstein coordinate system.  Note that optimizing in terms of the regularity and the decay of the initial data is of fundamental importance for applications to non-linear problems. 
 \label{remark2intro}
 \end{remark}

\subsection{Summary of the main results}
\label{sintro}

In this section, we will state the main theorems that are proved in this paper. We will employ the notation that is introduced in Section \ref{sec:prelim}. In Section \ref{sec:prelim} we will moreover introduce the precise spacetime backgrounds to which the theorems apply, which are equipped with metrics of the form
\begin{equation*}
g=-D(r)du^2-2drdu+r^2(d\theta^2+\sin^2\theta d\varphi^2),
\end{equation*}
in Bondi coordinates $(u,r,\theta,\varphi)$, where the function $D$ depends only on $r$ and satisfies the properties outlined in Section \ref{sec:prelim}. In Section \ref{sec:prelim} we will moreover introduce the alternative coordinate charts $(\tau,r,\theta,\varphi)$ and $(\widetilde{\tau},\rho,\theta,\varphi)$ used in the theorems below.

The vector field $T$ is the Killing vector field corresponding to time-translation. We will also consider the angular momentum operators $\Omega_i$, with $i=1,2,3$ (see Section \ref{sec:basicineq} for a definition) and
\begin{equation*}
\Omega^{k}=\Omega_1^{k_1}\Omega_2^{k_2}\Omega_3^{k_3},
\end{equation*}
where $k=(k_1,k_2,k_3)\in \N_0^3$. We will frequently use $\slashed{\nabla}$ to denote the restriction of the covariant derivative to the round spheres that foliate the spacetime.

We use $\psi$ to denote solutions to \eqref{waveequation} arising from the initial value problem in Theorem \ref{thm:extuniq}, and we denote with $\phi=r\psi$ the corresponding Friedlander radiation fields. See Section \ref{sec:basicineq} for details regarding the spherical decomposition in angular frequencies, and in particular, the decomposition
\[\psi=\psi_{0}+\psi_{1}.\]

The first Newman--Penrose constant is denoted by
\begin{equation*}
I_0[\psi]=\lim_{r \to \infty} \int_{\s^2}r^2\partial_r \phi|_{u=0}(r,\theta,\varphi)\,\sin \theta d\theta d\varphi.
\end{equation*}
See Section \ref{sec:1stnpconstant} for more details regarding the first Newman--Penrose constant.

All our theorems apply to the Schwarzschild metric and more generally to the sub-extremal Reissner--Nordstr\"{o}m metric. See \cite{redshift} for a proof of the integrated local energy decay estimate for such metrics. Our $r^{p}$-weighted hierarchies in the far away region apply to the extremal Reissner--Nordstr\"{o}m metric as well (see \cite{aretakis1} for the derivation of a degenerate Morawetz estimate on such backgrounds). 

\subsubsection{New hierarchies of $r^p$-weighted estimates}
\label{sec:summaryhierarchies}
A key result in this paper is the discovery of \emph{new} hierarchies of $r^p$-weighted estimates.

For spherically symmetric solutions $\psi$, we extend the hierarchy of Dafermos--Rodnianski (see Section \ref{sec:dafrodhier}) and we distinguish between the cases where the initial data satisfies $I_0[\psi] \neq 0$ and $I_0[\psi]=0$.

We obtain the following hierarchies of estimates:
\begin{theorem}[$r^p$-weighted estimates for $\psi_0$ with $I_0\neq 0$]
\label{thm:hierpsi01}
Let $\psi$ be a spherically symmetric solution to \eqref{waveequation} emanating from initial data given as in Theorem \ref{thm:extuniq} in the region $\mathcal{A}=\{r\geq R\}$ with $\boldsymbol{I_0 [\psi ] \neq 0}$, and take $p\in(0,3)$. Then there exists an $R>0$ such that for any $0\leq u_1<u_2$
\begin{equation}\label{en00est}
\begin{split}
\int_{{\mathcal{N}}_{u_2}} r^p(\partial_r\phi )^2\,d\omega dr+p \int_{\mathcal{A}_{u_1}^{u_2}} r^{p-1}(\partial_r\phi )^2\,d\omega drdu \leq&\: C\int_{{\mathcal{N}}_{u_1}} r^p(\partial_r\phi )^2\,d\omega dr\\
&+C \int_{\Sigma_{u_1}} J^T [\psi ] \cdot n_{u_1}\, d\mu_{{u_1}} ,
\end{split}
\end{equation}
where $C=C(D,R)>0$ is a constant.
\end{theorem}

\begin{theorem}[$r^p$-weighted estimates for $\psi_0$ with $I_0=0$]
\label{thm:hierpsi02}
Let $\psi$ be a spherically symmetric solution to \eqref{waveequation} emanating from initial data as in Theorem \ref{thm:extuniq} in the region $\mathcal{A}=\{r\geq R\}$ with $\boldsymbol{I_0 [\psi ] =0}$, and take $p\in(0,4)$. Then there exists an $R>0$ such that for any $0\leq u_1<u_2$
\begin{equation*}
\begin{split}
\int_{{\mathcal{N}}_{u_2}} &r^p(\partial_r\phi )^2 dr+p\int_{\mathcal{A}_{u_1}^{u_2}} r^{p-1}(\partial_r\phi )^2\,drdu \leq \frac{C}{(4-p)^2}\int_{{\mathcal{N}}_{u_1}} r^p(\partial_r\phi )^2 dr\\
&+ \frac{C}{(4-p)^2}\int_{\Sigma_{u_1}} J^T [\psi ] \cdot n_{{u_1}} d\mu_{{u_1}},
\end{split}
\end{equation*}
where $C=C(D,R)>0$.

Moreover, for $p\in [4,5)$ we have that
\begin{equation*}
\begin{split}
\int_{{\mathcal{N}}_{u_2}} &r^p(\partial_r\phi )^2 dr+p\int_{\mathcal{A}_{u_1}^{u_2}} r^{p-1}(\partial_r\phi )^2\,drdu \leq C\int_{{\mathcal{N}}_{u_1}} r^p(\partial_r\phi )^2 dr\\
&+ C\int_{\Sigma_{u_1}} J^T [\psi ] \cdot n_{{u_1}} d\mu_{{u_1}} + C  \frac{E_{0;\rm aux}^{\delta}[\psi]}{(1+u_1 )^{1 -{2\delta}}} , 
\end{split}
\end{equation*}
for any $\delta > 0$, with
\begin{equation*}
E_{0;\rm aux}^{\delta}[\psi]= \sum_{l\leq 4} \int_{\Sigma_0} J^N [T^l\psi ] \cdot n_0\, d\mu_{\Sigma_0}+ \sum_{l\leq 3}\int_{\mathcal{N}_0} r^{4-l-\delta} (\partial_r  T^l\phi  )^2 \, d\omega dr<\infty,
\end{equation*}
where $C=C(D,R,\delta )>0$ is a constant.
\end{theorem}

Theorem \ref{thm:hierpsi01} is proved in Section \ref{sec:hiernpnotzero} and Theorem \ref{thm:hierpsi02} is proved in Section \ref{sec:hiernpzero}.

For the remaining part of the solution $\psi_1=\psi-\psi_0$ we instead construct a new hierarchy for the variable
\begin{equation*}
\Phi \doteq r^2\partial_r\phi=r^2\partial_r(r\psi).
\end{equation*}
That is to say, we \emph{commute} once with $r^2\partial_r$.

We obtain the following hierarchy of estimates for $r^2\partial_r(r\psi_1)$.

\begin{theorem}[$r^p$-weighted estimates for $r^2\partial_r(r\psi_1)$]
\label{thm:hierdrpsi1}
Let $\psi $ be a solution to \eqref{waveequation} emanating from initial data given as in Theorem \ref{thm:extuniq} in the region $\mathcal{A}=\{r\geq R\}$. We assume that $\psi$ is supported on angular frequencies $\ell\geq 1$.

Take $p\in (-4,2]$ and assume that
\begin{equation*}
\sum_{|k|\leq 2}\int_{\Sigma}J^T[\Omega^k\psi]\cdot n_{\Sigma}\,d\mu_{\Sigma}<\infty,
\end{equation*}
and
\begin{align*}
\lim_{r \to \infty }\sum_{|k|\leq 2}\int_{\s^2}(\Omega^k\phi)^2\,d\omega\big|_{u'=0}<&\:\infty,\\
\lim_{r\to \infty} \int_{\s^2}\Phi^2\,d\omega\big|_{u'=0}<&\:\infty.
\end{align*}

Then there exists an $R>0$ such that for any $0\leq u_1<u_2$
\begin{equation*}
\begin{split}
\int_{\mathcal{N}_{u_2}}& r^p(\partial_r\Phi)^2\, d\omega dr+\int_{\mathcal{A}_{u_1}^{u_2}} (p+4)r^{p-1}(\partial_r\Phi)^2+(2-p)r^{p-1}|\snabla \Phi|^2\,d\omega drdu \\
\leq&\: C\int_{\mathcal{N}_{u_1}} r^p(\partial_r\Phi)^2\,d\omega dr+C\sum_{l\leq 1}\int_{\Sigma_{u_1}}J^T[T^{l}\psi]\cdot n_{{u_1}}\,d\mu_{\Sigma_{u_1}},
\end{split}
\end{equation*}
where $C\doteq C(D,R)>0$ is a constant.
\end{theorem}

By commuting once more with $r^2\partial_r$ and considering the variable
\begin{equation*}
\Phi_{(2)} \doteq r^2\partial_r\Phi=r^2\partial_r(r^2\partial_r(r\psi)),
\end{equation*}
we obtain and additional hierarchy of $r^p$-weighted estimates.

\begin{theorem}[$r^p$-weighted estimates for $(r^2\partial_r)^2(r\psi_1)$]
\label{thm:hierdrpsi2}
Let $\psi$ be a solution to \eqref{waveequation} emanating from initial data given as in Theorem \ref{thm:extuniq} in the region $\mathcal{A}=\{r\geq R\}$. We assume that $\psi$ is supported on angular frequencies $\ell\geq 1$.

Take $p\in(-6,1)$ and assume that
\begin{equation*}
\sum_{|k|\leq 4}\int_{\Sigma}J^T[\Omega^k\psi]\cdot n_{\Sigma}\,d\mu_{\Sigma}<\infty
\end{equation*}
and
\begin{align*}
\lim_{r \to \infty }\sum_{|k|\leq 4}\int_{\s^2}(\Omega^k\phi)^2\,d\omega\big|_{u'=0}<&\:\infty,\\
\lim_{r \to \infty }\sum_{|k|\leq 2}\int_{\s^2}(\Omega^k\Phi)^2\,d\omega\big|_{u'=0}<&\:\infty,\\
\lim_{r \to \infty }r^{-1}\int_{\s^2}\Phi_{(2)}^2\,d\omega\big|_{u'=0}<&\:\infty.
\end{align*}

Then there exists an $R>0$ such that for any $0\leq u_1<u_2$
\begin{equation*}
\begin{split}
\int_{\mathcal{N}_{u_2}}& r^p(\partial_r\Phi_{(2)})^2\, d\omega dr+\int_{\mathcal{A}_{u_1}^{u_2}} r^{p-1}(\partial_r\Phi_{(2)})^2+(2-p)r^{p-1}|\snabla \Phi_{(2)}|^2\,d\omega drdu\\
 \leq&\: C(p+6)^{-1}(p-1)^{-2}\int_{\mathcal{N}_{u_1}} r^p(\partial_r\Phi_{(2)})^2\,d\omega dr+C(p+6)^{-1}\sum_{k\leq 2}\int_{\Sigma_{u_1}}J^T[T^{k}\psi]\cdot n_{\Sigma_{u_1}},
\end{split}
\end{equation*}
where $C\doteq C(D,R)>0$ is a constant.
\end{theorem}

Theorem \ref{thm:hierdrpsi1} and Theorem \ref{thm:hierdrpsi2} are proved in Section \ref{sec:hierpsi1a}.
\subsubsection{Decay statements}
We apply the hierachies from Section \ref{sec:summaryhierarchies} to obtain almost-sharp energy decay estimates and pointwise decay estimates.

Let $\epsilon>0$. We introduce the following initial energy norms for the spherical mean $\psi_0$ on the hypersurface $\Sigma_0$.
\begin{align*}
E_{0 , I_0 \neq 0}^{\epsilon}[\psi] \doteq &\:  \sum_{l=0}^3 \int_{\Sigma_0} J^N [T^l \psi ] \cdot n_0\, d\mu_{\Sigma_0} + \int_{\mathcal{N}_0} r^{3-\epsilon} (\partial_r \phi )^2 \,   d\omega  dr + \int_{\mathcal{N}_0} r^2 (\partial_r (T \phi ) )^2 \, d\omega dr \\ \nonumber
&+ \int_{\mathcal{N}_0} r (\partial_r (T^2 \phi ) )^2 \,  d\omega dr,\\
E_{0, I_0=0}^{\epsilon}[\psi] \doteq &\: \sum_{l=0}^5 \int_{\Sigma_0} J^N [T^l\psi ] \cdot n_0\, d\mu_{\Sigma_0}+ \int_{\mathcal{N}_0} r^{5-\epsilon} (\partial_r \phi )^2 \,  d\omega dr  +  \int_{\mathcal{N}_0} r^{4-\epsilon} (\partial_r (T \phi ) )^2 \,  d\omega dr\\ \nonumber
&+ \int_{\mathcal{N}_0} r^{3-\epsilon} (\partial_r (T^2 \phi ) )^2 \,  d\omega dr + \int_{\mathcal{N}_0} r^2 (\partial_r (T^3 \phi ) )^2 \,  d\omega  dr + \int_{\mathcal{N}_0} r(\partial_r (T^4 \phi ) )^2 \,  d\omega  dr .
\end{align*}

For each $\epsilon>0$ we also introduce the following weighted initial energy norms for $\psi_1=\psi-\psi_0$:
\begin{equation*}
\begin{split}
E^{\epsilon}_{1}[\psi]\doteq &\: \sum_{l\leq 5}\int_{\Sigma_{0}}J^N[T^l\psi]\cdot n_{0}\,d\mu_{\Sigma_0}+\sum_{l\leq 3}\int_{\mathcal{N}_0}r^{2}(\partial_rT^l\phi)^2+r^{}(\partial_rT^{l+1}\phi)^2\,d\omega dr\\
&+ \int_{\mathcal{N}_0} r^{2-\epsilon}(\partial_r\Phi)^2 \, d\omega dr+\int_{\mathcal{N}_{0}}r^{2-\epsilon}(\partial_rT{\Phi})^2+r^{1-\epsilon}(\partial_rT^2{\Phi})^2\,d\omega dr+\int_{\mathcal{N}_{0}}r^{1-\epsilon}(\partial_r{\Phi}_{(2)})^2\,d\omega dr.
\end{split}
\end{equation*}

\begin{theorem}[Energy decay for $\psi$]\label{thm:endecaypsi}
Let $\psi$ be a solution to \eqref{waveequation} emanating from initial data given as in Theorem \ref{thm:extuniq} in the region $\mathcal{A}=\{r\geq R\}$. Assume moreover that
\begin{align*}
\lim_{r \to \infty }\sum_{|l|\leq 4}\int_{\s^2}(\Omega^l\phi)^2\,d\omega\big|_{u'=0}<&\:\infty,\\
\lim_{r \to \infty }\sum_{|l|\leq 2}\int_{\s^2}(\Omega^l\Phi)^2\,d\omega\big|_{u'=0}<&\:\infty,\\
\lim_{r \to \infty }\int_{\s^2}r^{-1}\left(\Phi_{(2)}\right)^2\,d\omega\big|_{u'=0}<&\:\infty.
\end{align*}
\begin{itemize}

\item[\emph{(i)}] Assume that initially we have that $\boldsymbol{I_0 [\psi ] \neq 0}$ and $E_{0 , I_0 \neq 0}^{\epsilon}[\psi_0]+E^{\epsilon}_{1}[\psi_1]<\infty$.

Then, for all $\epsilon>0$, there exists a constant $C \doteq C(D,R,\epsilon)$ such that for all $u\geq 0$
\begin{equation}\label{est:endec1}
\int_{\Sigma_u} J^N [\psi ] \cdot n_{u} d\mu_{\Sigma_{u}} \leq C \frac{E_{0, I_0 \neq 0}^{\epsilon}[\psi_0]+E^{\epsilon}_{1}[\psi_1]}{(1+u )^{3-\epsilon}} .
\end{equation}

\item[\emph{(ii)}] Assume that initially we have that $\boldsymbol{I_0 [\psi ] = 0}$ and $E_{0, I_0=0}^{\epsilon}[\psi_0]+E^{\epsilon}_{1}[\psi_1]<\infty$.

Then, for all $\epsilon>0$, there exists a constant $C \doteq C(D,R,\epsilon)$ such that for all $u\geq 0$
\begin{equation}\label{est:endec2}
\int_{\Sigma_u} J^N [\psi ] \cdot n_{u} d\mu_{\Sigma_{u}} \leq C \frac{E_{0, I_0=0}^{\epsilon}[\psi_0]+E^{\epsilon}_{1}[\psi_1]}{(1+u )^{5-\epsilon}}. 
\end{equation}
\end{itemize}
\end{theorem}
Theorem \ref{thm:endecaypsi} is proved in Section \ref{sec:energydecay}.

The above energy decay statements can be used to obtain pointwise decay statements for the radiation field $\phi$.
\begin{theorem}[$L^{\infty}$-decay of $r\psi$]
\label{thm:pdecayradfield}
Let $\psi$ be a solution to \eqref{waveequation} emanating from initial data given as in Theorem \ref{thm:extuniq} in the region $\mathcal{A}=\{r\geq R\}$. Assume moreover that
\begin{align*}
\lim_{r \to \infty }\sum_{|l|\leq 6}\int_{\s^2}(\Omega^l\phi)^2\,d\omega\big|_{u'=0}<&\:\infty,\\
\lim_{r \to \infty }\sum_{|l|\leq 4}\int_{\s^2}(\Omega^l\Phi)^2\,d\omega\big|_{u'=0}<&\:\infty,\\
\lim_{r \to \infty }\sum_{|l|\leq 2}\int_{\s^2}r^{-1}\left(\Omega^l\Phi_{(2)}\right)^2\,d\omega\big|_{u'=0}<&\:\infty.
\end{align*}

Assume further that either $E^{\epsilon}_{0,I_0\neq0}[\psi_0]<\infty$, or $E^{\epsilon}_{0,I_0=0}[\psi_0]<\infty$, and also that $\sum_{|l|\leq 2}E_{1}^{\epsilon}[\Omega^l\psi_1]<\infty$. 
Then, for all $\epsilon>0$ and for $R>0$ suitably large there exists a constant $C=C(D,R,\epsilon)>0$ such that for all $\widetilde{\tau}\geq 0$ 
\begin{align*}
|r\psi|(\widetilde{\tau},\rho,\theta,\varphi)\leq C\sqrt{E^{\epsilon}_{0,I_0\neq0}[\psi_0]+\sum_{|\alpha|\leq 2}E^{\epsilon}_{1}[\Omega^{\alpha}\psi_1]}(1+\widetilde{\tau})^{-1+\epsilon} \quad \textnormal{if}\quad I_0[\psi]\neq 0,\\
|r\psi|(\widetilde{\tau},\rho,\theta,\varphi)\leq C\sqrt{E^{\epsilon}_{0,I_0=0}[\psi_0]+\sum_{|\alpha|\leq 2}E^{\epsilon}_{1}[\Omega^{\alpha}\psi_1]}(1+\widetilde{\tau})^{-2+\epsilon} \quad \textnormal{if}\quad I_0[\psi]=0.
\end{align*}
\end{theorem}

Theorem \ref{thm:pdecayradfield} is proved in Section \ref{sec:pointdecayradfield}.

In order to prove almost-sharp pointwise decay for $\psi$ itself, we will need to assume that $D'(r_+)=0$ on top of the assumptions in Section \ref{sec:prelim}, and we will need to introduce additional higher-order energies.

Let $\epsilon>0$ and $k\in \N_0$. We introduce the following higher-order initial energy norms for the spherical mean $\psi_0$ on the hypersurface $\Sigma_0$.
\begin{align}
E_{0, I_0\neq 0; k}^{\epsilon}[\psi] \doteq &\: \sum_{l\leq 3+3k}\int_{\Sigma_{0}}J^N[T^l\psi]\cdot n_{0}\,d\mu_{\Sigma_0}\\ \nonumber
&+\sum_{l\leq 2k}\int_{\mathcal{N}_{0}} r^{3-\epsilon}(\partial_rT^l\phi)^2\,  d\omega dr+r^{2}(\partial_rT^{l+1}\phi)^2+r(\partial_rT^{2+l}\phi)^2\,  d\omega dr\\ \nonumber
&+\sum_{\substack{m\leq k\\ l\leq 2k-2m+\min\{k,1\}}} \int_{\mathcal{N}_{0}}r^{2+2m-\epsilon}(\partial_r^{1+m}T^{l}\phi)^2\,   d\omega dr\\\nonumber
&+\int_{\mathcal{N}_{0}}r^{3+2k-\epsilon}(\partial_r^{1+k}\phi)^2\,  d\omega  dr,\\
E_{0, I_0= 0; k}^{\epsilon}[\psi] \doteq &\: \sum_{l\leq 5+3k}\int_{\Sigma_{0}}J^N[T^l\psi]\cdot n_{0}\,d\mu_{\Sigma_0}\\ \nonumber
&+\sum_{l\leq 2k}\int_{\mathcal{N}_{0}} r^{5-\epsilon}(\partial_rT^l\phi)^2+r^{4-\epsilon}(\partial_rT^{1+l}\phi)^2+r^{3-\epsilon}(\partial_rT^{2+l}\phi)^2\,  d\omega dr\\\nonumber
&+r^{2}(\partial_rT^{3+l}\phi)^2+r(\partial_rT^{4+l}\phi)^2\,   d\omega dr\\\nonumber
&+\sum_{\substack{m\leq k\\ l\leq 2k-2m+\min\{k,1\}}} \int_{\mathcal{N}_{0}}r^{4+2m-\epsilon}(\partial_r^{1+m}T^{l}\phi)^2\,   d\omega dr\\\nonumber
&+\int_{\mathcal{N}_{0}}r^{5+2k-\epsilon}(\partial_r^{1+k}\phi)^2\,   d\omega dr.\nonumber
\end{align}
For each $\epsilon>0$ and $k\in \N_0$ we also introduce the following higher-order weighted initial energy norms for $\psi_1$:
\begin{equation}
\begin{split}
E_{1;k}^{\epsilon}[\psi]\doteq &\:\sum_{\substack{|\alpha|\leq k\\ l+|\alpha|\leq 5+3k}}\int_{\Sigma_{0}}J^N[T^l\Omega^{\alpha}\psi]\cdot n_{0}\;d\mu_{\Sigma_0}\\
&+\sum_{ l\leq 3+2k}\int_{\mathcal{N}_0}r^{2}(\partial_rT^l\phi)^2+r^{}(\partial_rT^{1+l}\phi)^2\,d\omega dr\\
&+\sum_{ l\leq 2k+1}\int_{\mathcal{N}_{0}} r^{2-\epsilon}(\partial_rT^{l}{\Phi})^2+r^{1-\epsilon}(\partial_rT^{l+1}{\Phi})^2\,d\omega dr\\
&+\sum_{\substack{|\alpha|\leq k\\l+|\alpha|\leq 2k}}\int_{\mathcal{N}_{0}} r^{1-\epsilon}(\partial_rT^{l}\Omega^{\alpha}{\Phi}_{(2)})^2\,d\omega dr\\
&+\sum_{\substack{|\alpha|\leq \max\{0,k-1\}\\ m\leq \max\{k-1,0\}\\ l+|\alpha|\leq k-2m+\min\{k,1\}}}\int_{\mathcal{N}_{0}}r^{1+2m-\epsilon}(\partial_r^{1+m}\Omega^{\alpha}T^{l}{\Phi}_{(2)})^2\,d\omega dr\\
&+\sum_{\substack{|\alpha|\leq \max\{0,k-1\}, m\leq k\\ l+|\alpha|\leq 2k-2m+1}}\int_{\mathcal{N}_{0}} r^{2m-\epsilon}(\partial_r^{1+m}\Omega^{\alpha}T^{l}{\Phi}_{(2)})^2\,d\omega dr\\
&+\int_{\mathcal{N}_{0}} r^{1+2k-\epsilon}(\partial_r^{1+k}{\Phi}_{(2)})^2\,d\omega dr.
\end{split}
\end{equation}

\begin{theorem}[$L^{\infty}$-decay of $\psi$]
\label{thm:decaypsi}
Let $\psi$ be a solution to \eqref{waveequation} emanating from initial data given as in Theorem \ref{thm:extuniq} in the region $\mathcal{A}=\{r\geq R\}$.

Assume moreover that 
\begin{align*}
\lim_{r \to \infty }\sum_{|l|\leq 8}\int_{\s^2}(\Omega^l\phi)^2\,d\omega\big|_{u'=0}<&\:\infty,\\
\lim_{r \to \infty }\sum_{|l|\leq 6}\int_{\s^2}(\Omega^l\Phi)^2\,d\omega\big|_{u'=0}<&\:\infty,\\
\lim_{r \to \infty }\sum_{|l|\leq 4}\int_{\s^2}r^{-1}\left(\Omega^l\Phi_{(2)}\right)^2\,d\omega\big|_{u'=0}<&\:\infty,
\end{align*}
and
\begin{equation*}
\lim_{r \to \infty }\sum_{|l|\leq 2}\int_{\s^2}r^{3}\left(\partial_r\Omega^l\Phi_{(2)}\right)^2\,d\omega\big|_{u'=0}<\infty.
\end{equation*}

Assume further that $E^{\epsilon}_{0,I_0\neq0;1}[\psi_0]<\infty$, or $E^{\epsilon}_{0,I_0=0;1}[\psi_0]<\infty$, and also that $\sum_{|l|\leq 2}E_{1;1}^{\epsilon}[\Omega^{l}\psi_1]<\infty$.

Then, for all $\epsilon>0$ there exists a constant $C=C(D,R,\epsilon)>0$ such that for all $\widetilde{\tau}\geq 0$
\begin{align*}
|\psi|(\widetilde{\tau},\rho,\theta,\varphi)\leq C (1+\widetilde{\tau})^{-2+\epsilon}\left[\sqrt{E^{\epsilon}_{0,I_0\neq0;1}[\psi_0]}+\sum_{ |\alpha|\leq 2}\sqrt{E^{\epsilon}_{1;1}[\Omega^{\alpha}\psi_1]}\right] \quad \textnormal{if}\quad I_0[\psi]\neq 0,\\
|\psi|(\widetilde{\tau},\rho,\theta,\varphi)\leq C (1+\widetilde{\tau})^{-3+\epsilon}\left[\sqrt{E^{\epsilon}_{0,I_0=0;1}[\psi_0]}+\sum_{|\alpha|\leq 2}\sqrt{E^{\epsilon}_{1;1}[\Omega^{\alpha}\psi_1]}\right] \quad \textnormal{if}\quad I_0[\psi]= 0.
\end{align*}
\end{theorem}

Theorem \ref{thm:decaypsi} is proved in Section \ref{prop:pointdecpsi}.

\subsubsection{Higher-order pointwise decay statements}
The $r^p$-weighted hierarchies introduced in this paper can be extended (see Section \ref{sec:rphierTkpsi1} and \ref{sec:rphierTkpsi0}) to obtain also the almost-sharp energy and pointwise decay statements for solutions of the form $T^k\psi$, with $k\in \N$.

\begin{theorem}[{Energy decay for $T^k \psi$}]\label{thm:edecayTkpsi}
Let $\psi$ be a solution to \eqref{waveequation} emanating from initial data given as in Theorem \ref{thm:extuniq} in the region $\mathcal{A}=\{r\geq R\}$. Let $n\in \N$ and assume additionally that $D(r)=1-2Mr^{-1}+O_{n+2}(r^{-1-\beta})$ for some $\beta>0$.

Assume further that 
\begin{align*}
\lim_{r \to \infty }\sum_{|l|\leq 4+2n}\int_{\s^2}(\Omega^l\phi)^2\,d\omega\big|_{u'=0}<&\:\infty,\\
\lim_{r \to \infty }\sum_{|l|\leq 2+2n}\int_{\s^2}(\Omega^l\Phi)^2\,d\omega\big|_{u'=0}<&\:\infty,\\
\lim_{r \to \infty }\sum_{|l|\leq 2n}\int_{\s^2}r^{-1}\left(\Omega^l\Phi_{(2)}\right)^2\,d\omega\big|_{u'=0}<&\:\infty,
\end{align*}
and
\begin{equation*}
\lim_{r \to \infty }\sum_{|l|\leq 2n-2s}\int_{\s^2}r^{2s+1}\left(\partial_r^s\Omega^l\Phi_{(2)}\right)^2\,d\omega\big|_{u'=0}<\infty,
\end{equation*}
for each $1\leq s\leq n$.

Then the following statements hold:
\begin{itemize}
\item[\emph{(i)}] Assume that initially we have that $\boldsymbol{I_0 [\psi ] \neq 0}$ and $E_{0, I_0\neq 0; k}^{\epsilon}[\psi_0]+E_{1;k}^{\epsilon}[\psi_1]<\infty$,
Then, for all $\epsilon>0$ and all $k\leq n$, there exists a constant $C \doteq C(D,R,\epsilon,n)$ such that for all $u\geq 0$
\begin{equation*}
\int_{\Sigma_u} J^N [T^k \psi ] \cdot n_u d\mu_{\Sigma_u} \leq C \frac{E_{0, I_0\neq 0; k}^{\epsilon}[\psi_0]+E_{1;k}^{\epsilon}[\psi_1]}{(1+u )^{2k+3-\epsilon}} . 
\end{equation*}

\item[\emph{(ii)}]Assume that initially we have that $\boldsymbol{I_0 [\psi ] = 0}$ and $E_{0, I_0= 0; k}^{\epsilon}[\psi_0]+E_{1;k}^{\epsilon}[\psi_1]<\infty$.

Then, for all $\epsilon>0$, there exists a constant $C \doteq C(D,R,\epsilon)$ such that for all $u\geq 0$
\begin{equation*}
\int_{\Sigma_u} J^N [T^k \psi ] \cdot n_u d\mu_{\Sigma_u} \leq C \frac{E_{0, I_0= 0; k}^{\epsilon}[\psi_0]+E_{1;k}^{\epsilon}[\psi_1]}{(1+u )^{2k+5-\epsilon}} . 
\end{equation*}
\end{itemize}
\end{theorem}

Theorem \ref{thm:edecayTkpsi} is proved in Section \ref{sec:energydecay}.

We can use the above higher-order energy decay statements to obtain also pointwise decay statements for $T$-derivatives of the radiation field $\phi$:
\begin{theorem}[$L^{\infty}$-decay of $rT^k\psi$]
\label{thm:decayTkradfield}
Let $\psi$ be a solution to \eqref{waveequation} emanating from initial data given as in Theorem \ref{thm:extuniq} in the region $\mathcal{A}=\{r\geq R\}$. Assume that $D(r)=1-2Mr^{-1}+O_{n+2}(r^{-1-\beta})$ for some $n\in \N$ and $\beta>0$. Assume moreover that
\begin{align*}
\lim_{r \to \infty }\sum_{|l|\leq 6+2n}\int_{\s^2}(\Omega^l\phi)^2\,d\omega\big|_{u'=0}<&\:\infty,\\
\lim_{r \to \infty }\sum_{|l|\leq 4+2n}\int_{\s^2}(\Omega^l\Phi)^2\,d\omega\big|_{u'=0}<&\:\infty,\\
\lim_{r \to \infty }\sum_{|l|\leq 2+2n}\int_{\s^2}r^{-1}\left(\Omega^l\Phi_{(2)}\right)^2\,d\omega\big|_{u'=0}<&\:\infty,
\end{align*}
and
\begin{equation*}
\lim_{r \to \infty }\sum_{|l|\leq 2+2n-2s}\int_{\s^2}r^{2s+1}\left(\partial_r^s\Omega^l\Phi_{(2)}\right)^2\,d\omega\big|_{u'=0}<\infty,
\end{equation*}
for each $1\leq s\leq n$. Theorem \ref{thm:decayTkradfield} is proved in Section \ref{sec:pointdecTkrpsi}.

Assume further that either $E^{\epsilon}_{0,I_0\neq0;n}[\psi_0]<\infty$, or $E^{\epsilon}_{0,I_0=0;n}[\psi_0]<\infty$, and also that $\sum_{|l|\leq 2}E^{\epsilon}_{1;n}[\Omega^l\psi_1]<\infty$.

Then, for all $\epsilon>0$ and for $R>0$ suitably large there exists a constant $C=C(D,R,\epsilon,n)>0$ such that for all $k\leq n$ and $\widetilde{\tau}\geq 0$ 
\begin{align*}
|rT^k\psi|(\widetilde{\tau},\rho,\theta,\varphi)\leq C\sqrt{E^{\epsilon}_{0,I_0\neq0;k}[\psi_0]+\sum_{|\alpha|\leq 2}E^{\epsilon}_{1;k}[\Omega^{\alpha}\psi_1]}\widetilde{\tau}^{-1-k+\epsilon} \quad \textnormal{if}\quad I_0[\psi]\neq 0,\\
|rT^k\psi|(\widetilde{\tau},\rho,\theta,\varphi)\leq C\sqrt{E^{\epsilon}_{0,I_0=0;k}[\psi_0]+\sum_{|\alpha|\leq 2}E^{\epsilon}_{1;k}[\Omega^{\alpha}\psi_1]}\widetilde{\tau}^{-2-k+\epsilon} \quad \textnormal{if}\quad I_0[\psi]= 0.
\end{align*}
\end{theorem}

Finally, we also obtain higher-order pointwise decay statements for $\psi$, again.

\begin{theorem}[$L^{\infty}$-decay of $T^k\psi$]
\label{thm:decayTkpsi}
Let $\psi$ be a solution to \eqref{waveequation} emanating from initial data given as in Theorem \ref{thm:extuniq} in the region $\mathcal{A}=\{r\geq R\}$. Assume that $D(r)=1-2Mr^{-1}+O_{n+2}(r^{-1-\beta})$ for some $n\in \N$ and $\beta>0$. Assume moreover that
\begin{align*}
\lim_{r \to \infty }\sum_{|l|\leq 8+2n}\int_{\s^2}(\Omega^l\phi)^2\,d\omega\big|_{u'=0}<&\:\infty,\\
\lim_{r \to \infty }\sum_{|l|\leq 6+2n}\int_{\s^2}(\Omega^l\Phi)^2\,d\omega\big|_{u'=0}<&\:\infty,\\
\lim_{r \to \infty }\sum_{|l|\leq 4+2n}\int_{\s^2}r^{-1}\left(\Omega^l\Phi_{(2)}\right)^2\,d\omega\big|_{u'=0}<&\:\infty,
\end{align*}
and
\begin{equation*}
\lim_{r \to \infty }\sum_{|l|\leq 2+2n-2s}\int_{\s^2}r^{2s+1}\left(\partial_r^s\Omega^l\Phi_{(2)}\right)^2\,d\omega\big|_{u'=0}<\infty,
\end{equation*}
for each $1\leq s\leq n$.

Assume further that either $E^{\epsilon}_{0,I_0\neq0;n+1}[\psi_0]<\infty$, or $E^{\epsilon}_{0,I_0=0;n+1}[\psi_0]<\infty$, and also that $\sum_{|l|\leq 2}E^{\epsilon}_{1;n+1}[\Omega^l\psi_1]<\infty$. 

For all $\epsilon>0$ there exists a constant $C=C(D,R,\epsilon,n)>0$ such that for all $0\leq k\leq n$ and $\widetilde{\tau}\geq 0$
\begin{align*}
|T^k\psi|(\widetilde{\tau},\rho,\theta,\varphi)\leq C\sqrt{E^{\epsilon}_{0,I_0\neq0;k+1}[\psi_0]+\sum_{|\alpha|\leq 2}E^{\epsilon}_{1;k+1}[\Omega^{\alpha}\psi_1]}\widetilde{\tau}^{-2-k+\epsilon} \quad \textnormal{if}\quad I_0[\psi]\neq 0,\\
|T^k\psi|(\widetilde{\tau},\rho,\theta,\varphi)\leq C\sqrt{E^{\epsilon}_{0,I_0=0;k+1}[\psi_0]+\sum_{|\alpha|\leq 2}E^{\epsilon}_{1;k+1}[\Omega^{\alpha}\psi_1]}\widetilde{\tau}^{-3-k+\epsilon} \quad \textnormal{if}\quad I_0[\psi]= 0.
\end{align*}
\end{theorem}

Theorem \ref{thm:decayTkpsi} is proved in Section \ref{sec:pointdecTkpsi}.

\subsection{Applications}
\label{introappli}

We next present applications of our method in various topics in general relativity.

\subsubsection{General asymptotically flat spacetimes}
\label{bhiintro1}
The method can be applied to a general class of asymptotically flat spacetimes without any symmetries assumptions (including, of course, the full Kerr--Newman family of black holes). One still needs to use the decomposition \eqref{sphdecintro}. However, the main observation is that the \textit{coupled} weighted estimates for $\psi_{0}$ and $\psi_{1}$ close and all error terms can be controlled leading to hierarchies similar to those presented in this paper.  This will be demonstrated in an upcoming work. 

\subsubsection{Late time asymptotics of scalar fields}
\label{bhiintroasymptotics}

The estimates and techniques presented in this paper play a fundamental role in our companion paper \cite{paper2} where the exact late time asymptotics of general solutions to the wave equation on spherically symmetric backgrounds are derived. These asymptotics yield in particular sharp upper and lower bounds for the scalar fields.  Our work provides a first rigorous proof of Price's heuristics  (see \cite{price1972}) regarding the $\tau^{-3}$ and $\tau^{-2}$ power-laws for the asymptotic lower tail bounds for scalar and radiation fields, respectively,  arising from smooth compactly supported initial data on sub-extremal backgrounds. 

In a future work, we will investigate the relevance of our method to the study of late time asymptotics for the Teukolsky equation, and, more generally, the linearized Einstein equations.

\subsubsection{Non-linear applications}
\label{bhiintrononlinear}
 Due to the extensive range of applications of the vector field method to non-linear wave equations, we expect that the new approach for establishing improved decay rates will be useful for studying non-linear wave equations and, in particular, the Einstein equations. 
In an upcoming work we show that the new approach to decay can be used to yield improved (in fact sharp) decay rates for non-linear wave equations satisfying the null condition (in particular, our method applies for the wave map problem). See also the relevant Remark \ref{remark1intro}. We expect that this method will be relevant to the non-linear stability problem of the Kerr family.

\subsubsection{Interior of black holes and strong cosmic censorship}
\label{bhiintro}

A precise quantitative understanding of the decay behaviour of solutions to (\ref{waveequation}) in the \emph{exterior} of black hole spacetimes is central to understanding the extendibility of solutions beyond the Cauchy horizon in the black hole \emph{interior}.  Boundedness and blow-up statements for solutions to (\ref{waveequation}) in black hole interiors rely heavily on upper and lower bounds for solutions along the event horizon. See \cite{MD03,MD05c, MD12, Luk2015, LukSbierski2016, DafShl2016, Hintz2015, Franzen2014, Luk2016a, Luk2016b} for works in the interior of sub-extremal black holes. 

In the interior of extremal black hole spacetimes, the regularity of solutions to (\ref{waveequation}) at the inner horizon depends delicately on the \emph{precise} asymptotics of the solutions along the (outer) event horizon. In particular, in \cite{gajic} the third author showed that $C^1$ extendibility of spherically symmetric solutions in the interior of extremal Reissner--Nordstr\"om requires \emph{sharp} decay rates along the event horizon. Analogous results have also been obtained in the interior of extremal Kerr--Newman black holes \cite{gajic2}.
 
In the cosmological setting, in sub-extremal Reissner--Nordstr\"om-de Sitter spacetimes, the sharp decay rate of solutions to (\ref{waveequation}) is expected to be \emph{exponential}. Interestingly, also in this case, the precise exponents in the decay rates of solutions along the event horizon play an important role in the regularity properties of solutions at the Cauchy horizon; see \cite{Costa2015, Costa2014, Costa2014a, Hintz2015a, CostFranz2016}.

\subsubsection{Exterior of extremal black holes}
\label{ebhintro}
The quantitative decay behaviour of solutions to (\ref{waveequation}) on extremal black holes differs dramatically from the decay behaviour on \emph{sub}-extremal black holes. A mathematical study of the wave equation on extremal Reissner--Nordstr\"om and extremal Kerr was initiated by the second author in \cite{aretakis1, aretakis2,aretakis3, aretakis4, aretakis2012, aretakis2013} establishing in particular the existence of conserved quantities along the event horizon that form an obstruction to decay estimates.  
 It was shown that due to the existence of conserved constants, transversal derivatives of solutions along the event horizon generically \emph{do not} decay and higher-order transversal derivatives \emph{blow up} asymptotically in proper time.

Furthermore, the existence of conserved quantities along the event horizon of extremal Reissner--Nordstr\"om is related to a type of ``stable'' trapping of null geodesics along the event horizon, which forms an obstruction to integrated decay estimates (cf. in contrast, the trapping of null geodesics at the photon sphere in (sub)-extremal Reissner--Nordstr\"om can be considered ``unstable''); see \cite{aag1} for more information.

Subsequent work \cite{hm2012,ori2013,Sela2015} in the physics literature provided a modified weaker power-law of solutions to (\ref{waveequation}) on extremal Reissner--Nordstr\"om. For work on tails on extremal Kerr we refer  to the very interesting recent works \cite{zimmerman1, zimmerman2}.

The new approach applies to extremal black hole backgrounds, in contrast to previous physical space or Fourier analytic techniques for decay for the wave equation. Indeed, in an upcoming work \cite{aag7} we derive the late time asymptotics of scalar fields on extremal Reissner--Nordstr\"om and establish the precise influence of conserved quantities along the event horizon and null infinity on the decay rate.

Furthermore, improved decay rates for scalar fields on extremal Reissner--Nordstr\"om obtained by the new approach are an essential ingredient even for showing global well-posedness for non-linear wave equations satisfying the null condition on such backgrounds. This is accomplished in an upcoming work. See also \cite{yannis1,rizad}.

\subsubsection{Higher-order limits on null infinity}
\label{horintro}

The traditional vector field method yields (see \cite{moschidis1}) bounds of the Friedlander radiation field $\phi$ on null infinity for general asymptotically flat spacetimes. In a future work we will show that the new approach can be used to derive limiting bounds of higher-order derivatives
\[r^{2}L(\phi)\] 
and more generally
\[ \big(r^{2}L\big)^{k}(\phi),\]
for $k\geq 1$, on null infinity. Note that this implies that \textit{the higher-order Newman--Penrose constants associated to higher spherical harmonic parameters are well-defined on general asymptotically flat spacetimes}.

\subsection{Outline }

In Section \ref{sec:prelim} we present the assumptions for the relevant spacetimes. In Section \ref{sec:limitsnullinfty} we present basic properties of the Newman--Penrose constants and other important limiting quantities on null infinity. The hierarchy of $r$-weighted estimates for $\psi_{1}=\psi-\int_{\mathbb{S}^{2}}\psi$ is derived in Section \ref{sec:hierpsi1} and for the spherically symmetric mean in Section \ref{sec:hierpsi0}. In section \ref{sharpnessofhierarchy} we show that the range of our weighted hierarchies is sharp, that is they cannot be further extended for solutions arising from generic smooth compactly supported initial data (on spacetimes with strictly positive mass $M>0$). Finally, in Section \ref{sec:energydecay} and \ref{sec:pointwise} we obtain almost-sharp energy and pointwise decay rates for the scalar field, the radiation field and their $T^{k}$ derivatives for all $k\geq 1$.

\subsection{Acknowledgements} 

We would like to thank Mihalis Dafermos and Georgios Moschidis for several insightful discussions. The second author (S.A.) acknowledges support through NSF grant DMS-1600643 and a Sloan Research Fellowship. The third author (D.G.) acknowledges support by the European Research Council grant no. ERC-2011-StG 279363-HiDGR.

\section{Preliminaries}
\label{sec:prelim}
We next introduce the class of spacetimes to which our theorems apply moreover set the notation for the remaining sections of the paper.

\subsection{Assumptions on the geometry of the near infinity region $r\geq R$}
\label{sec:geomassm}
Let $R>0$ and consider the manifold-with-boundary $\mathcal{A'}=\R\times [R,\infty)\times \s^2$, equipped with a Lorentzian metric
\begin{equation}
\label{metricurcoords}
g=-D(r)du^2-2dudr+r^2(d\theta^2+\sin^2\theta d\varphi),
\end{equation}
where $u\in \R$, $r\in [R,\infty)$, $\theta\in (0,\pi)$, $\varphi\in (0,2\pi)$ and $D:[R,\infty)\to \R$ is a smooth function. 

Let us refer to the $(u,r,\theta,\varphi)$ coordinates as \emph{Bondi coordinates}; modulo standard degenerations of the $(\theta,\varphi)$ coordinates on $\s^2$, they cover $\mathcal{A'}$ globally. We refer to $\mathcal{A'}$ as the far-away region.

We will assume the following asymptotics for the function $D$:
\begin{equation}
\label{ass:quantitativedecD}
D(r)=1-\frac{2M}{r}+O_{3}(r^{-1-\beta}),
\end{equation}
for $M\in \R$, such that $M\geq 0$, and $\beta>0$. Here, we have applied ``big O'' notation, i.e. the term $O_k(r^{-l})$ consists of functions $f:[R,\infty)\to \R$ that satisfy the following property: for all $0\leq j \leq k$, there exist uniform constants $C_j>0$, such that
\begin{equation*}
\left|\frac{d^jf}{dr^j}\right|\leq C_j r^{-l-j}.
\end{equation*}
The asymptotics in \eqref{ass:quantitativedecD} are consistent with the assumption that $\mathcal{A}'$ can be foliated by asymptotically flat hypersurfaces, such that $M$ is the ADM mass corresponding to these hypersurfaces.  By assumption, the vector field $T\doteq \partial_u$ is a timelike Killing vector field.

We can define a smooth function $v:\mathcal{A'}\to \R$ such that $v=u+2r_*$, where $r_*:\mathcal{A}'\to \R$ is defined as follows:
\begin{align*}
r_*(R)=&\:R,\\
\frac{dr}{dr_*}=&\:D(r).
\end{align*}

In the coordinate chart $(v,r,\theta,\varphi)$ the metric can be expressed as follows:
\begin{equation}
\label{metricvrcoords}
g=-D(r)dv^2+2dvdr+r^2(d\theta^2+\sin^2\theta d\varphi).
\end{equation}
These coordinates are called \emph{ingoing Eddington--Finkelstein} coordinates.

\subsection{Assumptions on the global  geometry }
\label{sec:asmMsetminusA}
Let us now define $\mathcal{M}$, the manifolds(-with-boundary) of interest. We will consider two cases.

As a first case, let us extend $D$ to a function $D:(r_+,\infty)\to \R$, with $0<r_+<R$, such that $D(r)>0$ for $r\in (r_+,\infty)$ and $D(r_+)=0$. We define the manifold-with-boundary $\mathcal{M}_+$: $$\mathcal{M}_+=\R \times [r_+,\infty)\times   \s^2,$$ such that $\mathcal{M}_+$ is covered by the coordinate chart $(v,r,\theta,\varphi)$, with $v\in \R$, $r\in [r_+\infty)$, $\theta\in (0,\pi)$, $\varphi\in (0,2\pi)$. We equip $\mathcal{M}_+$ with the metric $g$ given by the expression \eqref{metricvrcoords}. The boundary $\mathcal{H}^+=\{(v,r,\theta,\varphi)\,:\,r=r_+\}$ is a null hypersurface, which from now on we will call the \emph{future event horizon} of the spacetime.

From the assumption on $D$ above, we can moreover define the manifold-with-boundary $\mathcal{M}_-$ by $$\mathcal{M}_-= \R \times [r_+,\infty) \times \s^2,$$ such that $\mathcal{M}_-$ is covered by the coordinate chart $(u,r,\theta,\varphi)$, with $u\in \R$, $r\in [r_+\infty)$, $\theta\in (0,\pi)$, $\varphi\in (0,2\pi)$. We equip $\mathcal{M}_-$ with the metric $g$ given by the expression \eqref{metricurcoords}. The boundary $\mathcal{H}^-=\{(u,r,\theta,\varphi)\,:\,r=r_+\}$ is a null hypersurface. We will refer to $\mathcal{H}^-$ as the \emph{past event horizon} of the spacetime.

We will denote $$\mathcal{M}=\mathcal{M}_+\cup \mathcal{M}_-=\mathcal{M}_+\cup \mathcal{H}^-=\mathcal{M}_-\cup \mathcal{H}^-.$$ See Figure \ref{fig:fullfoliations2} for the corresponding Penrose diagram. The minimum value of $r$ on $\mathcal{M}$ is denoted as $r_{\rm min}$, so $r_{\rm min}=r_+$ in this case. In this paper, we will only be dealing with the extension $\mathcal{M}_+$.

The domains of outer communication of Reissner--Nordstr\"om black hole spacetimes, for which
\begin{equation*}
D(r)=1-\frac{2M}{r}+\frac{e^2}{r^2},
\end{equation*}
where $|e|\leq M$ is a constant, are examples of a spacetime region satisfying the above assumptions on $\mathcal{M}$ in the case where $r_{\rm min}=r_+>0$.

As a second case, we extend $D$ to $D:[0,\infty)\to \R$, such that $D(r)\geq d_D$, for some constant $d_D>0$. We now consider the manifold $\mathcal{M}=\R\times \R^3$, and the submanifold $$\mathring{\mathcal{M}}=\mathcal{M}\setminus \{\R\times \{0\}\}=\R\times (0,\infty)\times \s^2,$$ which is covered by the coordinate chart $(v,r,\theta,\varphi)$, with $v\in \R$, $r\in (0,\infty)$, $\theta\in (0,\pi)$, $\varphi\in (0,2\pi)$.

We equip $\mathring{\mathcal{M}}$ with the metric $g$ given by the expression \eqref{metricvrcoords}. Note that $g$ can be extended to the entire manifold $\mathcal{M}$ after a suitable change of coordinates. Clearly $r_{\text{min}}=0$ in this case, where $r_{\text{min}}=0$  denotes the infimum of $r$ on $\mathring{\mathcal{M}}$. The Penrose diagram of such spacetimes is depicted in Figure \ref{fig:fullfoliations1}.

The case $D(r)=1$ corresponds to the Minkowski spacetime which clearly is an example of a spacetime satisfying the above assumptions.

\subsection{Foliations}
\label{sec:foliations}
Let $\widetilde{\Sigma}$ be an asymptotically flat spacelike hypersurface in $\mathcal{M}$ interesecting  $\mathcal{H}^{+}$ such that $\widetilde{\Sigma}\cap \mathcal{H}^+$ is isometric to a round sphere.

Consider the outgoing null hypersurface $\mathcal{N}_{u'}=\{(u,r,\theta,\varphi)\,:\, u=u',\: r\geq R\}$. Now, define the following spacetime regions contained in $\mathcal{A}'$: 
 $$\mathcal{A}^{u_2}_{u_1}=\bigcup_{u\in [u_1,u_2]}\mathcal{N}_{u}$$
 and the extended region  
 $$\mathcal{A}=\bigcup_{u\in[0,\infty)}\mathcal{N}_{u}.$$ 
\begin{figure}[h!]
\begin{center}
\includegraphics[width=1.9in]{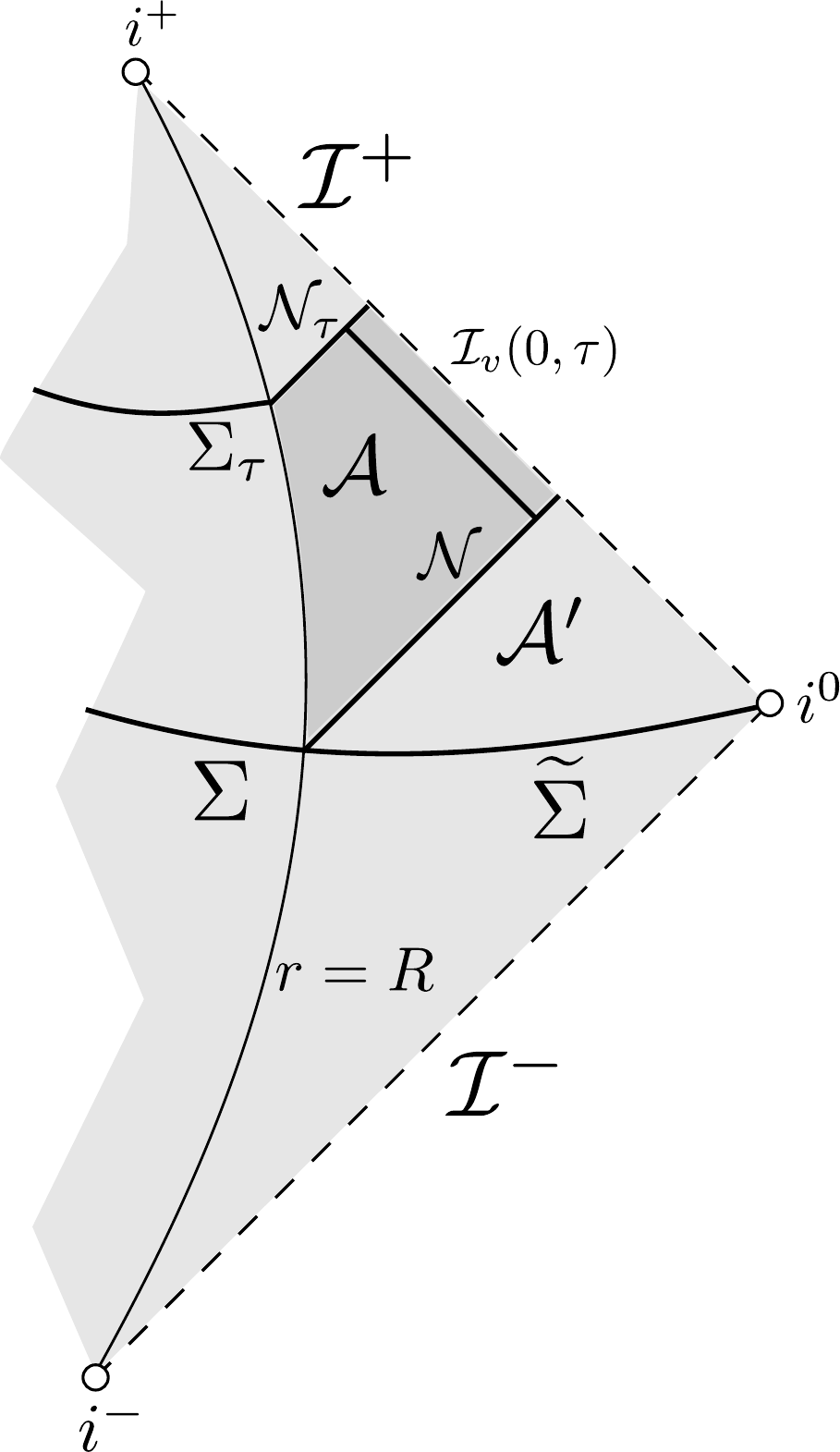}
\caption{\label{fig:foliations}The Penrose diagram of the far-away spacetime region $\mathcal{A'}=\left\{r\geq R\right\}$ embedded in a larger spacetime.}
\end{center}
\end{figure}

By shifting the coordinate $u$ by a constant, we can assume without loss of generality that $\widetilde{\Sigma}\cap \{r=R\}=\{u=0\}\cap \{r=R\}$. Denote $\mathcal{N}=\mathcal{N}_0$ and consider the spacelike-null hypersurface $$\Sigma\doteq \widetilde{\Sigma}\cap\{r\leq R\}\cup \mathcal{N}$$
and let $\tau$ be a smooth function on $J^+(\Sigma)$, such that $\tau|_{{\Sigma}}=0$, and $T(\tau)=1$. Note that in $\mathcal{A}$ it holds that $\tau=u$.

Let $\mathcal{I}_{v'}(\tau_1,\tau_2) \doteq \{v=v',\:\tau_1\leq u\leq \tau_2\}$ denote ingoing null segments in $\mathcal{A}^{\tau_2}_{\tau_1}$. Finally, we denote the main spacetime region on interest in this paper by $\mathcal{R}$, where
\begin{equation*}
\mathcal{R}=J^+(\Sigma)=\bigcup_{\tau\in [0,\infty)}\Sigma_{\tau}.
\end{equation*}
See Figure \ref{fig:foliations} for an illustration of the above foliation.

It will also be convenient to introduce a foliation with leaves that are \emph{hyperboloidal} spacelike hypersurface which terminate at null infinity. In order to construct the hyperboloidal foliation, we will first consider a vector field $$Y=\partial_r+h(r)\partial_v,$$ with $h: [r_{\rm min},\infty) \to \R$ a smooth function, such that:
\begin{align*}
\frac{1}{\max_{r_{\rm min}\:\leq r\leq R}D(r)}&\leq h(r)< \frac{2}{D(r)}\quad \textnormal{if}\: r\leq R,\\
0<\frac{2}{D(r)}-h(r)&=O_1(r^{-1-\eta})\quad \textnormal{if}\: r>R,
\end{align*}
for some $\eta>0$. Note that by construction $Y$ is spacelike, i.e.\ $$g(Y,Y)=h(r)(2-h(r)D(r))>0$$ for all $r\in [r_{r\min},\infty)$.

We will construct the hyperboloidal leaves by considering the integral curves $\gamma_Y \subset \mathcal{M}$ of $Y$. By considering as a parameter the function $r$, we have that $\gamma_Y: [r_{\rm min},\infty)\to \mathcal{R}$, with
\begin{equation*}
\gamma_Y(r)=(v_{Y}(r),r,\theta_0,\varphi_0),
\end{equation*}
in $(v,r,\theta,\varphi)$ coordinates, where we keep $\theta_0,\varphi_0$ fixed, we have that $\frac{dv_Y}{dr}=h(r)$ and moreover $v_{Y}(R)=v_{0}$. By choosing $v_{0}$ sufficiently large, we can guarantee that $\gamma_{Y}(R)\in\mathcal{R}$ as $v$ increases in the future direction along $r=R$.  

We first restrict to the region $\mathcal{R}\setminus \mathcal{A}'$. There, $v_Y$ is given by 
\begin{equation*}
v_Y(r)=v_0-\int_{r}^R h(r')\,dr'.
\end{equation*}
By choosing $v_0$ suitably large depending on $R$ and $h$, we can ensure that $v_{Y}(\tilde{r})$ is larger than $v(\Sigma\cap \left\{r=\tilde{r}\right\})$  for all $r_{\rm min}\leq \tilde{r}\leq R$. This implies that  $\gamma_{Y}(r)\in \mathcal{R}$ for all $r_{\rm min}\leq r\leq R$.

Now, consider the region $\mathcal{A}'$. Using that $u=v-2r_*$, we can express $\gamma_Y$ in $(u,r,\theta,\varphi)$ coordinates. We obtain:
\begin{equation*}
\gamma_Y(r)=(u_{Y}(r)=v_{Y}(r)-2r_*(r),r,\theta_0,\varphi_0),
\end{equation*}
with $\frac{du_{Y}}{dr}=h-\frac{2}{D}$. By the assumptions on $h$ above, we therefore have that $\frac{du_{Y}}{dr}<0$ and moreover
\begin{equation*}
\begin{split}
\left|u_{Y}(r)-u_{Y}(R)\right|=&\:\left|\int_{R}^r h(r')-\frac{2}{D(r')}\,dr'\right|\\
\leq &\: C_Y (R^{-\eta}-r^{-\eta})\leq C_Y R^{-\eta},
\end{split}
\end{equation*}
where $C_Y>0$ is a constant that depends on the choice of $h$. In particular, we observe that for $v_0$ suitably large (depending on $C_Y$ and $R$) $u_Y$ satisfies the in equality $u_{Y}(r)>0$, and hence we can conclude that also $\gamma_{Y}(r)\in \mathcal{R}$ for all $r\geq R$.

We define the \emph{spacelike hyperboloidal hypersurface} $\mathcal{S}$ as
\begin{equation*}
\mathcal{S}=\{(v,r,\theta,\varphi)\,:\, v=v_Y(r),\: r\in [r_{\rm min},\infty)\}.
\end{equation*}
By construction, we have that $\mathcal{S}\subset \mathcal{R}$.

We can now construct a \emph{hyperboloidal} foliation of a subset of $\mathcal{R}$, by taking as our leaves the spacelike hypersurfaces $\mathcal{S}_{\widetilde{\tau}'}=\{\widetilde{\tau}=\widetilde{\tau}'\}$, where $\widetilde{\tau}: \mathcal{R}\cap J^+(\mathcal{S})\to [0,\infty)$ is the function defined by: $\widetilde{\tau}|_{\mathcal{S}}=0$ and $T(\widetilde{\tau})=1$. Then,
\begin{equation*}
J^+(\mathcal{S})=\bigcup_{\widetilde{\tau}\in[0,\infty)}\mathcal{S}_{\widetilde{\tau}}.
\end{equation*}

It will also be useful to introduce coordinates on $\mathcal{R}\cap J^+(\mathcal{S})$, corresponding to the hyperboloidal foliation. We consider the coordinate chart $(\widetilde{\tau},\rho,\theta,\varphi)$, with $\rho=r|_{\mathcal{S}_0}$ and $\partial_{\rho}= Y$. By construction, we can estimate
\begin{equation}
\label{eq:comptimes}
\tau-\tau_0\leq \widetilde{\tau}\leq \tau+\tau_0,
\end{equation}
for some $\tau_0=\tau_0(D, R,\mathcal{S}_0,\Sigma)>0$.


\begin{figure}[h!]
\begin{center}
\includegraphics[width=3in]{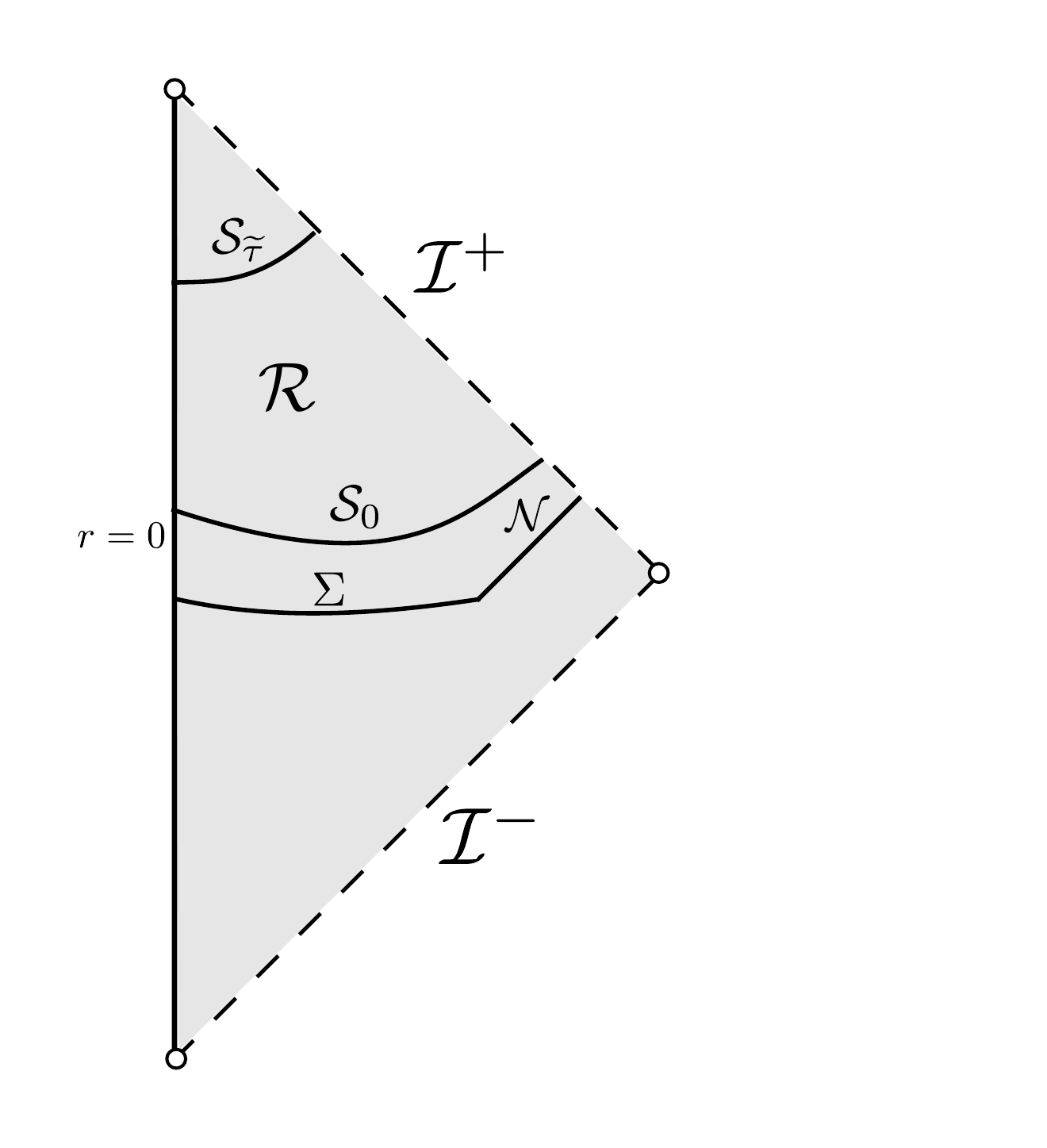}
\caption{\label{fig:fullfoliations1}A Penrose diagram of $\mathcal{M}$ in the case $r_{\rm min}=0$.}
\end{center}
\end{figure}


\begin{figure}[h!]
\begin{center}
\includegraphics[width=3in]{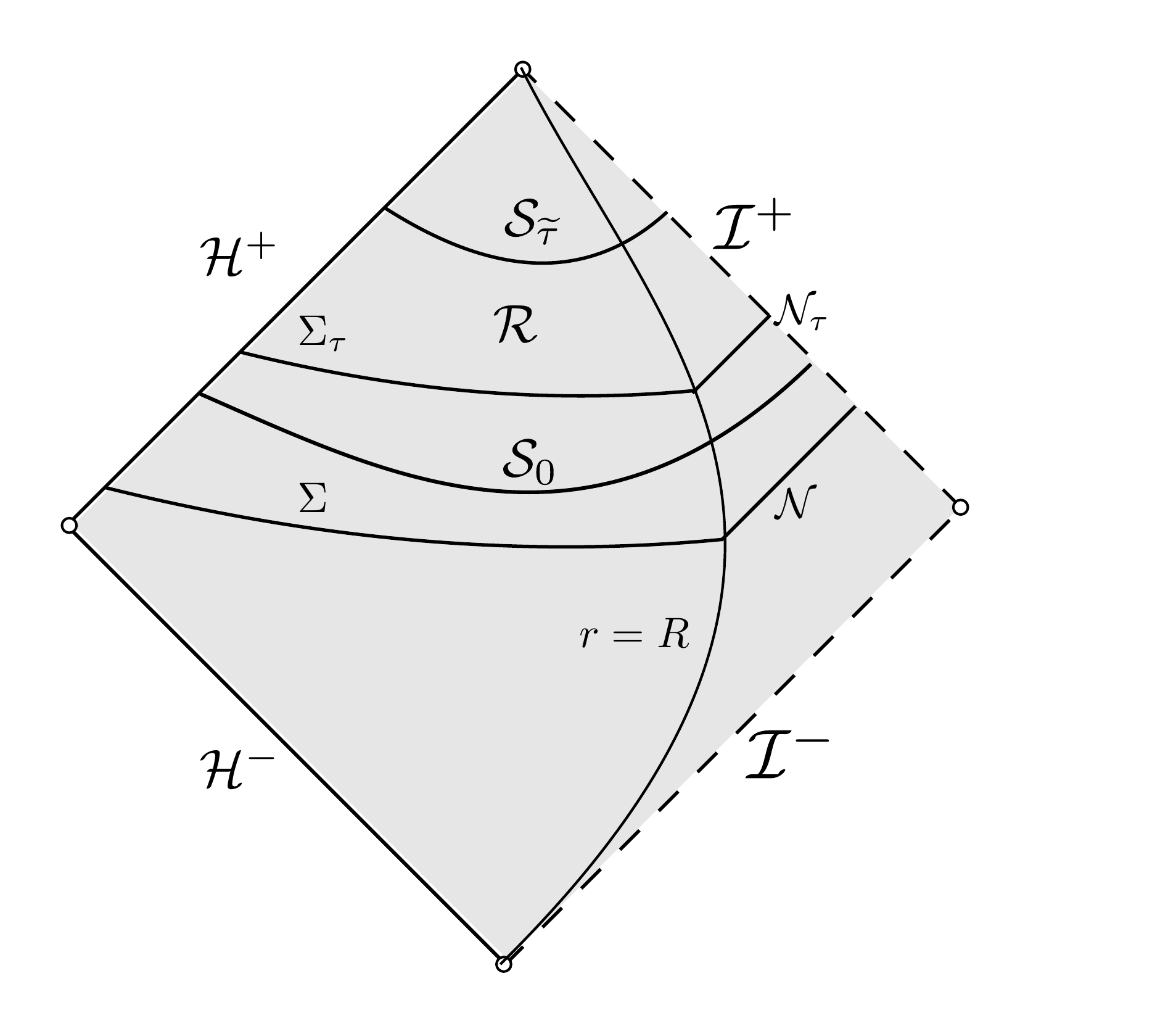}
\caption{\label{fig:fullfoliations2}Penrose diagram of $\mathcal{M}$ in the case $r_{\rm min}=r_+$.}
\end{center}
\end{figure}

\subsubsection{Additional notational conventions}
In this section, we will introduce several notational conventions that will be used throughout the paper. First of all, we write
\begin{equation*}
f \lesssim A,
\end{equation*}
for a positive function $f$ and a constant $A$, when there exists a \emph{uniform} constant $C>0$, depending only on the geometry of $(\mathcal{R},g)$, such that
\begin{equation*}
f\leq C\cdot A.
\end{equation*}

We also write
\begin{equation*}
f \sim A
\end{equation*}
when there exist \emph{two} uniform constants $C>c>0$ such that
\begin{equation*}
c\cdot A \leq f\leq C\cdot A.
\end{equation*}

The induced volume form on the spacelike-null hypersurfaces $\Sigma_{\tau}$ will be denoted by $d\mu_{\Sigma_{\tau}}$ and the induced volume form on the hyperboloidal hypersurfaces  $\mathcal{S}_{\widetilde{\tau}}$ will be denoted by $d\mu_{\mathcal{S}_{\widetilde{\tau}}}$. Here, where we let $d\mu_{\Sigma_{\tau}}|_{\mathcal{N}_{\tau}}=r^2d\omega dr$ along $\mathcal{N}_{\tau}$. Moreover, $\underline{L}$ will denote the ingoing null vector field satisfying $\underline{L}(u)=1$ and $L$ the outgoing null vector field satisfying ${L}(v)=1$.

We will additionally denote the future-directed normal along $\Sigma_{\tau}$ by $n_{\Sigma_{\tau}}$, with the shorthand form $n_{\tau}$. We moreover let $n_{\tau}|_{\mathcal{N}_{\tau}}=\underline{L}$. Furthermore, we will denote the future-directed normal along $\mathcal{S}_{\widetilde{\tau}}$ by $n_{\mathcal{S}_{\widetilde{\tau}}}$ and $n_{\widetilde{\tau}}$.

\subsection{Assumptions for the wave equation}
\label{sec:waveassm}
We consider the spacetime region $(\mathcal{R},g)$ defined in Section \ref{sec:foliations} and formulate a global existence and uniqueness statement for the linear wave equation (\ref{waveequation}) with smooth mixed Cauchy and characteristic initial data on $\Sigma$.\footnote{For convenience, we formulated a global existence and uniqueness statement for smooth initial data. Instead, we could have taken our data to be less regular, i.e.\ $\Psi \in W^{k+1,2}_{\rm loc}$ or $\Psi' \in W^{k,2}_{\rm loc}$ for $k\geq 0$.}

\begin{theorem}
\label{thm:extuniq}
Let $\Psi\in C^{\infty}({\Sigma})$, $\Psi'\in C^{\infty}({\Sigma}\setminus \mathcal{N})$. Then there exists a unique smooth function $\psi: \mathcal{R} \to \R$ satisfying
\begin{align*}
 \square_g \psi&=0,\\
\psi|_{{\Sigma}}&=\Psi,\\
n_{\widetilde{\Sigma}}(\psi)|_{\widetilde{\Sigma}\setminus \Sigma'}&=\Psi'.
\end{align*}
\end{theorem}

We will moreover always assume that $\Psi\to 0$ as $r\to \infty$.

We denote the stress-energy tensor corresponding a function $f:\mathcal{R}\to\R$ by $\mathbf{T}[f]$. The stress-energy tensor $\mathbf{T}[f]$ is a symmetric 2-tensor, with components
\begin{equation*}
\mathbf{T}_{\alpha \beta}[f]=\partial_{\alpha}f \partial_{\beta}f-\frac{1}{2}{g}_{\alpha \beta} (g^{-1})^{\kappa \lambda}\partial_{\kappa}f \partial_{\lambda} f,
\end{equation*}
with respect to a coordinate basis. Note that $\textnormal{div}\,\mathbf{T}[\psi]=0$ for solutions $\psi$ to (\ref{waveequation}).

We define the energy current $J^V[f]$ with respect to a function $f$ and a vector field $V$ on $\mathcal{R}$ as the following contraction:
\begin{equation*}
J^V[f]\doteq \mathbf{T}(V,\cdot).
\end{equation*}

Let $W$ be another vector field on $\mathcal{R}$. We will make use of the following notation:
\begin{equation*}
J^V[\psi]\cdot W \doteq \mathbf{T}(V,W).
\end{equation*}

It will be convenient to split
\begin{equation*}
\textnormal{div}\,J^V[f]=K^V[f]+\mathcal{E}^V[f],
\end{equation*}
where
\begin{align}
\label{def:KV}
K^V[f] \doteq&\: \mathbf{T}^{\alpha \beta}\nabla_{\alpha}V_{\beta},\\
\label{def:EV}
\mathcal{E}^V[f]\doteq&\: V(f)\square_g f.
\end{align}
Note that in particular, $\mathcal{E}^V[\psi]=0$ for solutions $\psi$ to (\ref{waveequation}).

\subsubsection{Energy boundedness for scalar waves}
Let $N$ be a strictly timelike vector field $N$, such that $N=T$ in $\mathcal{A}$. In the $r_{\rm min}=0$ case, we can take $N=T$, whereas in the $r_{\rm min}=r_+>0$ case, we can construct $N$ such that $N=T-Y$ for $r_+\leq r\leq r_0$ and $N=T$ for $r\geq r_1$, with $r_+<r_0<r_1$, by employing a smooth cut-off function.

 We will assume the following energy boundedness statement for the wave equation as a ``black-box'': assume that
 \begin{equation*}
 \int_{\Sigma}J^N[\psi]\cdot n_{\Sigma}\,d\mu_{\Sigma}<\infty.
 \end{equation*}
Then there exists a uniform constant $C>0$, such that for all $v$
\begin{equation}
\label{ass:ebound}
\int_{\Sigma_{\tau}}J^N[\psi]\cdot n_{\tau}\,d\mu_{\Sigma_{\tau}}+\int_{\mathcal{I}_v(0,\tau)}J^N[\psi]\cdot \underline{L}\: r^2d\omega du\leq C \int_{\Sigma}J^N[\psi]\cdot n_{\Sigma}\,d\mu_{\Sigma}.
\end{equation}

\begin{remark}
Note that the geometric assumptions in Section \ref{sec:geomassm} and \ref{sec:asmMsetminusA} together with the energy boundedness assumption above do not fix the behaviour of $D$ in the spacetime region $\mathcal{R}\setminus \mathcal{A}$. A spacetime satisfying only these assumptions allows in particular for complicated trapping behaviour of null geodesics; see \cite{molog, Keir, aag1}. In \cite{molog, Keir}, in particular, it is shown that trapping can be the source of time-decay for $\psi$ with merely a logarithmic rate. We will, however, consider only spacetimes where the behaviour of $\psi$ in $\mathcal{R}\setminus \mathcal{A}$ does \emph{not} form an obstruction to decay estimates. Quantitatively, this amounts to making an additional assumption of suitable Morawetz estimates, or integrated local energy decay statements in the assumption below.
\end{remark}

\subsubsection{Morawetz estimates for scalar waves}
We assume that the region $\mathcal{R}\setminus \mathcal{A}$ does not form an obstruction to decay by imposing the following Morawetz estimate as another black-box: there exists a uniform constant $C>0$, such that for all $0<\tau_1<\tau_2$
\begin{equation}
\label{ass:morawetz}
\int_{\tau_1}^{\tau_2}\left(\int_{\Sigma_{\tau}\setminus \mathcal{N}_{\tau}}J^N[ \psi]\cdot n_{\tau}\,d\mu_{\Sigma_{\tau}}\right)\,d\tau\leq C \int_{\Sigma_{\tau_1}}J^N[\psi]\cdot n_{\tau_1}+J^N[T\psi]\cdot n_{\tau_1}\,d\mu_{\Sigma_{\tau_1}}.
\end{equation}
We will moreover assume local Morawetz estimates in $\mathcal{A}$ \emph{without} a loss of derivatives for higher-order derivatives of $\psi$ :
\begin{equation}
\label{ass:morawetzlocal}
\int_{\tau_1}^{\tau_2}\left(\int_{\mathcal{N}_{\tau}\cap \{\widetilde{R}\leq r\leq \widetilde{R}+1\}}J^T[ \partial^{\alpha}\psi]\cdot n_{\tau}\,d\mu_{\Sigma_{\tau}}\right)\,d\tau\leq C_{\alpha} \sum_{k\leq |\alpha|}\int_{{\Sigma}_{\tau_1}}J^N[T^{k}\psi]\cdot n_{\tau_1},
\end{equation}
for $|\alpha|\geq 0$ and $\widetilde{R}\geq R$, where $C_{\alpha}>0$ depends on $\widetilde{R}$ and the choice of $\alpha$.

The Morawetz estimates above have in particular been obtained for sub-extremal Reissner--Nordstr\"om spacetimes  (see, for instance, \cite{redshift,blu1}). The left hand side of \ref{ass:morawetz} has been shown to \emph{blow up} for generic smooth, compactly supported data on $\Sigma$ in extremal Reissner--Nordstr\"om, which are therefore not included in the class of spacetimes studied in this paper. However, a degenerate Morawetz estimate still holds in the extremal case (see \cite{aretakis1}). This allows us to appropriately adapt the methods developed in this paper to the extremal case separately (\cite{aag7}).

\begin{remark}
Whenever a spacetime contains trapped null geodesics, one cannot arrive at a Morawetz estimate like (\ref{ass:morawetz}) without ``losing derivatives'' on the right-hand side; see \cite{janpaper, ralston2}. However, the trapping of null geodesics in example Schwarzschild and Kerr (away from the event horizon) is a high-angular frequency phenomenon; that is to say, if we restrict $\psi$ to a fixed angular mode $\psi_{\ell}$, we can remove the loss of derivatives in \eqref{ass:morawetz}, at the expense of replacing the uniform constant on the right-hand side with a constant that grows polynomially in $\ell$. Note that we can also relax the assumption \eqref{ass:morawetz} to allow for the loss of \emph{any} finite number of derivatives on the right-hand side.
\end{remark}

\subsection{Spherical decompositions and Hardy and Poincar\'{e} inequalities}
\label{sec:basicineq}
In this section, we review some standard inequalities that play a central role in our proofs.

\begin{lemma}[Hardy inequalities]
\label{lm:Hardy}
Let $q\in \R$ and $f:[r_0,\infty)\to\R$ a $C^1$ function. Then,
\begin{align}
\label{eq:Hardy1}
\int_{r_0}^{\infty} r^q f^2(r)\,dr&\leq \frac{4}{(q+1)^2}\int_{r_0}^{\infty} r^{q+2}(\partial_rf)^2(r)\,dr\\
& \textnormal{for}\:q\neq -1,\:\textnormal{if} \: f(r_0)=0\:\textnormal{and}\:\lim_{r\to \infty} r^{q+1}f^2(r)=0, \nonumber\\
\label{eq:Hardy3}
\int_{r_{0}}^{\infty} f^2(r)\,dr&\leq 4\int_{r_{0}}^{\infty} (r-r_{0})^2 (\partial_rf)^2(r)\,dr \quad \textnormal{if}\: \lim_{r\to \infty} r f^2(r)=0.
\end{align}
\end{lemma}
\begin{proof}
Let $q\neq -1$. We integrate $\partial_r(r^{q+1}f^2)$ and use the boundary conditions on $f$ to obtain:
\begin{equation*}
0=\int_{r_0}^{\infty}\partial_r(r^{q+1}f^2)\,dr=(q+1)\int_{r_0}^{\infty}r^qf^2\,dr+2\int_{r_0}^{\infty}r^{q+1}f\partial_rf\,dr.
\end{equation*}
By applying a (weighted) Cauchy--Schwarz inequality, we can further estimate
\begin{equation*}
2\int_{r_0}^{\infty}r^{q+1}|f||\partial_rf|\,dr\leq 2\sqrt{\int_{r_0}^{\infty}r^qf^2\,dr}\sqrt{\int_{r_0}^{\infty}r^{q+2}(\partial_rf)^2\,dr}.
\end{equation*}
We therefore obtain the inequality
\begin{equation*}
\int_{r_0}^{\infty}r^qf^2\,dr\leq \frac{2}{q+1}\sqrt{\int_{r_0}^{\infty}r^qf^2\,dr}\sqrt{\int_{r_0}^{\infty}r^{q+2}(\partial_rf)^2\,dr}.
\end{equation*}
We arrive at (\ref{eq:Hardy1}) by dividing both sides of the above inequality by $\sqrt{\int_{r_0}^{\infty}r^qf^2\,dr}$ and then squaring both sides of the resulting inequality.

The estimate \eqref{eq:Hardy3} follows by considering $\partial_r((r-r_{0})f^2)$ and repeating the argument above.
\end{proof}

We can decompose any $C^2$ function $f: \mathcal{R}\to \R$ into \emph{spherical harmonic modes}, i.e.\
\begin{equation*}
f=\sum_{\ell'=0}^{\infty}f_{\ell=\ell'},
\end{equation*}
such that moreover, with respect to the coordinates $(v,r,\theta,\varphi)$ on $\mathcal{R}$:
\begin{equation*}
f_{\ell=\ell'}(v,r,\theta,\varphi)=\sum_{m=-\ell'}^{\ell'} f_{\ell' m}(u,v)Y_{\ell' m}(\theta,\varphi),
\end{equation*}
where $Y_{\ell m}$ are the spherical harmonics, which constitute a complete basis on $L^2(\s^2)$ of eigenfunctions for the spherical Laplacian $\slashed{\Delta}_{\s^2}$ with eigenvalues $-\ell(\ell+1)$ and $f_{\ell' m}$ are functions of only $v$ and $r$. In particular,
\begin{equation*}
\slashed{\Delta}f_{\ell=L}=r^{-2}\slashed{\Delta}_{\s^2}f_{\ell=L}=-r^{-2}L(L+1)f_{\ell=L}.
\end{equation*}
Furthermore, let $\ell,\ell'\geq 0$, then
\begin{equation*}
\int_{\s^2} \psi_{\ell=L} \psi_{\ell=\tilde{L}}\,d\omega=\frac{4\pi}{2L+1}\delta_{L \tilde{L}}\sum_{m=-L}^{L} |\psi_{L m}|^2.
\end{equation*}

We denote for $L\geq 1$
\begin{equation*}
\psi_{\ell \geq L}\doteq \psi-\sum_{\ell'=0}^{L-1}\psi_{\ell=\ell'}.
\end{equation*}

From the assumptions on the metric $g$ in Section \ref{sec:geomassm}, it follows that for all $L\geq 0$
\begin{equation*}
\square_g\psi_{\ell=L}=0,
\end{equation*}
so each spherical harmonic mode $\psi_{\ell=\ell'}$ and moreover each sum $\psi_{\ell \geq L}$ are solutions to (\ref{waveequation}).

The following lemma is a standard result corresponding to the spherical harmonic decomposition.
\begin{lemma}[Poincar\'e inequality on $\s^2$]
\label{lm:poincare}
Let $L \geq 1$. Then
\begin{equation}
\label{eq:poincare2}
\int_{\s^2}\psi_{\ell \geq L}^2\,d\omega  \leq \frac{1}{L(L+1)}r^{-2}\int_{\s^2}|\snabla \psi_{\ell \geq L}|^2\,d\omega.
\end{equation}
In the case $\psi$ is supported on a single harmonic mode the inequality becomes an equality:
\begin{equation}
\label{eq:poincare1}
\int_{\s^2}\psi_{\ell=L}^2\,d\omega=\frac{1}{L(L+1)}r^{-2}\int_{\s^2}|\snabla \psi_{\ell=L}|^2\,d\omega.
\end{equation}
Furthermore,
\begin{equation}
\label{eq:poincare3}
\int_{\s^2}r^2|\snabla \psi|^2\,d\omega  \leq \int_{\s^2}(\slashed{\Delta}_{\s^2} \psi)^2\,d\omega.
\end{equation}
\end{lemma}

We denote the three spacelike Killing vector asssociated to the spherical symmetry of the spacetime by $\Omega_i$, with $i=1,2,3$. They can be expressed in spherical coordinates as follows:
\begin{align*}
\Omega_1&=\sin \varphi \partial_{\theta}+\cot\theta \cos \varphi \partial_{\varphi},\\
\Omega_2&=-\cos \varphi \partial_{\theta}+\cot\theta \sin \varphi \partial_{\varphi},\\
\Omega_3&=\partial_{\varphi}.
\end{align*}

We will make frequent use of the following shorthand notation:
\begin{equation*}
\Omega^k=\Omega_1^{k_1}\Omega_2^{k_2}\Omega_3^{k_3},
\end{equation*}
where $k=(k_1,k_2,k_3)\in \N_0$.
\begin{lemma}
\label{lm:angmomop}
There exists a numerical constant $C>0$ such that
\begin{equation}
\label{eq:angmomop1}
\int_{\s^2}|\snabla f|^2\,d\omega\leq Cr^{-2}\sum_{|k|=1}\int_{\s^2}(\Omega^kf)^2\,d\omega.
\end{equation}
Furthermore,
\begin{equation}
\label{eq:angmomop2}
\int_{\s^2}(\slashed{\Delta} f)^2\,d\omega\leq Cr^{-2}\sum_{|k|=1}\int_{\s^2}|\snabla \Omega^kf|^2\,d\omega.
\end{equation}
\end{lemma}
By combing a standard Sobolev inequality, together with Lemma \ref{lm:angmomop}, we obtain the following Sobolev inequality on $\s^2$.

\begin{lemma}
\label{lm:sobolevs2}
There exists a numerical constant $C>0$ such that
\begin{equation}
\label{eq:sobolevs2}
\sup_{\omega \in \s^2} f^2(\omega)\leq C \sum_{|k|\leq 2}\int_{\s^2}(\Omega^kf)^2(\omega)\,d\omega.
\end{equation}
\end{lemma}

\begin{lemma}[Interpolation inequality]
\label{lm:interpolation}
Let $f: \R_+ \times [R,\infty)\to \R$ be a function such that the following inequalities hold:
\begin{align}
\label{eq:interp1}
\int_R^{\infty}r^{p-\epsilon}f^2(\tau,r)\,dr\leq &\:D_1(1+\tau)^{-q},\\
\label{eq:interp2}
\int_R^{\infty}r^{p+1-\epsilon}f^2(\tau,r)\,dr\leq&\: D_2(1+\tau)^{-q+1},
\end{align}
for some $\tau$-independent constants $D_1,D_2>0$, $q\in \R$ and $\epsilon\in (0,1)$.

Then
\begin{equation}
\label{eq:interp3}
\int_R^{\infty}r^{p}f^2(\tau,r)\,dr\lesssim \max\{D_1,D_2\}(1+\tau)^{-q+\epsilon}.
\end{equation}
\end{lemma}
\begin{proof}
We split
\begin{equation*}
\int_R^{\infty}r^{p}f^2(\tau,r)\,dr=\int_R^{\tau+R}r^{p}f^2(\tau,r)\,dr+\int^{\infty}_{\tau+R}r^{p}f^2(\tau,r)\,dr.
\end{equation*}
We use \eqref{eq:interp1} to estimate
\begin{equation*}
\begin{split}
\int_R^{\tau+R}r^{p}f^2(\tau,r)\,dr=\int_R^{\tau+R}r^{\epsilon}r^{p-\epsilon}f^2(\tau,r)\,dr\leq&\: (\tau+R)^{\epsilon}\int_R^{\tau+R}r^{p-\epsilon}f^2(\tau,r)\,dr\\
\leq &\: C D_1 (1+\tau)^{-q+\epsilon}.
\end{split}
\end{equation*}
We use \eqref{eq:interp2} to estimate
\begin{equation*}
\begin{split}
\int^{\infty}_{\tau+R}r^{p}f^2(\tau,r)\,dr=\int^{\infty}_{\tau+R}r^{-1+\epsilon}r^{p+1-\epsilon}f^2(\tau,r)\,dr\leq&\: (\tau+R)^{-1+\epsilon}\int^{\infty}_{\tau+R}r^{p+1-\epsilon}f^2(\tau,r)\,dr\\
\leq &\: C D_2 (1+\tau)^{-q+\epsilon}.
\end{split}
\end{equation*}
We obtain \eqref{eq:interp3} by adding the above two inequalities.
\end{proof}
\subsection{The Dafermos--Rodnianski hierarchy and applications}
\label{sec:dafrodhier}
 We state here precisely the $r^p$-hierarchies obtained by Dafermos and Rodnianski in \cite{newmethod}.
\begin{proposition}[\textbf{The Dafermos--Rodnianski $r^p$-weighted estimates}]
\label{prop:rpphiv1}
Let $\psi$ be a solution to \eqref{waveequation} corresponding to initial data from Theorem \ref{thm:extuniq} on $(\mathcal{R}, g)$. Denote $\phi=r\psi$. 

Let $p\in(0,2]$. Then there exists an $R>0$ such that for any $0\leq u_1<u_2$
\begin{equation}
\label{rphierpsidafrod}
\begin{split}
\int_{\mathcal{N}_{u_2}} &r^p(\partial_r\phi)^2\, d\omega dr+\int_{\mathcal{A}_{u_1}^{u_2}} pr^{p-1}(\partial_r\phi)^2+(2-p)r^{p-1}|\snabla  \phi|^2  \,d\omega drdu \\
&+\int_{\mathcal{I}^+}r^p|\snabla \phi|^2\,d\omega du \leq \\
\leq&\: C\int_{\mathcal{N}_{u_1}} r^p(\partial_r\phi)^2\,d\omega dr+C\int_{\Sigma_{u_1}}J^T[\psi]\cdot n_{\Sigma_{u_1}}\,d\mu_{{u_1}},
\end{split}
\end{equation}
where $C=C(D,R)>0$ is a constant.
\end{proposition}
As an application of the above $r^p$-weighted inequalities, one can obtain the following decay estimates for $r\psi$ and $r^{1/2} \psi$.
\begin{proposition}\label{DafRoddecay}
Let $\psi$ be a solution to \eqref{waveequation} corresponding to initial data from Theorem \ref{thm:extuniq} on $(\mathcal{R}, g)$.
Then there exists a constant $C=C(D,R)>0$ such that for all $u\geq 0$ we have that:
\begin{equation}\label{Ndecay2}
\int_{\Sigma_{u}} J^N [\psi ] \cdot n_{\Sigma_{u}} d\mu_{\Sigma_{u}} \leq C(1+u)^{-2}E_{\rm dr}[\psi] ,
\end{equation}
with
\begin{equation*}
E_{\rm dr}[\psi] \doteq \sum_{k \leq 2} \int_{\Sigma_0} J^N [T^k \psi ] \cdot n_{\Sigma_0} d\mu_{\Sigma_0} + \int_{\mathcal{N}_0} r (\partial_r (T\phi ))^2 d\omega dr + \int_{\mathcal{N}_0} r^2 (\partial_r \phi )^2 d\omega dr. 
\end{equation*}
and moreover
\begin{align*}
\sup_{\Sigma_{\tau}}r^{\frac{1}{2}}|\psi|\leq &\:C(1+{\tau})^{-1} \sum_{
 |\alpha|\leq 2}\sqrt{E_{\rm dr}[\Omega^{\alpha}\psi]},\\
\sup_{\Sigma_{\tau}}|r\psi|\leq &\:C(1+{\tau})^{-\frac{1}{2}} \sum_{ |\alpha|\leq 2}\sqrt{E_{\rm dr}[\Omega^{\alpha}\psi]}.
\end{align*}
\end{proposition}

\section{Limits at null infinity}
\label{sec:limitsnullinfty}
In this section we will derive time-dependent estimates for the radiation field $\phi \doteq r\psi$ and its naturally $r$-weighted higher-order derivatives that will be applied in the subsequent sections. We will use these to derive the asymptotic behaviour at $\mathcal{I}^+$ of $r\psi$ and its $r$-weighted derivatives in the $r$-variable.

Note that in this section we can relax it is sufficient to consider $(\mathcal{R},g)$ that only satisfy the assumptions in Section \ref{sec:geomassm}. That is to say, we do not need to invoke the energy boundedness assumption and the Morawetz estimates from Section \ref{sec:waveassm}.

\subsection{Commuting $\square_g$ with $r$-weighted vector fields}

We first derive the appropriate equations for $r\psi$ and its derivatives. We denote the ingoing and outgoing null vector fields by $\underline{L}$ and $L$, respectively, which in double null coordinates $(u,v,\theta,\varphi)$ are given by $\underline{L}=\partial_u$ and $L=\partial_v$, whereas in Bondi coordinates $(u,r,\theta,\varphi)$ we have that
\begin{align*}
\underline{L}=&\:\partial_u-\frac{1}{2}D\partial_r,\\
L=&\:\frac{1}{2}D\partial_r.
\end{align*}

In $(u,r,\theta,\varphi)$ coordinates, the equation \eqref{waveequation} becomes:
\begin{equation*}
\begin{split}
\square_g\psi&=-2\partial_u\partial_r\psi+D\partial_r^2\psi-2r^{-1}\partial_u\psi+(D'+2r^{-1}D)\partial_r\psi+\slashed{\Delta}\psi=0.
\end{split}
\end{equation*}
See \eqref{eq:derivwaveeq} in Appendix \ref{app:commmultp} for a derivation.

\begin{lemma}
\label{lm:commute0time}
Let $\psi$ be a solution to \eqref{waveequation} emanating from initial data given as in Theorem \ref{thm:extuniq} on $(\mathcal{R}, g)$. Then $\phi = r\psi$ satisfies the following equation:
\begin{equation}
\label{waveoperatorphi}
\square_g\phi=-2r^{-1}\partial_u\phi+2Dr^{-1}\partial_r\phi+D'r^{-1}\phi.
\end{equation}
Furthermore,
\begin{equation}
\label{eq:sDeltapsiinftyv0}
2\partial_u\partial_r\phi=\partial_r(D\partial_r\phi)-D'r^{-1}\phi+\slashed{\Delta} \phi,
\end{equation}
which is equivalent to
\begin{equation}
\label{eq:sDeltapsiinfty}
2\underline{L}(\partial_r\phi)=D'\partial_r\phi-D'r^{-1}\phi-\slashed{\Delta} \phi.
\end{equation}
\end{lemma}
\begin{proof}
See Appendix \ref{app:commmultp}.
\end{proof}
It will be necessary to additionally consider two types of first-order radiation fields:
\begin{align*}
\Phi \doteq &\: r^2\partial_r(r\psi),\\
\widetilde{\Phi} \doteq &\: r(r-M)\partial_r(r\psi).
\end{align*}
Using that $\square_g\psi=0$, we can similarly derive equations for $\Phi$ and $\widetilde{\Phi}$.

\begin{lemma}
\label{lm:commute1time}
Let $\psi$ be a solution to \eqref{waveequation} emanating from initial data given as in Theorem \ref{thm:extuniq} on $(\mathcal{R}, g)$. 

Then $\Phi = r^2\partial_r\phi$ and $\widetilde{\Phi}= r(r-M)\partial_r\phi$ satisfy the following equations:
\hspace{1pt}
\begin{itemize}\setlength\itemsep{1em}
\item[\emph{(i)}]
\begin{equation}
\label{eq:boxPhiinfty}
\begin{split}
\square_g\Phi&=r^{-1}\left[4D-D'r\right]\partial_r\Phi-2r^{-1}\partial_u\Phi+r^{-1}[-D''r+3D'-2Dr^{-1}]\Phi\\
&+r[D''+D'r^{-1}]\phi,
\end{split}
\end{equation}
and moreover
\begin{equation}
\label{equationforPhiinfty}
\begin{split}
-2\partial_u\partial_r\Phi&=-D\partial_r^2\Phi+r^{-1}\left[2D-2D'r\right]\partial_r\Phi-\slashed{\Delta}\Phi\\
&+r^{-1}[-D''r+3D'-2Dr^{-1}]\Phi+r[D''+D'r^{-1}]\phi,
\end{split}
\end{equation}
which is equivalent to
\begin{equation}
\label{equationforPhiinftyv2}
\begin{split}
2\underline{L}(\partial_r\Phi)&=-r^{-1}\left[2D-2D'r\right]\partial_r\Phi+\slashed{\Delta}\Phi-r^{-1}[-D''r+3D'-2Dr^{-1}]\Phi\\
&-r[D''+D'r^{-1}]\phi.
\end{split}
\end{equation}

\item[\emph{(ii)}]
\begin{equation}
\label{eq:boxtildePhiinfty}
\begin{split}
\square_g\widetilde{\Phi}=&\:r^{-1}(4D-D'r+MD(r-M)^{-1})\partial_r\widetilde{\Phi}-2r^{-1}\partial_u\widetilde{\Phi}-M\slashed{\Delta}\phi\\
&+r^{-1}[-D''r+3D'-2Dr^{-1}-MD(r-M)^{-2}\\
&+M(r-M)^{-1}(D'-Dr^{-1})]\widetilde{\Phi}+[(r-M)D''+D']\phi,
\end{split}
\end{equation}
and moreover
\begin{equation}
\label{equationfortildePhiinfty}
\begin{split}
-2\partial_u\partial_r\widetilde{\Phi}=&\:-D\partial_r^2\widetilde{\Phi}+r^{-1}\left[2D-2D'r+MD(r-M)^{-1}\right]\partial_r\widetilde{\Phi}-M\slashed{\Delta}\phi-\slashed{\Delta}\widetilde{\Phi}\\
&+r^{-1}[-D''r+3D'-2Dr^{-1}-MD(r-M)^{-2}\\
&+M(r-M)^{-1}(D'-Dr^{-1})]\widetilde{\Phi}+[(r-M)D''+D']\phi.
\end{split}
\end{equation}
\end{itemize}
\end{lemma}
\begin{proof}
See Appendix \ref{app:commmultp}.
\end{proof}

Finally, we consider a second-order radiation field, which we denote by $\Phi_{(2)}$ and it is given by
\begin{equation*}
\Phi_{(2)}=(r^2\partial_r)^2(r\psi).
\end{equation*}

\begin{lemma}
\label{lm:commute2time}
Let $\psi$ be a solution to \eqref{waveequation} emanating from initial data given as in Theorem \ref{thm:extuniq} on $(\mathcal{R}, g)$. Then $\Phi_{(2)}=r^2\partial_r\Phi$ satisfies the following equation:
\begin{equation}
\label{eq:Phi2infty}
\begin{split}
\square_g\Phi_{(2)}&=r^{-1}[6D-2D'r]\partial_r\Phi_{(2)}-2r^{-1}\partial_u\Phi_{(2)}+r^{-1}[-6Dr^{-1}-3D'' r+7D']\Phi_{(2)}\\
&+r[-D'''r+2D''+2D'r^{-1}]\Phi+r^3[D'''+4D''r^{-1}+2D'r^{-2}]\phi
\end{split}
\end{equation}
and moreover
\begin{equation}
\label{equationforPhi2}
\begin{split}
2 \partial_u\partial_r\Phi_{(2)}=&\:D\partial_r^2\Phi_{(2)}-[4Dr^{-1}+D']\partial_r\Phi_{(2)}+\slashed{\Delta}\Phi_{(2)}-r^{-1}[-6Dr^{-1}-3D'' r+7D']\Phi_{(2)}\\
&-r[-D'''r+2D''+2D'r^{-1}]\Phi-r^3[D'''+4D''r^{-1}+2D'r^{-2}]\phi,
\end{split}
\end{equation}
which is equivalent to
\begin{equation}
\label{equationforPhi2v2}
\begin{split}
2 \underline{L}(\partial_r\Phi_{(2)})=&\:-[4Dr^{-1}+D']\partial_r\Phi_{(2)}+\slashed{\Delta}\Phi_{(2)}-r^{-1}[-6Dr^{-1}-3D'' r+7D']\Phi_{(2)}\\
&-r[-D'''r+2D''+2D'r^{-1}]\Phi-r^3[D'''+4D''r^{-1}+2D'r^{-2}]\phi.
\end{split}
\end{equation}
\end{lemma}
\begin{proof}
See Appendix \ref{app:commmultp}.
\end{proof}

We can apply the above lemmas together with the energy boundedness assumption of Section \ref{sec:waveassm} to obtain $u$-dependent estimates for the radiation field and its higher-order derivatives.
\begin{proposition}
\label{prop:step0radfields}
Let $\psi$ be a solution to \eqref{waveequation} emanating from initial data given as in Theorem \ref{thm:extuniq} on $(\mathcal{R}, g)$.
\begin{itemize}
\item[\emph{(i)}]
If we assume that
\begin{equation}
\label{eq:assinitialTenergy}
\int_{\Sigma}J^T[\psi]\cdot n_{\Sigma}\,d\mu_{\Sigma}<\infty
\end{equation}
and
\begin{equation}
\label{eq:asslimitphi}
\lim_{r \to \infty }\int_{\s^2}\phi^2\,d\omega\big|_{u'=0}<\infty,\\
\end{equation}
then for all $u\geq 0$, we have along each $\mathcal{N}_u$ that
\begin{equation}
\label{est:0radfieldudep}
\int_{\s^2}\phi^2\,d\omega\big|_{u'=u}\leq \: C_0(u),
\end{equation}
with a $u$-dependent constant $C_0(u)>0$.
\item[\emph{(ii)}]
If we additionally assume that
\begin{equation*}
\sum_{|k|\leq 2}\int_{\Sigma}J^T[\Omega^k\psi]\cdot n_{\Sigma}\,d\mu_{\Sigma}<\infty,
\end{equation*}
and
\begin{align}
\lim_{r \to \infty }\sum_{|k|\leq 2}\int_{\s^2}(\Omega^k\phi)^2\,d\omega\big|_{u'=0}<&\:\infty,\\
\label{eq:asslimitphi1}
\lim_{r\to \infty} \int_{\s^2}\Phi^2\,d\omega\big|_{u'=0}<&\:\infty,
\end{align}
then we also have that
\begin{equation*}
\lim_{r \to \infty }\int_{\s^2} \phi^2\,d\omega\big|_{u'=u}<\infty
\end{equation*}
and
\begin{align}
\label{est:1radfieldudep}
\int_{\s^2}\Phi^2\,d\omega\big|_{u'=u}\leq& \: C_1(u),\\
\label{1radfieldfin}
\lim_{r\to \infty} \int_{\s^2}\Phi^2\,d\omega\big|_{u'=u}<&\:\infty,
\end{align}
with a $u$-dependent constant $C_1(u)>0$. 
\item[\emph{(iii)}]
If we moreover assume that
\begin{equation*}
\sum_{|k|\leq 4}\int_{\Sigma}J^T[\Omega^k\psi]\cdot n_{\Sigma}\,d\mu_{\Sigma}<\infty
\end{equation*}
and
\begin{align}
\lim_{r \to \infty }\sum_{|k|\leq 4}\int_{\s^2}(\Omega^k\phi)^2\,d\omega\big|_{u'=0}<&\:\infty,\\
\lim_{r \to \infty }\sum_{|k|\leq 2}\int_{\s^2}(\Omega^k\Phi)^2\,d\omega\big|_{u'=0}<&\:\infty,\\
\label{eq:asslimitphi2}
\lim_{r \to \infty }\int_{\s^2}\left(r^{-n}\Phi_{(2)}\right)^2\,d\omega\big|_{u'=0}<&\:\infty,
\end{align}
for some $0\leq n\leq 2$, then we additionally have that
\begin{align}
\label{est:2radfieldudep}
\int_{\s^2}\left(r^{-n}\Phi_{(2)}\right)^2\,d\omega\big|_{u'=u} \leq &\: C_2(u),\\
\label{2radfieldfin}
\lim_{r \to \infty }\int_{\s^2}\left(r^{-n}\Phi_{(2)}\right)^2\,d\omega\big|_{u'=u}<&\:\infty,
\end{align}
with a $u$-dependent constant $C_2(u)>0$.
\end{itemize}
\end{proposition}
\begin{proof}
It will be convenient to work in $(u,v,\theta,\varphi)$ coordinates. By the fundamental theorem of calculus and Cauchy--Schwarz, we can estimate
\begin{equation*}
|\psi|^2(u,v,\theta,\varphi)\leq |\psi|^2(0,v,\theta,\varphi)+\int_{0}^u\,du'\cdot \int_0^u (\underline{L}\psi)^2(u',v,\theta,\varphi)\,du'.
\end{equation*}
After integrating over $\s^2$ and using (\ref{eq:Tcurrent2}), we therefore obtain
\begin{equation*}
\int_{\s^2}\psi^2(u,v,\theta,\varphi)\,d\omega \leq \int_{\s^2}\psi^2(0,v,\theta,\varphi)\,d\omega+ u\int_{\mathcal{I}_v(0,u)}J^T[\psi]\cdot \underline{L}\,r^2d\omega du.
\end{equation*}

Boundedness of the energy on the right-hand side follows by application of the divergence theorem together with the Killing property of $T$ and assumption \eqref{eq:assinitialTenergy}.

Similarly, we can apply the fundamental theorem of calculus to obtain
\begin{equation}
\label{eq:phi0fundthmcalc}
\phi(u,v,\theta,\varphi)=\phi(0,v,\theta,\varphi)+\int_{0}^u \underline{L}(r\psi)(u',v,\theta,\varphi)\,du'.
\end{equation}

We can further estimate by Cauchy--Schwarz
\begin{equation*}
\int_{\s^2}\left(\int_{0}^u |\underline{L}(r\psi)|\,du'\right)^2\,d\omega \leq u \int_{\mathcal{I}_v(0,u)}J^T[\psi]\cdot \underline{L}\,r^2d\omega du+\int_{\s^2}\left(\int_0^uD |\psi|\,du'\right)^2\,d\omega.
\end{equation*}
After applying the Sobolev inequality \eqref{lm:sobolevs2}, the estimate for $\psi$ above and the assumptions on the initial data imply that there exists a constant $C_0(u)>0$, such that
\begin{equation*}
\int_{\s^2}\phi^2(u,v,\theta,\varphi)\,d\omega\leq C_0(u)
\end{equation*}
for all $u\geq 0$.

We apply the fundamental theorem of calculus again to obtain
\begin{equation}
\label{eq:fundpartialrphi}
\partial_r\phi(u,v,\theta,\varphi)=\partial_r\phi(0,v,\theta,\varphi)+\int_0^u \underline{L}(\partial_r\phi)(u',v,\theta,\varphi)\,du',
\end{equation}
where $\partial_r=2D^{-1}\partial_v$. By (\ref{eq:sDeltapsiinfty}) we can estimate
\begin{equation*}
|\underline{L}(\partial_r\phi)|\leq CMr^{-2}|\partial_r\phi|+Cr^{-3}|\phi|+Cr^{-2}|\slashed{\Delta}_{\s^2} \phi|.
\end{equation*}
Consequently,
\begin{equation*}
\underline{L}\left(\int_{\s^2}(\partial_r\phi)^2\,d\omega \right)\leq C\int_{\s^2}r^{-2}(\partial_r\phi)^2\,d\omega+\int_{\s^2}r^{-4}\phi^2+r^{-2}(\slashed{\Delta}_{\s^2} \phi)^2\,d\omega.
\end{equation*}

We use \eqref{eq:angmomop1} and \eqref{eq:angmomop2} to estimate the $(\slashed{\Delta}_{\s^2} \phi)^2$ integral by using the above estimates for $\psi$ applied also to $\Omega \psi$ and $\Omega^2\psi$. We then apply a standard Gr\"onwall inequality to obtain
\begin{equation*}
\int_{\s^2}(\partial_r\phi)^2\,d\omega\big|_{u'=u}\leq C\left(1+\int_{\s^2}(\partial_r\phi)^2\,d\omega\big|_{u'=0}\right).
\end{equation*}

Subsequently, we can conclude that
\begin{equation*}
\lim_{v\to \infty }\underline{L}(\partial_r\phi)(u,v,\theta,\varphi)=0
\end{equation*}
for all $u\geq 0$ and therefore, by the bounded convergence theorem applied to \eqref{eq:fundpartialrphi},
\begin{equation*}
\lim_{r\to \infty }\int_{\s^2}(\partial_r\phi)^2(u,v,\theta,\varphi)\,d\omega=0.
\end{equation*}
More generally, we have that
\begin{equation}
\label{eq:Lbarrndrphi}
\underline{L}(r^n\partial_r\phi)=\frac{1}{2}\left(-nr^{-1}D+D'\right)r^n\partial_r\phi-\frac{1}{2}D'r^{n-1}\phi-\frac{1}{2}r^{n-2}\slashed{\Delta}_{\s^2} \phi,
\end{equation}
for $0\leq n\leq 2$. By similar arguments we can therefore similarly conclude that
\begin{equation}
\label{est:kderphiudep}
\int_{\s^2}r^{2n}(\partial_r\phi)^2\,d\omega\big|_{u'=u}\leq C\left(1+\int_{\s^2}r^{2n}(\partial_r\phi)^2\,d\omega\big|_{u'=0}\right).
\end{equation}
for $n\leq 2$.

We moreover have that $\partial_r\phi$ is integrable in $r$ for $k\leq 2$ by applying (\ref{est:kderphiudep}) with $n= 2$. Hence,
\begin{equation*}
\lim_{r\to\infty} \int_{\s^2}\phi^2(u,v,\theta,\varphi)\,d\omega<\infty.
\end{equation*}
The above limit and (\ref{est:kderphiudep}) for $n\leq 2$ allow us to apply the bounded convergence theorem once more and conclude that
\begin{equation*}
\lim_{r\to\infty}\int_{\s^2}r^{4}(\partial_r\phi)^2\,d\omega\big|_{u'=u}<\infty.
\end{equation*}

In order to prove (\ref{est:2radfieldudep}), we proceed similarly by using the fundamental theorem of calculus together with (\ref{equationforPhiinftyv2}) and the results above. We obtain (\ref{2radfieldfin}) by applying the bounded convergence theorem, as above.
\end{proof}
\begin{proposition}
\label{prop:hostep0radfields}
Let $\psi$ be a solution to \eqref{waveequation} as in Proposition \ref{prop:step0radfields}. Fix $n\geq 1$ and $0\leq k\leq 2$. If we assume that \begin{equation*}
\sum_{|l|\leq 4+2n}\int_{\Sigma}J^T[\Omega^l\psi]\cdot n_{\Sigma}\,d\mu_{\Sigma}<\infty,
\end{equation*}
and
\begin{align*}
\lim_{r \to \infty }\sum_{|l|\leq 4+2n}\int_{\s^2}(\Omega^l\phi)^2\,d\omega\big|_{u'=0}<&\:\infty,\\
\lim_{r \to \infty }\sum_{|l|\leq 2+2n}\int_{\s^2}(\Omega^l\Phi)^2\,d\omega\big|_{u'=0}<&\:\infty,\\
\lim_{r \to \infty }\sum_{|l|\leq 2n}\int_{\s^2}\left(r^{-k}\Omega^l\Phi_{(2)}\right)^2\,d\omega\big|_{u'=0}<&\:\infty,
\end{align*}
and additionally
\begin{equation*}
\lim_{r \to \infty }\sum_{|l|\leq 2n-2s}\int_{\s^2}\left(r^{m}\partial_r^s\Omega^l\Phi_{(2)}\right)^2\,d\omega\big|_{u'=0}<\infty,
\end{equation*}
for each $1\leq s\leq n$ and $0\leq m\leq s+1-k$, then we have that
\begin{align*}
\lim_{r \to \infty }\sum_{|l|\leq 2n-2s}\int_{\s^2}\left(r^{m}\partial_r^s\Omega^l\Phi_{(2)}\right)^2\,d\omega\big|_{u'=u}\leq&\: C_{3;s}(u),\\
\lim_{r \to \infty }\sum_{|l|\leq 2n-2s}\int_{\s^2}\left(r^{m}\partial_r^s\Omega^l\Phi_{(2)}\right)^2\,d\omega\big|_{u'=u}<&\: \infty,
\end{align*}
for each $1\leq s\leq n$ and $0\leq m\leq s+1-k$, with $C_{3;s}(u)>0$ a $u$-dependent constant.
\end{proposition}
\begin{proof}
The proof proceeds analogously to the proof of Proposition \ref{prop:step0radfields} by considering $\square_g(r^{2-k}\partial_r\Phi_{(2)})$ and $\square_g(\partial_r^{s-1}(r^{2-k}\partial_r\Phi_{(2)}))$.
\end{proof}
\begin{remark}
Proposition \ref{prop:step0radfields} and Proposition \ref{prop:hostep0radfields} should be thought of as ``preliminary'' pointwise estimates for the radiation fields at infinity. In particular, we will obtain in Section \ref{sec:pointdecayradfield} far more refined estimates for $\phi$ that are uniform in $u$ and even provide $u$-decay for $\phi$.
\end{remark}

\subsection{The first Newman--Penrose constant}
\label{sec:1stnpconstant}

Next, we can consider the quantity 
\begin{equation}\label{firstNP}
r^2 \partial_r (r\psi_0 )\big|_{\mathcal{I}^+}(u) \doteq  \lim_{r \to \infty} r^2 \partial_r (r\psi_0 )(u,r) \mbox{ for $ \psi_0 = \frac{1}{4\pi}\int_{\mathbb{S}^2} \psi \, d\omega ,$}
\end{equation}
which is known as the first Newman--Penrose \textit{quantity} (see \cite{NP1,np2}). By Proposition \ref{prop:step0radfields}, it is well-defined for suitably decaying data. In the proposition below, we show that the first Newman--Penrose quantity is actually a \emph{constant}, i.e.\ it is independent of $u$ and determined by initial data on $\Sigma$.
\begin{proposition}\label{consNP}
Let $\psi$ be a solution to \eqref{waveequation} emanating from initial data given as in Theorem \ref{thm:extuniq} on a spacetime $(\mathcal{R}, g)$. Then the first Newman--Penrose quantity defined in \eqref{firstNP} is independent of $u$.
\end{proposition}
\begin{proof}
By the results of Lemma \ref{prop:step0radfields} together with \eqref{eq:Lbarrndrphi} with $n=2$, we have that
\begin{align*}
|\underline{L}(r^2\partial_r\phi_0)|\leq &\:C(u)\quad \textnormal{for}\,r\geq R,\\
\lim_{r\to \infty}\underline{L}(r^2\partial_r\phi_0)(u,r)=&\:0\quad \textnormal{for all}\,u\geq 0.
\end{align*}
The above equality holds because the only non-trivial term on the right-hand side of \eqref{eq:Lbarrndrphi} as $r\to\infty$ is $\slashed{\Delta}_{\s^2} \phi$ which vanishes since $\psi$ is assumed to be spherically symmetric. By the bounded convergence theorem, we therefore have that
\begin{equation*}
r^2 \partial_r (r\psi_0 )\big|_{\mathcal{I}^+}(u)=\lim_{r\to \infty}r^2\partial_r\phi_0(u,r)=\lim_{r\to \infty}r^2\partial_r\phi_0(0,r)=r^2 \partial_r (r\psi_0 )\big|_{\mathcal{I}^+}(0),
\end{equation*}
for all $u\geq 0$.
\end{proof}

Let us now define
\begin{equation*}
I_0[\psi]\doteq \lim_{r\to \infty}r^2 \partial_r (r\psi_0 )\Big|_{\Sigma}.
\end{equation*}
Then, in light of Proposition \ref{firstNP}, we have that
\begin{equation*}
r^2 \partial_r (r\psi_0 )\big|_{\mathcal{I}^+}(u)=I_0[\psi]
\end{equation*}
for all $u\geq 0$. We therefore refer to $I_0[\psi]$ as the first Newman--Penrose constant.

We conclude this section by stating an auxiliary result that can be applied to solutions $\psi=\psi_0$ to \eqref{waveequation} that emanate from data for which $I_0 [ \psi ] = 0$.

\begin{proposition}\label{consNP0}
Let $\psi$ be a solution to \eqref{waveequation} emanating from initial data given as in Theorem \ref{thm:extuniq} on $(\mathcal{R}, g)$. Let $0\leq n\leq 3$ and assume additionally that
\begin{equation}\label{ass:r3}
 \left| r^n \partial_r (r \psi_0 ) \right|_{\mathcal{N}} \leq C , 
\end{equation}
for some constant $C < \infty$ where $\psi_0 = \int_{\s^2} \psi$, then we have that
\begin{equation}\label{est:r3}
 \left| r^n \partial_r (r \psi_0 ) \right|_{\mathcal{N}_u} \leq K(u) , 
\end{equation}
for all $0 \leq u < \infty$, where $K(u) < \infty$.
\end{proposition}
\begin{proof}
The estimate \eqref{est:r3} follows analogously to the results in Lemma \ref{prop:step0radfields}. We apply the results of Lemma \ref{prop:step0radfields} together with \eqref{eq:Lbarrndrphi} with $0\leq n\leq 3$. Here, it is critical that the term $r^{n-2}\slashed{\Delta}_{\s^2}\phi$ that appears on the right-hand side of \eqref{eq:Lbarrndrphi} vanishes for spherically symmetric $\psi$. The estimate \eqref{est:r3} then follows from the fundamental theorem of calculus together with a Gr\"onwall inequality.
\end{proof}
\section{Commutator estimates for $\psi_{1}=\psi-\frac{1}{4\pi}\int_{\mathbb{S}^{2}}\psi$}
\label{sec:hierpsi1}
In this section, we will derive new hierarchies of $r^p$-weighted estimates for $\psi_{1}=\psi-\frac{1}{4\pi}\int_{\mathbb{S}^{2}}\psi\,d\omega$, improving upon the Dafermos--Rodnianski hierarchy from Section \ref{sec:dafrodhier} and also the hierarchies obtained in \cite{volker1,moschidis1}. \textbf{Unless specifically stated otherwise, we will denote $\psi_1$ by $\psi$ in this section.}

\subsection{The hierarchy of estimates for $\psi_{1}$}
\label{sec:hierpsi1a}
In this section, we will use the vector field $r^2\partial_r$ as a commutation vector field to derive an additional hierarchy of $r^p$-weighted estimates for $r^2\partial_r(r\psi_1)$ and for $(r^2\partial_r)^2(r\psi_1)$. 

\begin{proposition}[\textbf{$r^p$-weighted estimates for $r^2\partial_r(r\psi_1)$}]
\label{prop:rpPhiv1}
Let $\psi$ be a solution to \eqref{waveequation} emanating from initial data given as in Theorem \ref{thm:extuniq} on $(\mathcal{R}, g)$.

Take $p\in (-4,2]$. Denote $$\Phi=r^2\partial_r\phi=r^2\partial_r(r\psi)$$ and assume that
\begin{equation*}
\sum_{|k|\leq 2}\int_{\Sigma}J^T[\Omega^k\psi]\cdot n_{\Sigma}\,d\mu_{\Sigma}<\infty,
\end{equation*}
and
\begin{align*}
\lim_{r \to \infty }\sum_{|k|\leq 2}\int_{\s^2}(\Omega^k\phi)^2\,d\omega\big|_{u'=0}<&\:\infty,\\
\lim_{r\to \infty} \int_{\s^2}\Phi^2\,d\omega\big|_{u'=0}<&\:\infty.
\end{align*}

Then there exists an $R>0$ such that for any $0\leq u_1<u_2$
\begin{equation}
\label{rphierpsilgeq1}
\begin{split}
\int_{\mathcal{N}_{u_2}}& r^p(\partial_r\Phi)^2\, d\omega dr+\int_{\mathcal{A}_{u_1}^{u_2}} (p+4)r^{p-1}(\partial_r\Phi)^2+(2-p)r^{p-1}|\snabla \Phi|^2\,d\omega drdu \\
\leq&\: C\int_{\mathcal{N}_{u_1}} r^p(\partial_r\Phi)^2\,d\omega dr+C\sum_{l\leq 1}\int_{\Sigma_{u_1}}J^T[T^{l}\psi]\cdot n_{{u_1}}\,d\mu_{\Sigma_{u_1}},
\end{split}
\end{equation}
where $C\doteq C(D,R)>0$ is a constant and we can take $R=(p+4)^{-1}R_0(D)>0$, with $R_0(D)>0$ a constant.
\end{proposition}
\begin{proof}
Let $\chi: (0,\infty)\to \R$ be a smooth cut-off function, such that
\begin{align*}
\chi(r) =&\: 0 \quad \textnormal{for}\: r\leq R,\\
\chi(r)=&\: 1 \quad \textnormal{for}\: r\geq R+1.\\
\end{align*}
Consequently $\chi \Phi (R)=0$, and $\chi \Phi (r)=\Phi(r)$ for $r\geq R+1$. Consider the vector field multiplier $V=r^{p-2}\partial_r$. We apply the divergence theorem in the region $\mathcal{A}_{u_1}^{u_2}$ on $\chi \Phi $ to obtain
\begin{equation*}
\begin{split}
\int_{\mathcal{N}_{u_2}}& J^V[\chi \Phi]\cdot L\, r^2 d\omega dr+\int_{\mathcal{I}^+}J^V[\chi \Phi]\cdot \underline{L}\, r^2d\omega du+\int_{\mathcal{A}_{u_1}^{u_2}} \textnormal{div}J^V[\chi \Phi]\,r^2d\omega drdu\\
&=\int_{\mathcal{N}_{u_1}}J^V[\chi \Phi]\cdot L\,r^2d\omega dr,
\end{split}
\end{equation*}
where we used that the boundary term along the hypersurface $\{r=R\}$ vanishes due to the choice of cut-off $\chi$. In Appendix \ref{app:commmultp} it is shown that
\begin{align*}
r^2 J^V[\chi \Phi ]\cdot L&=r^{p}D^2(\partial_r(\chi\Phi ))^2,\\
r^2 J^V[\chi \Phi ]\cdot \underline{L}\big|_{\mathcal{I}^+}&=r^{p}|\snabla \Phi|^2.
\end{align*}
We split
\begin{equation*}
\textnormal{div}J^V[\chi \Phi ]=K^V[\chi \Phi ]+\mathcal{E}^V[\chi \Phi ].
\end{equation*}
Recall from Appendix \ref{app:commmultp} that we have
\begin{equation}
\label{eq:KVchiphi0}
K^V[\chi \Phi ]=\frac{1}{2}r^{p-3}\left[D(p-4)-D'r\right](\partial_r (\chi \Phi ))^2+2r^{p-3}\partial_r(\chi \Phi )\partial_u (\chi \Phi ) .
\end{equation}
We also have that
\begin{equation*}
\mathcal{E}^V[\chi \Phi ]=V(\chi \Phi )\square_g(\chi \Phi )=V(\chi \Phi )\square_g (\chi\Phi ).
\end{equation*}
From (\ref{eq:boxPhiinfty}) it therefore follows that
\begin{equation*}
\begin{split}
\mathcal{E}^V[\chi \Phi]=&\:r^{p-3}\left[4D-D'r\right](\partial_r (\chi \Phi))^2-2r^{p-3}\partial_u(\chi\Phi)\partial_r(\chi\Phi)\\
&+r^{p-3}[-D''r+3D'-2Dr^{-1}]\chi\Phi\partial_r(\chi\Phi)+r^{p-1}[D''+D'r^{-1}]\chi\phi\partial_r(\chi\Phi)\\
&+\sum_{|\alpha_1|+|\alpha_2|\leq 1}\mathcal{R}_{\chi}[\partial^{\alpha_1}\Phi\cdot \partial^{\alpha_2}\Phi],
\end{split}
\end{equation*}
where we use the notation $\mathcal{R}_{\chi} [f]$ to denote all terms that are linear in $f$, multiplied by factors depending non-trivially on the derivatives $\chi'$ or $\chi''$.

Hence,
\begin{equation*}
\begin{split}
\textnormal{div}J^V[\chi \Phi]=&\:\frac{1}{2}r^{p-3}\left[(p+4)D-3D'r\right](\partial_r (\chi \Phi))^2+\frac{1}{2}(2-p)r^{p-3}|\snabla(\chi\Phi)|^2\\
&+ r^{p-3}[-D''r+3D'-2Dr^{-1}]\chi\Phi\partial_r(\chi\Phi)+r^{p-1}[D''+D'r^{-1}]\chi \phi\partial_r(\chi \Phi)\\
&+\sum_{|\alpha_1|\leq 1,\,|\alpha_2|\leq 1}\mathcal{R}_{\chi}[\partial^{\alpha_1}\Phi\cdot \partial^{\alpha_2}\Phi].
\end{split}
\end{equation*}

We can write 
\begin{equation*}
r^2\textnormal{div}J^V[\chi \Phi]=J_0+J_1+J_2+\sum_{|\alpha_1|+|\alpha_2|\leq 1}\mathcal{R}_{\chi}[\partial^{\alpha_1}\Phi\cdot \partial^{\alpha_2}\Phi],
\end{equation*}
with
\begin{align*}
J_0[\chi \Phi]&\doteq\frac{1}{2}r^{p-1}\left[(p+4)D-3D'r\right](\partial_r(\chi \Phi))^2+\frac{1}{2}r^{p-3}(2-p)r^2|\snabla(\chi \Phi)|^2,\\
J_1[\chi \Phi]&\doteq r^{p-1}[-D''r+3D'-2Dr^{-1}]\chi \Phi\partial_r(\chi \Phi),\\
J_2[\chi \Phi]&\doteq r^{p+1}[D''+D'r^{-1}]\chi\phi\partial_r(\chi \Phi).
\end{align*}

As $\chi^{(k)}$ is supported in the region $R\leq r\leq R+1$ for $k\geq 1$, $\mathcal{R}_{\chi}[\partial^{\alpha_1}\Phi\cdot \partial^{\alpha_2}\Phi]$ must also only be supported in the region $R\leq r\leq R+1$.  By the Morawetz estimate (\ref{ass:morawetzlocal}), we can therefore estimate
\begin{equation*}
\begin{split}
\int_{\mathcal{A}^{u_2}_{u_1}}\sum_{|\alpha_1|\leq 1,\,|\alpha_2|\leq 1}\mathcal{R}_{\chi}[\partial^{\alpha_1}\Phi\cdot \partial^{\alpha_2}\Phi]\,d\omega dr du\leq&\: C\sum_{|\alpha|\leq 2}\int_{\mathcal{A}^{u_2}_{u_1}\cap \{r\leq R+1\}}(\partial^{\alpha}\psi)^2\,d\omega dr du\\
\leq &\: C\sum_{l\leq 1}\int_{\Sigma_{u_1}}J^T[T^{l}\psi]\cdot n_{\Sigma_{u_1}}\,d\mu_{\Sigma_{u_1}}.
\end{split}
\end{equation*}

Note that $J_0$ has a positive (good) sign if $-4<p\leq 2$ and $R\geq (p+4)^{-1}R_0(D)$, for $R_0(D)>0$ a suitably large constant. Moreover,
\begin{align*}
D&=1-\frac{2M}{r}+O_3(r^{-1-\beta}),\\
D'&=\frac{2M}{r^2}+O_2(r^{-1-\beta}),\\
D''&=-\frac{4M}{r^3}+O_1(r^{-3-\beta}).
\end{align*}

The leading order term of $J_1$ is therefore $-2r^{p-2}\chi \Phi\partial_r(\chi \Phi)$. We integrate by parts to obtain
\begin{equation}
\label{estJ1infty}
\begin{split}
-\int_{\mathcal{A}^{u_2}_{u_1}}2r^{p-2}\chi \Phi\partial_r(\chi \Phi)\,d\omega drdu=&-\int_{\mathcal{I}^+\cap\{u_1\leq u \leq u_2\}}r^{p-2}(\chi \Phi)^2\,d\omega du\\
&+(p-2)\int_{\mathcal{A}^{u_2}_{u_1}}r^{p-3}(\chi \Phi)^2\,d\omega dudr.
\end{split}
\end{equation}

Recall also that
\begin{equation*}
\int_{\mathcal{I}^+} J^V[\chi \Phi]\cdot \underline{L}\,r^2d\omega du =\frac{1}{2}\int_{\mathcal{I}^+} r^{p-2}|\snabla (\chi \Phi)|^2d\omega du.
\end{equation*}
By the Poincar\'e inequality (\ref{eq:poincare2}), we can further estimate
\begin{equation*}
\frac{1}{2}\int_{\mathcal{I}^+} r^{p}|\snabla (\chi\Phi_{\ell\geq L})|^2d\omega du \leq \frac{1}{2}L(L+1)\int_{\mathcal{I}^+}r^{p-2}(\chi\Phi_{\ell\geq L})^2\,d\omega du.
\end{equation*}
Hence, for $\psi=\psi_{\ell\geq 1}$ the boundary term in (\ref{estJ1infty}) gets absorbed by the flux term along $\mathcal{I}^+$. Moreover, if $\psi=\psi_{\ell=1}$ an \emph{exact} cancellation occurs! 

We can also apply the Poincar\'e inequality to the bulk term in (\ref{estJ1infty}) to show that it gets absorbed by the $\snabla (\chi \Phi)$ term in $J_0$ if $p\leq 2$. If $\psi=\psi_{\ell=1}$, an exact cancellation also occurs for the bulk term.

The remaining terms can be estimated by applying a Cauchy--Schwarz inequality:
\begin{equation*}
\begin{split}
\int_{\mathcal{A}^{u_2}_{u_1}}{O}(r^{p-3})\chi\Phi\partial_r(\chi\Phi)\,d\omega dudr\leq&\: \epsilon \int_{\mathcal{A}^{u_2}_{u_1}}r^{p-1} (p+4)(\partial_r(\chi\Phi))^2\,d\omega dudr\\
&+C_{\epsilon}(p+4)^{-1} \int_{\mathcal{A}^{u_2}_{u_1}}r^{p-5}(\chi \Phi)^2\,d\omega dudr,
\end{split}
\end{equation*}
where the first term on the right-hand side can be absorbed into $J_0$ and the second term can be estimated by applying the Hardy inequality (\ref{eq:Hardy1}):
\begin{equation}
\label{eq:auxeqrpPhi}
\begin{split}
 (p+4)^{-1}\int_{\mathcal{A}^{u_2}_{u_1}}r^{p-5}(\chi\Phi)^2\,d\omega dudr\leq &\:C(p-4)^{-2}(p+4)^{-2}\int_{\mathcal{A}^{u_2}_{u_1}}(p+4)r^{p-3}(\partial_r(\chi\Phi))^2\,d\omega dudr,
 \end{split}
\end{equation}
for $p<4$, using that Proposition \ref{prop:step0radfields} implies that $\lim_{r\to \infty}r^{p-4} \int_{\s^2}\Phi^2(u,r,\theta,\varphi)\,d\omega=0$ if we use the assumptions on the initial data in the statement of the current proposition.

The right-hand side of \eqref{eq:auxeqrpPhi} can be absorbed into $J_0$ if $R>(p+4)^{-1}(p-4)^{-1}R_0$, with $R_0=R_0(D)>0$ suitably large.

We are left with estimating $J_2$. We can write
\begin{equation*}
J_2[\chi\Phi]=(-2Mr^{p-2}+{O}(r^{p-2-\beta}))\chi \phi\partial_r(\chi \Phi).
\end{equation*}
First, note that the integral of ${O}(r^{p-2-\beta})\chi \phi\partial_r(\chi \Phi)$ can be easily estimated via Cauchy--Schwarz and \eqref{eq:Hardy1} and absorbed into $J_0$.

We integrate the leading-order term in $J_2$ by parts to obtain
\begin{equation}
\label{eq:maintermJ2}
\begin{split}
\int_{\mathcal{A}^{u_2}_{u_1}}&-2Mr^{p-2}\chi \phi\partial_r(\chi \Phi)\,d\omega drdu=-2M\int_{\mathcal{I}^+}r^{p-2}\phi \Phi\,d\omega du\\
&-2M(2-p)\int_{\mathcal{A}^{u_2}_{u_1}}r^{p-3}\chi \phi\chi \Phi\,d\omega drdu+\int_{\mathcal{A}^{u_2}_{u_1}}2Mr^{p-2}\partial_r(\chi \phi)\chi \Phi\,d\omega drdu\\
=&\:-2M\int_{\mathcal{I}^+}r^{p-2}\phi \Phi\,d\omega du-2M(2-p)\int_{\mathcal{A}^{u_2}_{u_1}}r^{p-3}\chi \phi\chi \Phi\,d\omega drdu\\
&+\int_{\mathcal{A}^{u_2}_{u_1}}2Mr^{p-4}(\chi \Phi)^2\,d\omega dr du\\
&+\int_{\mathcal{A}^{u_2}_{u_1}}\mathcal{R}_{\chi}[\phi \cdot \partial_r\phi]\,d\omega drdu.
\end{split}
\end{equation}
The third term on the very right-hand side of \eqref{eq:maintermJ2} has a good sign and the fourth term can be estimated using (\ref{ass:morawetzlocal}). In order to estimate the remaining terms, we will use that there exists a unique smooth function $f$ on $\s^2$ such that $\slashed{\Delta}_{\s^2} f=\Phi$ since $\int_{\s^2}\Phi\,d\omega=0$. 

Therefore, we can integrate by parts on $\s^2$ to obtain
\begin{equation}
\label{eq:boundaryI+}
\int_{\s^2}\phi \cdot \Phi\,d\omega =\int_{\s^2}\slashed{\Delta}_{\s^2}\phi\cdot f\,d\omega.
\end{equation}
Now we use 
\begin{equation}
\label{eq:Deltaphi}
\begin{split}
r^2\slashed{\Delta} \phi=2 \underline{L}\Phi+(2Dr^{-1}-D')\Phi_1+D'r\phi,
\end{split}
\end{equation}
where $\underline{L}=\partial_u-\frac{D}{2}\partial_r$ is the ingoing null generator (i.e.\ $\partial_u$ in double null Eddington--Finkelstein coordinates).

Using that $\phi$, $r^2 \slashed{\Delta}\phi$, $\Phi$ and $\underline{L}\phi$ have a well-defined limit on $\mathcal{I}^+$ by Proposition \ref{prop:step0radfields}, it follows that

\begin{equation*}
r^2\slashed{\Delta} {\phi}|_{\mathcal{I}^+}=2 \underline{L}{\Phi}|_{\mathcal{I}^+}.
\end{equation*}
Hence, we can use (\ref{eq:boundaryI+}) and \eqref{eq:Deltaphi} to obtain
\begin{equation*}
\begin{split}
-2M\int_{\mathcal{I}^+}r^{p-2}\phi \Phi\,d\omega du=&-4M\int_{\mathcal{I}^+} r^{p-2}\underline{L}(\slashed{\Delta}_{\s^2}f) f\,d\omega du\\
=&2M\int_{\mathcal{I}^+} \underline{L}(r^{p-2}r^2|\slashed{\nabla} f|^2)\,d\omega du-(2-p)2M\int_{\mathcal{I}^+}r^{p-3}r^2|\slashed{\nabla} f|^2\,d\omega du\\
=&2M\lim_{r\to\infty}\int_{\s^2}r^{p-2}r^2|\slashed{\nabla} f|^2(u_2,r,\theta,\varphi)\,d\omega \\
&-2M\lim_{r\to\infty}\int_{\s^2}r^{p-2}r^2|\slashed{\nabla} f|^2(u_1,r,\theta,\varphi)\,d\omega \\
&-(2-p)2M\int_{\mathcal{I}^+}r^{p-3}r^2|\slashed{\nabla} f|^2\,d\omega du.
\end{split}
\end{equation*}

The first term on the very right-hand side has the correct sign and the third term has to vanish for $p<3$ because $\Phi$ is finite at $\mathcal{I}^+$. To estimate the second term, we apply the Poincar\'e inequality \eqref{eq:poincare3} followed by the fundamental theorem of calculus together with Cauchy--Schwarz:
\begin{equation}
\label{est:boundest1}
\begin{split}
2M\lim_{r\to\infty}\int_{\s^2}&r^{p-2}\chi^2r^2|\slashed{\nabla} f|^2(u_1,r,\theta,\varphi)\,d\omega\leq M\lim_{r\to\infty}\int_{\s^2}r^{p-2}\chi^2\Phi^2(u_1,r,\theta,\varphi)\,d\omega\\
 =&\:M\int_{\mathcal{N}_{u_1}}\partial_r(r^{p-2}\chi^2\Phi^2)\,d\omega dr\\
\leq&\frac{M}{2}\int_{\mathcal{N}_{u_1}}r^p(\partial_r(\chi\Phi))^2\,d\omega dr+\frac{M}{2}\int_{\mathcal{N}_{u_1}}r^{p-4}\chi^2\Phi^2\,d\omega dr\\
&-M(2-p)\int_{\mathcal{N}_{u_1}}r^{p-3}\chi^2\Phi^2\,d\omega dr\\
\leq \:&C(3-p)^{-1}\int_{\mathcal{N}_{u_1}}r^2(\partial_r(\chi\Phi))^2\,d\omega dr-M(2-p)\int_{\mathcal{N}_{u_1}}r^{p-3}\chi^2\Phi^2\,d\omega dr,
\end{split}
\end{equation}
where we applied \eqref{eq:Hardy1} in the last inequality. The first term on the very right-hand side of (\ref{est:boundest1}) is a constant multiple of a flux term that already appears in the divergence identity for $J^V[\Phi]$. The second term has a good sign for $p\leq 2$.

We can similarly estimate the second term on the very right-hand side of (\ref{eq:maintermJ2})
\begin{equation}
\label{eq:additionalpneq2}
\begin{split}
-2M(2-p)\int_{\mathcal{A}^{u_2}_{u_1}}r^{p-3}\chi^2 \phi \Phi\,d\omega dudr=&-2M(2-p)\int_{\mathcal{A}^{u_2}_{u_1}}r^{p-3}r^2 \chi^2 \slashed{\Delta}\phi f\,d\omega dudr\\
=&M(2-p)\int_{\mathcal{A}^{u_2}_{u_1}}r^{p-3}\underline{L}(\chi r^2|\slashed{\nabla} f|^2))\,d\omega dudr\\
&+M(2-p)\int_{\mathcal{A}^{u_2}_{u_1}}(2Dr^{p-4}-D'r^{p-3})\chi^2r^2|\slashed{\nabla} f|^2\,d\omega dudr\\
&+M(2-p)\int_{\mathcal{A}^{u_2}_{u_1}}r^{p-2}D'\chi^2 \phi f\,d\omega dudr\\
&+\sum_{|\alpha_i|\leq 1}\int_{\mathcal{A}^{u_2}_{u_1}}\mathcal{R}_{\chi}[\partial^{\alpha_1}\phi\cdot \partial^{\alpha_2}\phi]\,d\omega drdu.
\end{split}
\end{equation}
The second term on the right-hand side of (\ref{eq:additionalpneq2}) has a good sign if $p\leq 2$. The third term can be estimated using Cauchy--Schwarz, \eqref{eq:poincare2} and \eqref{eq:poincare3}:
\begin{equation*}
\int_{\mathcal{A}^{u_2}_{u_1}}r^{p-2}D'\chi^2\phi f\,d\omega dudr\leq \epsilon \int_{\mathcal{A}^{u_2}_{u_1}}r^{p-4}\chi^2\Phi^2\,d\omega dudr+C_{\epsilon}\int_{\mathcal{A}^{u_2}_{u_1}}r^{p-4}\chi^2\phi^2\,d\omega dudr.
\end{equation*}
We absorb the $\Phi^2$ term into the second term on the right-hand side of (\ref{eq:additionalpneq2}) and estimate the $\phi^2$ term by applying the Hardy inequality \eqref{eq:Hardy1} twice.

By plugging all the estimates above into the divergence identity for $\Phi$, we obtain
\begin{equation}
\label{eq:rpestPhichiPhi}
\begin{split}
\int_{\mathcal{N}_{u_2}} r^p(\partial_r(\chi\Phi))^2\, d\omega dr+\int_{\mathcal{A}_{u_1}^{u_2}} (p+4)r^{p-1}(\partial_r(\chi\Phi))^2\,d\omega drdu \leq&\: C\int_{\mathcal{N}_{u_1}} r^p(\partial_r(\chi\Phi))^2\,d\omega dr\\
&+C\sum_{k\leq 1}\int_{\Sigma_{u_1}}J^T[T^{k}\psi]\cdot n_{{u_1}}\,d\mu_{\Sigma_{u_1}}.
\end{split}
\end{equation}
Note that we can add the term $(2-p)r^{p-1}|\snabla \chi \Phi|^2$ inside the spacetime integral on the left-hand side. Indeed, for $\Phi_{\ell \geq 2}$ this term arises from $J_0$, whereas for $\Phi_{\ell=1}$ we can apply the equality \eqref{eq:poincare1} together with the Hardy inequality \eqref{eq:Hardy1} to add $r^{p-1}|\snabla \chi \Phi_{\ell=1}|^2$ to the spacetime term.

In order to remove the cut-off $\chi$ in \eqref{eq:rpestPhichiPhi} we note that we can apply Hardy's inequality \eqref{eq:Hardy1} to estimate the following
\begin{equation*}
\begin{split}
\int_{R}^{\infty}r^p(\partial_r(\chi\Phi ))^2\, dr\big|_{u=u'}\leq &\: 2\int_{R}^{\infty}r^p\chi^2 (\partial_r\Phi )^2\, dr\big|_{u=u'} +2\int_{R}^{R+1}r^p\chi'^2\Phi^2\, dr\big|_{u=u'}\\
\leq &\:2\int_{R}^{\infty}r^p (\partial_r\Phi )^2\, dr\big|_{u=u'}+ C\int_{R}^{\infty}\psi^2+(\partial_r\psi)^2+(\partial_r^2\psi)^2\, dr\big|_{u=u'}.\\
\leq &\:2\int_{R}^{\infty}r^p (\partial_r\Phi )^2\, dr\big|_{u=u'}+ C\sum_{l\leq 1}\int_{\Sigma_{u'}}J^T[\partial_r^l\psi ]\cdot n_{u'} \,d\mu_{{u'}}.
\end{split}
\end{equation*}
Note that we can replace the term $J^T[\partial_r^l\psi ]$ above by $J^T[T^l\psi ]$ by applying a standard elliptic estimate.

Similarly, we can estimate
\begin{equation*}
\begin{split}
\int_{R}^{\infty}r^p (\partial_r\Phi )^2\, dr\big|_{u=u'}=&\:\int_{R+1}^{\infty}r^p (\partial_r(\chi \Phi ))^2\, dr\big|_{u=u'}+\int_{R}^{R+1}r^p (\partial_r\Phi )^2\, dr\big|_{u=u'}\\
\leq &\:\int_{R+1}^{\infty}r^p (\partial_r(\chi \Phi ))^2\, dr\big|_{u=u'}+C\sum_{l\leq 1}\int_{\Sigma_{u'}}J^T[\partial_r^l\psi ]\cdot n_{u'} \,d\mu_{{u'}}.\qedhere
\end{split}
\end{equation*}
\end{proof}

By considering the quantity $\widetilde{\Phi}=r(r-M)\partial_r\phi$ instead of $\Phi=r^2\partial_r\phi$ and restricting to the single angular mode $\psi=\psi_{\ell=1}$, we can in fact extend the range of $p$ to $p\in(-4,4)$.
\begin{proposition}[\textbf{Extended $r^p$-weighted estimates for $r(r-M)\partial_r(r\psi_{\ell=1})$}]
\label{prop:rpphil=1}
Let $\psi$ be a solution to \eqref{waveequation} emanating from initial data given as in Theorem \ref{thm:extuniq} on $(\mathcal{R}, g)$, which is supported on the single angular mode with $\ell=1$.

Take $p\in(-4,4)$ and denote $$\widetilde{\Phi}=r(r-M)\partial_r\phi=r(r-M)\partial_r(r\psi)$$ and assume that
\begin{equation*}
\int_{\Sigma}J^T[\psi]\cdot n_{\Sigma}\,d\mu_{\Sigma}<\infty,
\end{equation*}
and
\begin{align*}
\lim_{r \to \infty }\sum_{|k|\leq 2}\int_{\s^2}\phi^2\,d\omega\big|_{u'=0}<&\:\infty,\\
\lim_{r\to \infty} \int_{\s^2}\widetilde{\Phi}^2\,d\omega\big|_{u'=0}<&\:\infty.
\end{align*}

Then there exists an $R>0$ such that for any $0\leq u_1<u_2$
\begin{equation}
\label{rphierpsileq1}
\begin{split}
\int_{\mathcal{N}_{u_2}} r^p(\partial_r\widetilde{\Phi})^2\, d\omega dr+\int_{\mathcal{A}_{u_1}^{u_2}} (p+4)r^{p-1}(\partial_r\widetilde{\Phi})^2\,d\omega drdu \leq&\: C\int_{\mathcal{N}_{u_1}} r^p(\partial_r\widetilde{\Phi})^2\\
&+C\sum_{k\leq 1}\int_{\Sigma_{u_1}}J^T[T^{k}\psi]\cdot n_{\Sigma_{u_1}}\,d\mu_{\Sigma_{u_1}},
\end{split}
\end{equation}
where $C\doteq C(D,R)>0$ is a constant and we can take $R=(p+4)^{-1}(p-4)^{-1}R_0(D)>0$, with $R_0(D)>0$ a constant.
\end{proposition}
\begin{proof}
Let $\chi$ be the cut-off introduced in the proof of Proposition \ref{prop:rpPhiv1}. Consider the vector field multiplier $V=r^{p-2}\partial_r$ and apply the divergence theorem to the corresponding current $J^V[\chi \widetilde{\Phi}]$:
\begin{equation*}
\begin{split}
\int_{\mathcal{N}_{u_2}}& J^V[\chi \widetilde{\Phi}]\cdot L\, r^2 d\omega dr+\int_{\mathcal{I}^+}J^V[\chi \widetilde{\Phi}]\cdot \underline{L}\, r^2d\omega du+\int_{\mathcal{A}_{u_1}^{u_2}} \textnormal{div}J^V[\chi \widetilde{\Phi}]\,r^2d\omega drdu\\
&=\int_{\mathcal{N}_{u_1}}J^V[\chi \widetilde{\Phi}]\cdot L\,r^2d\omega dr,
\end{split}
\end{equation*}

From (\ref{eq:boxtildePhiinfty}) it follows that
\begin{equation*}
\begin{split}
\mathcal{E}^V[\chi \widetilde{\Phi}]=&\:r^{p-3}(4D-D'r+MD(r-M)^{-1})(\partial_r(\chi\widetilde{\Phi}))^2-2r^{p-3}\partial_u(\chi\widetilde{\Phi})\partial_r(\chi\widetilde{\Phi})\\
&-Mr^{p-2}\chi\slashed{\Delta}\phi\partial_r(\chi \widetilde{\Phi})+r^{p-3}[-D''r+3D'-2Dr^{-1}-MD(r-M)^{-2}\\
&+M(r-M)^{-1}(D'-Dr^{-1})]\chi\widetilde{\Phi}\partial_r(\chi\widetilde{\Phi})+r^{p-2}[(r-M)D''+D']\chi \phi\partial_r(\chi\widetilde{\Phi})\\
&+\sum_{|\alpha_1|\leq 1,\,|\alpha_2|\leq 1}\mathcal{R}_{\chi}[\partial^{\alpha_1}\widetilde{\Phi}\cdot \partial^{\alpha_2}\widetilde{\Phi}].
\end{split}
\end{equation*}
Hence,
\begin{equation*}
\begin{split}
\textnormal{div}J^V[\chi \widetilde{\Phi}]=&\:\frac{1}{2}r^{p-3}\left[(p+4)D-3D'r+MD(r-M)^{-1}\right](\partial_r (\chi \widetilde{\Phi}))^2\\
&+\frac{1}{2}(2-p)r^{p-3}|\snabla(\chi\widetilde{\Phi})|^2+ r^{p-3}[-D''r+3D'-2Dr^{-1}-MD(r-M)^{-2}\\
&+M(r-M)^{-1}(D'-Dr^{-1})]\chi\widetilde{\Phi}\partial_r(\chi\widetilde{\Phi})\\
&+r^{p-1}[D''+D'r^{-1}-Mr^{-1}D'']\chi \phi\partial_r(\chi \widetilde{\Phi})-Mr^{p-2}\chi\slashed{\Delta}\phi\partial_r(\chi \widetilde{\Phi})\\
&+\sum_{|\alpha_1|\leq 1,\,|\alpha_2|\leq 1}\mathcal{R}_{\chi}[\partial^{\alpha_1}\widetilde{\Phi}\cdot \partial^{\alpha_2}\widetilde{\Phi}].
\end{split}
\end{equation*}

We can write $r^2\textnormal{div}J^V[\chi \widetilde{\Phi}]=J_0+J_1+J_2+\mathcal{R}_{\chi}[\partial^{\alpha_1}\widetilde{\Phi}\cdot \partial^{\alpha_2}\widetilde{\Phi}]$, with
\begin{align*}
J_0[\chi \widetilde{\Phi}]&\doteq\frac{1}{2}r^{p-1}\left[(p+4)+{O}(r^{-1})\right](\partial_r(\chi \widetilde{\Phi}))^2+\frac{1}{2}r^{p-3}(2-p)r^2|\snabla(\chi \widetilde{\Phi})|^2,\\
J_1[\chi \widetilde{\Phi}]&\doteq r^{p-1}[-2r^{-1}+{O}(r^{-2})]\chi \widetilde{\Phi}\partial_r(\chi \widetilde{\Phi}),\\
J_2[\chi \widetilde{\Phi}]&\doteq r^{p+1}[D''+D'r^{-1}+{O}(r^{-4})]\chi\phi\partial_r(\chi \widetilde{\Phi})-Mr^{p}\chi\slashed{\Delta}\phi\partial_r(\chi \widetilde{\Phi}).
\end{align*}
We can estimate the spacetime integral of $\mathcal{R}_{\chi}[\partial^{\alpha_1}\widetilde{\Phi}\cdot \partial^{\alpha_2}\widetilde{\Phi}]$ in the same way as in the proof of Proposition \ref{prop:rpPhiv1}.

Moreover, we can improve the estimates of $J_i$, if we restrict to $\psi=\psi_1$. The $(\partial_r(\chi \widetilde{\Phi}))^2$ term in $J_0$ has the right sign for $p>0$, if $R\geq (p+4)^{-1}R_0$, for $R_0=R_0(D)>0$ suitably large. We use the Poincar\'e inequality (\ref{eq:poincare1}) to write the remaining term as follows:
\begin{equation}
\label{eq:J1l=1}
\int_{\mathcal{A}^{u_2}_{u_1}}\frac{1}{2}r^{p-3}(2-p)r^2|\snabla(\chi \widetilde{\Phi}_1)|^2\,d\omega dr du=(2-p)\int_{\mathcal{A}^{u_2}_{u_1}}r^{p-3}(\chi \widetilde{\Phi}_1)^2\,d\omega dr du.
\end{equation}
Moreover, we apply integration by parts to rewrite the leading order term in $J_1$:
\begin{equation}
\label{eq:J1l=1b}
\begin{split}
\int_{\mathcal{A}^{u_2}_{u_1}}-2r^{p-2}\chi\widetilde{\Phi}_1\partial_r(\chi \widetilde{\Phi}_1)\,d\omega du dr=&-\int_{\mathcal{I}^+} r^{p-2}(\chi\widetilde{\Phi}_1)^2\,d\omega du\\
&+\int_{\mathcal{A}^{u_2}_{u_1}}(p-2)r^{p-3}\chi\widetilde{\Phi}_1\partial_r(\chi \widetilde{\Phi}_1)\,d\omega du dr.
\end{split}
\end{equation}
Note that the second term on the right-hand of (\ref{eq:J1l=1b}) cancels out exactly the term on the right-hand side of (\ref{eq:J1l=1}). Similarly, the first term on the right-hand side of (\ref{eq:J1l=1b}) cancels out exactly the $|\snabla (\chi \tilde{\Phi})|^2$ term in the flux integral along $\mathcal{I}^+$ that appears in the divergence identity for $J^V[\chi \widetilde{\Phi}]$.

We are left with estimating the higher-order terms in $r^{-1}$ in $J_1$; we apply Cauchy--Schwarz:
\begin{equation*}
r^{p-3}\chi\widetilde{\Phi}\partial_r(\chi \widetilde{\Phi})\leq \epsilon (p+4) r^{p-1} (\partial_r(\chi \widetilde{\Phi}))^2+C_{\epsilon}(p+4)^{-1}r^{p-5}(\chi \widetilde{\Phi})^2
\end{equation*}
and absorb the spacetime integral of the first term on the right-hand side into $J_0$, whereas we apply a Hardy inequality to deal with the second term:
\begin{equation*}
C_{\epsilon}(p+4)^{-1}\int_{\mathcal{A}^{u_2}_{u_1}}r^{p-5}(\chi \widetilde{\Phi})^2\,d\omega dr du \leq C_{\epsilon}(p+4)^{-1}(p-4)^{-2}\int_{\mathcal{A}^{u_2}_{u_1}}r^{p-3}(\partial_r(\chi \widetilde{\Phi}))^2\,d\omega dr du,
\end{equation*}
where we require $p<4$ and moreover, use that $\lim_{r\to \infty}r^{p-4}\int_{\s^2}\widetilde{\Phi}^2\,d\omega=0$ for $p<4$, by Proposition \ref{prop:step0radfields}. We can absorb the right-hand side above into $J_0$ for suitably large $R_0>0$.

Recall that,
\begin{align*}
D&=1-\frac{2M}{r}+O_3(r^{-1-\beta}),\\
D'&=\frac{2M}{r^2}+{O}_2(r^{-2-\beta}),\\
D''&=-\frac{4M}{r^3}+{O}_1(r^{-3-\beta}),
\end{align*}
so we can write
\begin{equation*}
\begin{split}
J_2[\chi \widetilde{\Phi}]=&\:r^{p+1}\left(-\frac{4M}{r^3}+\frac{2M}{r^3}+{O}(r^{-4})\right)\chi\phi_1\partial_r(\chi \widetilde{\Phi}_1)-Mr^{p}\chi\slashed{\Delta}\phi_1\partial_r(\chi \widetilde{\Phi}_1)\\
=&\:{O}(r^{p-3})\chi\phi_1\partial_r(\chi \widetilde{\Phi}_1).
\end{split}
\end{equation*}
The precise definition of $\tilde{\Phi}$ and the restriction $\psi=\psi_1$ are crucial for the above cancellation of the leading order terms in $r^{-1}$ in $J_2$.

We apply Cauchy--Schwarz and Hardy to estimate
\begin{equation*}
\begin{split}
\int_{\mathcal{A}^{u_2}_{u_1}}{O}(r^{p-3})\chi\phi\partial_r(\chi\widetilde{\Phi})\,d\omega dudr\leq&\: \epsilon \int_{\mathcal{A}^{u_2}_{u_1}}(p+4)r^{p-1}(\partial_r(\chi\widetilde{\Phi}))^2\,d\omega dudr\\
&+C_{\epsilon}(p+4)^{-1} \int_{\mathcal{A}^{u_2}_{u_1}}r^{p-5}\chi^2\phi^2\,d\omega dudr\\
\leq&\: \epsilon \int_{\mathcal{A}^{u_2}_{u_1}}(p+4)r^{p-1}(\partial_r(\chi\widetilde{\Phi}))^2\,d\omega dudr\\
&+C_{\epsilon}(p-4)^2(p+4)^{-1} \int_{\mathcal{A}^{u_2}_{u_1}}r^{p-3}(\partial_r(\chi\phi))^2\,d\omega dudr,
\end{split}
\end{equation*}
for $p<4$, where we used that $\lim_{r\to \infty}r^{p-4}\int_{\s^2}\phi^2\,d\omega=0$, by Proposition \ref{prop:step0radfields}.

We can apply Hardy once more to obtain
\begin{equation*}
\begin{split}
(p-4)^2(p+4)^{-1}& \int_{\mathcal{A}^{u_2}_{u_1}}r^{p-3}(\partial_r(\chi\phi))^2\,d\omega dudr\leq C(p-4)^2(p+4)^{-1} \int_{\mathcal{A}^{u_2}_{u_1}}r^{p-3}\chi^2(\partial_r\phi)^2\,d\omega dudr\\
&+ \int_{\mathcal{A}^{u_2}_{u_1}}\mathcal{R}_{\chi}[\phi^2]\,d\omega dudr\\
\leq&\: C(p-4)^2(p-6)^{-2}(p+4)^{-1} \int_{\mathcal{A}^{u_2}_{u_1}}r^{p-5}(\partial_r(\chi\widetilde{\Phi}))^2\,d\omega dudr\\
&+  \int_{\mathcal{A}^{u_2}_{u_1}}\mathcal{R}_{\chi}[\phi^2]\,d\omega dudr,
\end{split}
\end{equation*}
using that $\lim_{r\to \infty}r^{p-6}\widetilde{\Phi}^2=0$ for $p<6$. We can apply the above estimates to the divergence identity to obtain the $r^p$-weighted estimates in the proposition.
\end{proof}
\begin{remark}\label{rmk:I1}
We need to require $\lim_{r\to \infty} r^2\partial_r(r(r-M)\partial_r(r\psi_{\ell=1})) (0,r , \theta, \varphi ) = 0$ only if we want to take $p\geq 3$ in Proposition \ref{prop:rpphil=1}. For $p < 3$ it is enough to consider
\begin{equation}
\lim_{r\to \infty} r^2\partial_r(r(r-M)\partial_r(r\psi_{\ell=1})) (0,r , \theta, \varphi ) < \infty.
\end{equation}
This limit, evaluated along $\mathcal{N}_{\tau}$, corresponds precisely to the \emph{second} Newman--Penrose constant $I_1[\psi]$, which one can show is conserved along $\mathcal{I}^+$. One can compare this with requiring that $I_0[\psi]=0$ in order to be able to take $p\geq 3$ in Proposition \ref{en01} for $\psi_0$. 
\end{remark}
We can obtain additional $r^p$-weighted estimates if we commute $\square_g$ once more with $r^2\partial_r$. First, we restrict to $\psi_{\ell\geq 2}$
\begin{proposition}[\textbf{$r^p$-weighted estimates for $(r^2\partial_r)^2(r\psi_{\ell \geq 2})$}]
\label{prop:l2commr2v0}
Let $\psi$ be a solution to \eqref{waveequation} emanating from initial data given as in Theorem \ref{thm:extuniq} on $(\mathcal{R}, g)$.

Let $\psi$ be supported on angular modes with $\ell \geq 2$ and take $p\in(-8,2)$. Denote $$\Phi_{(2)}=r^2\partial_r\Phi=r^2\partial_r(r^2\partial_r(r\psi)).$$ and assume that
\begin{equation*}
\sum_{|k|\leq 4}\int_{\Sigma}J^T[\Omega^k\psi]\cdot n_{\Sigma}\,d\mu_{\Sigma}<\infty
\end{equation*}
and
\begin{align*}
\lim_{r \to \infty }\sum_{|k|\leq 4}\int_{\s^2}(\Omega^k\phi)^2\,d\omega\big|_{u'=0}<&\:\infty,\\
\lim_{r \to \infty }\sum_{|k|\leq 2}\int_{\s^2}(\Omega^k\Phi)^2\,d\omega\big|_{u'=0}<&\:\infty,\\
\lim_{r \to \infty }\int_{\s^2}\Phi_{(2)}^2\,d\omega\big|_{u'=0}<&\:\infty,
\end{align*}

Then there exists an $R>0$ such that for any $0\leq u_1<u_2$
\begin{equation}
\label{rphierPhi2lgeq2av0}
\begin{split}
\int_{\mathcal{N}_{u_2}}& r^p(\partial_r\Phi_{(2)})^2\, d\omega dr+\int_{\mathcal{A}_{u_1}^{u_2}} (p+8)r^{p-1}(\partial_r\Phi_{(2)})^2+(2-p)r^{p-1}|\snabla \Phi_{(2)}|^2\,d\omega drdu\\
 \leq&\: C\int_{\mathcal{N}_{u_1}} r^p(\partial_r\Phi_{(2)})^2\,d\omega dr+C\sum_{k\leq 2}\int_{\Sigma_{u_1}}J^T[T^{k}\psi]\cdot n_{{u_1}}\,d\mu_{\Sigma_{u_1}},
\end{split}
\end{equation}
where $C\doteq C(D,R)>0$ is a constant and we can take $R=(p+8)^{-1}(p-2)^{-1}R_0(D)>0$, with $R_0(D)>0$ a constant.
\end{proposition}
\begin{proof}
From (\ref{eq:Phi2infty}) it follows that
\begin{equation*}
\begin{split}
\mathcal{E}^V[\chi \Phi_{(2)}]=&\:r^{p-3}[6D-2D'r](\partial_r(\chi\Phi_{(2)}))^2-2r^{p-3}\partial_u(\chi\Phi_{(2)})\partial_r(\chi\Phi_{(2)})\\
&+r^{p-3}[-6Dr^{-1}-3D'' r+7D']\chi\Phi_{(2)}\partial_r(\chi\Phi_{(2)})\\
&+r^{p-1}[-D'''r+2D''+2D'r^{-1}]\chi\Phi\partial_r(\chi\Phi_{(2)})\\
&+r^{p+1}[D'''+4D''r^{-1}+2D'r^{-2}]\chi\phi\partial_r(\chi\Phi_{(2)})+\sum_{|\alpha_1|\leq 1,\,|\alpha_2|\leq 1}\mathcal{R}_{\chi}[\partial^{\alpha_1}\Phi_{(2)}\cdot \partial^{\alpha_2}\Phi_{(2)}].
\end{split}
\end{equation*}
We obtain
\begin{equation*}
\begin{split}
\textnormal{div}J^V[\chi \Phi_{(2)}]&=\frac{1}{2}r^{p-3}\left[D(p+8)-5D'r\right](\partial_r(\chi\Phi_{(2)}))^2+\frac{1}{2}r^{p-3}(2-p)|\snabla(\chi\Phi_{(2)})|^2\\
&+r^{p-3}[-6Dr^{-1}-3D'' r+7D']\chi\Phi_{(2)}\partial_r(\chi\Phi_{(2)})\\
&+r^{p-1}[-D'''r+2D''+2D'r^{-1}]\chi\Phi\partial_r((\chi\Phi_{(2)}))\\
&+r^{p+1}[D'''+4D''r^{-1}+2D'r^{-2}]\chi\phi\partial_r((\chi\Phi_{(2)})+\sum_{|\alpha_1|\leq 1,\,|\alpha_2|\leq 1}\mathcal{R}_{\chi}[\partial^{\alpha_1}\Phi_{(2)}\cdot \partial^{\alpha_2}\Phi_{(2)}].
\end{split}
\end{equation*}
We can write
\begin{equation*}
r^2\textnormal{div}J^V[\chi\Phi_{(2)}]=\sum_{k=0}^3 J_k[\chi\Phi_{(2)}]+\sum_{|\alpha_1|\leq 1,\,|\alpha_2|\leq 1}\mathcal{R}_{\chi}[\partial^{\alpha_1}\Phi_{(2)}\cdot \partial^{\alpha_2}\Phi_{(2)}],
\end{equation*}
where
\begin{align*}
J_0[\chi \Phi_{(2)}]&\doteq \frac{1}{2}r^{p-1}\left[D(p+8)-5D'r\right](\partial_r(\chi\Phi_{(2)}))^2+\frac{1}{2}r^{p-3}(2-p)r^2|\snabla(\chi\Phi_{(2)})|^2,\\
J_1[\chi \Phi_{(2)}]&\doteq r^{p-1}[-6Dr^{-1}-3D'' r+7D']\chi\Phi_{(2)}\partial_r(\chi\Phi_{(2)}),\\
J_2[\chi \Phi_{(2)}]&\doteq r^{p+1}[-D'''r+2D''+2D'r^{-1}]\chi\Phi\partial_r(\chi\Phi_{(2)}),\\
J_3[\chi \Phi_{(2)}]&\doteq r^{p+3}[D'''+4D''r^{-1}+2D'r^{-2}]\chi\phi\partial_r(\chi\Phi_{(2)}).
\end{align*}

By the Morawetz estimate (\ref{ass:morawetzlocal}), we can estimate

\begin{equation*}
\begin{split}
\int_{\mathcal{A}^{u_2}_{u_1}}\sum_{|\alpha_1|\leq 1,\,|\alpha_2|\leq 1}\mathcal{R}_{\chi}[\partial^{\alpha_1}\Phi_{(2)}\cdot\partial^{\alpha_2}\Phi_{(2)}]\,d\omega dr du\leq&\: C\sum_{|\alpha|\leq 3}\int_{\{R\leq r\leq R+1\}}(\partial^{\alpha}\psi)^2\,d\omega dr du\\
\leq &\: C\sum_{k\leq 2}\int_{\Sigma_{u_1}}J^T[T^{k}\psi]\cdot n_{\Sigma_{u_1}}\,d\mu_{\Sigma_{u_1}}.
\end{split}
\end{equation*}

We have that
\begin{align*}
D&=1-\frac{2M}{r}+{O}_3(r^{-1-\beta}),\\
D'&=\frac{2M}{r^2}+{O}_2(r^{-2-\beta}),\\
D''&=-\frac{4M}{r^3}+{O}_1(r^{-3-\beta}),\\
D'''&=\frac{12M}{r^4}+{O}(r^{-4-\beta}).
\end{align*}

The terms in $J_0$ are non-negative definite if $-8<p\leq 2$ and $R\geq (p+8)^{-1}R_0$, with $R_0=R_0(D)>0$ suitably large. The leading-order term in $J_1$ can be estimated by integration by parts:
\begin{equation}
\label{eq:J1phi2infty}
\begin{split}
-\int_{\mathcal{A}^{u_2}_{u_1}}6r^{p-2}\chi\Phi_{(2)}\partial_r(\chi\Phi_{(2)})\,d\omega drdu=&-3\int_{\mathcal{I}^+}r^{p-2}(\chi\Phi_{(2)})^2\,d\omega du\\
&+3(p-2)\int_{\mathcal{A}^{u_2}_{u_1}}3r^{p-3}(\chi\Phi_{(2)})^2\,d\omega drdu.
\end{split}
\end{equation}
We can decompose into orthogonal parts $\psi_1=\psi_{\ell=1}+\psi_{\ell\geq 2}$.

Suppose that $\psi=\psi_{\ell\geq 2}$. Then we can estimate, using the Poincar\'e inequality (\ref{eq:poincare1}),
\begin{equation*}
\frac{1}{2}(2-p)\int_{\mathcal{A}^{u_2}_{u_1}}r^{p-3}r^{2}|\snabla(\chi\Phi_{(2)})|^2\,d\omega drdu\leq 3(2-p) \int_{\mathcal{A}}r^{p-3}(\chi\Phi_{(2)})^2\,d\omega drdu,
\end{equation*}
where the inequality becomes an equality if $\psi=\psi_{\ell=2}$. Hence, the bulk term on the right-hand side of (\ref{eq:J1phi2infty}) gets absorbed (or exactly cancelled) by the $\snabla (\chi\Phi_{(2)})$ term in $J_0$. Similarly, we can absorb the boundary integral on the right-hand side of (\ref{eq:J1phi2infty}) by the flux integral of $J^V[\chi \Phi_{(2)}]\cdot \underline{L}$ at $\mathcal{I}^+$, where an exact cancellation occurs again if $\psi=\psi_{\ell=2}$.

The remaining terms in $J_1$ can be estimated by applying a Cauchy--Schwarz inequality:
\begin{equation*}
\begin{split}
\int_{\mathcal{A}^{u_2}_{u_1}}{O}(r^{p-3})\chi\Phi_{(2)}\partial_r(\chi\Phi_{(2)})\,d\omega dudr\leq&\: \epsilon \int_{\mathcal{A}^{u_2}_{u_1}}(p+4)r^{p-1}(\partial_r(\chi\Phi_{(2)}))^2\,d\omega dudr\\
&+C_{\epsilon}(p+4)^{-1} \int_{\mathcal{A}^{u_2}_{u_1}}r^{p-1}\chi^2(\partial_r\Phi)^2\,d\omega dudr,
\end{split}
\end{equation*}
where the first term on the right-hand side can be absorbed into $J_0$ for suitably large $R_0$ and the second term can be estimated by applying \eqref{eq:Hardy1} as follows:
\begin{equation*}
 \int_{\mathcal{A}^{u_2}_{u_1}}r^{p-1}\chi^2(\partial_r\Phi)^2\,d\omega dudr\leq C \int_{\mathcal{A}^{u_2}_{u_1}}r^{p-3}(\partial_r(\chi\Phi_{(2)}))^2\,d\omega dudr,
\end{equation*}
if $p<4$, and absorbed into $J_0$, where we used that $\lim_{r\to \infty}r^{p-4}\int_{\s^2}\Phi_{(2)}^2\,d\omega=0$.
 
Consider now $J_2$. We write
\begin{equation*}
J_2=O(r^{p-2})\chi\Phi\partial_r(\chi \Phi_{(2)})
\end{equation*}
and apply Cauchy--Schwarz to estimate
\begin{equation*}
\begin{split}
\int_{\mathcal{A}^{u_2}_{u_1}}{O}(r^{p-2})\chi\Phi\partial_r(\chi \Phi_{(2)})\,d\omega drdu\leq&\: \epsilon \int_{\mathcal{A}^{u_2}_{u_1}}(p+8)r^{p-1}(\partial_r(\chi \Phi_{(2)}))^2\,d\omega drdu\\
&+{C}_{\epsilon}(p+8)^{-1}\int_{\mathcal{A}^{u_2}_{u_1}}r^{p-3}(\chi \Phi)^2\,d\omega drdu.
\end{split}
\end{equation*}
The first term on the right-hand side can be absorbed into $J_0$, whereas the second term can be estimated by applying a \eqref{eq:Hardy1}:
\begin{equation*}
\int_{\mathcal{A}^{u_2}_{u_1}}r^{p-3}(\chi \Phi)^2\,d\omega drdu\leq C(p-2)^{-2}\int_{\mathcal{A}^{u_2}_{u_1}}r^{p-1}(\partial_r(\chi \Phi))^2\,d\omega drdu,
\end{equation*}
if $p<2$.

We can further estimate
\begin{equation*}
\int_{\mathcal{A}^{u_2}_{u_1}}r^{p-1}(\partial_r(\chi \Phi))^2\,d\omega drdu\leq C\int_{\mathcal{A}^{u_2}_{u_1}}r^{p-1}\chi^2(\partial_r \Phi)^2\,d\omega drdu+\sum_{|\alpha_1|+|\alpha_2|\leq 1}\mathcal{R}_{\chi}[\partial^{\alpha_1}\phi\cdot \partial^{\alpha_2}\phi]
\end{equation*}
and apply another Hardy inequality to absorb the first term on the right-hand side into $J_0$, as above, now taking $R\geq (p+8)^{-1}(p-2)^{-1}R_0$, with $R_0=R_0(D)>0$ suitably large.

Consider now $J_3$. Observe that
\begin{equation*}
D'''+4D''r^{-1}+2D'r^{-2}= {O}(r^{-5}).
\end{equation*}
Therefore, 
\begin{equation*}
J_3={O}(r^{p-2})\chi\phi\partial_r(\chi \Phi_{(2)}),
\end{equation*}
and we can apply Cauchy--Schwarz to estimate
\begin{equation*}
\begin{split}
\int_{\mathcal{A}^{u_2}_{u_1}}{O}(r^{p-2})\chi\phi\partial_r(\chi \Phi_{(2)})\,d\omega drdu\leq &\: \epsilon \int_{\mathcal{A}^{u_2}_{u_1}}(p+8)r^{p-1}(\partial_r(\chi \Phi_{(2)}))^2\,d\omega drdu\\
&+{C}_{\epsilon}(p+8)^{-1}\int_{\mathcal{A}^{u_2}_{u_1}}r^{p-3}(\chi \phi)^2\,d\omega drdu.
\end{split}
\end{equation*}
The first term on the right-hand side can be absorbed into $J_0$, whereas the second term can be estimated by applying a Hardy inequality:
\begin{equation*}
\int_{\mathcal{A}^{u_2}_{u_1}}r^{p-3}(\chi \phi)^2\,d\omega drdu\leq C(p-2)^2\int_{\mathcal{A}^{u_2}_{u_1}}r^{p-1}(\partial_r(\chi \phi))^2\,d\omega drdu,
\end{equation*}
if $p<2$. Furthermore,
\begin{equation*}
\int_{\mathcal{A}^{u_2}_{u_1}}r^{p-1}(\partial_r(\chi \phi))^2\,d\omega drdu\leq C\int_{\mathcal{A}^{u_2}_{u_1}}r^{p-1}\chi^2(\partial_r \phi)^2+\mathcal{R}_{\chi}[\phi^2]\,d\omega drdu,
\end{equation*}
and the first term on the right-hand side can be absorbed into $J_0$ after applying two more Hardy inequalities, as above.

We now apply the divergence identity to $J^V$ to obtain
\begin{equation*}
\begin{split}
\int_{\mathcal{N}_{u_2}}& r^p(\partial_r(\chi\Phi_{(2)}))^2\, d\omega dr+\int_{\mathcal{A}_{u_1}^{u_2}} (p+8)r^{p-1}(\partial_r(\chi\Phi_{(2)}))^2\,d\omega drdu \\
\leq&\: C\int_{\mathcal{N}_{u_1}} r^p(\partial_r(\chi\Phi_{(2)}))^2\,d\omega dr+C\sum_{k\leq 2}\int_{\Sigma_{u_1}}J^T[T^{k}\psi]\cdot n_{{u_1}}\,d\mu_{\Sigma_{u_1}},
\end{split}
\end{equation*}
if $R\geq (p+8)^{-1}(p-2)^{-1}R_0$ with $R_0=R_0(D)>0$ suitably large.

Note that we can add the term  $(2-p)r^{p-1}|\snabla \chi \Phi_{(2)}|^2$ to the bulk term integrand on the left-hand side. Indeed, for $(\Phi_{(2)})_{\ell \geq 3}$ this term appears with a good sign in $J_0+J_1$, whereas for $(\Phi_{(2)})_{\ell=2}$ we can apply the equality \eqref{eq:poincare1} together with the Hardy inequality \eqref{eq:Hardy1} to add $r^{p-1}|\snabla \chi \Phi_{(2)}|^2$ to the bulk term.

We can remove the cut-off $\chi$ from the above estimates as in Proposition \ref{prop:rpPhiv1}.
\end{proof}
\begin{remark}
By estimating $J_2$ similarly to the proof of Proposition \ref{prop:rpPhiv1}, we can in fact increase the range of $p$ to $p\in (-8,2]$ in Proposition \ref{prop:l2commr2}.
\end{remark}
By combining Proposition \ref{prop:rpphil=1} with Proposition \ref{prop:l2commr2}, we obtain the following hierarchy of estimates for $\psi_1$.

\begin{proposition}[\textbf{$r^p$-weighted estimates for $(r^2\partial_r)^2(r\psi_1)$}]
\label{prop:l2commr2}
Let $\psi$ be a solution to \eqref{waveequation} emanating from initial data given as in Theorem \ref{thm:extuniq} on $(\mathcal{R}, g)$.

Take $p\in(-6,1)$. Denote $$\Phi_{(2)}=r^2\partial_r\Phi=r^2\partial_r(r^2\partial_r(r\psi))$$ and assume that
\begin{equation*}
\sum_{|k|\leq 4}\int_{\Sigma}J^T[\Omega^k\psi]\cdot n_{\Sigma}\,d\mu_{\Sigma}<\infty
\end{equation*}
and
\begin{align*}
\lim_{r \to \infty }\sum_{|k|\leq 4}\int_{\s^2}(\Omega^k\phi)^2\,d\omega\big|_{u'=0}<&\:\infty,\\
\lim_{r \to \infty }\sum_{|k|\leq 2}\int_{\s^2}(\Omega^k\Phi)^2\,d\omega\big|_{u'=0}<&\:\infty,\\
\lim_{r \to \infty }r^{-1}\int_{\s^2}\Phi_{(2)}^2\,d\omega\big|_{u'=0}<&\:\infty.
\end{align*}

Then there exists an $R>0$ such that for any $0\leq u_1<u_2$
\begin{equation}
\label{rphierPhi2lgeq2a}
\begin{split}
\int_{\mathcal{N}_{u_2}}& r^p(\partial_r\Phi_{(2)})^2\, d\omega dr+\int_{\mathcal{A}_{u_1}^{u_2}} r^{p-1}(\partial_r\Phi_{(2)})^2+(2-p)r^{p-1}|\snabla \Phi_{(2)}|^2\,d\omega drdu\\
 \leq&\: C(p+6)^{-1}(p-1)^{-2}\int_{\mathcal{N}_{u_1}} r^p(\partial_r\Phi_{(2)})^2\,d\omega dr+C(p+6)^{-1}\sum_{k\leq 2}\int_{\Sigma_{u_1}}J^T[T^{k}\psi]\cdot n_{\Sigma_{u_1}},
\end{split}
\end{equation}
where $C\doteq C(D,R)>0$ is a constant and we can take $R=(p+6)^{-1}R_0(D)>0$, with $R_0(D)>0$ a constant.
\end{proposition}
\begin{proof}
The only difference with the proof of Proposition \ref{prop:l2commr2v0} occurs in the estimate of $J_1$ for $\psi=\psi_{\ell=1}$. In that case, we use Proposition \ref{prop:rpphil=1} and Cauchy--Schwarz to estimate
\begin{equation*}
\begin{split}
-\int_{\mathcal{A}^{u_2}_{u_1}}&6r^{p-2}\chi\Phi_{(2)}\partial_r(\chi\Phi_{(2)})\,d\omega drdu\leq \epsilon \int_{\mathcal{A}^{u_2}_{u_1}}r^{p-1}(\partial_r(\chi\Phi_{(2)}))^2\,d\omega drdu+C_{\epsilon }\int_{\mathcal{A}^{u_2}_{u_1}}r^{p-3}(\chi\Phi_{(2)})^2\,d\omega drdu\\
\leq&\: \epsilon \int_{\mathcal{A}^{u_2}_{u_1}}r^{p-1}(p+8)(\partial_r(\chi\Phi_{(2)}))^2\,d\omega drdu+C_{\epsilon }(p+8)^{-1}\int_{\mathcal{A}^{u_2}_{u_1}}r^{p+1}(\partial_r(\chi\widetilde{\Phi}))^2\,d\omega drdu\\
&+C_{\epsilon}(p+8)^{-1}\int_{\mathcal{A}^{u_2}_{u_1}}\mathcal{R}_{\chi}[\widetilde{\Phi}^2]\,d\omega drdu\\
\leq&\: \epsilon \int_{\mathcal{A}^{u_2}_{u_1}}r^{p-1}(p+8)(\partial_r(\chi\Phi_{(2)}))^2\,d\omega drdu+C_{\epsilon }(p+6)^{-1}(p+8)^{-1}\int_{\mathcal{N}_{u_1}}r^{p+2}(\partial_r(\chi\widetilde{\Phi}))^2\,d\omega dr\\
&+C_{\epsilon }(p+6)^{-1}(p+8)^{-1}\sum_{k\leq 1}\int_{\Sigma_{u_1}}J^T[T^{k}\psi]\cdot n_{\Sigma_{u_1}},
\end{split}
\end{equation*}
if $p<2$ and $R\geq (p-2)^{-1}(p+6)^{-1}R_0$, with $R_0=R_0(D)>0$ suitably large. If $p<1$, we can apply $\eqref{eq:Hardy1}$ to further estimate
\begin{equation*}
\begin{split}
\int_{\mathcal{N}_{u_1}}r^{p+2}(\partial_r(\chi\widetilde{\Phi}))^2\,d\omega dr\leq&\: C \int_{\mathcal{N}_{u_1}}r^{p+2}\chi^2(\partial_r\Phi)^2+\chi^2 r^{p-2}\Phi^2\,d\omega dr+C\int_{\Sigma_{u_1}}J^T[\psi]\cdot n_{\Sigma_{u_1}}\\
\leq&\: C(p-1)^{-2}\int_{\mathcal{N}_{u_1}}r^{p}(\partial_r(\chi \Phi_{(2)})^2\,d\omega dr,
\end{split}
\end{equation*}
for $p<1$.
\end{proof}
\subsection{Extended hierarchies for $\partial_r^{k+1}\Phi_{(2)}$ and $\partial_r^{k}T\Phi_{(2)}$, $k\geq 0$}
\label{sec:rphierTkpsi1}
In this section, we will obtain additional $r^p$-weighted estimates for the quantity $\partial_r^{k}T\Phi_{(2)}$, with $k\geq 0$, using the $r^p$-weighted estimates from the previous sections. As an intermediate step, we will obtain $r^p$-weighted estimates for the commuted quantities $\partial_r^k\Phi_{(2)}$.
\begin{lemma}
\label{lm:boxdrkPsi2}
Let $D=1-\frac{2M}{r}+O_{n+2}(r^{-1-\beta})$ for some $n\in \N$. Then
\begin{equation}
\label{eq:boxdrkPsi2}
\begin{split}
\square_g(\partial_r^k \Phi_{(2)})=&\:\left[6r^{-1}+O(r^{-2})\right]\partial_r^{k+1}\Phi_{(2)}-2r^{-1}\partial_u\partial_r^k\Phi_{(2)}+\sum_{j=0}^{k}O(r^{-2-j})\partial_{r}^{k-j}\Phi_{(2)}\\
&+2kr^{-1}\slashed{\Delta}(\partial_r^{k-1}\Phi_{(2)})+\sum_{j=0}^{\max\{k-2,0\}} k(k-1)O(r^{-2-j})\slashed{\Delta}\partial_r^{k-j-2}\Phi_{(2)}\\
&+O(r^{-k-2})\Phi+O(r^{-k-2})\phi,
\end{split}
\end{equation}
for all $0\leq k\leq n$.
\end{lemma}
\begin{proof}
See Appendix \ref{app:commmultp}.
\end{proof}
\begin{proposition}
\label{prop:rpPhi2commr}
Let $\psi$ be a solution to \eqref{waveequation} emanating from initial data given as in Theorem \ref{thm:extuniq} on $(\mathcal{R}, g)$.

Assume that $D=1-\frac{2M}{r}+O_{n+2}(r^{-1-\beta})$ with $n\in \N$. Assume also that 
\begin{equation*}
\sum_{|l|\leq 4+2n}\int_{\Sigma}J^T[\Omega^l\psi]\cdot n_{\Sigma}\,d\mu_{\Sigma}<\infty,
\end{equation*}
and moreover,
\begin{align*}
\lim_{r \to \infty }\sum_{|l|\leq 4+2n}\int_{\s^2}(\Omega^l\phi)^2\,d\omega\big|_{u'=0}<&\:\infty,\\
\lim_{r \to \infty }\sum_{|l|\leq 2+2n}\int_{\s^2}(\Omega^l\Phi)^2\,d\omega\big|_{u'=0}<&\:\infty,\\
\lim_{r \to \infty }\sum_{|l|\leq 2n}\int_{\s^2}r^{-1}\left(\Omega^l\Phi_{(2)}\right)^2\,d\omega\big|_{u'=0}<&\:\infty,
\end{align*}
and, if $n\geq 1$,
\begin{equation*}
\lim_{r \to \infty }\sum_{|l|\leq 2n-2s}\int_{\s^2}r^{2s+1}\left(\partial_r^s\Omega^l\Phi_{(2)}\right)^2\,d\omega\big|_{u'=0}<\infty,
\end{equation*}
for each $1\leq s\leq n$.

Let $0\leq k \leq n$ and take $p\in(-6+2k,1+2k)$. Then there exists an $R>0$ such that for any $0\leq u_1<u_2$
\begin{equation}
\label{eq:pPhi2commr}
\begin{split}
\int_{\mathcal{N}_{u_2}}& r^p(\partial_r^{k+1}\Phi_{(2)})^2\, d\omega dr\\
&+\int_{\mathcal{A}_{u_1}^{u_2}} r^{p-1}(\partial_r^{k+1}\Phi_{(2)})^2+r^{p-1}|\snabla \partial_r^k \Phi_{(2)}|^2 \,d\omega drdu \\
\leq&\: C(p-1-2k)^{-2}\sum_{\substack{j\leq k}}\int_{\mathcal{N}_{u_1}} r^{p-2j}(\partial_r^{k-j+1}\Phi_{(2)})^2\,d\omega dr\\
&+C\sum_{j\leq k+2}\int_{\Sigma_{u_1}}J^T[T^{j}\psi]\cdot n_{\Sigma_{u_1}}\,d\mu_{{u_1}},
\end{split}
\end{equation}
where $C\doteq C(k,D,R)>0$ is a constant and we can take $R=(p+6-2k)^{-1}(p-1-2k)^{-1}R_0(D)>0$, with $R_0(k,D)>0$ a constant.
\end{proposition}
\begin{proof}
Let $V=r^{p-2}\partial_r$. Then by \eqref{eq:boxdrkPsi2}:
\begin{equation*}
\begin{split}
r^2\mathcal{E}^V[\chi \partial_r^{k}\Phi_{(2)}]=&\:\left[6r^{p-1}+O(r^{-2})\right](\partial_r(\chi \partial_r^k\Phi_{(2)}))^2-2\chi r^{p-1}\partial_u\partial_r^k\Phi_{(2)}\partial_r(\chi \partial_r^k\Phi_{(2)})\\
&+\sum_{j=0}^{k}O(r^{p-2-j})\chi\partial_{r}^{k-j}\Phi_{(2)}\partial_r(\chi \partial_r^k\Phi_{(2)})\\
&+2kr^{p-1}\slashed{\Delta}(\chi\partial_r^{k-1}\Phi_{(2)})\partial_r(\chi \partial_r^k\Phi_{(2)})\\
&+\sum_{j=0}^{\max\{k-1,0\}} k(k-1)O(r^{p-2-j})\slashed{\Delta}(\chi\partial_r^{k-j-2}\Phi_{(2)})\partial_r(\chi \partial_r^k\Phi_{(2)})\\
&+O(r^{p-k-2})\chi\Phi\partial_r(\chi \partial_r^k\Phi_{(2)})+O(r^{p-k-1})\chi\phi\partial_r(\chi \partial_r^k\Phi_{(2)}),
\end{split}
\end{equation*}
so that:
\begin{equation*}
\begin{split}
r^2\textnormal{div} J^V[\chi \partial_r^k\Phi_{(2)}]=&\:\frac{1}{2}[(p+8)r^{p-1}+O(r^{p-4})](\partial_r(\chi\partial_r^k \Phi_{(2)}))^2+\frac{1}{2}(2-p)r^{p-1}|\snabla \chi \partial_r^k\Phi_{(2)}|^2\\
&+\sum_{j=0}^{k}O(r^{p-2-j})\chi\partial_{r}^{k-j}\Phi_{(2)}\partial_r(\chi \partial_r^k\Phi_{(2)})\\
&+2kr^{p-1}\slashed{\Delta}(\chi\partial_r^{k-1}\Phi_{(2)})\partial_r(\chi \partial_r^k\Phi_{(2)})\\
&+\sum_{j=0}^{\max\{k-1,0\}} k(k-1)O(r^{p-2-j})\slashed{\Delta}(\chi\partial_r^{k-j-2}\Phi_{(2)})\partial_r(\chi \partial_r^k\Phi_{(2)})\\
&+O(r^{p-k-2})\chi\Phi\partial_r(\chi \partial_r^k\Phi_{(2)})+O(r^{p-k-1})\chi\phi\partial_r(\chi \partial_r^k\Phi_{(2)})\\
&+\sum_{|\alpha_1|,|\alpha_2|\leq 1}\mathcal{R}_{\chi}[\partial^{\alpha_1}\partial_r^{k}\Phi_{(2)}\cdot \partial^{\alpha_2}\partial_r^k\Phi_{(2)}]\\
= &\: \sum_{i=0}^{k+2}J_i[\chi \partial_r^{k}\widetilde{\Phi}]+\sum_{|\alpha_1|,|\alpha_2|\leq 1}\mathcal{R}_{\chi}[\partial^{\alpha_1}\partial_r^{k}\Phi_{(2)}\cdot \partial^{\alpha_2}\partial_r^k\Phi_{(2)}],
\end{split}
\end{equation*}
where
\begin{align*}
J_0[\chi\partial_r^k\Phi_{(2)}] \doteq&\: \frac{1}{2}[(p+8)r^{p-1}+O(r^{p-4})](\partial_r(\chi\partial_r^k \Phi_{(2)}))^2+\frac{1}{2}(2-p)r^{p-1}|\snabla \chi \partial_r^k \Phi_{(2)}|^2,\\
J_1[\chi\partial_r^k\Phi_{(2)}] \doteq&\:2kr^{p-1}\slashed{\Delta}(\chi\partial_r^{k-1}\Phi_{(2)})\partial_r(\chi \partial_r^k\Phi_{(2)}),\\
J_{i+2}[\chi\partial_r^k\Phi_{(2)}] \doteq &\: \chi \left[O(r^{p-2-i})\partial_{r}^{k-i}\Phi_{(2)}+k(k-1)O(r^{p-2-i})\slashed{\Delta}\partial_r^{k-i-2}\Phi_{(2)}\right]\partial_r(\chi \partial_r^k\Phi_{(2)})\\
& \textnormal{for}\quad 0\leq i \leq \max\{0,k-2\},\\
J_{k+1}[\chi\partial_r^k\Phi_{(2)}] \doteq &\:kO(r^{p-2-i})\chi \partial_{r}\Phi_{(2)}\partial_r(\chi \partial_r^k\Phi_{(2)}),\\
J_{k+2}[\chi\partial_r^k\Phi_{(2)}] \doteq &\: O(r^{p-2-k})\chi \Phi_{(2)} \partial_r(\chi \partial_r^k\Phi_{(2)}),\\
J_{k+3}[\chi\partial_r^k\Phi_{(2)}] \doteq &\: O(r^{p-k-2})\chi\Phi\partial_r(\chi \partial_r^k\Phi_{(2)}),\\
J_{k+4}[\chi\partial_r^k\Phi_{(2)}]\doteq &\:O(r^{p-k-2})\chi\phi\partial_r(\chi \partial_r^k\Phi_{(2)}).
\end{align*}
We will prove by induction that \eqref{eq:pPhi2commr} holds for all $0\leq k\leq n$. First, \eqref{eq:pPhi2commr} holds for $k=0$ by Proposition \ref{prop:l2commr2}. Let us suppose \eqref{eq:pPhi2commr} holds for all $0\leq k\leq n$. We want to show that then \eqref{eq:pPhi2commr} also holds for $k=n+1$. We therefore fix $k=n+1$.

First of all, we can estimate by \eqref{ass:morawetzlocal}:
\begin{equation*}
\begin{split}
\int_{\mathcal{A}^{u_2}_{u_1}}\sum_{|\alpha_1|\leq 1,\,|\alpha_2|\leq 1}\mathcal{R}_{\chi}[\partial^{\alpha_1}\partial_r^{n+1}\Phi_{(2)}\cdot \partial^{\alpha_2}\partial_r^{n+1}\Phi_{(2)}]\,d\omega dr du\leq&\: C\sum_{|\alpha|\leq n+4}\int_{\mathcal{A}^{u_2}_{u_1}\cap \{r\leq R+1\}}(\partial^{\alpha}\psi)^2\,d\omega dr du\\
\leq &\: C\sum_{j\leq n+3}\int_{\Sigma_{u_1}}J^T[T^{j}\psi]\cdot n_{\Sigma_{u_1}}\,d\omega dr.
\end{split}
\end{equation*}

The first term in $J_0$ is positive for all $p\geq -8$ if we take $R\geq (p+8)^{-1}R_0$ with $R_0=R_0(D)>0$ suitably large. The second term in $J_0$ is negative if $p>2$, so we will estimate it by using $J_1$. Consider $J_1$ and integrate by parts in $r$ and $\s^2$ to obtain:
\begin{equation}
\label{eq:intbypartsranglaplace}
\begin{split}
\int_{\mathcal{A}^{u_2}_{u_1}}&2(n+1)r^{p-3}\slashed{\Delta}_{\s^2}(\chi\partial_r^{n}\Phi_{(2)})\partial_r(\chi \partial_r^{n+1}\Phi_{(2)})\,d\omega dr du\\
=&\:\int_{\mathcal{I}^+(u_1,u_2)}2(n+1)r^{p-3}\slashed{\Delta}_{\s^2}(\partial_r^{n}\Phi_{(2)}) \partial_r^{n+1}\Phi_{(2)}\,d\omega du\\
&-\int_{\mathcal{A}^{u_2}_{u_1}}2(p-3)(n+1)r^{p-4}\slashed{\Delta}_{\s^2}(\chi\partial_r^{n}\Phi_{(2)})\chi \partial_r^{n+1}\Phi_{(2)}\,d\omega dr du\\
&-\int_{\mathcal{A}^{u_2}_{u_1}}2(n+1)r^{p-3}\slashed{\Delta}_{\s^2}(\partial_r(\chi\partial_r^{n+1}\Phi_{(2)}))\chi \partial_r^{n+1}\Phi_{(2)}\,d\omega dr du\\
=&\:\int_{\mathcal{I}^+(u_1,u_2)}2(n+1)r^{p-3}\slashed{\Delta}_{\s^2}(\partial_r^{n}\Phi_{(2)})\cdot \partial_r^{n+1}\Phi_{(2)}\,d\omega du\\
&+\int_{\mathcal{A}^{u_2}_{u_1}}2(p-3)(n+1)r^{p-2}\slashed{\nabla}(\chi\partial_r^{n}\Phi_{(2)})\cdot \snabla(\chi \partial_r^{n+1}\Phi_{(2)})\,d\omega dr du\\
&+\int_{\mathcal{A}^{u_2}_{u_1}}2(n+1)r^{p-1}|\slashed{\nabla}\partial_r(\chi\partial_r^{n+1}\Phi_{(2)})|^2\,d\omega dr du.
\end{split}
\end{equation}
The integral along $\mathcal{I}^+$ vanishes by Proposition \ref{prop:step0radfields} and Proposition \ref{prop:hostep0radfields}, using that the assumptions on the initial data in this proposition (for $n$ replaced by $n+1$) imply that 
\begin{equation*}
\lim_{r\to \infty}\int_{\s^2}r^{p-4}(\partial_r^n\slashed{\Delta}_{\s^2}\Phi_{(2)})^2+ r^{p-2}(\partial_r^{n+1}\Phi_{(2)})^2\,d\omega\big|_{u'=u}=0
\end{equation*}
at each $u\geq 0$ for $p-2<2(n+1)+1$.

Moreover, the third term on the very right-hand side of \eqref{eq:intbypartsranglaplace} has a good sign, and if we combine it with the second term of $J_0$, we are left with
\begin{equation}
\label{eq:angtermgoodsign}
\frac{1}{2}(2+4(n+1)-p)r^{p-1}|\snabla \chi \partial_r^{n+1} \Phi_{(2)}|^2,
\end{equation}
which has a good sign if $p<2+4(n+1)$. We are left with controlling the second term on the very very right-hand side of \eqref{eq:intbypartsranglaplace}. We apply Cauchy--Schwarz to estimate this term as follows:
\begin{equation*}
\begin{split}
r^{p-2}\slashed{\nabla}(\chi\partial_r^{n}\Phi_{(2)})\cdot \snabla(\chi \partial_r^{n+1}\Phi_{(2)})\leq&\: \epsilon (2+4(n+1)-p)r^{p-1}|\snabla \chi \partial_r^{n+1} \Phi_{(2)}|^2\\
&+C_{\epsilon}(2+4(n+1)-p)^{-1}r^{p-3}|\snabla \chi \partial_r^n \Phi_{(2)}|^2,
\end{split}
\end{equation*}
where we absorb the first term on the right-hand side into \eqref{eq:angtermgoodsign} by choosing $\epsilon$ suitably small, and we can estimate the second term by employing \eqref{eq:pPhi2commr} with $k=n$ if we restrict $2n-6<p-2<2n+1$ and take $R\geq (p+6-2(n+1))^{-1}(p-1-2(n+1))^{-1}R_0$. 

To estimate $J_{k+1}$ and the first term in each $J_{i+2}$ for $0\leq i\leq \max\{n-1,0\}$, we first apply a Cauchy--Schwarz inequality as follows:
\begin{equation*}
|J_{i+2}|\leq \epsilon r^{p-1}(\partial_r(\chi\partial_r^{n+1}\Phi_{(2)}))^2+ C_{\epsilon}r^{p-2i-3}(\partial_{r}^{n+1-i}\Phi_{(2)})^2.
\end{equation*}
We can estimate the second term by applying \eqref{eq:poincare2} and \eqref{eq:pPhi2commr} with $k=n-i$, if $2k-6< p-2-2i< 1+2k$ or equivalently, $2(n+1)-6<p< 1+2(n+1)$ and moreover $R\geq (p+6-2(n+1))^{-1}(p-1-2(n+1))^{-1}R_0$.

We estimate the second term in each $J_{i+2}$ by first integrating by parts:
\begin{equation*}
\begin{split}
\int_{\mathcal{A}^{u_2}_{u_1}}&O(r^{p-2-i})\slashed{\Delta}\chi\partial_r^{n+1-i-2}\Phi_{(2)}\partial_r(\chi \partial_r^{n+1}\Phi_{(2)})\,d\omega dr du\\
=&\:\int_{\mathcal{I}^+(u_1,u_2)}O(r^{p-2-i})\slashed{\Delta}\partial_r^{n+1-i-2}\Phi_{(2)}\partial_r^{n+1}\Phi_{(2)}\,d\omega du\\
&+\int_{\mathcal{A}^{u_2}_{u_1}}O(r^{p-3-i})(\slashed{\Delta}\chi\partial_r^{n+1-i-2}\Phi_{(2)})\chi \partial_r^{n+1}\Phi_{(2)}\,d\omega dr du\\
&+\int_{\mathcal{A}^{u_2}_{u_1}}O(r^{p-2-i})(\slashed{\Delta}\partial_r(\chi\partial_r^{k-i-2}))\chi \partial_r^{n+1}\Phi_{(2)}\,d\omega dr du\\
=&\:\int_{\mathcal{I}^+(u_1,u_2)}O(r^{p-4-i})\slashed{\Delta}_{\s^2}(\partial_r^{n+1-i-2}\Phi_{(2)}) \partial_r^{n+1}\Phi_{(2)}\,d\omega du\\
&+\int_{\mathcal{A}^{u_2}_{u_1}}O(r^{p-3-i})\slashed{\nabla}(\partial_r^{n+1-i-2}\Phi_{(2)})\cdot \snabla(\chi \partial_r^{n+1}\Phi_{(2)})\,d\omega dr du\\
&+\int_{\mathcal{A}^{u_2}_{u_1}}O(r^{p-2-i})\slashed{\nabla}(\partial_r(\chi\partial_r^{n+1-i-2}\Phi_{(2)}))\cdot \snabla(\chi \partial_r^{n+1}\Phi_{(2)})\,d\omega dr du.
\end{split}
\end{equation*}
Again, once can easily see that the integral along $\mathcal{I}^+$ vanishes by Proposition \ref{prop:hostep0radfields} once we split
\begin{equation*}
O(r^{p-4-i})\slashed{\Delta}_{\s^2}(\partial_r^{n+1-i-2}\Phi_{(2)})\partial_r^{n+1}\Phi_{(2)}\leq O(r^{p-2})(\partial_r^{n+1}\Phi_{(2)})^2+ O(r^{p-6-2i})(\slashed{\Delta}_{\s^2}(\partial_r^{n-1-i}\Phi_{(2)}))^2
\end{equation*}
and restrict $p-2<2(n+1)+1$. Now, we estimate the remaining terms by applying Cauchy--Schwarz:
\begin{align*}
O(r^{p-3-i})&\slashed{\nabla}(\chi\partial_r^{n+1-i-2}\Phi_{(2)})\cdot \snabla(\chi \partial_r^{n+1}\Phi_{(2)})\leq \epsilon (2+4(n+1)-p)r^{p-1}|\snabla \chi \partial_r^{n+1} \Phi_{(2)}|^2\\
&+C_{\epsilon}(2+4(n+1)-p)^{-1}r^{p-5-2i}|\snabla \chi \partial_r^{n+1-i-2} \Phi_{(2)}|^2,\\
O(r^{p-2-i})&\slashed{\nabla}(\chi\partial_r^{n}\Phi_{(2)})\cdot \snabla(\chi \partial_r^{n+1}\Phi_{(2)})\leq \epsilon (2+4(n+1)-p)r^{p-1}|\snabla \chi \partial_r^{n+1} \Phi_{(2)}|^2\\
&+C_{\epsilon}(2+4(n+1)-p)^{-1}r^{p-3-2i}|\snabla \partial_r(\chi \partial_r^{n+1-i-2} \Phi_{(2)})|^2,
\end{align*}
where we absorb the terms with a factor $\epsilon$ in front and we use \eqref{eq:pPhi2commr} with $0\leq k\leq n-1$ to estimate the remaining terms. These estimates are certainly valid under the restrictions $2(n+1)-6<p< 1+2(n+1)$ and $R\geq (p+6-2(n+1))^{-1}(p-1-2(n+1))^{-1}R_0$.

We are left with the terms $J_{n+2}$, $J_{n+3}$ and $J_{n+4}$. $J_{n+2}$ can be estimated exactly as above. To estimate $J_{n+3}$ and $J_{n+4}$ we apply again Cauchy--Schwarz:
\begin{align*}
 O(r^{p-n-3})\chi{\Phi}\partial_r(\chi \partial_r^{n+1}{\Phi}_{(2)})\leq&\: \epsilon  r^{p-1}(\partial_r(\chi \partial_r^{n+1}{\Phi}_{(2)}))^2+ C_{\epsilon}r^{p-2n-5}(\chi{\Phi})^2,\\
O(r^{p-n-3})\chi{\Phi}\partial_r(\chi \partial_r^n{\Phi})\leq&\: \epsilon  r^{p-1}(\partial_r(\chi \partial_r^{n+1}{\Phi}_{(2)}))^2+ C_{\epsilon}r^{p-2n-5}(\chi\phi)^2.
\end{align*}
By \eqref{eq:Hardy1}, we have that
\begin{equation*}
\int_{\mathcal{A}^{u_2}_{u_1}}r^{p-2n-5}(\chi{\Phi})^2\,d\omega dr du\leq C\int_{\mathcal{A}^{u_2}_{u_1}}r^{p-2n-3}(\partial_r(\chi{\Phi}))^2\,d\omega dr du,
\end{equation*}
using that $\lim_{r\to \infty}\int_{\s^2}r^{p-2n-4}{\Phi}^2\,d\omega=0$, for $p< 2n+4=2(n+1)+2$.

We can apply \eqref{eq:Hardy1} once more to estimate the right-hand side by Proposition \ref{prop:l2commr2} with $p$ replaced by $p-2(n+1)$. We apply \eqref{eq:Hardy1} three times to estimate the spacetime integral of $r^{p-2n-5}(\chi\phi)^2$ and then  Proposition  \ref{prop:l2commr2}, where we need that $\lim_{r\to \infty}\int_{\s^2}r^{p-2n-4}\phi^2\,d\omega=0$, and $\lim_{r\to \infty}\int_{\s^2}r^{p-2n-4}{\Phi}^2\,d\omega=0$, which holds for $p< 2(n+1)+2$.

Finally, we put the above estimates together and apply the divergence theorem on $\textnormal{div}J^V[\chi \partial_r^{n+1}{\Phi}_{(2)}]$ to conclude that \eqref{eq:pPhi2commr} also holds for $k=n+1$.
\end{proof}

\begin{proposition}
\label{eq:hierTkpsi1}
Let $\psi$ be a solution to \eqref{waveequation} emanating from initial data given as in Theorem \ref{thm:extuniq} on $(\mathcal{R}, g)$.

Assume that $D=1-\frac{2M}{r}+O_{n+2}(r^{-1-\beta})$ with $n\in \N$, $n\geq 1$. Assume also that 
\begin{equation*}
\sum_{ |l|\leq 4+2n}\int_{\Sigma}J^T[\Omega^l\psi]\cdot n_{\Sigma}\,d\mu_{\Sigma}<\infty,
\end{equation*}
and moreover,
\begin{align*}
\lim_{r \to \infty }\sum_{|l|\leq 4+2n}\int_{\s^2}(\Omega^l\phi)^2\,d\omega\big|_{u'=0}<&\:\infty,\\
\lim_{r \to \infty }\sum_{|l|\leq 2+2n}\int_{\s^2}(\Omega^l\Phi)^2\,d\omega\big|_{u'=0}<&\:\infty,\\
\lim_{r \to \infty }\sum_{|l|\leq 2n}\int_{\s^2}r^{-1}\left(\Omega^l\Phi_{(2)}\right)^2\,d\omega\big|_{u'=0}<&\:\infty,
\end{align*}
and
\begin{equation*}
\lim_{r \to \infty }\sum_{|l|\leq 2n-2s}\int_{\s^2}r^{2s+1}\left(\partial_r^s\Omega^l\Phi_{(2)}\right)^2\,d\omega\big|_{u'=0}<\infty,
\end{equation*}
for each $1\leq s\leq n$.

Let $1\leq k \leq n$ and take $p\in(-6+2k,2k+1)$. Then there exists an $R>0$ such that for any $0\leq u_1<u_2$
\begin{equation}
\label{eq:pPhi2commT}
\begin{split}
&\int_{\mathcal{A}_{u_1}^{u_2}}r^{p-1}(\partial_r^{k}T\Phi_{(2)})^2 \,d\omega drdu \\
\leq&\: C(p+6-2k)^{-3}(p-1-2k)^{-1}\int_{\mathcal{N}_{u_1}} r^{p}(\partial_r^{k+1}\Phi_{(2)})^2\,d\omega dr\\
&+ C(p+6-2k)^{-3}(p-1-2k)^{-1}\sum_{\substack{|\alpha|\leq 1\\j\leq k-1}}\int_{\mathcal{N}_{u_1}} r^{p-2-2j}(\partial_r^{k-j}\Omega^{\alpha}\Phi_{(2)})^2\,d\omega dr\\
&+C(p+6-2k)^{-3}\sum_{\substack{|\alpha|\leq 1\\j+|\alpha|\leq k+2}}\int_{\Sigma_{u_1}}J^T[T^{j}\Omega^{\alpha}\psi]\cdot n_{\Sigma_{u_1}}\,d\mu_{{u_1}},
\end{split}
\end{equation}
where $C\doteq C(k,D,R)>0$ is a constant and we can take $R=(p+6-2k)^{-1}(p-1-2k)^{-1}R_0(D)>0$, with $R_0(k,D)>0$ a constant.
\end{proposition}
\begin{proof}
Let $D=1-\frac{2M}{r}+O_{k+2}(r^{-1-\beta})$, with $k\in \N$. We can commute \eqref{equationfortildePhiinfty} with $\partial_r^k$ and $T$, with $k\in \N$, to arrive at
\begin{equation}
\label{eq:commPhi2partialrT}
\begin{split}
\partial_r^{k}T\Phi_{(2)}=&\:[1+O(r^{-1})]\partial_r^{k+1}\Phi_{(2)}+\sum_{j=0}^{k-1} O(r^{-j})\slashed{\Delta}(\partial_r^{k-1-j}\Phi_{(2)})+\sum_{j=0}^{k} O(r^{-j-1})\partial_r^{k-j}\Phi_{(2)}\\
&+O(r^{-k-2})\Phi+O(r^{-k-1})\phi.
\end{split}
\end{equation}

We square both sides of \eqref{eq:commPhi2partialrT} and integrate the resulting quadratic terms, after multiplying them with a factor $r^{p-1}$. We insert moreover the cut-off $\chi$ to obtain:
\begin{equation*}
\begin{split}
\int_{\mathcal{A}^{u_2}_{u_1}}&r^{p-1}(\partial_r (\chi T \partial_r^{k-1}\Phi_{(2)}))^2,d\omega dr du\\
&\leq\: C\int_{\mathcal{A}^{u_2}_{u_1}}\chi^2r^{p-1}(\partial_r^{k+1}\Phi_{(2)})^2\,d\omega dr du\\
&+\int_{\mathcal{A}^{u_2}_{u_1}}\chi^2\sum_{j=0}^{k-1}O(r^{p-2j-5})\left(\slashed{\Delta}_{\s^2}(\partial_r^{k-1-j}\Phi_{(2)})\right)^2\,d\omega dr du,\\
&+C\sum_{j=0}^{k}\int_{\mathcal{A}^{u_2}_{u_1}}r^{p-3-2j}\chi^2(\partial_r^{k-j} \Phi_{(2)})^2\,d\omega dr du\\
&+C\int_{\mathcal{A}^{u_2}_{u_1}}\chi^2r^{p-5-2k}{\Phi}^2+\chi^2r^{p-3-2k}\phi^2\,d\omega dr du\\
&+C \sum_{0\leq j\leq k+1}\int_{\mathcal{A}^{u_2}_{u_1}}(\chi')^2 (T^j{\Phi}_{(2)})^2\,d\omega dr du.
\end{split}
\end{equation*}
We can apply \eqref{eq:Hardy1} (multiple times) to absorb the $\chi^2(T^{l-1}\phi)^2$, $\chi^2(T^{l-1}\Phi)^2$ and $\chi^2(T^{l-1}\Phi_{(2)})^2$ terms into the remaining terms, as long as $p<2k+2$. 

Furthermore, we estimate
\begin{equation*}
\begin{split}
\int_{\mathcal{A}^{u_2}_{u_1}}&\chi^2\sum_{j=0}^{k-1}O(r^{p-2j-5})\left(\slashed{\Delta}_{\s^2}(\partial_r^{k-1-j}\Phi_{(2)})\right)^2\,d\omega dr du\\
\leq&\:\int_{\mathcal{A}^{u_2}_{u_1}}\chi^2\sum_{|\alpha|\leq 1}\sum_{j=0}^{k-1}O(r^{p-2j-3})|\snabla\left(\partial_r^{k-1-j}\Omega^{\alpha}\Phi_{(2)}\right)|^2\,d\omega dr du
\end{split}
\end{equation*}

We can therefore obtain \eqref{eq:pPhi2commT} by applying \eqref{eq:pPhi2commr} with $-6<p<2k+1$ to $\phi$ and $\sum_{|\alpha|=1}\Omega^{\alpha} \phi$.
\end{proof}

\section{Commutator estimates for $\psi_{0}=\frac{1}{4\pi}\int_{\mathbb{S}^{2}}\psi$}
\label{sec:hierpsi0}

In this section our goal is to prove $r^p$-weighted estimates for the spherically symmetric part $\psi_0$ of $\psi$. \textbf{Unless specifically stated otherwise, we will denote $\psi_0$ by $\psi$ in this section.} 

\subsection{The restricted hierarchy}
\label{sec:hiernpnotzero}

First we assume that the first Newman--Penrose constant of the linear wave $\psi$ that we are dealing with is non-zero, i.e. we assume that initially we have that $I_0 [\psi ] \neq 0$, for $I_0 [\psi ]$ given as in \eqref{firstNP}. In that case we can establish the following hierarchy of estimates.

\begin{proposition}[\textbf{$r^p$-weighted estimates for $\psi_0$ with $I_0 [\psi ] \neq 0$}]\label{en00}
Let $\psi$ be a solution to \eqref{waveequation} emanating from initial data given as in Theorem \ref{thm:extuniq} on $(\mathcal{R}, g)$.

Assume that $\boldsymbol{I_0 [\psi ] \neq 0}$ and take $p\in(0,3)$. Then there exists an $R>0$ such that for any $0\leq u_1<u_2$
\begin{equation}\label{en00est}
\begin{split}
\int_{{\mathcal{N}}_{u_2}} r^p(\partial_r\phi )^2\,d\omega dr+p \int_{\mathcal{A}_{u_1}^{u_2}} r^{p-1}(\partial_r\phi )^2\,d\omega drdu \leq&\: C\int_{{\mathcal{N}}_{u_1}} r^p(\partial_r\phi )^2\,d\omega dr\\
&+C \int_{\Sigma_{u_1}} J^T [\psi ] \cdot n_{u_1}\, d\mu_{{u_1}} ,
\end{split}
\end{equation}
where $C=C(D,R)>0$ is a constant and we can take $R= p^{-1}R_0$ with $R_0=R_0(D)>0$ a constant.
\end{proposition}
\begin{proof}
Let $\chi$ be the cut-off function introduced in the proof of Proposition \ref{prop:rpPhiv1}. Consider the vector field multiplier $V=r^{p-2}\partial_r$. We apply the divergence theorem in the region $\mathcal{A}_{u_1}^{u_2}$ on $\chi \phi $ to obtain
\begin{equation}
\begin{split}
\label{eq:divthm0}
\int_{\mathcal{N}_{u_2}}& J^V[\chi \phi ]\cdot L\, r^2  dr+\int_{\mathcal{I}^+}J^V[\chi \phi ]\cdot \underline{L}\, r^2 du+\int_{\mathcal{A}_{u_1}^{u_2}} \textnormal{div}J^V[\chi \phi ]\,r^2 drdu\\
&=\int_{\mathcal{N}_{u_1}}J^V[\chi \phi ]\cdot L\,r^2 dr,
\end{split}
\end{equation}
where we used that the boundary term along the hypersurface $\{r=R\}$ vanishes due to the choice of cut-off $\chi$. Recall from Appendix \ref{app:commmultp} that we have
\begin{equation}
\label{eq:KVchiphi0}
K^V[\chi \phi ]=\frac{1}{2}r^{p-3}\left[D(p-4)-D'r\right](\partial_r (\chi \phi ))^2+2r^{p-3}\partial_r(\chi \phi )\partial_u (\chi \phi ) .
\end{equation}
By \eqref{waveoperatorphi} we also have that
\begin{equation}
\label{eq:EVchiphi0}
\begin{split}
\mathcal{E}^V[\chi \phi ]=&\:2Dr^{p-3}(\partial_r(\chi\phi ))^2-2r^{p-3}\partial_u(\chi\phi ) \cdot \partial_r(\chi\phi )+D'r^{p-3}( \chi\phi ) \cdot \partial_r(\chi \phi )\\
&+\sum_{|\alpha_1 | \leq 1 ,| \alpha_2 | \leq 1}\mathcal{R}_{\chi}[\partial^{\alpha_1} \phi \cdot \partial^{\alpha_2}\phi ].
\end{split}
\end{equation}
We add \eqref{eq:KVchiphi0} and \eqref{eq:EVchiphi0} to obtain
\begin{equation*}
\begin{split}
\textnormal{div}\,J^V[\phi ]=&\:\frac{1}{2}r^{p-3}\left[pD-D'r\right](\partial_r(\chi \phi ))^2+D'r^{p-3}( \chi\phi ) \cdot \partial_r(\chi\phi )\\
&+\sum_{|\alpha_1 | \leq 1 , |\alpha_2 | \leq 1}\mathcal{R}_{\chi} [\partial^{\alpha_1} \phi \cdot \partial^{\alpha_2}\phi ].
\end{split}
\end{equation*}
We further write
\begin{equation*}
r^2 \cdot \textnormal{div} J^V[\chi\phi ]=J_0[\chi\phi ]+J_1[\chi\phi ]+\sum_{|\alpha_1 | \leq 1 ,|\alpha_2 | \leq 1}\mathcal{R}_{\chi} [\partial^{\alpha_1} \phi \cdot \partial^{\alpha_2}\phi ],
\end{equation*}
where
\begin{align*}
J_0[\chi\phi ]& \doteq\frac{1}{2}r^{p-1}\left[pD-D'r\right](\partial_r(\chi\phi ))^2 ,\\
J_1[\chi\phi]&\doteq D'r^{p-1}( \chi\phi ) \cdot \partial_r(\chi\phi).
\end{align*}
If $R$ is chosen suitably large we can apply the Morawetz estimate \eqref{ass:morawetzlocal} to estimate
\begin{equation*}
\sum_{|\alpha_1 | \leq 1 ,|\alpha_2 | \leq 1}\int_{\mathcal{A}^{u_2}_{u_1}}\mathcal{R}_{\chi} [\partial^{\alpha_1} \phi \cdot \partial^{\alpha_2}\phi ] \,r^2 drdu\leq C \int_{\Sigma_{u_1}} J^T[\psi ]\cdot n_{u_1} \,d\mu_{{u_1}},
\end{equation*}
where $C>0$ depends in particular on $R>0$ and the choice of cut-off function $\chi$. We immediately obtain that $J_0$ is positive-definite for $0<p < 3$ and $R>0$ suitably large. Indeed, we can expand in the following way for any $\beta > 0$
\begin{align*}
D&=1-\frac{2M}{r}+O(r^{-1-\beta}),\\
D'&=\frac{2M}{r^2}+O(r^{-2-\beta}).
\end{align*}
The leading order term in $J_1$ can be estimated by integrating by parts:
\begin{equation*}
\begin{split}
\int_{\mathcal{A}^{u_2}_{u_1}}r^{p-3}( \chi\phi ) \cdot \partial_r(\chi \phi )\, drdu&=\frac{1}{2}\int_{\mathcal{A}^{u_2}_{u_1}}\partial_r(r^{p-3}(\chi\phi )^2)\,drdu-\frac{1}{2}(p-3) \int_{\mathcal{A}^{u_2}_{u_1}}r^{p-4}(\chi\phi )^2\, drdu\\
&=\frac{1}{2}\int_{\mathcal{I}^+}r^{p-3}(\chi\phi )^2\, du+\frac{1}{2} (3-p)\int_{\mathcal{A}^{u_2}_{u_1}}r^{p-4}(\chi\phi )^2\, drdu.
\end{split}
\end{equation*}
Both terms above have a good sign if $p<3$. The lower-order terms in $J_1$ can be absorbed into the right-hand side above if $R>0$ is suitably large. We put the above estimates into \eqref{eq:divthm0} to estimate
\begin{equation*}
\begin{split}
\int_{\mathcal{N}_{u_2}}& r^p(\partial_r(\chi\phi ))^2\,  dr+\int_{\mathcal{A}_{u_1}^{u_2}} r^{p-1}(\partial_r(\chi\phi ))^2  drdu\\
 \leq&\: \int_{\mathcal{N}_{u_1}} r^p(\partial_r(\chi\phi ))^2\, dr+C\int_{\Sigma_{u_1}}J^T[\psi ]\cdot n_{u_1} \,d\mu_{{u_1}},
\end{split}
\end{equation*}
which is the estimate that we want but for $\chi \phi$ instead of $\phi$. We can remove the cut-off $\chi$ in the above estimate as in the proof of Proposition \ref{prop:rpPhiv1}.
\end{proof}

\subsection{The full hierarchy}
\label{sec:hiernpzero}

We can extend the range of the power $p$ in Proposition \ref{en00}, if we assume additionally that the first Newman--Penrose constant $I_0[\psi]$ vanishes.

\begin{proposition}[\textbf{$r^p$-weighted estimates for $\psi_0$ with $I_0 [\psi ] = 0$}]\label{en01}
Let $\psi$ be a solution to \eqref{waveequation} emanating from initial data given as in Theorem \ref{thm:extuniq} on $(\mathcal{R}, g)$.

Assume that $\boldsymbol{I_0 [\psi ] =0}$ and take $p\in(0,4)$. Then there exists an $R>0$ such that for any $0\leq u_1<u_2$
\begin{equation}\label{en01est}
\begin{split}
\int_{{\mathcal{N}}_{u_2}} &r^p(\partial_r\phi )^2 dr+p\int_{\mathcal{A}_{u_1}^{u_2}} r^{p-1}(\partial_r\phi )^2\,drdu \leq \frac{C}{(4-p)^2}\int_{{\mathcal{N}}_{u_1}} r^p(\partial_r\phi )^2 dr\\
&+ \frac{C}{(4-p)^2}\int_{\Sigma_{u_1}} J^T [\psi ] \cdot n_{{u_1}} d\mu_{{u_1}},
\end{split}
\end{equation}
where $C=C(D,R)>0$ is a constant and we can take $R= p^{-1}R_0$ with $R_0=R_0(D)>0$ a constant.

Moreover, for $p\in [4,5)$ we have that
\begin{equation}\label{en01estv2}
\begin{split}
\int_{{\mathcal{N}}_{u_2}} &r^p(\partial_r\phi )^2 dr+p\int_{\mathcal{A}_{u_1}^{u_2}} r^{p-1}(\partial_r\phi )^2\,drdu \leq C\int_{{\mathcal{N}}_{u_1}} r^p(\partial_r\phi )^2 dr\\
&+ C\int_{\Sigma_{u_1}} J^T [\psi ] \cdot n_{{u_1}} d\mu_{{u_1}} + C  \frac{E_{0;\rm aux}^{\delta}[\psi]}{(1+u_1 )^{1 -{2\delta}}} , 
\end{split}
\end{equation}
for any $\delta > 0$ and
\begin{equation}\label{E0def}
E_{0;\rm aux}^{\delta}[\psi]= \sum_{l\leq 4} \int_{\Sigma_0} J^N [T^l\psi ] \cdot n_0\, d\mu_{\Sigma_0}+ \sum_{l\leq 3}\int_{\mathcal{N}_0} r^{4-l-\delta} (\partial_r  T^l\phi  )^2 \, d\omega dr<\infty,
\end{equation}
with $C=C(D,R,\delta )>0$ is a constant and we can take $R= p^{-1}R_0$ with $R_0=R_0(D)>0$ a constant.
\end{proposition}

\begin{proof}[Proof of \eqref{en01est}]
We have that
\begin{equation*}
r^2\textnormal{div} J^V[\chi\phi]=J_0[\chi\phi]+J_1[\chi\phi]+\sum_{|\alpha_1| \leq 1, |\alpha_2|\leq 1}\mathcal{R}_{\chi}[\partial^{\alpha_1} \phi\partial^{\alpha_2}\phi],
\end{equation*}
where
\begin{align*}
J_0[\chi\phi]&\doteq\frac{1}{2}r^{p-1}\left[pD-D'r\right](\partial_r(\chi\phi))^2,\\
J_1[\chi\phi]& \doteq D'r^{p-1}\chi\phi\partial_r(\chi\phi).
\end{align*}

First, let us restrict to the case $p\in(0,4)$. Instead of estimating the term $J_1$ by integrating by parts as in the proof of Proposition \ref{prop:rpphiv1}, we apply an $r$-weighted Cauchy--Schwarz inequality, with $\epsilon>0$ suitably small, to obtain
\begin{equation*}
|J_1|\leq \epsilon r^{p-1}p(\partial_r(\chi \phi ) )^2+ C_{\epsilon}p^{-1}r^{p-5}(\chi \phi )^2,
\end{equation*}
where $C_{\epsilon}>0$ is a constant that depends only on $\epsilon>0$. Therefore,
\begin{equation*}
\int_{\mathcal{A}^{u_2}_{u_1}} J_1\,drdu\leq \epsilon \int_{\mathcal{A}^{u_2}_{u_1}}p r^{p-1}(\partial_r(\chi \phi ) )^2\,drdu+C_{\epsilon}p^{-1}\int_{\mathcal{A}^{u_2}_{u_1}}r^{p-5}(\chi \phi )^2\,drdu.
\end{equation*}
The first term on the right-hand side can be absorbed into $J_0$. The second term can be estimated using the Hardy inequality \eqref{eq:Hardy1}
\begin{equation*}
\int_{\mathcal{A}^{u_2}_{u_1}}r^{p-5}(\chi \phi )^2\,drdu\leq \frac{4}{(4-p)^2}\int_{\mathcal{A}^{u_2}_{u_1}}r^{p-3}(\partial_r(\chi \phi ) )^2\,drdu,
\end{equation*}
for $p<4$ and the previous Dafermos--Rodnianski hierarchy. Observe that in order to apply \eqref{eq:Hardy1} we need that:
\begin{equation*}
\lim_{r\to \infty} r^{p-4}\phi^2=0,
\end{equation*}
which follows from Proposition \ref{prop:step0radfields}, for any $0<p<4$ because in particular $I_0[\psi]<\infty$.

We can remove the cut-off working as in Proposition \ref{prop:rpphiv1}.
\end{proof}

Let $\epsilon>0$. Then we can use \eqref{en01est} to obtain the following pointwise decay estimate for $R>0$ suitably large (depending on $\epsilon$)
\begin{equation}\label{dec320}
|\phi |(u,r) \leq C E_{0;\rm aux}^{\epsilon}[\psi](1+u)^{-\frac{3}{2} + \frac{\epsilon}{2}}.
\end{equation}
See Lemma \ref{eq:pointrtkpsi0v0} for a derivation. Using estimate \eqref{dec320} we can now alter the estimate \eqref{en01est} to extend the range of $p$ to $p \in (0,5)$.

\begin{proof}[Proof of \eqref{en01estv2}]
We now estimate the term $J_1$ by applying a Cauchy--Schwarz inequality with weights in $r$ \emph{and} weights in $u$
\begin{equation*}
|J_1|\leq \epsilon' r^{p} u^{-1-\eta}(\partial_r(\chi \phi ) )^2+ Cr^{p-6} u^{1+\eta}(\chi \phi )^2,
\end{equation*}
where we can take $\eta>0$ and $\epsilon'>0$ arbitrarily small. We estimate the first term as follows
\begin{equation*}
\int_{\mathcal{A}^{u_2}_{u_1}} u^{-1-\eta}r^p(\partial_r(\chi \phi ) )^2\,drdu\leq u_1^{-\eta}\sup_{u_1\leq u \leq u_2}\int_{\mathcal{N}_{u}}r^p(\partial_r(\chi \phi ) )^2\,dr.
\end{equation*}
By using \eqref{dec320}, we have that
\begin{equation*}
\int_{\mathcal{A}^{u_2}_{u_1}} u^{1+\eta}r^{p-6}(\chi \phi )^2\,drdu\leq CE_{0;\rm aux}^{\delta}[\psi]\int_{\mathcal{A}^{u_2}_{u_1}} u^{-2+\eta+\delta}r^{p-6}\,drdu\leq C\frac{E_{0; \rm aux}[\psi]}{(1+u_1 )^{1-2\delta}},
\end{equation*}
for $p<5$, for a constant $C$ depending on $R > 0$, and for $\eta=\delta$. 

We use the above estimates for $J_1$ when applying the divergence theorem with respect to the current $J^V$, with $V=r^{p-2}\partial_r$, but we ignore the spacetime terms with a good sign in order to obtain
\begin{equation*}
\int_{\mathcal{N}_{u_2}} r^{p}(\partial_r(\chi \phi ) )^2 dr \leq \epsilon' \sup_{u_1\leq u \leq u_2}\int_{\mathcal{N}_{u}}r^{p}(\partial_r (\chi\phi ) )^2\,dr+C\int_{\mathcal{N}_{u_1}} r^{p}(\partial_r (\chi \phi ) )^2\,dr+
\end{equation*}
$$ + C\int_{\Sigma_{u_1}}J^T [\psi ] \cdot n_{{u_1}} d\mu_{{u_1}} + C \frac{E_{0;\rm aux}^{\delta}[\psi]}{(1+u_1 )^{1-2\delta}}. $$

We can replace the term on the left-hand side by $\sup_{u_1\leq u \leq u_2}\int_{\mathcal{N}_{u}}r^{p}(\partial_r(\chi \phi ) )^2\,dr$, without changing the estimate, and absorb the term with a factor $\epsilon'$ into the left-hand side. Finally, we apply the divergence theorem once more, but now we include the spacetime integral on the left-hand side. We are left with
\begin{equation*}
\begin{split}
p \int_{\mathcal{A}_{u_1}^{u_2}}&r^{p-1}(\partial_{r}(\chi \phi ) )^{2}drdu+\sup_{u_1\leq u \leq u_2}\int_{\mathcal{N}_{u}} r^{p}(\partial_r (\chi \phi ) )^2 dr \leq C\int_{\mathcal{N}_{u_1}} r^{p}(\partial_r (\chi \phi ) )^2\,dr\\
&+ C\int_{\Sigma_{u_1}} J^T [\psi ] \cdot n_{{u_1}} d\mu_{{u_1}} + C \frac{E_{0;\rm aux}[\psi]}{(1+u_1 )^{1-2\delta}} ,
\end{split}
\end{equation*}
for $p<5$. The cut-off is removed in the same way as before.  
\end{proof}

\subsection{Aside: a hierarchy for $r^2 \partial_r (r\psi_0)$ with $I_0[\psi]=0$}
Below we provide some additional estimates for a spherically symmetric wave $\psi$ with vanishing first Newman--Penrose $I_0 [\psi ] = 0$ after commuting $\square_g$ with $r^2 \partial_r$ and using again the multiplier vector field $r^{p-2} \partial_r$. The estimates in this section will not be needed for the remainder of the paper.

\begin{proposition}[\textbf{$r^p$-weighted estimates for $r^2 \partial_r (r\psi_0)$ with $I_0 [\psi_0 ] = 0$}]\label{prop:r2rphi0}
Let $\psi$ be a solution to \eqref{waveequation} emanating from initial data given as in Theorem \ref{thm:extuniq} on $(\mathcal{R}, g)$.

Assume that $\boldsymbol{I_0[\psi]=0}$ and denote 
$$ \Phi = r^2 \cdot \partial_r \phi = r^2 \cdot \partial_r \phi_0 . $$

Then, for $p=2$ we have that
\begin{equation}
\label{ent0est:com1p2}
\begin{split}
\int_{{\mathcal{N}}_{u_2}}& r^p ( \partial_r \Phi )^2 \, dr+\int_{\mathcal{A}_{u_1}^{u_2}} r^{p-1}(\partial_r \Phi )^2\, drdu \leq C \int_{{\mathcal{N}}_{u_1}} r^p(\partial_r \Phi )^2 dr \\
&+ C\sum_{k\leq 1} \int_{\Sigma_{u_1}} J^T [T^k \psi ] \cdot n_{\Sigma_{u_1}} d\mu_{{u_1}}.
\end{split}
\end{equation}
Furthermore, if
\begin{equation}
\label{ass:r3prop}
\lim_{r\to \infty}r^{\frac{1}{2}}\Phi(0,r)<\infty,
\end{equation}
we have for $p\in [2,3)$ that
\begin{equation}
\label{ent0est:com1}
\begin{split}
\int_{{\mathcal{N}}_{u_2}}& r^p ( \partial_r \Phi )^2 \, dr+\int_{\mathcal{A}_{u_1}^{u_2}} r^{p-1}(\partial_r \Phi )^2\, drdu \\
&\leq C \int_{{\mathcal{N}}_{u_1}} r^p(\partial_r \Phi )^2 dr + C\sum_{k\leq 1} \int_{\Sigma_{u_1}} J^T [T^k \psi ] \cdot n_{\Sigma_{u_1}} d\mu_{{u_1}} + C \frac{E_{0,\rm comm}[\psi]}{(1+u_1 )^{1-\delta}} ,
\end{split}
\end{equation}
for any $\delta > 0$, with 
\begin{equation*}
\begin{split}
E_{0;\rm comm}[\psi] =&\: \sum_{k \leq 4} \int_{\Sigma_0} J^T[T^k \psi]\cdot n_0 \, d\mu_{\Sigma_0} + \sum_{l\leq 2} \int_{\mathcal{N}_0} r^{3-l} (\partial_r (T^l \phi ) )^2 \, dr \\
&+ \int_{\mathcal{N}_0} r^{2} (\partial_r \Phi )^2 \, dr,
\end{split}
\end{equation*}
$C=C(D,R,\delta )>0$ a constant and $R= p^{-1}R_0$ with $R_0=R_0(D)>0$ a constant.
\end{proposition}
\begin{proof}

According to the computations in Appendix \ref{app:commmultp} we obtain
\begin{equation*}
\begin{split}
r^{2}\cdot \textnormal{div} J^V [ \chi \Phi] &=J_0+J_1+J_2 + \sum_{|\alpha_1 | \leq 1 , |\alpha_2| \leq 2} \mathcal{R}_{\chi} [\partial^{\alpha_1} \Phi \cdot \partial^{\alpha_2} \Phi ] ,
\end{split}
\end{equation*}
 with
\begin{align*}
J_0[\chi\Phi]&\doteq \frac{1}{2}r^{p-1}\left[pD-3D'r\right](\partial_r(\chi \Phi ) )^2,\\
J_1[\chi\Phi]&\doteq r^{p-1}[-D''r+3D'-2Dr^{-1}](\chi \Phi) \cdot \partial_r (\chi \Phi ) ,\\
J_2[\chi\Phi]&\doteq r^{p+1}[D''+D'r^{-1}]( \chi \phi ) \cdot  \partial_r (\chi \Phi) ,
\end{align*}
for $\chi$ the same cut-off that we used in the previous propositions. First we note that for the term $\mathcal{R}_{\chi}$ by using the Morawetz estimate \eqref{ass:morawetzlocal} we have that
\begin{equation*}
\sum_{ |\alpha_1 | \leq 1, |\alpha_2 | \leq 1}\int_{\mathcal{A}^{u_2}_{u_1}}\mathcal{R}_{\chi} [\partial^{\alpha_1}  \Phi \cdot \partial^{\alpha_2} \Phi ]  \, drdu\leq C \sum_{k \leq 1} \int_{\Sigma_{u_1}} J^T[T^k \psi ]\cdot n_{u_1}\,d\mu_{{u_1}},
\end{equation*}
where $C>0$ depends in particular on $R>0$ and the choice of cut-off function $\chi$.

We note that $J_0$ is positive definite for $p > 0$. For the term $J_1$ we notice that its leading order term is
$$ -2 r^{p-2} (\chi \Phi ) \cdot \partial_r (\chi \Phi) . $$
We have that
\begin{equation}\label{comm0:j1}
 -2 \int_{\mathcal{A}_{u_1}^{u_2}} r^{p-2} (\chi \Phi ) \cdot \partial_r (\chi \Phi) drdu = - \int_{\mathcal{I}^{+}} r^{p-2} (\chi \Phi )^2 du + (p-2) \int_{\mathcal{A}_{u_1}^{u_2}} r^{p-3} (\chi \Phi )^2 drdu . 
 \end{equation}
For $p<3$ the term at null infinity is zero by Proposition \ref{consNP0} (which can be applied due to assumption \eqref{ass:r3prop}). Moreover, for $p \in [2,3)$ the spacetime term in the right hand side of \eqref{comm0:j1} has the right sign. 

For $J_2$ the leading order term in $r^{-1}$ is
$$ -2M r^{p-2} \cdot (\chi \phi ) \cdot \partial_r (\chi \Phi ) . $$
For $p=2$ we integrate by parts and we have that 
$$ -\int_{\mathcal{A}_{u_1}^{u_2}} 2M (\chi \phi ) \cdot \partial_r (\chi \Phi ) \, drdu = - \int_{\mathcal{I}^{+}} 2M (\chi \phi ) \cdot (\chi \Phi ) \, du + \int_{\mathcal{A}_{u_1}^{u_2}} 2M \chi^2 r^2 (\partial_r \phi )^2 \, drdu + $$ $$ + \int_{\mathcal{A}_{u_1}^{u_2}} 2M \chi' r^2 (\chi \phi) \cdot \partial_r \phi  \, drdu . $$
The first term vanishes again by Proposition \ref{consNP0}, the second term has a good sign, and the third one can be estimated by \eqref{ass:morawetzlocal}. We can therefore arrive at the estimate \eqref{ent0est:com1p2}.

We can use \eqref{ent0est:com1p2} to arrive at a result that is analogous to (ii) of Lemma \ref{lm:auxdecaypsi0}, but with $E^{\epsilon}_{0;\rm aux}$ replaced by $E_{0;\rm comm}$ and $r^{4-\epsilon}(\partial_r\phi)^2$ replaced by $r^2(\partial_r\Phi)^2$. A proof of this statement is straightforward and is omitted from this paper.

Then, via arguments similar to those in Lemma \ref{sec:pointdecayradfield}, we can arrive at the following pointwise bound
\begin{equation}\label{dec20prel}
|\phi |(u,r) \leq C \cdot \frac{\sqrt{E_{0; \rm comm}[\psi]}}{(1+u)^{3/2 - \delta/2}} \mbox{ for any $\delta > 0$}.
\end{equation}

Now for $p\in [2,3)$ we apply a $u$-weighted Cauchy--Schwarz inequality and we have that for $\eta > 0$ and $\epsilon' > 0$
$$ -2M r^{p-2} \cdot (\chi \phi ) \cdot \partial_r (\chi \Phi ) \leq \epsilon' \cdot u^{1+\eta} \cdot r^{p-4} (\chi \phi )^2 + C u^{-1-\eta} \cdot r^p \left( \partial_r (\chi \Phi ) \right)^2 . $$
We estimate the second term of the last expression as follows
\begin{equation*}
\int_{\mathcal{A}^{u_2}_{u_1}} u^{-1-\eta}r^p(\partial_r (\chi \Phi ) )^2\,drdu\leq u_1^{-\eta}\sup_{u_1\leq u \leq u_2}\int_{\mathcal{N}_{u}}r^p(\partial_r (\chi \Phi ) )^2\,dr.
\end{equation*}
By using \eqref{dec20prel}, we have that
\begin{equation*}
\int_{\mathcal{A}^{u_2}_{u_1}} u^{1+\eta}r^{p-4}(\chi \phi )^2\,drdu\leq CE_0\int_{\mathcal{A}^{u_2}_{u_1}} u^{-3+\eta+\epsilon}r^{p-4}\,drdu\leq C\frac{E_{0; \rm comm}}{(1+u_1 )^{1-\delta}},
\end{equation*}
for $2 < p<3$, for a constant $C$ depending on $R > 0$, and for $\delta = \epsilon + \eta$.

We use the above estimates for $J_2$ when applying the divergence theorem with respect to the current $J^V$, with $V=r^{p-2}\partial_r$, and we ignore the spacetime terms with a good sign in order to obtain
\begin{equation*}
\int_{\mathcal{N}_{u_2}} r^{p}(\partial_r(\chi \Phi ) )^2 dr \leq \epsilon' \sup_{u_1\leq u \leq u_2}\int_{\mathcal{N}_{u}}r^{p}(\partial_r(\chi \Phi ) )^2\,dr+C\int_{\mathcal{N}_{u_1}} r^{p}(\partial_r (\chi \Phi ) )^2\,dr+
\end{equation*}
$$ + C\sum_{k \leq 1} \int_{\Sigma_{u_1}}J^T [T^k \psi ] \cdot n_{{u_1}} d\mu_{{u_1}} + C \frac{E_{0;\rm comm}[\psi]}{(1+u_1 )^{1-\delta}}. $$

We can replace the term on the left-hand side by $\sup_{u_1\leq u \leq u_2}\int_{\mathcal{N}_{u}}r^{p}(\partial_r (\chi \Phi ) )^2\,dr$, without changing the estimate, and absorb the term with a factor $\epsilon'$ into the left-hand side. Finally, we apply the divergence theorem once more, but now we include the spacetime integral on the left-hand side. We are left with
\begin{equation*}
\begin{split}
(p-2) \int_{\mathcal{A}_{u_1}^{u_2}}&r^{p+1}(\chi \partial_{r} \phi )^{2}drdu+\int_{\mathcal{A}_{u_1}^{u_2}} r^{p-1}(\partial_{r} (\chi \Phi ) )^{2}drdu +\sup_{u_1\leq u \leq u_2}\int_{\mathcal{N}_{u}} r^{p}(\partial_r (\chi \Phi ) )^2 dr \leq \\
&\leq C\int_{\mathcal{N}_{u_1}} r^{p}(\partial_r (\chi \Phi ) )^2\,dr+ C\sum_{k\leq 1} \int_{\Sigma_{u_1}} J^T [T^k \psi ] \cdot n_{{u_1}} d\mu_{{u_1}} + C \frac{E_{0;\rm comm}[\psi]}{(1+u_1 )^{1-\delta}} ,
\end{split}
\end{equation*}
for $2 <  p<3$. The cut-off can be removed similarly as before after noticing that
$$  \int_{\mathcal{N}_{u}} r^{p}(\partial_r (\chi \Phi ) )^2 dr \leq C \int_{\mathcal{N}_{u}} r^{p}(\partial_r \Phi  )^2 dr + C\sum_{k\leq 1} \int_{\Sigma_{u_1}} J^T [T^k \psi ] \cdot n_{{u_1}} d\mu_{{u_1}} , $$
and
$$  \int_{\mathcal{N}_{u}} r^{p}(\partial_r  \Phi  )^2 dr \leq C \int_{\mathcal{N}_{u}} r^{p}(\partial_r (\chi \Phi )  )^2 dr + C\sum_{k\leq 1} \int_{\Sigma_{u_1}} J^T [T^k \psi ] \cdot n_{{u_1}} d\mu_{{u_1}} . $$
\end{proof}
\begin{remark}
It should be noted that the $r$-weighted energies on the right-hand sides of \eqref{ent0est:com1p2} and \eqref{ent0est:com1} can be finite even if $I_0 [\psi ]$ is non-vanishing (of course, we need to in particular use the pointwise condition $I_0[\psi]=0$ to prove Proposition \ref{prop:r2rphi0} in the first place). This is in contrast to Proposition \ref{en01} where the $r$-weighted energies on the right-hand sides of \eqref{en01est} and \eqref{en01est} are not finite if $I_0 [\psi ]$ is non-vanishing. This makes Proposition \ref{prop:r2rphi0} more useful for nonlinear applications, where the Newman--Penrose quantity need to be conserved.
\end{remark}

\subsection{The extended hierarchy for $\partial_r^{k+1} \psi_0$ and $\partial_r^kT\psi$, $k \geq 0$}
\label{sec:rphierTkpsi0}
In this section we will prove some additional estimates for $\partial_r^{k+1} \psi_0$ and $\partial_r^kT\psi_0$, with $k \in \mathbb{N}_0$. Note that we always have $I_0 [T^k \psi ] = 0$ (even when $I_0 [\psi ] \neq 0$) so the estimates provided by Proposition \ref{en01} can directly be applied to $T^k \psi$. We will however prove alternative hierarchies in this section.

\begin{lemma}
\label{lm:commboxdrkpsi0}
Let $\psi$ be a solution to \eqref{waveequation} emanating from initial data given as in Theorem \ref{thm:extuniq} on $(\mathcal{R}, g)$. Then we have that for any $k \in \mathbb{N}$
\begin{equation}\label{eq:partialkphi}
\Box_g (\partial_r^k \phi ) = \left( \frac{2}{r} + O (r^{-2} ) \right) \cdot \partial_r^{k+1} \phi + \sum_{m=0}^k O (r^{-m-3} ) \cdot  \partial_r^{k-m} \phi - \frac{2}{r} \partial_u \partial_r^k \phi . 
\end{equation}
\end{lemma}
\begin{proof}
See Appendix \ref{app:commmultp}.
\end{proof}

\begin{proposition}
\label{enr0}
Let $\psi$ be a solution to \eqref{waveequation} emanating from initial data given as in Theorem \ref{thm:extuniq} on $(\mathcal{R}, g)$.

Let $n\in \N_0$ and assume that $D=1-\frac{2M}{r}+O_{n+2}(r^{-1-\beta})$. 
\begin{itemize}
\item[\emph{(i)}] Assume that $\boldsymbol{I_0 [\psi ] = 0}$. 
Let $0\leq k \leq n$ and take $p\in(2k,4+2k)$. Then there exists an $R>0$ such that for any $0\leq u_1<u_2$
\begin{equation}
\label{enr0est}
\begin{split}
\int_{\mathcal{N}_{u_2}}& r^p(\partial_r^{k+1}\phi)^2\, dr+\int_{\mathcal{A}_{u_1}^{u_2}} pr^{p-1}(\partial_r^{k+1}\phi)^2\, drdu \\
\leq&\: C(p-4-2k)^{-1}\sum_{j=0}^k\int_{\mathcal{N}_{u_1}} r^{p-2j}(\partial_r^{k-j+1}\phi)^2\, dr\\
&+C(p-4-2k)^{-1}\sum_{j\leq k}\int_{\Sigma_{u_1}}J^T[T^{j}\psi]\cdot n_{\Sigma_{u_1}}\,d\mu_{\Sigma_{u_1}},
\end{split}
\end{equation}
where $C\doteq C(k,D,R)>0$ is a constant and we can take $R=p^{-1}(p-4-2k)^{-1}R_0(D)>0$, with $R_0(k,D)>0$ a constant.

Furthermore, take $p\in[2k+4,5+2k)$, then there exists an $R>0$ such that for any $0\leq u_1<u_2$ and any $\delta \in (0,1)$
\begin{equation}
\label{enr0estp5}
\begin{split}
\int_{\mathcal{N}_{u_2}}& r^p(\partial_r^{k+1}\phi)^2\, dr+\int_{\mathcal{A}_{u_1}^{u_2}} pr^{p-1}(\partial_r^{k+1}\phi)^2\, drdu \leq C\sum_{j=0}^k\int_{\mathcal{N}_{u_1}} r^{p-2j}(\partial_r^{k-j+1}\phi)^2\, dr\\
&+C\sum_{j\leq k}\int_{\Sigma_{u_1}}J^T[T^{j}\psi]\cdot n_{\Sigma_{u_1}}\,d\mu_{\Sigma_{u_1}}+CE^{\delta}_{0,\rm aux}[\psi](1+u_1)^{-1-\delta},
\end{split}
\end{equation}
where $C\doteq C(k,D,R)>0$ is a constant and we can take $R=p^{-1}R_0(D)>0$, with $R_0(k,D)>0$ a constant.
\item[\emph{(ii)}]
Assume that $\boldsymbol{I_0 [\psi ]\neq  0}$. Then \eqref{enr0est} holds with the restricted range $p\in (2k,2k+3)$.
\end{itemize}
\end{proposition}
\begin{proof}
Let $V=r^{p-2}\partial_r$. We have that
$$ K^V[\chi  \partial_r^k \phi ]=\frac{1}{2}r^{p-3}\left[D(p-4)-D'r\right]\left(\partial_r( \chi  \partial_r^k \phi ) \right)^2+2r^{p-3}\partial_r(\chi  \partial_r^k \phi ) \cdot \partial_u(\chi  \partial_r^k \phi ), $$
and using \eqref{eq:partialkphi}, we moreover have that
\begin{equation*}
 \begin{split}
\mathcal{E}^V[\chi \partial_r^k \phi ]&=r^{p-2}\partial_r(\chi \partial_r^k \phi ) \square_g(\chi \partial_r^k \phi )\\
&=-2r^{p-3}\partial_u(\chi \partial_r^k \phi ) \cdot \partial_r(\chi \partial_r^k \phi ) +r^{p-2}\left(\frac{2}{r} +O(r^{-2} ) \right)\left(\partial_r(\chi \partial_r^k \phi )\right)^2+ \\
&+ \sum_{m=0}^k r^{p-2} O(r^{-m-3} ) (\chi \partial_r^{k-m} \phi ) \cdot \partial_r(\chi  \partial_r^k \phi ) + \sum_{|\alpha_1|\leq 1, |\alpha_2|\leq 1}\mathcal{R}_{\chi}[\partial^{\alpha_1}   \partial_r^k \phi \cdot\partial^{\alpha_2} \partial_r^k \phi ] .
\end{split} 
\end{equation*}
Therefore,
\begin{equation*}
r^2\textnormal{div}J^V[\chi\partial_r^k\phi]=J_0[\chi \partial_r^k\phi]+\sum_{m=1}^{k+1}J_m[\chi \partial_r^k\phi]+ \sum_{|\alpha_1|\leq 1, |\alpha_2|\leq 1}\mathcal{R}_{\chi}[\partial^{\alpha_1}   \partial_r^k \phi \cdot\partial^{\alpha_2} \partial_r^k \phi ] ,
\end{equation*}
where
\begin{align*}
J_0[\chi \partial_r^k\phi]\doteq &\:\frac{1}{2}r^{p-1}\left[p+O(r^{-1})\right]\left(\partial_r( \chi  \partial_r^k \phi ) \right)^2,\\
J_{l+1}[\chi \partial_r^k\phi]\doteq &\:O(r^{p-3-l} ) (\chi \partial_r^{k-l} \phi ) \cdot \partial_r(\chi  \partial_r^k \phi ), \quad \textnormal{for}\: 0\leq l\leq \max\{k-1,0\},\\
J_{k+1}[\chi \partial_r^k\phi]\doteq &\: O(r^{p-4-k})\chi\phi \cdot \partial_r(\chi  \partial_r^k \phi ).
\end{align*}
We will first prove by induction that \eqref{enr0est} holds for all $0\leq k\leq n$. First of all, \eqref{enr0est} holds for $k=0$ by Proposition \ref{en01} . Let us now suppose \eqref{enr0est} holds for all $0\leq k\leq n$. We want to show that then \eqref{enr0est} also holds for $k=n+1$. We therefore fix $k=n+1$.

First of all, we can estimate by \eqref{ass:morawetzlocal}:
\begin{equation*}
\begin{split}
\sum_{|\alpha_1|\leq 1,\,|\alpha_2|\leq 1}\int_{\mathcal{A}^{u_2}_{u_1}}\mathcal{R}_{\chi}[\partial^{\alpha_1}\partial_r^{n+1}{\phi}\cdot \partial^{\alpha_2}\partial_r^{n+1}{\phi}]\, dr du\leq&\: C\sum_{|\alpha|\leq n+2}\int_{\mathcal{A}^{u_2}_{u_1}\cap \{r\leq R+1\}}(\partial^{\alpha}\psi)^2\,d\omega dr du\\
\leq &\: C\sum_{j\leq n+1}\int_{\Sigma_{u_1}}J^T[T^{j}\psi]\cdot n_{\Sigma_{u_1}}\,d\mu_{\Sigma_{u_1}}.
\end{split}
\end{equation*}

The term in $J_0$ is positive for all $p> 0$ if we take $R\geq p^{-1}R_0$, with $R_0=R_0(D,n)>0$ suitably large. We estimate $J_{l+1}
$ with $1\leq l\leq n$ by applying a Cauchy--Schwarz inequality:
\begin{equation*}
|J_{l+1}|\leq \epsilon r^{p-1}p\left(\partial_r( \chi  \partial_r^{n+1} \phi ) \right)^2+C_{\epsilon}p^{-1}r^{p-5-2l} (\chi  \partial_r^{n+1-l} \phi )^2
\end{equation*} 
The first term on the right-hand side can be absorbed into $J_0$ for suitably small $\epsilon$. We estimate the second term as follows:
\begin{equation*}
C_{\epsilon}p^{-1}r^{p-5-2l} (\chi  \partial_r^{n+1-m} \phi )^2\leq C_{\epsilon} p r^{p-3-2l}(\chi \partial_r^{n+1-m}\phi)^2,
\end{equation*}
if we take $R\geq p^{-1}R_0$. We can estimate the right-hand side by applying \eqref{enr0est} with $k=n+1-l$ if $2k<p-2-2m< 4+2k$, or equivalently, $2(n+1)<p<4+2(n+1)$, provided we take $R\geq p^{-1}(p-4-2(n+1))^{-1}R_0$.

We are left with estimating $J_{n+2}$. We apply once again Cauchy--Schwarz to estimate
\begin{equation*}
|J_{n+2}|\leq \epsilon r^{p-1}p\left(\partial_r( \chi  \partial_r^{n+1} \phi ) \right)^2+C_{\epsilon}p^{-1}r^{p-7-2(n+1)} (\chi \phi )^2,
\end{equation*}
and we estimate further
\begin{equation*}
C_{\epsilon}p^{-1}r^{p-7-2(n+1)} (\chi  \phi )^2\leq C_{\epsilon} p r^{p-5-2(n+1)}(\chi \phi)^2,
\end{equation*}
by taking $R\geq p^{-1}R_0$. Now, we apply \eqref{eq:Hardy1} to estimate the spacetime integral of the right-hand side as follows:
\begin{equation*}
\int_{\mathcal{A}^{u_2}_{u_1}} pr^{p-5-2(n+1)}(\chi \phi)^2\, dr du\leq C(p-2-2m)^{-1}\int_{\mathcal{A}^{u_2}_{u_1}} pr^{p-3-2(n+1)}(\partial_r(\chi\phi))^2\, dr du
\end{equation*}
where we use that $p<4+2(n+1)$. The term on the right-hand side can be estimated by applying Proposition \ref{en01}.

Finally, we put the above estimates together and apply the divergence theorem on $\textnormal{div}J^V[\chi \partial_r^{n+1}\phi]$ to conclude that \eqref{enr0est} also holds for $k=n+1$. We easily obtain part (ii) of the proposition by restricting the range of $p$ to $p<3+2k$.

In order to prove \eqref{enr0estp5} we repeat the arguments above, with the only difference being the estimate of $J_{n+2}$. Here, we apply a $u$-weighted Cauchy--Schwarz inquality
\begin{equation*}
|J_{n+2}|\leq \epsilon r^{p}u^{-1-\eta}\left(\partial_r( \chi  \partial_r^{n+1} \phi ) \right)^2+C_{\epsilon}u^{1+\eta}r^{p-9-2(n+1)} (\chi \phi )^2
\end{equation*}
and use the pointwise bound \eqref{dec320}, as in the proof of Proposition \ref{en01}, to extend the range of $p$ to $2(n+1)<p<5+2(n+1)$ for $k=n+1$.
\end{proof}

We will use Proposition \ref{enr0} in order to present an augmented hierarchy of $r^p$--weighted estimates for $T^k \psi$ with $k \geq 1$, where $\psi$ is spherically symmetric.
\begin{proposition}\label{ent0}
Let $\psi$ be a solution to \eqref{waveequation} emanating from initial data given as in Theorem \ref{thm:extuniq} on $(\mathcal{R}, g)$.

Let $n\in \N_0$ and assume that $D=1-\frac{2M}{r}+O_{n+2}(r^{-1-\beta})$. 
\begin{itemize}
\item[\emph{(i)}] Assume that $\boldsymbol{I_0 [\psi ] = 0}$. 
Let $1\leq k \leq n$ and take $p\in(2k,4+2k)$. Then there exists an $R>0$ such that for any $0\leq u_1<u_2$
\begin{equation}
\label{ent0est}
\begin{split}
\int_{\mathcal{A}_{u_1}^{u_2}}pr^{p-1}(\partial_r^{k}T\phi)^2\, drdu \leq&\: C(p-4-2k)^{-1}\sum_{j=0}^k\int_{\mathcal{N}_{u_1}} r^{p-2j}(\partial_r^{k-j+1}\phi)^2\, dr\\
&+C(p-4-2k)^{-1}\sum_{j\leq k}\int_{\Sigma_{u_1}}J^T[T^{j}\psi]\cdot n_{\Sigma_{u_1}}\,d\mu_{\Sigma_{u_1}},
\end{split}
\end{equation}
where $C\doteq C(k,D,R)>0$ is a constant and we can take $R=p^{-1}(p-4-2k)^{-1}R_0(D)>0$, with $R_0(k,D)>0$ a constant.
\item[\emph{(ii)}]
Assume that $\boldsymbol{I_0 [\psi ]\neq  0}$. Then \eqref{ent0est} holds with the restricted range $p\in (2k,2k+3)$.
\end{itemize}
\end{proposition}

\begin{proof}
Let $D=1-\frac{2M}{r}+O_{k+2}(r^{-1-\beta})$, with $k\in \N$. We can commute \eqref{eq:sDeltapsiinftyv0} with $\partial_r^k$, with $k\in \N$, $k\geq1$, to arrive at
\begin{equation*}
\partial_r^k T \phi = [1+O(r^{-1})] \partial_r^{k+1}  \phi + \sum_{m=0}^k O (r^{-m-2} ) \partial_r^{k-m} \phi.
\end{equation*}

We therefore obtain
\begin{equation*}
\begin{split}
\int_{\mathcal{A}_{u_1}^{u_2}} pr^{p-1}(\chi \partial_r^k T \phi )^2\, drdu\leq&\: C\int_{\mathcal{A}_{u_1}^{u_2}}p r^{p-1}(\partial_r(\chi \partial_k \phi )^2\, drdu\\
&+C\sum_{m=0}^{k-1}\int_{\mathcal{A}_{u_1}^{u_2}} pr^{p-1-2m-4}\chi^2(\partial_r^{k-m} \phi )^2\, drdu\\
&+C\int_{\mathcal{A}_{u_1}^{u_2}} pr^{p-1-2k-4}\chi^2\phi^2\, drdu\\
&+\int_{\mathcal{A}_{u_1}^{u_2}} \chi'^2(\partial_r^k\phi)^2\,drdu.
\end{split}
\end{equation*}
We apply Proposition \ref{enr0} with $k=n$ to estimate the first term on the right-hand side and Proposition \ref{enr0} with $k$ replaced by $k-m$ and $2(k-m)< p-2m<4+2(k-m)$ to estimate the next $k$ terms. In order to estimate the term with a factor $\phi^2$ we apply \eqref{eq:Hardy1} (using that $p<4+2k$), followed by Proposition \ref{en01} with $p$ replaced by $p-2k$. The final term with a factor $\chi'^2$ can be estimated by applying \eqref{ass:morawetzlocal} as usual.
\end{proof}

\section{Sharpness of the hierarchy for $\psi_0$}
\label{sharpnessofhierarchy}
In this section we will show that the range of $p$ is sharp in Proposition \ref{en00} and \ref{en01}. 

\begin{proposition}\label{prop:sharpn0}
Let $\psi$ be a solution to \eqref{waveequation} emanating from initial data given as in Theorem \ref{thm:extuniq} on $(\mathcal{R}, g)$.

Assume that $\boldsymbol{I_0 [\psi ] \neq 0}$. Then the range of $p$ in estimate \eqref{en00est} is sharp, i.e.\ for any fixed $0\leq u_0 < \infty$ it holds that
\begin{equation}\label{est:sharpn0}
 \int_{\mathcal{A}_{u_0}^{\infty}} r^{2} (\partial_r \phi )^2 \, drdu = \infty .
\end{equation}

\end{proposition}
\begin{proof}
From the estimates in Section \ref{sec:limitsnullinfty} and the assumption that $\boldsymbol{I_0 [\psi ] \neq 0}$, we can estimate
\begin{equation*}
r^4(\partial_r\phi)^2(u_0,r)\geq \frac{1}{2} I_0^2[\psi],
\end{equation*}
for $r\geq \widetilde{R}(u_0)$, with $\widetilde{R}(u_0)>0$ suitably large. 

From the above, it follows that 
\begin{equation}\label{est:auxsharp1}
\int_{\mathcal{N}_{u_0}} r^3 (\partial_r \phi )^2 \, dr = \infty .
\end{equation}
By applying the divergence theorem as in Proposition \ref{en00} with $p=3$ but instead in the region $\mathcal{\bar{A}}_{u_0}^{V} \doteq \mathcal{A}_{u_0}^{\infty} \cap \{ v \leq V \}$ for any fixed $u_0 \geq 0$ and for any $V \geq V_0 (u_0 , R)$ (with $\chi$ a cut-off as in the previous propositions), we obtain:
$$ \int_{\mathcal{\bar{A}}_{u_0}^{V}} r^2 (3D - D' r)  (\partial_r (\chi \phi ) )^2 \, drdu + \int_{\mathcal{\bar{A}}_{u_0}^{V}} r^2 D' (\chi \phi) \cdot (\partial_r (\chi \phi ) ) \, drdu = \int_{\mathcal{N}_{u_0} \cap \{ v \leq V \}} r^3 (\partial_r \phi )^2 \, dr \Rightarrow $$ 
\begin{equation}\label{est:auxsharp2}
I + II + III = \int_{\mathcal{N}_{u_0} \cap \{ r \leq V \}} r^3 (\partial_r \phi )^2 \, dr,
\end{equation}
where

\begin{equation*}
I= \int_{\mathcal{\bar{A}}_{u_0}^{V}} r^2 3D  (\partial_r (\chi \phi ) )^2 \, drdu \geq c(u_0)  \int_{\mathcal{\bar{A}}_{u_0}^{V}} r^2 (\partial_r (\chi \phi ) )^2,
\end{equation*}
for some constant $c(u_0)>0$, and for all $V>\infty$
$$ II = -  \int_{\mathcal{\bar{A}}_{u_0}^{V}} D' r^3  (\partial_r (\chi \phi ) )^2 \, drdu = -  \int_{\mathcal{\bar{A}}_{u_0}^{V}} \left( 2Mr + O (r^{1-\beta} ) \right)  (\partial_r (\chi \phi ) )^2 \, drdu \leq C(u_0), $$
for some constant $C(u_0)>0$, by Proposition \ref{en00} with $p=2$.

Furthermore, we have that for $0<\eta <1$ and for all $V\infty$
\begin{align*}
III = \int_{\mathcal{\bar{A}}_{u_0}^{V}} r^2 D' (\chi \phi) \cdot (\partial_r (\chi \phi ) ) \, drdu \leq&\: C \int_{\mathcal{\bar{A}}_{u_0}^{V}} \frac{1}{r^{2-\eta}}(\chi \phi)^2 \, drdu +  C \int_{\mathcal{\bar{A}}_{u_0}^{V}} r^{2-\eta}  (\partial_r (\chi \phi ) )^2 \, drdu\\
 \leq&\: C \int_{\mathcal{\bar{A}}_{u_0}^{V}} r^{\eta} (\partial_r (\chi \phi ) )^2 \, drdu+C\int_{\mathcal{\bar{A}}_{u_0}^{V}} r^{2-\eta} (\partial_r (\chi \phi ) )^2 \, drdu\\
 \leq&\: C(u_0),
\end{align*}
for some constant $C(u_0)>0$, , by Proposition \ref{en00} with $p=3-\eta$ and an application of \eqref{eq:Hardy1}.

We note that by taking $V \rightarrow \infty$ the right-hand side of \eqref{est:auxsharp2} becomes infinite due to \eqref{est:auxsharp1}. Since $II+III\to \infty$, the term $I$ of the left-hand side of  of \eqref{est:auxsharp2} therefore goes to infinity as $V\to \infty$.

The statement of the proposition then follows after removing the cut-off $\chi$,  by estimating the terms involving $\chi'$ with Cauchy--Schwarz and the Morawetz estimate \eqref{ass:morawetzlocal}.  
\end{proof}

We consider now the case of vanishing first Newman--Penrose constant. For this we will need the following auxiliary lemma.
\begin{lemma}\label{dataT}
If $M\neq 0$, there exists spherically symmetric initial data for a solution $\psi$ to \eqref{waveequation} such that
\begin{equation}
\label{ass:cptsupp}
\psi|_{\mathcal{N}_0}=0,\\
\end{equation}
and moreover, $\psi=T\widetilde{\psi}$, where
\begin{equation*}
r^2\partial_r(r\widetilde{\psi} )(0,r)=I_0 [\widetilde{\psi} ]+O_3(r^{-\alpha})
\end{equation*}
and
\begin{equation*}
\partial_v(r\widetilde{\psi} )|_{\mathcal{N}_0}(v)=2I_0 [\widetilde{\psi} ] v^{-2}+O_3(v^{-2-\alpha}).
\end{equation*}
with $I_0 [\widetilde{\psi} ] \neq 0$.
\end{lemma}
\begin{proof}
We want to construct a spherically symmetric function $\widetilde{\psi}_0$ on $\mathcal{N}$ such that $T\widetilde{\psi}=0$. Since $\widetilde{\psi}$ is a solution to \eqref{waveequation}, condition (\ref{ass:cptsupp}) implies that
\begin{equation*}
D\partial_r^2\widetilde{\psi}+(D'+2r^{-1}D)\partial_r\widetilde{\psi} =0
\end{equation*}
on $\mathcal{N}$; see \eqref{eq:derivwaveeq}.

We multiply both sides by $r^2$ and rearrange terms to obtain
\begin{equation*}
\partial_r(Dr^2\partial_r\widetilde{\psi} )(u,r)=0.
\end{equation*}
for $r\geq R$. In particular, we require
\begin{equation*}
Dr^2\partial_r\widetilde{\psi} (0,r)=C_0
\end{equation*}
for $r\geq R$ and $C_0\in \R$. Let $\lim_{r\to \infty} \widetilde{\psi} (0,r)=0$, then we can integrate to obtain
\begin{equation*}
\widetilde{\psi} (0,r)=-C_0\int_{r}^{\infty} D^{-1}(r')r'^{-2}\,dr'.
\end{equation*}
By our assumptions on $D$ we have that for any $\beta > 0$
\begin{equation*}
D^{-1}=1+\frac{2M}{r}+O_3(r^{-1-\beta}),
\end{equation*}
and hence
\begin{equation*}
\widetilde{\psi} (0,r)=-C_0r^{-1}-C_0Mr^{-2}+O_3(r^{-2-\beta}).
\end{equation*}
In particular,
\begin{align*}
r\widetilde{\psi} (0,r)=&-C_0-C_0Mr^{-1}+O_3(r^{-1-\beta}),\\
r^2\partial_r(r\widetilde{\psi} )(0,r)=&\:C_0M+O_3(r^{-\beta}).
\end{align*}
We have that $I_0 [\widetilde{\psi} ] =\lim_{r\to \infty}r^2\partial_r(r\widetilde{\psi} )(0,r)=C_0M$, so we can take $I_0 [\psi ] \neq 0$ if and only if $M\neq 0$. By the assumptions on the metric $g$ from Section \ref{sec:geomassm}, we have that $\partial_vr=\frac{1}{2}D$ and $v|_{\mathcal{N}}=2r_*(r)$, where 
\begin{equation*}
r_*(r)=r-2M\log \left(\frac{r}{R}\right)+O_3(r^{-\beta})+O_3(R^{-\beta}).
\end{equation*}
So we can estimate
\begin{equation*}
v^2\partial_v(r\widetilde{\psi} )|_{\mathcal{N}_0}(v)=2I_0 [\widetilde{\psi} ]+O_3(v^{-\beta}). \qedhere
\end{equation*}
\end{proof}

\begin{remark}
In our companion paper \cite{paper2}, the construction of $\widetilde{\psi}_0$ such that $T\widetilde{\psi}_0=\psi_0$ in Lemma \ref{dataT} is carried out for more general initial data for $\psi$ and the constant $C_0\in \R$ is expressed \emph{explicitly} in terms of this initial data.
\end{remark}

\begin{proposition}\label{prop:sharp0}
Let $\psi$ be a solution to \eqref{waveequation} emanating from initial data given as in Theorem \ref{thm:extuniq} on $(\mathcal{R}, g)$.

Assume that $\boldsymbol{I_0 [\psi ] = 0}$. Then the range of $p$ in estimate \eqref{en00est} is \textbf{generically} sharp, i.e. for any fixed $0\leq u_0 < \infty$ it holds that
\begin{equation}\label{est:sharpn02}
 \int_{\mathcal{A}_{u_0}^{\infty}} r^{4} (\partial_r \phi )^2 \, drdu = \infty
 \end{equation}
 for $\psi$ arising from generic, compactly supported initial data on $\Sigma$.
 \end{proposition}
\begin{proof}
We consider the set of solutions constructed in Lemma \ref{dataT}, i.e. we look at $T\widetilde{\psi}$ for a spherically symmetric linear wave $\widetilde{\psi}$ with $I_0 [\widetilde{\psi} ] \neq 0$ and we denote $\widetilde{\phi}=r\widetilde{\psi}$. We consider the region $\mathcal{\bar{A}}^V_{u_0}$ that was introduced in the proof of Proposition \ref{prop:sharpn0}, and we have for $p=5$ that
\begin{align*}
\int_{\mathcal{\bar{A}}_{u_0}^{V}} r^4 (\partial_r (T \widetilde{\phi} ) )^2 \, drdu =&\: \int_{\mathcal{\bar{A}}_{u_0}^{V}} r^4 \left( (\partial_r (D \partial_r \widetilde{\phi} ) )^2 + \frac{(D' )^2}{r^2} \widetilde{\phi}^2 - \frac{2D'}{r} \partial_r (D \partial_r \widetilde{\phi} ) \cdot \widetilde{\phi} \right) \, drdu\\
 =&\: I + II + III ,
 \end{align*}
and it can be easily checked that $II + III < \infty$. We then have that
\begin{align*}
I =&\: \int_{\mathcal{\bar{A}}_{u_0}^{V}} r^4 (\partial_r (D \partial_r \widetilde{\phi} ) )^2 \, drdu  = \int_{\mathcal{\bar{A}}_{u_0}^{V}} r^4 \left( D^2 (\partial_r^2 \widetilde{\phi})^2 + (D' )^2 (\partial_r \widetilde{\phi} )^2 + 2 D D' (\partial_r^2 \widetilde{\phi} ) \cdot (\partial_r \widetilde{\phi} ) \right) \, drdu \\
=&\: IV + V + VI ,
\end{align*}
and it can be easily checked again that $V + VI < \infty$, while for $IV$, by the Hardy inequality \eqref{eq:Hardy1}, there exists a constant $C>0$ such that
$$IV \geq C \int_{\mathcal{\bar{A}}_{u_0}^{V}} r^2 (\partial_r \widetilde{\phi} )^2 \, drdu = \infty , $$
where in the last step we used the previous result of Proposition \ref{prop:sharpn0}, as $I_0 [\widetilde{\psi}] \neq 0$.

Now consider a \emph{general} solution $\psi$ to \eqref{waveequation} emanating from initial data given as in Theorem \ref{thm:extuniq} on $(\mathcal{R}, g)$. By linearity of \eqref{waveequation}, the function
\begin{equation*}
\psi^{\epsilon}=\psi+\epsilon T\widetilde{\psi}
\end{equation*}
is then also a solution \eqref{waveequation} and for $\epsilon>0$ suitably small, the initial data of $\psi^{\epsilon}$ on $\Sigma$ lies arbitrarily close to the initial data for $\psi$ on $\Sigma$ with respect to any (weighted) initial data norm on $\Sigma$, as the initial data of $\epsilon T\widetilde{\psi}$ is compactly supported. Therefore, solutions to \eqref{waveequation} emanating from \emph{generic} initial data (with respect to any weighted energy norm) satisfy \eqref{est:sharpn02}.
\end{proof}
 
\section{Energy decay estimates}
\label{sec:energydecay}
In this section we obtain energy decay estimates for solutions $\psi$ to \eqref{waveequation} with respect to the timelike vector field $N$ as an application of the hierarchies of $r^p$-weighted estimates developed in Section \ref{sec:hierpsi1} and Section \ref{sec:hierpsi0}. 

Additionally, we show how improved energy decay estimates for solutions of the form $T^k\psi$, with $k\in \N_0$, follow naturally in this formalism from the additional hierarchies of $r^p$-weighted estimates for $\partial_r^{k+1}(r\psi)$ and $\partial_r^{k}T(r\psi)$.

We split $\psi=\psi_0+\psi_1$ and prove energy decay estimates for $\psi_0$ in Section \ref{sec:edecaypsi0} and energy decay estimates of $\psi_1$ in Section \ref{sec:edecaypsi1}. 
\subsection{Energy decay for  $\psi_{0}$}
We first establish polynomial decay for the $N$-energy of spherically symmetric solutions $\psi_0$. We distinguish the cases where $I_0[\psi]\neq 0$ and $I_0[\psi]=0$.
\label{sec:edecaypsi0}
\begin{proposition}[\textbf{Energy decay for $\psi_0$}]\label{prop:endec1}
Let $\psi$ be a spherically symmetric solution to \eqref{waveequation} emanating from initial data given as in Theorem \ref{thm:extuniq} on $(\mathcal{R}, g)$.
\begin{itemize}
\item[\emph{(i)}] Assume that initially we have that $\boldsymbol{I_0 [\psi ] \neq 0}$ and $E_{0 , I_0 \neq 0}^{\epsilon}[\psi]<\infty$,
with
\begin{equation*}
\begin{split}
E_{0 , I_0 \neq 0}^{\epsilon}[\psi] =&\:  \sum_{l=0}^3 \int_{\Sigma_0} J^N [T^l \psi ] \cdot n_0\, d\mu_{\Sigma_0} + \int_{\mathcal{N}_0} r^{3-\epsilon} (\partial_r \phi )^2 \,  dr + \int_{\mathcal{N}_0} r^2 (\partial_r (T \phi ) )^2 \, dr \\
&+ \int_{\mathcal{N}_0} r (\partial_r (T^2 \phi ) )^2 \,  dr.
\end{split}
\end{equation*}
Then, for all $\epsilon>0$, there exists a constant $C \doteq C(D,R,\epsilon)$ such that for all $u\geq 0$
\begin{equation}\label{est:endec1}
\int_{\Sigma_u} J^N [\psi ] \cdot n_{u} d\mu_{\Sigma_{u}} \leq C \frac{E_{0, I_0 \neq 0}^{\epsilon}[\psi]}{(1+u )^{3-\epsilon}} .
\end{equation}

\item[\emph{(ii)}] Assume that initially we have that $\boldsymbol{I_0 [\psi ] = 0}$ and $E_{0, I_0=0}^{\epsilon}[\psi]<\infty$, with
\begin{equation*}
\begin{split}
E_{0, I_0=0}^{\epsilon}[\psi] =&\: \sum_{l=0}^5 \int_{\Sigma_0} J^N [T^l\psi ] \cdot n_0\, d\mu_{\Sigma_0}+ \int_{\mathcal{N}_0} r^{5-\epsilon} (\partial_r \phi )^2 \,  dr  +  \int_{\mathcal{N}_0} r^{4-\epsilon} (\partial_r (T \phi ) )^2 \,  dr\\
&+ \int_{\mathcal{N}_0} r^{3-\epsilon} (\partial_r (T^2 \phi ) )^2 \,  dr + \int_{\mathcal{N}_0} r^2 (\partial_r (T^3 \phi ) )^2 \, dr + \int_{\mathcal{N}_0} r(\partial_r (T^4 \phi ) )^2 \, dr .
\end{split}
\end{equation*}

Then, for all $\epsilon>0$, there exists a constant $C \doteq C(D,R,\epsilon)$ such that for all $u\geq 0$
\begin{equation}\label{est:endec2}
\int_{\Sigma_u} J^N [\psi ] \cdot n_{u} d\mu_{\Sigma_{u}} \leq C \frac{E_{0, I_0=0}^{\epsilon}[\psi]}{(1+u )^{5-\epsilon}}. 
\end{equation}
\end{itemize}
\end{proposition}
\begin{proof}
(i) $\boldsymbol{I_0 [\psi ] \neq 0}$. By Proposition \ref{DafRoddecay}, we can estimate
$$ \int_{\Sigma_u} J^N [\psi ] \cdot n_u d\mu_{\Sigma_u} \leq C (1+u )^{-2}\bigg[ \sum_{l \leq 2} \int_{\Sigma_u} J^N [T^l \psi ] \cdot n_0\, d\mu_{\Sigma_0} + \int_{\mathcal{N}_0} r^2 (\partial_r \phi )^2 \, dr + \int_{\mathcal{N}_0} r (\partial_r T \phi )^2 \, dr\bigg], $$
for any $u \geq 0$, for a constant $C \doteq C(D,R)$. 

We apply the estimate of Proposition \ref{en00} with $p=3-\epsilon$ and the pigeonhole principle (the mean value theorem on dyadic intervals) to obtain
\begin{equation*}
\begin{split}
\int_{\mathcal{N}_u} r^{2-\epsilon} (\partial_r \phi )^2  \,  dr \leq&\: C {(1+u )^{-2}}\Bigg[\sum_{l \leq 2} \int_{\Sigma_u} J^N [T^l \psi ] \cdot n_0\, d\mu_{\Sigma_0} + \int_{\mathcal{N}_0} r^2 (\partial_r \phi )^2 \, dr \\
&+ \int_{\mathcal{N}_0} r (\partial_r T \phi )^2 \,  dr\Bigg]+C {(1+u )^{-1}}\int_{\mathcal{N}_0} r^{3-\epsilon} (\partial_r \phi )^2  \,  dr.
\end{split}
\end{equation*}
and subsequently, we apply Proposition \ref{en00} with $p=2-\epsilon$, the pigeonhole principle and Lemma \ref{lm:interpolation} with $p=1$ and $q=2$ to obtain
\begin{equation*}
\begin{split}
\int_{\mathcal{N}_u} r (\partial_r \phi )^2  \,  dr \leq&\: C {(1+u )^{-2+\epsilon}}\Bigg[\sum_{l \leq 2} \int_{\Sigma_u} J^N [T^l \psi ] \cdot n_0\, d\mu_{\Sigma_0} + \int_{\mathcal{N}_0} r^2 (\partial_r \phi )^2 \, dr \\
&+ \int_{\mathcal{N}_0} r (\partial_r T \phi )^2 \,  dr\Bigg]+C {(1+u )^{-2+\epsilon}}\int_{\mathcal{N}_0} r^{3-\epsilon} (\partial_r \phi )^2  \,  dr.
\end{split}
\end{equation*}

Now, we combine Proposition \ref{en00} with $p=1-\epsilon$ together with \eqref{ass:morawetz} to estimate
$$ \int_{u_1}^{u_2} \int_{\Sigma_u} J^N [\psi ] \cdot n_{u} d\mu_{\Sigma_u} du \leq C\sum_{l\leq 1}\int_{\Sigma_{u_1}} J^N [T^l\psi ] \cdot n_{u_1} d\mu_{\Sigma_u}  +C {(1+u )^{-2+\epsilon}}\int_{\mathcal{N}_0} r^{3-\epsilon} (\partial_r \phi )^2  \, dr . $$
The energy decay estimate \eqref{est:endec1} now follows after applying once more the pigeonhole principle, using Proposition \ref{DafRoddecay} for $T^l\psi$ with $l\leq 1$ together with energy boundedness \eqref{ass:ebound}.

(ii) $\boldsymbol{I_0 [\psi ] = 0}$. In the case $I_0[\psi]=0$ we can add two additional estimate in the hierarchy of estimates. We consider the cases $p=4-\epsilon$ and $p=5-\epsilon$ from Proposition \ref{en01}. We can therefore improve the decay above by applying as above (multiple times) the pigeonhole principle and then Lemma \ref{lm:interpolation} to estimate the flux integral of $r(\partial_r\phi)^2$ by $r^{1-\epsilon}(\partial_r\phi)^2$ as above, followed by \eqref{ass:morawetz} together with the already established $N$-energy decay estimate (for $\psi$ and $T\psi)$.
\end{proof}
The following lemma concerns decay estimates of certain $r^p$-weighted energy fluxes that will be necessary when deriving pointwise estimates for the radiation field. The lemma contains in particular a sub-optimal decay result that is used to increase the range of $p$ in  \eqref{en00est} from $p\in(0,4)$ to $p\in(0,5)$.
\begin{lemma}
\label{lm:auxdecaypsi0}
Let $\psi=\psi_0$ be a solution to \eqref{waveequation} emanating from initial data given as in Theorem \ref{thm:extuniq} on $(\mathcal{R}, g)$. 
\begin{itemize}
\item[\emph{(i)}]
Assume that $\boldsymbol{I_0[\psi]\neq 0}$, and
\begin{equation*}
E^{\epsilon}_{0;I_0\neq 0}[\psi]<\infty.
\end{equation*}
Then we can estimate
\begin{align}
\label{endec:r2n0v0}
\int_{\mathcal{N}_u} r^2 (\partial_r \phi )^2 \, d\omega dr \leq&\: C(1+ u )^{-1+\epsilon}\left(E_{\rm dr}[\psi]+\int_{\mathcal{N}_0}r^{3-\epsilon}(\partial_r(r\psi))^2\, dr\right),\\
\label{endec:r2n0bv0}
\int_{\mathcal{N}_u} (\partial_r \phi )^2 \, dr \leq&\: C(1+ u )^{-3+\epsilon}E^{\epsilon}_{0;I_0\neq 0}[\psi],
\end{align}
for $C=C(D,R,\epsilon)>0$.
\item[\emph{(ii)}]
Assume that $\boldsymbol{I_0[\psi]=0}$ and $\lim_{r\to \infty}T\psi(0,r)<\infty$, and moreover assume that
\begin{equation*}
\begin{split}
E^{\epsilon}_{0;\textnormal{aux}}[\psi]=&\: \sum_{l\leq 4} \int_{\Sigma_0} J^N [T^l\psi ] \cdot n_0\, d\mu_{\Sigma_0}+ \sum_{l\leq 3}\int_{\mathcal{N}_0} r^{4-l-\epsilon} (\partial_r  T^l\phi  )^2 \,  dr<\infty.
\end{split}
\end{equation*}
Then we can estimate for all $\epsilon>0$
\begin{align}\label{endec:r2n0}
\int_{\mathcal{N}_u} r^2 (\partial_r \phi )^2 \, dr \leq&\: C (1+ u )^{-2+\epsilon}\left(E_{\rm dr}[\psi]+\int_{\mathcal{N}_0}r^{4-\epsilon}(\partial_r(r\psi))^2\, dr\right) ,\\
\label{endec:r2n0b}
\int_{\mathcal{N}_u} (\partial_r \phi )^2 \,  dr \leq&\: C (1+ u )^{-4+\epsilon}E^{\epsilon}_{0;\textnormal{aux}}[\psi],
\end{align}
for $C=C(D,R,\epsilon)>0$.
\item[\emph{(iii)}] Assume that $\boldsymbol{I_0[\psi]=0}$, $\lim_{r\to \infty}T\psi(0,r)<\infty$, $\lim_{r\to \infty}T^2\psi(0,r)<\infty$ and moreover
\begin{equation*}
E^{\epsilon}_{0;I_0= 0}[\psi]<\infty.
\end{equation*}
Then we can estimate
\begin{align}\label{endec:r2n0v2}
\int_{\mathcal{N}_u} r^2 (\partial_r \phi )^2 \,  dr \leq&\: C(1+ u )^{-3+\epsilon}\left(E^{\epsilon}_{0;I_0\neq 0}[\psi]+\int_{\mathcal{N}_0}r^{5-\epsilon}(\partial_r(r\psi))^2\, dr\right),\\
\label{endec:r2n0bv2}
\int_{\mathcal{N}_u} (\partial_r \phi )^2 \,  dr \leq&\: C(1+ u )^{-5+\epsilon}E^{\epsilon}_{0;I_0= 0}[\psi],
\end{align}
for $C=C(D,R,\epsilon)>0$.
\end{itemize}
\end{lemma}

\begin{proof}
The estimates \eqref{endec:r2n0v0}, \eqref{endec:r2n0bv0}, \eqref{endec:r2n0} and \eqref{endec:r2n0bv0} appear in the proof of Proposition \ref{prop:endec1} (after using Lemma \ref{lm:interpolation} to transfer $r$-degeneracy into $u$-degeneracy). Moreover, both the estimates \eqref{endec:r2n0} and \eqref{endec:r2n0b} appear in the proof of Proposition \ref{prop:endec1} as intermediate results obtained by using the hierarchy of \eqref{en00est} only up to $p=4-\epsilon$.
\end{proof}

The energy $E^{\epsilon}_{0,\textnormal{aux}}[\psi]$ will appear in the pointwise decay estimate required to extend the range of $p$ in Proposition \ref{en01} from $p\in(0,4)$ to $p\in(0,5)$; see Lemma \ref{lm:auxpointdecay}.

\subsection{Energy decay for  $T^{k}\psi_{0}$}
\label{sec:edecayTkpsi0}
In this section, we obtain improved energy decay estimates for $T^k\psi_0$, using the $r^p$-weighted estimates from Section \ref{sec:rphierTkpsi0}.

\begin{proposition}[\textbf{Energy decay for $T^k \psi_0$}]\label{prop:endec2}
Let $\psi$ be a spherically symmetric solution to \eqref{waveequation} emanating from initial data given as in Theorem \ref{thm:extuniq} on $(\mathcal{R}, g)$.

Let $n\in \N$ and assume additionally that $D(r)=1-2Mr^{-1}+O_{n+2}(r^{-1-\beta})$ for some $\beta>0$.
\begin{itemize}
\item[\emph{(i)}] Assume that initially we have that $\boldsymbol{I_0 [\psi ] \neq 0}$ and $E_{0, I_0\neq 0; k}^{\epsilon}[\psi]<\infty$, with
\begin{equation*}
\begin{split}
E_{0, I_0\neq 0; k}^{\epsilon}[\psi] \doteq &\: \sum_{l\leq 3+3k}\int_{\Sigma_{0}}J^N[T^l\psi]\cdot n_{0}\,d\mu_{\Sigma_0}\\
&+\sum_{l\leq 2k}\int_{\mathcal{N}_{0}} r^{3-\epsilon}(\partial_rT^l\phi)^2\,dr+r^{2}(\partial_rT^{l+1}\phi)^2+r(\partial_rT^{2+l}\phi)^2\,dr\\
&+\sum_{\substack{m\leq k\\ l\leq 2k-2m+\min\{k,1\}}} \int_{\mathcal{N}_{0}}r^{2+2m-\epsilon}(\partial_r^{1+m}T^{l}\phi)^2\, dr\\
&+\int_{\mathcal{N}_{0}}r^{3+2k-\epsilon}(\partial_r^{1+k}\phi)^2\, dr.
\end{split}
\end{equation*}
Then, for all $\epsilon>0$ and all $k\leq n$, there exists a constant $C \doteq C(D,R,\epsilon,n)$ such that for all $u\geq 0$
\begin{equation}\label{est:tendec1}
\int_{\Sigma_u} J^N [T^k \psi ] \cdot n_u d\mu_{\Sigma_u} \leq C \frac{E_{0, I_0\neq 0; k}^{\epsilon}[\psi]}{(1+u )^{2k+3-\epsilon}} . 
\end{equation}

\item[\emph{(ii)}]Assume that initially we have that $\boldsymbol{I_0 [\psi ] = 0}$ and $E_{0, I_0= 0; k}^{\epsilon}[\psi]<\infty$, with
\begin{equation*}
\begin{split}
E_{0, I_0= 0; k}^{\epsilon}[\psi] \doteq &\: \sum_{l\leq 5+3k}\int_{\Sigma_{0}}J^N[T^l\psi]\cdot n_{0}\,d\mu_{\Sigma_0}\\
&+\sum_{l\leq 2k}\int_{\mathcal{N}_{0}} r^{5-\epsilon}(\partial_rT^l\phi)^2+r^{4-\epsilon}(\partial_rT^{1+l}\phi)^2+r^{3-\epsilon}(\partial_rT^{2+l}\phi)^2\\
&+r^{2}(\partial_rT^{3+l}\phi)^2+r(\partial_rT^{4+l}\phi)^2\, dr\\
&+\sum_{\substack{m\leq k\\ l\leq 2k-2m+\min\{k,1\}}} \int_{\mathcal{N}_{0}}r^{4+2m-\epsilon}(\partial_r^{1+m}T^{l}\phi)^2\, dr\\
&+\int_{\mathcal{N}_{0}}r^{5+2k-\epsilon}(\partial_r^{1+k}\phi)^2\, dr.
\end{split}
\end{equation*}
 Then, for all $\epsilon>0$, there exists a constant $C \doteq C(D,R,\epsilon)$ such that for all $u\geq 0$
\begin{equation}\label{est:tendec2}
\int_{\Sigma_u} J^N [T^k \psi ] \cdot n_u d\mu_{\Sigma_u} \leq C \frac{E_{0, I_0= 0; k}^{\epsilon}[\psi]}{(1+u )^{2k+5-\epsilon}} . 
\end{equation}
\end{itemize}
\end{proposition}
\begin{proof}
We will only prove (ii). The estimate in (i) can be proven analogously by restricting everywhere the range of $p$.

We will prove \eqref{est:tendec2} by induction. The case $k=0$ follows from \eqref{est:endec2}. Now suppose \eqref{est:tendec2} holds for all $k\leq n$. Then \eqref{est:tendec2} also holds for $\psi$ replaced by $T\psi$. 

By \eqref{ent0est} with $p=4+2(n+1)-\epsilon$ we have an additional estimate for $T^{n+1}\psi$ in our hierarchy of estimates compared to the hierarchy of estimates for $T^n\psi$: indeed, for all $0\leq u_1<u_2$,
\begin{equation*}
\begin{split}
\int_{\mathcal{A}^{u_2}_{u_1}}r^{5+2n-\epsilon}(\partial_r^{1+n}T\phi)^2\, dr du\leq&\: C\sum_{m\leq n+1}\int_{\mathcal{N}_{u_1}}r^{6+2n-\epsilon-2m}(\partial_r^{2+n-m}\phi)^2\, dr\\
&+ \sum_{l\leq n+1}\int_{\Sigma_{u_1}}J^N[T^l\psi]\cdot n_{u_1}\,d\mu_{\Sigma_{u_1}}.
\end{split}
\end{equation*}
As a consequence, analogous to the arguments in Proposition \ref{prop:endec1}, we can improve the decay rate in \eqref{est:tendec2} with $k=n$ applied to $T\psi$ by one power to obtain

\begin{equation*}
\begin{split}
\int_{\Sigma_{u}} J^N [T^{n+1} \psi ] \cdot n_u\, d\mu_{\Sigma_u} \leq&\: C(1+u)^{2n+6-\epsilon}\Bigg[\sum_{l\leq 7+3n}\int_{\Sigma_{0}}J^N[T^l\psi]\cdot n_{\Sigma_0}\,d\mu_{\Sigma_0}\\
&+\sum_{l\leq 2n+1}\int_{\mathcal{N}_{0}} r^{5-\epsilon}(\partial_rT^{1+l}\phi)^2+r^{4-\epsilon}(\partial_rT^{2+l}\phi)^2+r^{3-\epsilon}(\partial_rT^{3+l}\phi)^2\\
&+r^{2}(\partial_rT^{4+l}\phi)^2+r(\partial_rT^{5+l}\phi)^2\,dr\\
&+\sum_{\substack{m\leq k\\ l\leq 2k-2m+2}} \int_{\mathcal{N}_{0}}r^{4+2m-\epsilon}(\partial_r^{1+m}T^{l}\phi)^2\, dr\\
&+\sum_{l\leq 1}\int_{\mathcal{N}_{0}}r^{5+2k-\epsilon}(\partial_r^{1+k}T^l\phi)^2\, dr\Bigg]\\
&+C(1+u)^{2n+6-\epsilon}\sum_{m\leq n+1}\int_{\mathcal{N}_{0}}r^{6+2n-\epsilon-2m}(\partial_r^{2+n-m}\phi)^2\, dr.
\end{split}
\end{equation*}
We can add another estimate to our hierarchy by applying \eqref{enr0estp5} with $p=5+2(n+1)-\epsilon$:
\begin{equation*}
\begin{split}
\int_{\mathcal{A}^{u_2}_{u_1}}r^{6+2n-2m-\epsilon}(\partial_r^{2+n-m}\phi)^2\, dr du\leq&\: C\sum_{m\leq n+1}\int_{\mathcal{N}_{u_1}}r^{7+2n-\epsilon-2m}(\partial_r^{2+n-m}\phi)^2\, dr\\
&+ \sum_{l\leq n+1}\int_{\Sigma_{u_1}}J^N[T^l\psi]\cdot n_{\tau_1}\,d\mu_{\Sigma_{u_1}}+CE_{0;\rm aux}[\psi].
\end{split}
\end{equation*}
We can now improve the energy decay rate by one more power:
\begin{equation*}
\begin{split}
\int_{\Sigma_{u}} J^N [T^{n+1} \psi ] \cdot n_u d\mu_{\Sigma_u} \leq&\: C(1+u)^{2n+7-\epsilon}\Bigg[\sum_{l\leq 8+3n}\int_{\Sigma_{0}}J^N[T^l\psi]\cdot n_{0}\\
&+\sum_{l\leq 2n+2}\int_{\mathcal{N}_{0}} r^{5-\epsilon}(\partial_rT^{1+l}\phi)^2+r^{4-\epsilon}(\partial_rT^{2+l}\phi)^2+r^{3-\epsilon}(\partial_rT^{3+l}\phi)^2\\
&+r^{2}(\partial_rT^{4+l}\phi)^2+r(\partial_rT^{5+l}\phi)^2\, dr\\
&+\sum_{\substack{m\leq k\\ l\leq 2k-2m+3}} \int_{\mathcal{N}_{0}}r^{4+2m-\epsilon}(\partial_r^{1+m}T^{l}\phi)^2\, dr\\
&+\sum_{l\leq 2} \int_{\mathcal{N}_{0}}r^{5+2k-\epsilon}(\partial_r^{1+k}T^l\phi)^2\, dr\Bigg]\\
&+C(1+u)^{2n+7-\epsilon}\sum_{m\leq n+1,\,l\leq 1}\int_{\mathcal{N}_{0}}r^{6+2n-\epsilon-2m}(\partial_r^{2+n-m}T^l\phi)^2\,dr\\
&+C(1+u)^{2n+7-\epsilon}\sum_{m\leq n+1}\int_{\mathcal{N}_{0}}r^{7+2n-\epsilon-2m}(\partial_r^{2+n-m}\phi)^2\, dr.
\end{split}
\end{equation*}
By rearranging the above terms, we arrive at \eqref{est:tendec2} with $k=n+1$ which completes the proof.
\end{proof}
We easily obtain the following decay estimates for $r$-weighted integrals of $T^k\psi$, in analogy with Lemma \ref{lm:auxdecaypsi0}.
\begin{lemma}
\label{lm:auxdecaypsi0Tk}
Let $\psi$ be a spherically symmetric solution to \eqref{waveequation} emanating from initial data given as in Theorem \ref{thm:extuniq} on $(\mathcal{R}, g)$. Let $n\in \N$ and assume additionally that $D(r)=1-2Mr^{-1}+O_{n+2}(r^{-1-\beta})$ for some $\beta>0$ and $\boldsymbol{I_0[\psi]=0}$.

Then we can estimate for all $\epsilon>0$ and $k\leq n$
\begin{align}\label{endec:r2n0ineq0tk}
\int_{\mathcal{N}_u} r^2 (\partial_rT^k \phi )^2 \, dr \leq&\: C \frac{E^{\epsilon}_{0,I_0\neq 0;k}[\psi]}{(1+ u )^{1+2k-\epsilon}} ,\\
\int_{\mathcal{N}_u} (\partial_r T^k\phi )^2 \, dr \leq&\: C \frac{E^{\epsilon}_{0,I_0\neq 0;k}[\psi]}{(1+ u )^{3+2k-\epsilon}} ,
\label{endec:r2n0btk}
\end{align}
for $C=C(D,R,\epsilon,n)>0$ if the energy norms on the right-hand side are finite, and moreover, if we assume $I_0[\psi]=0$, then
\begin{align}\label{endec:r2n0tk}
\int_{\mathcal{N}_u} r^2 (\partial_rT^k \phi )^2 \, dr \leq&\: C \frac{E^{\epsilon}_{0,I_0=0;k}[\psi]}{(1+ u )^{3+2k-\epsilon}} ,\\
\int_{\mathcal{N}_u} (\partial_r T^k\phi )^2 \, dr \leq&\: C \frac{E^{\epsilon}_{0,I_0=0;k}[\psi]}{(1+ u )^{5+2k-\epsilon}},
\label{endec:r2n0btk}
\end{align}
if the energy norms on the right-hand side are finite.
\end{lemma}

\subsection{Energy decay for $\psi_{1}$}
In this section we establish polynomial decay for the $N$-energy of $\psi_1$.
\label{sec:edecaypsi1}
\begin{proposition}[\textbf{Energy decay for $\psi_{1}$}]\label{decl2}
Let $\psi$ be a solution to \eqref{waveequation} with $\int_{\s^2}\psi\,d\omega=0$ emanating from initial data given as in Theorem \ref{thm:extuniq} on $(\mathcal{R}, g)$. 

Assume that $E^{\epsilon}_{1}[\psi]<\infty$, with 
\begin{equation*}
\begin{split}
E^{\epsilon}_{1}[\psi]\doteq &\: \sum_{l\leq 5}\int_{\Sigma_{0}}J^N[T^l\psi]\cdot n_{0}\,d\mu_{\Sigma_0}+\sum_{l\leq 3}\int_{\mathcal{N}_0}r^{2}(\partial_rT^l\phi)^2+r^{}(\partial_rT^{l+1}\phi)^2\,d\omega dr\\
&+ \int_{\mathcal{N}_0} r^{2-\epsilon}(\partial_r\Phi)^2 \, d\omega dr+\int_{\mathcal{N}_{0}}r^{2-\epsilon}(\partial_rT{\Phi})^2+r^{1-\epsilon}(\partial_rT^2{\Phi})^2\,d\omega dr+\int_{\mathcal{N}_{0}}r^{1-\epsilon}(\partial_r{\Phi}_{(2)})^2\,d\omega dr
\end{split}
\end{equation*}
and moreover
\begin{align*}
\lim_{r \to \infty }\sum_{|l|\leq 4}\int_{\s^2}(\Omega^l\phi)^2\,d\omega\big|_{u'=0}<&\:\infty,\\
\lim_{r \to \infty }\sum_{|l|\leq 2}\int_{\s^2}(\Omega^l\Phi)^2\,d\omega\big|_{u'=0}<&\:\infty,\\
\lim_{r \to \infty }\int_{\s^2}r^{-1}\left(\Phi_{(2)}\right)^2\,d\omega\big|_{u'=0}<&\:\infty.
\end{align*}

Then, for all $\epsilon> 0$ there exists a constant $C=C(D,R,\epsilon)>0$ such that for all $u\geq 0$
\begin{equation}
\label{eq:edecaylgeq1}
\begin{split}
\int_{\Sigma_{u}}J^N[\psi]\cdot n_{u}\,d\mu_{\Sigma_u}\leq&\: CE^{\epsilon}_{1}[\psi](1+u)^{-5+\epsilon}.
\end{split}
\end{equation}
\end{proposition}
\begin{proof}
Note first of all that the assumption
\begin{equation*}
\sum_{|k|\leq 4}\int_{\Sigma}J^T[\Omega^k\psi]\cdot n_{\Sigma}\,d\mu_{\Sigma}<\infty
\end{equation*}
in Proposition \ref{prop:l2commr2} follows from 
\begin{equation*}
\sum_{|k|\leq 4}\int_{\Sigma}J^T[T^k\psi]\cdot n_{\Sigma}\,d\mu_{\Sigma}<\infty,
\end{equation*}
after applying standard elliptic estimates, which in turn follows from $E^{\epsilon}_{1}[\psi]<\infty$. Moreover, the assumption
\begin{equation*}
\lim_{r\to \infty}r^{-1}\int_{\s^2}\Phi_{(2)}^2\,d\omega<\infty
\end{equation*}
also follows from $E^{\epsilon}_{1}[\psi]<\infty$, after applying the fundamental theorem of calculus in the $r$ variable.

By Proposition \ref{DafRoddecay} we have that
\begin{equation*}
\int_{\Sigma_{u}}J^N[\psi]\cdot n_{u}\,d\mu_{\Sigma_u}\leq C(1+u)^{-2}\left[\sum_{|l|\leq 2}\int_{\Sigma_{0}}J^N[T^l\psi]\cdot n_{0}\,d\mu_{\Sigma_0}+\sum_{|l|\leq 1}\int_{\mathcal{N}_0}r^{2-l}(\partial_rT^l\phi)^2\,d\omega dr\right].
\end{equation*}
By applying \eqref{rphierpsilgeq1} with $p=1-\epsilon$ we moreover have that for all $0\leq u_1<u_2$
\begin{equation*}
\begin{split}
\int_{\mathcal{A}^{u_2}_{u_1}}r^{-\epsilon}(\partial_r(\chi{\Phi}))^2\,d\omega dr du\leq&\: C\int_{\mathcal{N}_{u_1}} r^{1-\epsilon}(\partial_r{\Phi})^2\,d\omega dr+C\sum_{k\leq 1}\int_{\Sigma_{u_1}}J^T[T^{k}\psi]\cdot n_{u_1}\,d\mu_{\Sigma_{u_1}},
\end{split}
\end{equation*}
for $R=R(D)>0$ suitably large. 

We can apply the Hardy inequality \eqref{eq:Hardy1} to estimate
\begin{equation*}
\int_{\mathcal{A}^{u_2}_{u_1}}\chi^2r^{2-\epsilon}(\partial_r{\phi})^2\,d\omega dr du\leq C\int_{\mathcal{A}^{u_2}_{u_1}}r^{-\epsilon}(\partial_r(\chi{\Phi}))^2\,d\omega dr du,
\end{equation*}
where use use that $\lim_{r \to \infty }r^{-1-\epsilon}{\Phi}^2=0$.

Together with \eqref{rphierpsilgeq1} for $p=1-\epsilon$ and $p=2-\epsilon$ we therefore obtain a hierarchy consisting of three estimates. In particular, we apply the pigeonhole principle to infer that there exists a dyadic sequence $\{u_j\}$ such that
\begin{align*}
\int_{\mathcal{N}_{u_j}}r^{2-\epsilon}(\partial_r\phi)^2\,d\omega dr \leq &\:C(1+u_j)^{-1}\left[\int_{\mathcal{N}_{0}} r^{1-\epsilon}(\partial_r{\Phi})^2\,d\omega dr+\sum_{|l|\leq 1}\int_{\Sigma_{0}}J^N[T^l\psi]\cdot n_{{0}}\,d\mu_{\Sigma_0}\right],\\
\int_{\mathcal{N}_{u_j}}r^{1-\epsilon}(\partial_r\phi)^2\,d\omega dr\leq&\: C(1+u_j)^{-2}\left[\int_{\mathcal{N}_{0}} r^{1-\epsilon}(\partial_r{\Phi})^2\,d\omega dr+\sum_{l\leq 1}\int_{\Sigma_{0}}J^N[T^l\psi]\cdot n_{0}\,d\mu_{\Sigma_0}\right]\\
&+C(1+u_j)^{-1}\int_{\Sigma_{u_{j-1}}}J^N[\psi]\cdot n_{u_{j-1}}\,d\mu_{\Sigma_{u_{j-1}}}.
\end{align*}
By Lemma \ref{lm:interpolation} we can interchange the $r^{\epsilon}$ degeneracy on the left-hand side of the second inequality for an $\epsilon$-loss in the decay rate on the right-hand side to obtain
\begin{equation*}
\begin{split}
\int_{\mathcal{N}_{u_j}}r(\partial_r\phi)^2\,d\omega dr\leq &\:C(1+u_j)^{-2+\epsilon}\left[\int_{\mathcal{N}_{0}} r^{1-\epsilon}(\partial_r{\Phi})^2\,d\omega dr+ \sum_{|l|\leq 1}\int_{\Sigma_{0}}J^N[T^l\psi]\cdot n_{0}\,d\mu_{\Sigma_0}\right]\\
&+C(1+u_j)^{-3+\epsilon}\left[\sum_{|l|\leq 2}\int_{\Sigma_{0}}J^N[T^l\psi]\cdot n_{0}\,d\mu_{\Sigma_0}+\sum_{|l|\leq 1}\int_{\mathcal{N}_0}r^{2-l}(\partial_rT^l\phi)^2\,d\omega dr\right].
\end{split}
\end{equation*}
\eqref{Ndecay2}

Together with the Morawetz estimates \eqref{ass:morawetz} and \eqref{ass:morawetzlocal}, \eqref{rphierpsilgeq1} with $p=1$ and \eqref{ass:ebound} we therefore obtain for all $u\geq 0$ the estimate:
\begin{equation*}
\begin{split}
\int_{\Sigma_{u}}J^N[\psi]\cdot n_{u}\,d\mu_{u}\leq&\: C(1+u)^{-3+\epsilon}\Bigg[\sum_{|l|\leq 3}\int_{\Sigma_{0}}J^N[T^l\psi]\cdot n_{0}\,d\mu_{\Sigma_0}+\sum_{|l|\leq 1}\int_{\mathcal{N}_0}r^{2-l}(\partial_rT^{l}\phi)^2+r(\partial_rT^{l+1}\phi)^2\,d\omega dr\\
&+\int_{\mathcal{N}_{0}} r^{1-\epsilon}(\partial_r{\Phi})^2\,d\omega dr\Bigg].
\end{split}
\end{equation*}

We can further apply \eqref{rphierpsilgeq1} with $p=2-\epsilon$ to obtain a fourth estimate in our hierarchy and repeat the arguments above to obtain
\begin{equation*}
\begin{split}
\int_{\Sigma_{u}}J^N[\psi]\cdot n_{\tau}\,d\mu_{u}\leq&\: C(1+u)^{-4+\epsilon}\Bigg[\sum_{|l|\leq 4}\int_{\Sigma_{0}}J^N[T^l\psi]\cdot n_{\Sigma_0}\,d\mu_0+\sum_{|l|\leq 2}\int_{\mathcal{N}_0}r^{2}(\partial_rT^l\phi)^2+r^{}(\partial_rT^{l+1}\phi)^2\,d\omega dr\\
&+\int_{\mathcal{N}_{0}} r^{2-\epsilon}(\partial_r{\Phi})^2+r^{1-\epsilon}(\partial_rT{\Phi})^2\,d\omega dr\Bigg].
\end{split}
\end{equation*}

We apply \eqref{eq:Hardy1} once more to estimate
\begin{equation*}
\int_{\mathcal{A}^{\tau_2}_{\tau_1}}\chi^2r^{2-\epsilon}(\partial_r{\Phi})^2\,d\omega dr du\leq C\int_{\mathcal{A}^{\tau_2}_{\tau_1}}r^{-\epsilon}(\partial_r(\chi{\Phi}_{(2)}))^2\,d\omega dr du,
\end{equation*}
where we used that $\lim_{r\to \infty}r^{-1-\epsilon}\Phi_{(2)}^2=0$.

We can now apply \eqref{rphierPhi2lgeq2a} with $p=1-\epsilon$ to obtain an additional estimate in our hierarchy and arrive at \eqref{eq:edecaylgeq1}.
\end{proof}

\begin{remark}
Instead of using the $r^p$-weighted estimate from Proposition \ref{prop:l2commr2} in the proof of Proposition \ref{decl2}, we can instead split $\psi_1=\psi_{\ell=1}+\psi_{\ell\geq 2}$ and use Proposition \ref{prop:rpphil=1} and \ref{prop:l2commr2v0}. In this way we arrive in particular at a version of Proposition \ref{decl2}, where we can take $\epsilon=0$, but where we require the \emph{second} Newman--Penrose constant $I_1[\psi]$ to vanish; see also Remark \ref{rmk:I1}.
\end{remark}

The following lemma concerns decay estimates of certain $r^p$-weighted energy fluxes that will be necessary when deriving pointwise estimates for the radiation field. It contains estimates that appear in the proof of Proposition \ref{decl2}, after additional application of Lemma \ref{lm:interpolation}. We will omit the proof here.
\begin{lemma}
\label{lm:auxdecaypsi1}
Let $\psi$ be a solution to \eqref{waveequation} with $\int_{\s^2}\psi\,d\omega=0$ emanating from initial data given as in Theorem \ref{thm:extuniq} on $(\mathcal{R}, g)$.

Then, for all $\epsilon> 0$ there exists a suitably large $R>0$ and a constant $C=C(D,R,\epsilon)>0$ such that for all $u\geq 0$
\begin{align}\label{endec:r2n0ineq1}
\int_{\mathcal{N}_u} r^2 (\partial_rT^k \phi )^2 \, dr \leq&\: C \frac{E^{\epsilon}_{1}[\psi]}{(1+ u )^{3-\epsilon}} ,\\
\int_{\mathcal{N}_u} (\partial_r T^k\phi )^2 \, dr \leq&\: C \frac{E^{\epsilon}_{1}[\psi]}{(1+ u )^{5-\epsilon}},
\label{endec:r2n1b}
\end{align}
if the energy norms on the right-hand side are assumed to be finite.
\end{lemma}

\subsection{Energy decay for $T^{k}\psi_{1}$}
\label{sec:energydecaytkpsi1}
In this section, we obtain improved energy decay estimates for $T^k\psi_1$, using the $r^p$-weighted estimates from Section \ref{sec:rphierTkpsi1}.

\label{sec:edecayTkpsi1}
\begin{proposition}[\textbf{Energy decay for $T^k\psi_{1}$}]
\label{prop:edecayTkpsi1}
Let $\psi$ be a solution to \eqref{waveequation} with $\int_{\s^2}\psi\,d\omega=0$ emanating from initial data given as in Theorem \ref{thm:extuniq} on $(\mathcal{R}, g)$.

Let $n\in \N$ and assume additionally that $D(r)=1-2Mr^{-1}+O_{n+2}(r^{-1-\beta})$ for some $\beta>0$.

Assume further that 
\begin{align*}
\lim_{r \to \infty }\sum_{|l|\leq 4+2n}\int_{\s^2}(\Omega^l\phi)^2\,d\omega\big|_{u'=0}<&\:\infty,\\
\lim_{r \to \infty }\sum_{|l|\leq 2+2n}\int_{\s^2}(\Omega^l\Phi)^2\,d\omega\big|_{u'=0}<&\:\infty,\\
\lim_{r \to \infty }\sum_{|l|\leq 2n}\int_{\s^2}r^{-1}\left(\Omega^l\Phi_{(2)}\right)^2\,d\omega\big|_{u'=0}<&\:\infty,
\end{align*}
and
\begin{equation*}
\lim_{r \to \infty }\sum_{|l|\leq 2n-2s}\int_{\s^2}r^{2s+1}\left(\partial_r^s\Omega^l\Phi_{(2)}\right)^2\,d\omega\big|_{u'=0}<\infty,
\end{equation*}
for each $1\leq s\leq n$.

Assume moreover that $E_{1;k}^{\epsilon}[\psi]<\infty$, with
\begin{equation*}
\begin{split}
E_{1;k}^{\epsilon}[\psi]\doteq &\:\sum_{\substack{|\alpha|\leq k\\ l+|\alpha|\leq 5+3k}}\int_{\Sigma_{0}}J^N[T^l\Omega^{\alpha}\psi]\cdot n_{0}\;d\mu_{\Sigma_0}\\
&+\sum_{ l\leq 3+2k}\int_{\mathcal{N}_0}r^{2}(\partial_rT^l\phi)^2+r^{}(\partial_rT^{1+l}\phi)^2\,d\omega dr\\
&+\sum_{ l\leq 2k+1}\int_{\mathcal{N}_{0}} r^{2-\epsilon}(\partial_rT^{l}{\Phi})^2+r^{1-\epsilon}(\partial_rT^{l+1}{\Phi})^2\,d\omega dr\\
&+\sum_{\substack{|\alpha|\leq k\\l+|\alpha|\leq 2k}}\int_{\mathcal{N}_{0}} r^{1-\epsilon}(\partial_rT^{l}\Omega^{\alpha}{\Phi}_{(2)})^2\,d\omega dr\\
&+\sum_{\substack{|\alpha|\leq \max\{0,k-1\}\\ m\leq \max\{k-1,0\}\\ l+|\alpha|\leq k-2m+\min\{k,1\}}}\int_{\mathcal{N}_{0}}r^{1+2m-\epsilon}(\partial_r^{1+m}\Omega^{\alpha}T^{l}{\Phi}_{(2)})^2\,d\omega dr\\
&+\sum_{\substack{|\alpha|\leq \max\{0,k-1\}, m\leq k\\ l+|\alpha|\leq 2k-2m+1}}\int_{\mathcal{N}_{0}} r^{2m-\epsilon}(\partial_r^{1+m}\Omega^{\alpha}T^{l}{\Phi}_{(2)})^2\,d\omega dr\\
&+\int_{\mathcal{N}_{0}} r^{1+2k-\epsilon}(\partial_r^{1+k}{\Phi}_{(2)})^2\,d\omega dr.
\end{split}
\end{equation*}

Then, for all $\epsilon> 0$ there exists a constant $C=C(D,R,\epsilon,n)>0$ such that for all $k\leq n$ and $u\geq 0$
\begin{equation}
\label{eq:edecayTklgeq2}
\begin{split}
\int_{\Sigma_{\tau}}J^N[T^k\psi]\cdot n_{u}\:d\mu_{\Sigma_{u}}\leq&\: CE_{1;k}^{\epsilon}[\psi](1+\tau)^{-5-2k+\epsilon}.
\end{split}
\end{equation}
\end{proposition}
\begin{proof}
We will prove \eqref{eq:edecayTklgeq2} by induction. The case $k=0$ follows from \eqref{eq:edecaylgeq1}. Now suppose \eqref{eq:edecayTklgeq2} holds for $k\leq n$. Then \eqref{eq:edecayTklgeq2} also holds for $\psi$ replaced by $T\psi$. 

By Proposition \ref{eq:hierTkpsi1} with $p=2(n+1)-\epsilon$ we have an additional estimate for $T^{n+1}\psi$ in our hierarchy of estimates compared to the hierarchy of estimates for $T^n\psi$: indeed, for all $0\leq u_1<u_2$,
\begin{equation}
\label{eq:pPhi2commT}
\begin{split}
&\int_{\mathcal{A}_{u_1}^{u_2}}r^{1+2n-\epsilon}(\partial_r^{n+1}T\Phi_{(2)})^2 \,d\omega drdu \\
\leq&\: C\int_{\mathcal{N}_{u_1}} r^{2+2n-\epsilon}(\partial_r^{n+2}\Phi_{(2)})^2\,d\omega dr\\
&+ C\sum_{\substack{|\alpha|\leq 1, j+|\alpha|\leq n}}\int_{\mathcal{N}_{u_1}} r^{2n-2j-\epsilon}(\partial_r^{n+1-j}\Omega^{\alpha}\Phi_{(2)})^2\,d\omega dr\\
&+C\sum_{\substack{|\alpha|\leq 1, j+|\alpha|\leq n+3}}J^T[T^{j}\Omega^{\alpha}\psi]\cdot n_{\Sigma_{u_1}}\,d\mu_{\Sigma_{u_1}}.
\end{split}
\end{equation}

We also apply \eqref{eq:pPhi2commr} with $p=1+2m-\epsilon$ to estimate
\begin{equation*}
\begin{split}
\sum_{\substack{|\alpha|\leq \max\{0,n-1\}, m\leq n\\ l+|\alpha|\leq 2n-2m+1}}&\int_{\mathcal{A}_{u_1}^{u_2}} r^{2m-\epsilon}(\partial_r^{1+m}\Omega^{\alpha}T^{l+1}{\Phi}_{(2)})^2\,d\omega dr\\
\leq&\: C\sum_{\substack{|\alpha|\leq \max\{0,n-1\}, m\leq n\\ l+|\alpha|\leq 2n-2m+1}}\int_{\mathcal{N}_{u_1}} r^{1+2m-\epsilon}(\partial_r^{1+m}\Omega^{\alpha}T^{l+1}{\Phi}_{(2)})^2\,d\omega dr\\
&+C\sum_{\substack{|\alpha|\leq \max\{0,n-1\}, j+|\alpha|\leq n+2}}J^T[T^{j}\Omega^{\alpha}\psi]\cdot n_{\Sigma_{u_1}}\,d\mu_{\Sigma_{u_1}}.
\end{split}
\end{equation*}

As a consequence, analogous to the arguments in Proposition \ref{decl2}, we can improve the decay rate in \eqref{eq:edecayTklgeq2} with $k=n$ applied to $T\psi$ by one power to obtain
\begin{equation*}
\begin{split}
\int_{\Sigma_{u}}&J^N[T^{n+1}\psi]\cdot n_{\Sigma_u}\,d\mu_{u}\leq C(1+u)^{-6-2n+\epsilon}\Bigg[\sum_{\substack{|\alpha|\leq n\\ l+|\alpha|\leq 7+3n}}\int_{\Sigma_{0}}J^N[T^{l}\Omega^{\alpha}\psi]\cdot n_{0}\,d\mu_{\Sigma_0}\\
&+\sum_{l\leq 3+2n+1}\int_{\mathcal{N}_0}r^{2}(\partial_rT^{l}\phi)^2+r^{}(\partial_rT^{1+l}\phi)^2\,d\omega dr\\
&+\sum_{ l\leq 2n+2}\int_{\mathcal{N}_{0}} r^{2-\epsilon}(\partial_rT^{l}{\Phi})^2+r^{1-\epsilon}(\partial_rT^{1+l}{\Phi})^2\,d\omega dr\\
&+\sum_{\substack{|\alpha|\leq n\\l+|\alpha|\leq 2n+1}}\int_{\mathcal{N}_{0}} r^{1-\epsilon}(\partial_rT^{l}\Omega^{\alpha}{\Phi}_{(2)})^2\,d\omega dr\\
&+\sum_{\substack{|\alpha|\leq \max\{n-1,0\}\\ m\leq \max\{n-1,0\}\\ l+|\alpha|\leq 2n-2m+1}}\int_{\mathcal{N}_{0}}r^{1+2m-\epsilon}(\partial_r^{1+m}\Omega^{\alpha}T^{l}{\Phi}_{(2)})^2\,d\omega dr\\
&+\sum_{\substack{|\alpha|\leq n, m\leq n\\ l+|\alpha|\leq 2n-2m+2}}\int_{\mathcal{N}_{0}} r^{2m-\epsilon}(\partial_r^{1+m}\Omega^{\alpha}T^{l}{\Phi}_{(2)})^2\,d\omega dr\\
&+\sum_{l\leq 1}\int_{\mathcal{N}_{0}} r^{1+2n-\epsilon}(\partial_r^{1+n}T^l{\Phi}_{(2)})^2\,d\omega dr\Bigg]\\
&+C(1+u)^{-6-2n+\epsilon}\sum_{\substack{|\alpha|\leq \max\{n-1,0\}, m\leq n\\ l+|\alpha|\leq 2n-2m+2}}\int_{\mathcal{N}_{0}} r^{1+2m-\epsilon}(\partial_r^{1+m}\Omega^{\alpha}T^{l}{\Phi}_{(2)})^2\,d\omega dr\\
&+C(1+u)^{-6-2n+\epsilon}\int_{\mathcal{N}_{0}} r^{2+2n-\epsilon}(\partial_r^{n+2}\Phi_{(2)})^2\,d\omega dr.
\end{split}
\end{equation*}
We apply Proposition \ref{eq:hierTkpsi1} and \eqref{eq:pPhi2commr} once more to obtain another estimate in our hierarchy. Therefore, we can gain one more power in $(1+u)^{-1}$:
\begin{equation*}
\begin{split}
\int_{\Sigma_{u}}&J^N[T^{n+1}\psi]\cdot n_{u}\,d\mu_{\Sigma_u}\leq C(1+u)^{-7-2n+\epsilon}\Bigg[\sum_{ l\leq 8+3n}\int_{\Sigma_{0}}J^N[T^{l}\psi]\cdot n_{0}\,d\mu_{\Sigma_0}\\
&+\sum_{l\leq 3+2n+2}\int_{\mathcal{N}_0}r^{2}(\partial_rT^{l}\phi)^2+r^{}(\partial_rT^{1+l}\phi)^2\,d\omega dr\\
&+\sum_{l\leq 2n+3}\int_{\mathcal{N}_{0}} r^{2-\epsilon}(\partial_rT^{l}{\Phi})^2+r^{1-\epsilon}(\partial_rT^{1+l}{\Phi})^2\,d\omega dr\\
&+\sum_{\substack{|\alpha|\leq n\\l+|\alpha|\leq 2n+2}}\int_{\mathcal{N}_{0}} r^{1-\epsilon}(\partial_rT^{l}\Omega^{\alpha}{\Phi}_{(2)})^2\,d\omega dr\\
&+\sum_{\substack{|\alpha|\leq n\\ m\leq n\\ l+|\alpha|\leq 2n-2m+2}}\int_{\mathcal{N}_{0}}r^{1+2m-\epsilon}(\partial_r^{1+m}\Omega^{\alpha}T^{l}{\Phi}_{(2)})^2\,d\omega dr\\
&+\sum_{\substack{|\alpha|\leq n, m\leq n+1\\ l+|\alpha|\leq 2n-2m+3}}\int_{\mathcal{N}_{0}} r^{2m-\epsilon}(\partial_r^{1+m}\Omega^{\alpha}T^{l+1}{\Phi}_{(2)})^2\,d\omega dr\\
&+\sum_{l\leq 2}\int_{\mathcal{N}_{0}} r^{2+2n-\epsilon}(\partial_r^{n+2}T^l\Phi_{(2)})^2\,d\omega dr\Bigg]\\
&+C(1+u)^{-7-2n+\epsilon}\sum_{j\leq n+2}\int_{\mathcal{N}_{0}} r^{3+2n-\epsilon-2j}(\partial_r^{n+2-j}\Phi_{(2)})^2\,d\omega dr.
\end{split}
\end{equation*}
The right-hand side above is equal to the right-hand side of \eqref{eq:edecayTklgeq2} with $k=n+1$.
\end{proof}

\begin{lemma}
\label{lm:auxdecaypsi1Tk}
Let $\psi$ be a solution to \eqref{waveequation} with $\int_{\s^2}\psi\,d\omega=0$ emanating from initial data given as in Theorem \ref{thm:extuniq} on $(\mathcal{R}, g)$. Let $n\in \N$ and assume additionally that $D(r)=1-2Mr^{-1}+O_{n+2}(r^{-1-\beta})$ for some $\beta>0$.

Then, for all $\epsilon> 0$ there exists a suitably large $R>0$ and a constant $C=C(D,R,\epsilon,n)>0$ such that for all $k\leq n$ and $u\geq 0$
\begin{align}\label{endec:r2n0ineq1tk}
\int_{\mathcal{N}_u} r^2 (\partial_rT^k \phi )^2 \, dr \leq&\: C \frac{E^{\epsilon}_{1;k}[\psi]}{(1+ u )^{3+2k-\epsilon}} ,\\
\int_{\mathcal{N}_u} (\partial_r T^k\phi )^2 \, dr \leq&\: C \frac{E^{\epsilon}_{1;k}[\psi]}{(1+ u )^{5+2k-\epsilon}}.
\label{endec:r2n1tk}
\end{align}
\end{lemma}

\subsection{Elliptic estimates}
\label{sec:ellipticest}
In this section, we derive an elliptic estimate for $\psi$ that we we will subsequently use to extend the decay estimates for the energy fluxes with respect to $J^N[T^{k+1}\psi]$ from Section \ref{sec:edecayTkpsi0} and Section \ref{sec:edecayTkpsi1} to the energy fluxes with respect to $D^2J^N[YT^k\psi]$. 

\textbf{In contrast with the previous sections, we will require the additional assumption $D'(r_+)>0$ in the $r=r_+$ case in this section.} In other words, we consider only black hole spacetimes $(\mathcal{R},g)$ such that the surface gravity of the future event horizon $\mathcal{H}^+$ is \emph{positive}. This property is in particular satisfied in sub-extremal Reissner--Nordstr\"om black holes (and in fact, this is precisely the property that characterises sub-extremality).

\begin{lemma}[\textbf{A degenerate elliptic estimate for $\psi$}]
Let $\psi$ be a solution to \eqref{waveequation} on $(\mathcal{R}, g)$, such that $D'(r_+)>0$ in the $r_{\rm min}=r_+$ case. Assume moreover that
\begin{align*}
\lim_{\rho\to \infty}r^{1/2}T\psi=&0,\\
\lim_{\rho\to \infty}r^{1/2}\partial_r\psi=&0.
\end{align*}

Then we can estimate with respect to $(\rho,\theta,\varphi)$ coordinates:
\begin{equation}
\label{eq:ellipticpsi}
\begin{split}
\int_{r_{\rm min}}^{\infty}&\int_{\s^2} D^2r^2(\partial_{\rho}^2\psi)^2+Dr^2|\snabla \partial_{\rho}\psi|^2+r^2|\slashed{\nabla}^2\psi|^2\,d\omega d\rho\Big|_{\widetilde{\tau}=\widetilde{\tau}'}\\
\leq&\: C(D)\int_{r_{\rm min}}^{\infty}\int_{\s^2}r^2(\partial_{\rho} T \psi)^2+O(r^{-2\eta})(T^2\psi)^2\,d\omega  d\rho\Big|_{\widetilde{\tau}=\widetilde{\tau}'},
\end{split}
\end{equation}
for all $\widetilde{\tau}'\in [0,\infty)$.
\end{lemma}
\begin{proof}
By \eqref{waveequation}, we have that
\begin{equation}
\label{eq:waveequationv2}
r^{-2}\partial_r(Dr^2\partial_r\psi)+\slashed{\Delta}\psi=-2\partial_r\partial_v\psi-2r^{-1}\partial_v\psi.
\end{equation}

We can rewrite in terms of $\rho$ derivatives:
\begin{equation*}
r^{-2}\partial_{\rho}(Dr^2\partial_{\rho}\psi)=r^{-2}\partial_r(Dr^2\partial_r\psi)+2hD\partial_{\rho}T\psi-Dh^2T^2\psi+r^{-2}\partial_r(Dr^2h)T\psi,
\end{equation*}
where we used that $\partial_{\rho}=Y=\partial_r+h(r)T$.

Hence,
\begin{equation}
\label{eq:waveequationv3}
\begin{split}
r^{-2}\partial_{\rho}(Dr^2\partial_{\rho}\psi)+\slashed{\Delta}\psi=&\:(2hD-2)\partial_{\rho}T\psi+(2h-Dh^2)T^2\psi\\
&+(r^{-2}\partial_r(Dr^2h)-2r^{-1})T\psi\\
=&\:(2+O(r^{-1-\eta}))\partial_{\rho}T\psi+O(r^{-1-\eta})T^2\psi+(2r^{-1}+O(r^{-2-\eta}))T\psi,
\end{split}
\end{equation}
where we used the properties of $h$ from Section \ref{sec:foliations} to arrive at the second equality of \eqref{eq:waveequationv3}.

By \eqref{eq:waveequationv3}, we have that
\begin{equation}
\label{eq:waveeqest}
\begin{split}
\int_{r_{\rm min}}^{\infty}\int_{\s^2} &r^{-4}(\partial_{\rho}(Dr^2\partial_{\rho}\psi))^2r^2+r^{-2}(\slashed{\Delta}_{\s^2}\psi)^2+2r^{-2}\partial_{\rho}(Dr^2\partial_{\rho}\psi)\slashed{\Delta}_{\s^2}\psi\,d\omega d\rho\\
\leq&\: C\int_{r_{\rm min}}^{\infty}\int_{\s^2}r^2(\partial_{\rho} T \psi)^2+O(r^{-2\eta})(T^2\psi)^2+(T\psi)^2\,d\omega d\rho.
\end{split}
\end{equation}

We can apply \eqref{eq:Hardy3} to estimate
\begin{equation*}
\int_{r_{\rm min}}^{\infty}(T\psi)^2\,d\rho\leq 4\int_{r_{\rm min}}^{\infty}(r-r_{\rm min})^2(\partial_{\rho} T\psi)^2\,d\rho,
\end{equation*}
so that
\begin{equation}
\label{eq:waveeqestv2}
\begin{split}
\int_{r_{\rm min}}^{\infty}\int_{\s^2} &r^{-4}(\partial_{\rho}(Dr^2\partial_{\rho}\psi))^2r^2+r^{-2}(\slashed{\Delta}_{\s^2}\psi)^2+2r^{-2}\partial_{\rho}(Dr^2\partial_{\rho}\psi)\slashed{\Delta}_{\s^2}\psi\,d\omega d\rho\\
\leq&\: C\int_{r_{\rm min}}^{\infty}\int_{\s^2}r^2(\partial_{\rho} T \psi)^2+O(r^{-2\eta})(T^2\psi)^2\,d\omega d\rho.
\end{split}
\end{equation}

We first consider the mixed derivative term on the left-hand side of \eqref{eq:waveeqestv2}. We integrate over $\s^2$ and integrate by parts in $\rho$ and on $\s^2$:
\begin{equation}
\begin{split}
\label{eq:ibpangular}
\int_{r_{\rm min}}^{\infty} \int_{\s^2}2r^{-2}\partial_{\rho}(Dr^2\partial_{\rho}\psi)\slashed{\Delta}_{\s^2}\psi\,d\omega d\rho=&\:\int_{r_{\rm min}}^{\infty} \int_{\s^2}4r^{-1}D\partial_{\rho}\psi\slashed{\Delta}_{\s^2}\psi-2D\partial_{\rho}\psi\slashed{\Delta}_{\s^2}\partial_{\rho}\psi\,d\omega d\rho\\
=&\:\int_{r_{\rm min}}^{\infty} \int_{\s^2}4r^{-1}D\partial_{\rho}\psi\slashed{\Delta}_{\s^2}\psi+2D|\snabla_{\s^2}\partial_{\rho}\psi|^2\,d\omega d\rho,
\end{split}
\end{equation}
where we used that all resulting boundary terms vanish by Proposition \ref{prop:step0radfields}: indeed, if $r_{\rm min}=r_+$, we use that  $D(r_+)=0$ and if $r_{\rm min}=0$, we use that $\lim_{r\to 0}\slashed{\Delta}_{\s^2}f=0$ for any function $f$ that is suitably regular at the centre of symmetry $\{r=0\}$. 

We now apply Cauchy--Schwarz and \eqref{eq:poincare2} to estimate the first term inside the integral on the very right-hand side above (more precisely, we apply \eqref{eq:poincare2} to the $\psi_{1}$ part of $\psi$, using orthogonality property of $\psi_{\ell}$):
\begin{equation*}
\begin{split}
\int_{r_{\rm min}}^{\infty}\int_{\s^2}\left|4r^{-1}D\partial_{\rho}\psi\slashed{\Delta}_{\s^2}\psi\right|\,d\omega d\rho\leq&\: \int_{r_{\rm min}}^{\infty}4D(\partial_{\rho}\psi)^2+r^{-2}D(\slashed{\Delta}_{\s^2}\psi)^2\,d\omega d\rho\\
\leq &\: \int_{r_{\rm min}}^{\infty}2D|\snabla_{\s^2}\partial_{\rho}\psi|^2+r^{-2}D(\slashed{\Delta}_{\s^2}\psi)^2\,d\omega d\rho
\end{split}
\end{equation*}

We use the above estimates together with \eqref{eq:waveeqest} to estimate:
\begin{equation}
\label{eq:mainineqelliptic}
\begin{split}
\int_{r_{\rm min}}^{\infty}\int_{\s^2} &r^{-2}(\partial_{\rho}(Dr^2\partial_{\rho}\psi))^2+r^{-2}(1-|D|)(\slashed{\Delta}_{\s^2}\psi)^2\,d\omega d\rho\\
\leq&\: C\int_{r_{\rm min}}^{\infty}(\partial_{\rho} T \psi)^2r^2+h(r)^2r^2(T^2\psi)^2\,d\omega d\rho.
\end{split}
\end{equation}

Since we assumed $|D|\leq 1$, we can now use the above estimate together with \eqref{eq:waveeqestv2} to estimate
\begin{equation}
\label{eq:mainineqelliptic3}
\begin{split}
\int_{r_{\rm min}}^{\infty}&\int_{\s^2} r^{-2}(\partial_{\rho}(Dr^2\partial_{\rho}\psi))^2+r^{2}(\slashed{\Delta}\psi)^2+Dr^2|\snabla\partial_{\rho}\psi|^2\,d\omega d\rho\\
\leq&\: C\int_{r_{\rm min}}^{\infty}\int_{\s^2}(\partial_{\rho} T \psi)^2r^2+r^2h(r)^2(T^2\psi)^2\,d\omega d\rho.
\end{split}
\end{equation}
Furthermore, we can decompose
\begin{equation*}
\begin{split}
r^{-2}(\partial_{\rho}(Dr^2\partial_{\rho}\psi))^2=&\: r^{-2}\left[Dr^2\partial_{\rho}\partial_{\rho}\psi+\partial_r(Dr^2)\partial_{\rho}\psi\right]^2\\
=&\: r^2D^2(\partial_{\rho}^2\psi)^2+r^{-2}(\partial_r(Dr^2))^2(\partial_{\rho}\psi)^2+2D\partial_r(Dr^2)\partial_{\rho}\psi\partial_{\rho}^2\psi.
\end{split}
\end{equation*}
We integrate the mixed term by parts:
\begin{equation*}
\begin{split}
\int_{r_{\rm min}}^{\infty}2D\partial_r(Dr^2)\partial_{\rho}\psi\partial_{\rho}^2\psi\,d\rho=&\:\int_{r_{\rm min}}^{\infty}D\partial_r(Dr^2)\partial_{\rho}((\partial_{\rho}\psi)^2)\,d\rho\\
= &\: -D\partial_r(Dr^2)(\partial_{\rho}\psi)^2|_{\rho=r_{\rm min}}-\int_{r_{\rm min}}^{\infty}\partial_{r}(D\partial_r(Dr^2))(\partial_{\rho}\psi)^2\,d\rho,
\end{split}
\end{equation*}
where we used that $\lim_{\rho\to \infty}D\partial_r(Dr^2)(\partial_{\rho}\psi)^2=0$.

Since moreover $D\partial_r(Dr^2)(\partial_{\rho}\psi)^2|_{\rho=r_{\rm min}}=0$ if $r_{\rm min}=r_+$ or $r_{\rm min}=0$, we are left with:
\begin{equation*}
\begin{split}
\int_{r_{\rm min}}^{\infty} r^2D^2(\partial_{\rho}^2\psi)^2\,d\rho=&\:\int_{r_{\rm min}}^{\infty} r^{-2}(\partial_{\rho}(Dr^2\partial_{\rho}\psi))^2\,d\rho\\
&+\int_{r_{\rm min}}^{\infty} \left[-r^{-2}(\partial_r(Dr^2))^2+\partial_r(D\partial_r(Dr^2))\right](\partial_{\rho}\psi)^2\,d\rho.
\end{split}
\end{equation*}
We can further write,
\begin{equation*}
\begin{split}
\partial_r(D\partial_r(Dr^2))=&\:r^{-2}\partial_r(Dr^2\partial_r(Dr^2))-2r^{-1}D\partial_r(Dr^2)\\
=&\:r^{-2}(\partial_r(Dr^2))^2+D\partial_r^2(Dr^2)-2r^{-1}D\partial_r(Dr^2).
\end{split}
\end{equation*}
Hence,
\begin{equation}
\label{eq:ellipticestv0}
\begin{split}
\int_{r_{\rm min}}^{\infty}\int_{\s^2} D^2r^2(\partial_{\rho}^2 \psi)^2+r^{-2}(\slashed{\Delta}_{\s^2}\psi)^2\,d\omega d\rho\leq&\: C(D)\int_{r_{\rm min}}^{\infty}\int_{\s^2}r^2(\partial_{\rho} T \psi)^2+o(r^{-2\eta})(T^2\psi)^2\,d\omega d\rho\\
&-\int_{r_{\rm min}}^{\infty}\int_{\s^2} F(r,D)(\partial_{\rho}\psi)^2\,d\omega d\rho,
\end{split}
\end{equation}
with
\begin{equation*}
F(D,r)=D(2r^{-1}\partial_r(Dr^2)-\partial_r^2(Dr^2)).
\end{equation*}
As $D=1-2Mr^{-1}+O(r^{-1-\beta})$, we find that there exists a $R>0$ suitably large, such that $F(D,r)\geq 0$ for all $r\geq R$. 

Suppose now that $r_{\rm min}=r_+$. For all $\epsilon>0$, there exists a $\delta>0$, such that for $|r-r_{+}|<\delta$
\begin{equation*}
|F(D,r)|\leq \epsilon.
\end{equation*}
As a consequence, we can apply \eqref{eq:Hardy3} to estimate
\begin{equation*}
\begin{split}
 \int_{r_{+}}^{r_++\delta}|F(D,r)|(\partial_{\rho}\psi)^2\,d\rho\leq &\:4\epsilon \int_{r_{\rm min}}^{\infty}(r-r_+)^2(\partial_{\rho}^2\psi)^2\,d\rho\\
 \leq & C \epsilon \int_{r_{\rm min}}^{\infty}D^2r^2(\partial_{\rho}^2\psi)^2\,d\rho,
 \end{split}
\end{equation*}
where we used in the last inequality that $D$ has a simple zero at $r=r_+$, as $D'(r_+)>0$ by assumption. For $\epsilon=\epsilon_0(D)$ suitably small, we can absorb the very right-hand side of the above equation into the left-hand side of \eqref{eq:ellipticestv0}, if we take $\delta=\delta_0(D)>0$ suitably small.

We are left with the region $\{r_++\delta_0\leq r\leq R\}$. Here, we can simply apply \eqref{eq:Hardy1} as follows:
\begin{equation*}
\begin{split}
\int_{r_++\delta_0}^R |F(D,r)|(\partial_r\psi)^2\,d\rho\leq &\:C(D,\delta_0,R)\int_{r_+}^{\infty}r^{-4}(Dr^2\partial_{\rho}\psi)^2\,d\rho\\
\leq& \:C(D,\delta_0,R)\int_{r_+}^{\infty}r^{-2}(\partial_{\rho}(Dr^2\partial_{\rho}\psi))^2\,d\rho.
\end{split}
\end{equation*}
We estimate the right-hand side by applying \eqref{eq:mainineqelliptic}.

Now, take $r_{\rm min}=0$. Then $D$ is uniformly bounded away from 0, so we can directly apply a Hardy inequality as above:
\begin{equation*}
\begin{split}
\int_{0}^R |F(D,r)|(\partial_{\rho}\psi_0)^2\,d\rho\leq &\:C(D,R)\int_{0}^{\infty}r^{-4}(Dr^2\partial_{\rho}\psi)^2\,d\rho\\
\leq& \:C(D,R)\int_{0}^{\infty}r^{-2}(\partial_{\rho}(Dr^2\partial_{\rho}\psi))^2\,d\rho.
\end{split}
\end{equation*}

The estimate \eqref{eq:ellipticpsi} now follows.
\end{proof}

\section{Pointwise decay estimates}
\label{sec:pointwise}
Pointwise decay estimates for $\psi$ and $r\psi$ follow from the energy decay estimates from Section \ref{sec:energydecay} together with a suitable application of the fundamental theorem of calculus and the Sobolev inequality \eqref{eq:sobolevs2}. In order to obtain an almost-sharp (with $\epsilon$ loss) pointwise decay rate for $\psi$, we additionally require the elliptic estimate from Section \ref{sec:ellipticest}.

\subsection{Pointwise decay for $\psi$ }
In this section we prove polynomial pointwise time-decay for $\psi$.
\begin{proposition}
\label{prop:pointdecpsi}
Let $\psi$ be a solution to \eqref{waveequation} emanating from initial data given as in Theorem \ref{thm:extuniq} on $(\mathcal{R}, g)$, such that $D'(r_+)>0$ in the $r_{\rm min}=r_+$ case.

Assume further that $E^{\epsilon}_{0,I_0\neq0;1}[\psi]<\infty$, $E^{\epsilon}_{0,I_0=0;1}[\psi]<\infty$ and $\sum_{|l|\leq 2}E_{1;1}^{\epsilon}[\Omega^{l}\psi]<\infty$, where these energies are defined in Proposition \ref{prop:endec2} and \ref{prop:edecayTkpsi1}. Assume also that 
\begin{align*}
\lim_{r \to \infty }\sum_{|l|\leq 8}\int_{\s^2}(\Omega^l\phi)^2\,d\omega\big|_{u'=0}<&\:\infty,\\
\lim_{r \to \infty }\sum_{|l|\leq 6}\int_{\s^2}(\Omega^l\Phi)^2\,d\omega\big|_{u'=0}<&\:\infty,\\
\lim_{r \to \infty }\sum_{|l|\leq 4}\int_{\s^2}r^{-1}\left(\Omega^l\Phi_{(2)}\right)^2\,d\omega\big|_{u'=0}<&\:\infty,
\end{align*}
and
\begin{equation*}
\lim_{r \to \infty }\sum_{|l|\leq 2}\int_{\s^2}r^{3}\left(\partial_r^s\Omega^l\Phi_{(2)}\right)^2\,d\omega\big|_{u'=0}<\infty,
\end{equation*}

Then, for all $\epsilon>0$ there exists a constant $C=C(D,R,\epsilon)>0$ such that for all $\widetilde{\tau}\geq 0$
\begin{align}
\label{eq:pointpsi1}
|\psi|(\widetilde{\tau},\rho,\theta,\varphi)\leq C (1+\widetilde{\tau})^{-2+\epsilon}\left[\sqrt{E^{\epsilon}_{0,I_0\neq0;1}[\psi]}+\sum_{ |\alpha|\leq 2}\sqrt{E^{\epsilon}_{1;1}[\Omega^{\alpha}\psi]}\right] \quad \textnormal{if}\quad I_0[\psi]\neq 0,\\
\label{eq:pointpsi2}
|\psi|(\widetilde{\tau},\rho,\theta,\varphi)\leq C (1+\widetilde{\tau})^{-3+\epsilon}\left[\sum_{l\leq 1}\sqrt{E^{\epsilon}_{0,I_0=0;1}[\psi]}+\sum_{|\alpha|\leq 2}\sqrt{E^{\epsilon}_{1;1}[\Omega^{\alpha}\psi]}\right] \quad \textnormal{if}\quad I_0[\psi]= 0,
\end{align}
\end{proposition}

\begin{proof}
We apply the fundamental theorem of calculus to $\psi^2$ in the $\rho$-direction and integrate over $\s^2$ to obtain
\begin{equation*}
\begin{split}
\int_{\s^2}\psi^2(\widetilde{\tau},\rho,\theta,\varphi)\,d\omega=&\:2\int_{\rho}^{\infty}\int_{\s^2} \psi\partial_{\rho}\psi\,d\omega d\rho\Big|_{\widetilde{\tau}'=\widetilde{\tau}}\\
\leq &\:2 \sqrt{\int_{\rho}^{\infty}\int_{\s^2} \psi^2\,d\omega d\rho}\cdot  \sqrt{\int_{\rho}^{\infty}\int_{\s^2} (\partial_{\rho}\psi)^2\,d\omega d\rho}\Big|_{\widetilde{\tau}'=\widetilde{\tau}}\\
\leq &\:C \sqrt{\int_{r_{\rm min}}^{\infty}\int_{\s^2} D^2r^2(\partial_\rho\psi)^2\,d\omega d\rho}\cdot  \sqrt{\int_{r_{\rm min}}^{\infty}\int_{\s^2} r^2D^2(\partial_\rho^2\psi)^2\,d\omega d\rho}\Big|_{\widetilde{\tau}'=\widetilde{\tau}}.
\end{split}
\end{equation*}
By the above equation together with \eqref{eq:ellipticpsi} and \eqref{eq:Tcurrent3} we can conclude that
\begin{equation*}
\int_{\s^2}\psi^2(\widetilde{\tau},\rho,\theta,\varphi)\,d\omega\leq \sqrt{\int_{\mathcal{S}_{\widetilde{\tau}}}J^T[\psi]\cdot n_{\widetilde{\tau}} \, d\mu_{\mathcal{S}_{\widetilde{\tau}}}}\cdot \sqrt{\int_{\mathcal{S}_{\widetilde{\tau}}}J^T[T\psi]\cdot n_{\widetilde{\tau}}\, d\mu_{\mathcal{S}_{\widetilde{\tau}}} }.
\end{equation*}
The estimates \eqref{eq:pointpsi1} and \eqref{eq:pointpsi2} now follow by applying Lemma \ref{lm:sobolevs2} and commuting $\square_g$ with $\Omega^{\alpha}$ and $T^k$ and using the energy decay from Proposition \ref{prop:endec1} and \ref{prop:endec2} for the $\psi_0$ part of $\psi$ and Proposition \ref{decl2} and \ref{prop:edecayTkpsi1} for the $\psi_1$ part. We use moreover that $\widetilde{\tau}\sim \tau$.
\end{proof}

\subsection{Pointwise decay for $T^{k}\psi$ }
\label{sec:pointdecTkpsi}
We moreover obtain improved decay estimates if we consider $T^k\psi$ with $k\geq 1$ instead of $\psi$.
\begin{proposition}
\label{prop:pointdecpsi}
Let $\psi$ be a solution to \eqref{waveequation} emanating from initial data given as in Theorem \ref{thm:extuniq} on $(\mathcal{R}, g)$, such that $D'(r_+)>0$ in the $r_{\rm min}=r_+$ case. Let $n\in \N_0$ and assume moreover that
$D(r)=1-2Mr^{-1}+O_{n+2}(r^{-1-\beta})$, with $\beta>0$.

Assume further that $E^{\epsilon}_{0,I_0\neq0;n+1}[\psi]<\infty$, $E^{\epsilon}_{0,I_0=0;n+1}[\psi]<\infty$ and $\sum_{|l|\leq 2}E^{\epsilon}_{1;n+1}[\Omega^l\psi]<\infty$, where these energies are defined in Proposition \ref{prop:endec2} and \ref{prop:edecayTkpsi1}. Assume also that 
\begin{align*}
\lim_{r \to \infty }\sum_{|l|\leq 8+2n}\int_{\s^2}(\Omega^l\phi)^2\,d\omega\big|_{u'=0}<&\:\infty,\\
\lim_{r \to \infty }\sum_{|l|\leq 6+2n}\int_{\s^2}(\Omega^l\Phi)^2\,d\omega\big|_{u'=0}<&\:\infty,\\
\lim_{r \to \infty }\sum_{|l|\leq 4+2n}\int_{\s^2}r^{-1}\left(\Omega^l\Phi_{(2)}\right)^2\,d\omega\big|_{u'=0}<&\:\infty,
\end{align*}
and
\begin{equation*}
\lim_{r \to \infty }\sum_{|l|\leq 2+2n-2s}\int_{\s^2}r^{2s+1}\left(\partial_r^s\Omega^l\Phi_{(2)}\right)^2\,d\omega\big|_{u'=0}<\infty,
\end{equation*}
for each $1\leq s\leq n$.

For all $\epsilon>0$ there exists a constant $C=C(D,R,\epsilon,n)>0$ such that for all $0\leq k\leq n$ and $\widetilde{\tau}\geq 0$
\begin{align}
\label{eq:pointrtkpsi1}
|T^k\psi|(\widetilde{\tau},\rho,\theta,\varphi)\leq C\sqrt{E^{\epsilon}_{0,I_0\neq0;k+1}[\psi]+\sum_{|\alpha|\leq 2}E^{\epsilon}_{1;k+1}[\Omega^{\alpha}\psi]}(1+\widetilde{\tau})^{-2-k+\epsilon} \quad \textnormal{if}\quad I_0[\psi]\neq 0,\\
\label{eq:pointrtkpsi2}
|T^k\psi|(\widetilde{\tau},\rho,\theta,\varphi)\leq C\sqrt{E^{\epsilon}_{0,I_0=0;k+1}[\psi]+\sum_{|\alpha|\leq 2}E^{\epsilon}_{1;k+1}[\Omega^{\alpha}\psi]}(1+\widetilde{\tau})^{-3-k+\epsilon} \quad \textnormal{if}\quad I_0[\psi]= 0.
\end{align}
\end{proposition}
\begin{proof}
We proceed as in the proof of Proposition \ref{prop:pointdecpsi}, using the energy decay estimates for $T^k\psi$ from Proposition \ref{prop:endec2} and \ref{prop:edecayTkpsi1}.
\end{proof}

\subsection{Pointwise decay for the radiation field $r\psi$}
\label{sec:pointdecayradfield}
Before we obtain the optimal (with loss in $\epsilon$) pointwise decay rate for $r\psi$, we first prove an intermediate decay result for $\psi_0$.
\begin{lemma}\label{lm:auxpointdecay}
Let $\psi$ be a spherically symmetric solution to \eqref{waveequation} emanating from initial data given as in Theorem \ref{thm:extuniq} on $(\mathcal{R}, g)$. 

For all $\epsilon>0$ and for $R>0$ suitably large there exists a constant $C=C(D,R,\epsilon)>0$ such that for all $u\geq 0$ and $r\geq R+1$:
\begin{align}
\label{eq:pointrtkpsi0v0}
|r\psi|(u,r,\theta,\varphi)\leq C\sqrt{E^{\epsilon}_{0;\textnormal{aux}}[\psi]}(1+u)^{-\frac{3}{2}+\epsilon} \quad \textnormal{if}\quad I_0[\psi]= 0,
\end{align}
where the energy $E^{\epsilon}_{0;\textnormal{aux}}[\psi]$ is defined in Lemma \ref{lm:auxdecaypsi0}.
\end{lemma}
\begin{proof}
Let $r\geq R+1$. We apply the fundamental theorem of calculus to $(\chi\phi)^2$ to obtain
\begin{equation*}
\begin{split}
\phi^2(u,r)=&\:2\int_{r}^{\infty}\chi\phi\partial_r(\chi\phi)\, dr'\Big|_{u'=u}\\
\leq &\:2 \sqrt{\int_{r}^{\infty} r^{-2}(\chi\phi)^2\, dr'}\cdot  \sqrt{\int_{r}^{\infty} r^2(\partial_r(\chi\phi))^2\, dr'}\Big|_{u'=u},
\end{split}
\end{equation*}
where we applied a Cauchy--Schwarz inequality. Now we can apply \eqref{eq:Hardy1} to obtain
\begin{equation}
\label{eq:fundtcalcphi}
\phi^2(\tau,r)\leq C \sqrt{\int_{r}^{\infty}(\partial_r(\chi\phi))^2\, dr'}\cdot  \sqrt{\int_{r}^{\infty} r^2(\partial_r(\chi\phi))^2\, dr'}\Big|_{u'=u}.
\end{equation}
By (ii) of Lemma \ref{lm:auxdecaypsi0}, we have that
\begin{align*}
\int_{R}^{\infty}(\partial_r(\chi T^k\phi))^2\, dr\Big|_{u'=u}\leq&\: CE^{\epsilon}_{0,\textnormal{aux}}[\psi](1+u)^{-4+\epsilon}\quad \textnormal{if}\quad I_0[\psi]= 0,\\
\int_{R}^{\infty}r^2(\partial_r(\chi T^k\phi))^2\, dr\Big|_{u'=u}\leq&\: CE^{\epsilon}_{0,\textnormal{aux}}[\psi](1+u)^{-2+\epsilon}\quad \textnormal{if}\quad I_0[\psi]= 0.
\end{align*}
The estimates \eqref{eq:pointrtkpsi0v0} now follows from the above estimates.
\end{proof}
\begin{proposition}
\label{prop:pointdecaradfield}
Let $\psi$ be a solution to \eqref{waveequation} emanating from initial data given as in Theorem \ref{thm:extuniq} on $(\mathcal{R}, g)$.

Assume further that $E^{\epsilon}_{0,I_0\neq0}[\psi]<\infty$, $E^{\epsilon}_{0,I_0=0}[\psi]<\infty$ and $\sum_{|l|\leq 2}E_{1}^{\epsilon}[\Omega^l\psi]<\infty$, where these energies are defined in Proposition \ref{prop:endec1} and \ref{decl2}, and also: 
\begin{align*}
\lim_{r \to \infty }\sum_{|l|\leq 6}\int_{\s^2}(\Omega^l\phi)^2\,d\omega\big|_{u'=0}<&\:\infty,\\
\lim_{r \to \infty }\sum_{|l|\leq 4}\int_{\s^2}(\Omega^l\Phi)^2\,d\omega\big|_{u'=0}<&\:\infty,\\
\lim_{r \to \infty }\sum_{|l|\leq 2}\int_{\s^2}r^{-1}\left(\Omega^l\Phi_{(2)}\right)^2\,d\omega\big|_{u'=0}<&\:\infty.
\end{align*}

Then, for all $\epsilon>0$ and for $R>0$ suitably large there exists a constant $C=C(D,R,\epsilon)>0$ such that for all $\widetilde{\tau}\geq 0$ 
\begin{align}
\label{eq:pointrpsi1}
|r\psi|(\widetilde{\tau},r,\theta,\varphi)\leq C\sqrt{E^{\epsilon}_{0,I_0\neq0}[\psi]+\sum_{|\alpha|\leq 2}E^{\epsilon}_{1}[\Omega^{\alpha}\psi]}(1+\widetilde{\tau})^{-1+\epsilon} \quad \textnormal{if}\quad I_0[\psi]\neq 0,\\
\label{eq:pointrpsi2}
|r\psi|(\widetilde{\tau},r,\theta,\varphi)\leq C\sqrt{E^{\epsilon}_{0,I_0=0}[\psi]+\sum_{|\alpha|\leq 2}E^{\epsilon}_{1}[\Omega^{\alpha}\psi]}(1+\widetilde{\tau})^{-2+\epsilon} \quad \textnormal{if}\quad I_0[\psi]=0.
\end{align}
\end{proposition}
\begin{proof}
Let $r\geq R+1$. We apply the fundamental theorem of calculus to $(\chi\phi)^2$ at fixed $(\theta,\varphi)$ and integrate over $\s^2$ to obtain
\begin{equation*}
\begin{split}
\int_{\s^2}\phi^2(\tau,r,\theta,\varphi)\,d\omega=&\:2\int_{r}^{\infty}\int_{\s^2} \chi\phi\partial_r(\chi\phi)\,d\omega dr'\Big|_{u=\tau}\\
\leq &\:2 \sqrt{\int_{r}^{\infty}\int_{\s^2} r^{-2}(\chi\phi)^2\,d\omega dr'}\cdot  \sqrt{\int_{r}^{\infty}\int_{\s^2} r^2(\partial_r(\chi\phi))^2\,d\omega dr'}\Big|_{u=\tau},
\end{split}
\end{equation*}
where we applied a Cauchy--Schwarz inequality. Now we can apply \eqref{eq:Hardy1} to obtain
\begin{equation}
\label{eq:fundtcalcphi}
\int_{\s^2}\phi^2(\tau,r,\theta,\varphi)\,d\omega\leq C \sqrt{\int_{r}^{\infty}\int_{\s^2}(\partial_r(\chi\phi))^2\,d\omega dr'}\cdot  \sqrt{\int_{r}^{\infty}\int_{\s^2} r^2(\partial_r(\chi\phi))^2\,d\omega dr'}\Big|_{u=\tau}.
\end{equation}
By (i) and (iii) of Lemma \ref{lm:auxdecaypsi0} and Lemma \ref{lm:auxdecaypsi1}, we have that
\begin{align*}
\int_{R}^{\infty}\int_{\s^2}(\partial_r(\chi \phi))^2\,d\omega dr\Big|_{u=\tau}\leq&\: C(E^{\epsilon}_{0;I_0\neq 0}[\psi]+E_{1}[\psi])(1+\tau)^{-3+\epsilon}\quad \textnormal{if}\quad I_0[\psi]\neq 0,\\
\int_{R}^{\infty}\int_{\s^2}(\partial_r(\chi \phi))^2\,d\omega dr\Big|_{u=\tau}\leq&\: C(E^{\epsilon}_{0;I_0\neq 0}[\psi]+E_{1}[\psi])(1+\tau)^{-5+\epsilon}\quad \textnormal{if}\quad I_0[\psi]= 0.
\end{align*}
Moreover, by (i) and (iii) of Lemma \ref{lm:auxdecaypsi0} and Lemma \ref{lm:auxdecaypsi1}, we have that
\begin{align*}
\int_{R}^{\infty}\int_{\s^2}r^2(\partial_r(\chi \phi))^2\,d\omega dr\Big|_{u=\tau}\leq&\: C(E^{\epsilon}_{0;I_0=0}[\psi]+E_{1}[\psi])(1+\tau)^{-1+\epsilon}\quad \textnormal{if}\quad I_0[\psi]\neq 0,\\
\int_{R}^{\infty}\int_{\s^2}r^2(\partial_r(\chi \phi))^2\,d\omega dr\Big|_{u=\tau}\leq&\: C(E^{\epsilon}_{0;I_0=0}[\psi]+E_{1}[\psi])(1+\tau)^{-3+\epsilon}\quad \textnormal{if}\quad I_0[\psi]= 0.
\end{align*}
The estimates \eqref{eq:pointrpsi1} and \eqref{eq:pointrpsi1} for $r\geq R+1$ now follow from the above estimates together with the Sobolev inequality on $\s^2$ in Lemma \ref{lm:sobolevs2} and \eqref{eq:comptimes}.

In order to extend \eqref{eq:pointrpsi1} and \eqref{eq:pointrpsi1} to $r<R+1$, we instead apply the fundamental theorem of calculus to $\psi$ along $\mathcal{S}_{\widetilde{\tau}}$, together with Cauchy--Schwarz, to obtain
\begin{equation*}
|\psi|(\widetilde{\tau},\rho,\theta,\varphi)\leq \int_{\rho}^{\infty}|\partial_{\rho}\psi|\,d\rho'\Big|_{\widetilde{\tau}'=\widetilde{\tau}}\leq \sqrt{\int_{\rho}^{\infty} \rho'^{-2}\,d\rho'}\sqrt{\int_{\rho}^{\infty}(\partial_r\psi)^2 \rho^2\,d\rho'}\Big|_{\widetilde{\tau}'=\widetilde{\tau}}.
\end{equation*}
Hence,
\begin{equation*}
r \int_{\s^2}\psi^2(\widetilde{\tau}, \rho,\theta,\varphi)\leq \int_{\mathcal{S}_{\widetilde{\tau}}}J^T[\psi]\cdot n_{\widetilde{\tau}}\,d\mu_{\mathcal{S}_{\widetilde{\tau}}}.
\end{equation*}
By the Sobolev inequality on $\s^2$ in Lemma \ref{lm:sobolevs2} together with Proposition \ref{prop:endec1} and \ref{decl2}, the estimates \eqref{eq:pointrpsi1} and \eqref{eq:pointrpsi2} therefore also hold for $r<R+1$.
\end{proof}

\subsection{Pointwise decay for $T^{k}(r\psi)$ }
\label{sec:pointdecTkrpsi}
We can moreover obtain improved decay estimates for the radiation fields $T^k(r\psi)$.
\begin{proposition}
Let $\psi$ be a solution to \eqref{waveequation} emanating from initial data given as in Theorem \ref{thm:extuniq} on $(\mathcal{R}, g)$. Let $n\in \N_0$ and assume that $D(r)=1-2Mr^{-1}+O_{n+2}(r^{-1-\beta})$, with $\beta>0$. 

Assume further that $E^{\epsilon}_{0,I_0\neq0;n}[\psi]<\infty$, $E^{\epsilon}_{0,I_0=0;n}[\psi]<\infty$ and $\sum_{|l|\leq 2}E^{\epsilon}_{1;n}[\Omega^l\psi]<\infty$ are defined in Proposition \ref{prop:endec2} and \ref{prop:edecayTkpsi1}. Assume also that:

\begin{align*}
\lim_{r \to \infty }\sum_{|l|\leq 6+2n}\int_{\s^2}(\Omega^l\phi)^2\,d\omega\big|_{u'=0}<&\:\infty,\\
\lim_{r \to \infty }\sum_{|l|\leq 4+2n}\int_{\s^2}(\Omega^l\Phi)^2\,d\omega\big|_{u'=0}<&\:\infty,\\
\lim_{r \to \infty }\sum_{|l|\leq 2+2n}\int_{\s^2}r^{-1}\left(\Omega^l\Phi_{(2)}\right)^2\,d\omega\big|_{u'=0}<&\:\infty,
\end{align*}
and
\begin{equation*}
\lim_{r \to \infty }\sum_{|l|\leq 2+2n-2s}\int_{\s^2}r^{2s+1}\left(\partial_r^s\Omega^l\Phi_{(2)}\right)^2\,d\omega\big|_{u'=0}<\infty,
\end{equation*}
for each $1\leq s\leq n$.

Then, for all $\epsilon>0$ and for $R>0$ suitably large there exists a constant $C=C(D,R,\epsilon,n)>0$ such that for all $k\leq n$ and $\widetilde{\tau}\geq 0$ 
\begin{align}
\label{eq:pointrtkpsi1}
|rT^k\psi|(\widetilde{\tau},r,\theta,\varphi)\leq C\sqrt{E^{\epsilon}_{0,I_0\neq0;k}[\psi]+\sum_{|\alpha|\leq 2}E^{\epsilon}_{1;k}[\Omega^{\alpha}\psi]}\widetilde{\tau}^{-1-k+\epsilon} \quad \textnormal{if}\quad I_0[\psi]\neq 0,\\
\label{eq:pointrtkpsi2}
|rT^k\psi|(\widetilde{\tau},r,\theta,\varphi)\leq C\sqrt{E^{\epsilon}_{0,I_0=0;k}[\psi]+\sum_{|\alpha|\leq 2}E^{\epsilon}_{1;k}[\Omega^{\alpha}\psi]}\widetilde{\tau}^{-2-k+\epsilon} \quad \textnormal{if}\quad I_0[\psi]= 0.
\end{align}
\end{proposition}
\begin{proof}
The proof proceeds almost identically to the proof of Proposition \ref{prop:pointdecaradfield}, with $\psi$ replaced by $T^k\psi$, using the $L^2$ decay estimates from Section \ref{sec:edecayTkpsi0} and \ref{sec:energydecaytkpsi1}.
\end{proof}
\appendix

\section{Useful Calculations}
\subsection{Commutation vector fields and vector field multipliers}
\label{app:commmultp}
Consider the stress-energy tensor $\mathbf{T}_{\alpha\beta}[f]=\partial_{\alpha}f\partial_{\beta}f-\frac{1}{2}g_{\alpha\beta}(g^{-1})^{\delta \gamma}\partial_{\delta} f\partial_{\gamma}f$. In $(u,r,\theta,\varphi)$ coordinates, we have that

\begin{align*}
\mathbf{T}_{uu}[f]&=(\partial_uf)^2+\frac{D}{2}\left[D(\partial_rf)^2-2\partial_rf\partial_uf+|\snabla f|^2\right],\\
\mathbf{T}_{rr}[f]&=(\partial_rf)^2,\\
\mathbf{T}_{ur}[f]&=\frac{1}{2}D(\partial_rf)^2+\frac{1}{2}|\snabla f|^2,\\
\mathbf{T}_{AB}[f]&=\partial_Af\partial_Bf-\frac{1}{2}\slashed{g}_{AB}\left[D(\partial_rf)^2-2\partial_rf\partial_uf+|\snabla f|^2\right],\\
\mathbf{T}^A_A[f]&=2\partial_rf\partial_uf-D(\partial_rf)^2.
\end{align*}

First, we consider the energy currents along null hypersurfaces.

\begin{proposition}
The corresponding energy currents with respect to the Killing vector field $T=\partial_u$ are given by
\begin{align}
\label{eq:Tcurrent1}
J^T[f]\cdot L=&\:\frac{1}{4}D^2(\partial_rf)^2+\frac{1}{4}D|\snabla f|^2=(Lf)^2+\frac{1}{4}D|\snabla f|^2,\\
\label{eq:Tcurrent2}
J^T[f]\cdot \underline{L}=&\:(\partial_u f-\frac{1}{2}D\partial_rf)^2+\frac{D}{4}|\snabla f|^2=(\underline{L}f)^2+\frac{D}{4}|\snabla f|^2.
\end{align}
Furthermore, let us denote with $g_{\mathcal{S}}$ the induced metric on $\mathcal{S}$, with $n_{\mathcal{S}}$ the corresponding normal vector field. Then we have that
\begin{equation}
\label{eq:Tcurrent3}
\sqrt{\det g_{\mathcal{S}}}J^T[f]\cdot n_{\mathcal{S}}=\left[\frac{1}{2}h(2-hD)(Tf)^2+\frac{1}{2}D(Yf)^2+|\snabla f|^2\right]r^2\sin\theta.
\end{equation}
and also that
\begin{equation}
\label{eq:Tcurrent4}
\sqrt{\det g_{\mathcal{S}}}J^N[f]\cdot n_{\mathcal{S}}\sim \left[(2-hD)(Tf)^2+(Yf)^2+|\snabla f|^2\right]r^2\sin\theta.
\end{equation}
\end{proposition}
\begin{proof}
The expressions \eqref{eq:Tcurrent1} and \eqref{eq:Tcurrent1} follow easily after using that, in $(u,r,\theta,\varphi)$ coordinates,
\begin{align*}
L=&\:\frac{1}{2}D\partial_r,\\
\underline{L}=&\:T-L=\partial_u-\frac{1}{2}D\partial_r.
\end{align*}

We are left with proving \eqref{eq:Tcurrent3}. We can express $\mathcal{S}=\{v-v_{Y}(r)=0\}$, where $\frac{dv_Y}{dr}=h$. Therefore, the corresponding induced metric ${g}_{\mathcal{S}}$ is given by
\begin{equation*}
g_{\mathcal{S}}=h(2-hD)dr^2+r^2(d\theta^2+\sin^2 \theta d\varphi^2).
\end{equation*}
Consequently,
\begin{equation*}
\sqrt{\det g_{\mathcal{S}}}=\sqrt{h(2-hD)}r^2\sin\theta.
\end{equation*}
We can express the vector field $Y$ tangential to $\mathcal{S}$ in $(u,r,\theta,\varphi)$ coordinates:
\begin{equation*}
Y=-\frac{2}{D}\underline{L}+hT=-\frac{2}{D}(T-L)+hT=\partial_r-\left(\frac{2}{D}-h\right)\partial_u.
\end{equation*}

Now, let us introduce the vector field $X=\partial_r+k(r)\partial_u$, where $k: [r_{\rm min},\infty)\to \infty$ is defined by requiring $g(X,Y)=0$, i.e.\
\begin{equation*}
0=g(\partial_r+k\partial_u,\partial_r-\left(2D^{-1}-h\right)\partial_u)=(2D^{-1}-h)+k(2-hD-1).
\end{equation*}
Hence
\begin{equation*}
k(r)=\frac{2D^{-1}-h}{hD-1}
\end{equation*}
and moreover,
\begin{equation*}
\begin{split}
g(X,X)=-2k-Dk^2=&\:-k(2+Dk)\\
=&\:-\frac{2D^{-1}-h}{hD-1}\left(\frac{2-hD}{hD-1}+2\right)\\
=&-\frac{h(2-hD)}{(hD-1)^2}.
\end{split}
\end{equation*}
Therefore,
\begin{equation*}
n_{\mathcal{S}}=\frac{X}{\sqrt{-g(X,X)}}=\frac{1}{\sqrt{h(2-hD)}}[(hD-1)\partial_r+(2D^{-1}-h)\partial_u].
\end{equation*}
From the above, we obtain
\begin{equation*}
\begin{split}
\sqrt{\det g_{\mathcal{S}}}J^T[f]\cdot n_{\mathcal{S}} =&\:\left[(hD-1)\mathbf{T}_{ur}[f]+(2D^{-1}-h)\mathbf{T}_{uu}[f]\right]r^2\sin \theta\\
\end{split}
\end{equation*}
We can express in terms of $T$ and $Y$ derivatives:
\begin{equation*}
\begin{split}
\mathbf{T}_{uu}[f]&\:=(\partial_uf)^2+\frac{D}{2}\left[D(\partial_rf)^2-2\partial_rf\partial_uf+|\snabla f|^2\right]\\
&\:=(Tf)^2+\frac{D}{2}\left[D(Yf)^2-(2-hD)h(Tf)^2+2(1-hD)Tf\cdot Yf+|\snabla f|^2\right]
\end{split}
\end{equation*}
and
\begin{equation*}
\begin{split}
\mathbf{T}_{ur}[f]&\:=\frac{1}{2}D(\partial_rf)^2+\frac{1}{2}|\snabla f|^2\\
&\:=\frac{1}{2}D(Yf)^2+(2-hD)Yf\cdot  Tf+\frac{1}{2}(2-hD)(2D^{-1}-h)(Tf)^2+\frac{1}{2}|\snabla f|^2.
\end{split}
\end{equation*}
After adding the above expressions in $\sqrt{\det g_{\mathcal{S}}}J^T[f]\cdot n_{\mathcal{S}}$ it follows that the $Yf\cdot Tf$ terms cancel and we are left with
\begin{equation*}
\sqrt{\det g_{\mathcal{S}}}J^T[f]\cdot n_{\mathcal{S}}=\left[\frac{1}{2}h(2-hD)(Tf)^2+\frac{1}{2}D(Yf)^2+|\snabla f|^2\right]r^2\sin\theta.
\end{equation*}
Note that in the $r_{\min}=r_+$ case, we can consider the energy flux with respect to $N=T+\frac{2}{D}\underline{L}$, rather than $T$ in a region $\{r_+\leq r\leq r_1\}$, with $|r_1-r_+|$ suitably small. Since $N$ is timelike in $\{r_+\leq r\leq r_1\}$, \eqref{eq:Tcurrent4} follows in $\{r\leq r_1\}$. In the remaining region, \eqref{eq:Tcurrent4} follows from \eqref{eq:Tcurrent3}.
\end{proof}

Let $V$ denote the \emph{vector field multiplier} $V=r^{p-2}\partial_r$ in $(u,r,\theta,\varphi)$ coordinates, with $p\in \R$. We have that
\begin{align*}
J^V[f]\cdot L=&\:\frac{1}{2}Dr^{p-2}(\partial_rf)^2,\\
J^V[f]\cdot \underline{L}=&\:\frac{1}{2}r^{p-2}|\snabla f|^2.
\end{align*}
Now, we consider the spacetime currents that appear in the divergence theorem. We have that $K^T[f]=0$ and $\mathcal{E}^T[\psi]=0$ if $\square_g\psi=0$.

Furthermore,
\begin{equation*}
K^V[\phi]=\mathbf{T}^{\alpha}_{\beta}(\nabla_{\alpha} V)^{\beta}=\mathbf{T}^{\alpha}_{\beta}(\nabla_{\alpha} (r^{p-2}\partial_r))^{\beta},
\end{equation*}
where $K^V[f]$ is defined in (\ref{def:KV}) and the connection coefficients $\nabla_{\partial_{\alpha}}\partial_{\beta}$ are given by
\begin{align*}
\nabla_{\partial_u}\partial_u&=-\frac{1}{2}D'\partial_u+\frac{1}{2}DD' \partial_r,\\
\nabla_{\partial_r}\partial_r&=0,\\
\nabla_{\partial_u}\partial_r&=\frac{1}{2}D'\partial_r,\\
(\nabla_{\partial_A}\partial_r)^B&=r^{-1}\delta^B_A,\quad \textnormal{where}\, A,B=\theta,\varphi.
\end{align*}

Consequently,
\begin{equation}
\label{eq:KVinfinity}
\begin{split}
K^V[f]&=(p-2)r^{p-3}\mathbf{T}^{r}_{r}+r^{p-2}\mathbf{T}^{\alpha}_{\beta}(\nabla_{\alpha}\partial_r)^{\beta}\\
&=\mathbf{T}^r_r(p-2)r^{p-3}+\frac{1}{2}D'r^{p-2}\mathbf{T}^u_r+r^{p-3}\mathbf{T}^A_A,\\
&=(g^{rr}\mathbf{T}_{rr}+g^{ur}\mathbf{T}_{ur})(p-2)r^{p-3}+\frac{1}{2}D'r^{p-2}g^{ur}\mathbf{T}_{rr}+r^{p-3}\mathbf{T}^A_A\\
&=D(p-2)r^{p-3}(\partial_r f)^2+(2-p)r^{p-3}\left[\frac{1}{2}D(\partial_r f)^2+\frac{1}{2}|\snabla f|^2\right]\\
&-\frac{1}{2}D'r^{p-2}(\partial_r f)^2-r^{p-3}D(\partial_r f)^2+2r^{p-3}\partial_r f\partial_u f\\
&=\frac{1}{2}r^{p-3}\left[D(p-4)-D'r\right](\partial_r f)^2+2r^{p-3}\partial_r f\partial_uf+\frac{1}{2}(2-p)r^{p-3}|\snabla f|^2.
\end{split}
\end{equation}
Consider moreover the \emph{commutation vector fields} $r^2\partial_r$ and $r(r-M)\partial_r$. In the remainder of this section we prove Lemma \ref{lm:commute0time}--\ref{lm:commute2time} and \ref{lm:boxdrkPsi2}.

\begin{proof}[Proof of Lemma \ref{lm:commute0time}]
We first need to express $\square_g\psi$ in $(u,r,\theta,\varphi)$ coordinates.
\begin{equation}
\label{eq:derivwaveeq}
\begin{split}
\square_g\psi&=(-\det g)^{-\frac{1}{2}}\partial_{\alpha}((-\det g)^{\frac{1}{2}}g^{\alpha\beta}\partial_{\beta}\psi)\\
&=\partial_u(g^{ur}\partial_r\psi)+r^{-2}\partial_r(r^2g^{rr}\partial_r\psi+r^2g^{ur}\partial_u\psi)\\
&=-2\partial_u\partial_r\psi+D\partial_r^2\psi-2r^{-1}\partial_u\psi+R\partial_r\psi+\slashed{\Delta}\psi,
\end{split}
\end{equation}
where $R:=r^{-2}(Dr^2)'=D'+2r^{-1}D$.

Therefore,
\begin{equation}
\label{eq:boxgphi}
\begin{split}
\square_g\phi&=\square_g(r\psi)=-2\partial_u\partial_r(r\psi)+D\partial_r^2(r\psi)-2r^{-1}\partial_u(r\psi)+R\partial_r(r\psi)+\slashed{\Delta}(r\psi)\\
&=r\square_g\psi-2\partial_u\psi+2D\partial_r\psi+R\psi\\
&=r\square_g\psi-2r^{-1}\partial_u\phi+2D\partial_r(r^{-1}\phi)+Rr^{-1}\phi\\
&=r\square_g\psi-2r^{-1}\partial_u\phi+2Dr^{-1}\partial_r\phi+\left[Rr^{-1}-2Dr^{-2}\right]\phi\\
&=-2r^{-1}\partial_u\phi+2Dr^{-1}\partial_r\phi+D'r^{-1}\phi,
\end{split}
\end{equation}
where we used that $\square_g\psi=0$ in the last equality.

In particular, by rearranging the terms above, we obtain
\begin{equation*}
\slashed{\Delta} \phi=2\partial_u\partial_r\phi-\partial_r(D\partial_r\phi)+D'r^{-1}\phi.
\end{equation*}
We rearrange the above terms to obtain (\ref{eq:sDeltapsiinfty}).
\end{proof}

\begin{proof}[Proof of Lemma \ref{lm:commute1time}]
For any function $f$ we have that
\begin{equation}
\label{eq:commdr}
\begin{split}
[\partial_r,\square_g]f&=-\square_g(\partial_rf)-\partial_r(2\partial_u\partial_rf+D\partial_r^2f-2r^{-1}\partial_uf+(D'+2r^{-1}D)\partial_rf+\slashed{\Delta} f)\\
&=D'\partial_r^2f+2r^{-2}\partial_uf+(D''+2D'r^{-1}-2Dr^{-2})\partial_r f-2r^{-1}\slashed{\Delta}f.
\end{split}
\end{equation}

Moreover, by differentiating \eqref{eq:boxgphi} in $r$, it follows that
\begin{equation*}
\begin{split}
\partial_r(\square_g\phi)&=\partial_r(-2r^{-1}\partial_u\phi+2Dr^{-1}\partial_r\phi+D'r^{-1}\phi)\\
&=2r^{-2}\partial_u\phi-2r^{-1}\partial_u\partial_r\phi+2r^{-1}(D'-Dr^{-1})\partial_r\phi+2Dr^{-1}\partial_r^2\phi+r^{-1}(D''-r^{-1}D')\phi\\
&+D'r^{-1}\partial_r\phi.
\end{split}
\end{equation*}
Putting the above two estimates together, we therefore obtain
\begin{equation*}
\begin{split}
\square_g(\partial_r\phi)&=-2r^{-1}\partial_u\partial_r\phi+(2Dr^{-1}-D')\partial_r^2\phi+\left(D'r^{-1}-D''\right)\partial_r\phi+2r^{-1}\slashed{\Delta}\phi\\
&+r^{-1}(D''-D'r^{-1})\phi.
\end{split}
\end{equation*}

Now let $\Phi=r^q\partial_r\phi$. We have that
\begin{equation*}
\begin{split}
\square_g\Phi&=\square_g(r^q\partial_r\phi)=-2\partial_u\partial_r(r^q\partial_r\phi)+D\partial_r^2(r^q\partial_r\phi)-2r^{-1}\partial_u(r^q\partial_r\phi)+R\partial_r(r^q\partial_r\phi)+\slashed{\Delta}(r^q\partial_r\phi)\\
&=r^q\square_g(\partial_r\phi)-2qr^{q-1}\partial_u\partial_r\phi+2qDr^{q-1}\partial_r^2\phi+q(q-1)Dr^{q-2}\partial_r\phi+qRr^{q-1}\partial_r\phi\\
&=r^{q-1}\left[2(q+1)D-D'r\right]\partial_r^2\phi-2(q+1)r^{q-1}\partial_u\partial_r\phi+2r^{q-1}\slashed{\Delta}\phi\\
&+r^{q-1}\left[-D''r+qR+D'+q(q-1)Dr^{-1}\right]\partial_r\phi+r^{q-1}[D''-D'r^{-1}]\phi.
\end{split}
\end{equation*}
Note that we can write
\begin{equation*}
\partial_r\Phi=\partial_r(r^q\partial_r\phi)=r^q\partial_r^2\phi+qr^{q-1}\partial_r\phi.
\end{equation*}
We fill in the above relation to obtain
\begin{equation}
\label{eq:waveeqPsi}
\begin{split}
\square_g\Phi&=r^{-1}\left[2(q+1)D-D'r\right]\partial_r\Phi-2(q+1)r^{-1}\partial_u\Phi+2r^{q-1}\slashed{\Delta}\phi\\
&+r^{-1}[D''r+qR+q(q-1)Dr^{-1}-2q(q+1)Dr^{-1}+D'q]\Phi+r^{q-1}(D''-D'r^{-1})\phi\\
&+r^{q-1}[D''-D'r^{-1}]\phi\\
&=r^{-1}\left[2(q+1)D-D'r\right]\partial_r\Phi-2(q+1)r^{-1}\partial_u\Phi+2r^{q-1}\slashed{\Delta}\phi\\
&+r^{-1}[-D''r+(2q+1)D'-q(q+1)Dr^{-1}]\Phi+r^{q-1}[D''-D'r^{-1}]\phi.
\end{split}
\end{equation}
Furthermore, we use (\ref{eq:sDeltapsiinfty}) to express
\begin{equation*}
\begin{split}
-2qr^{-1}\partial_u\Phi&=-2qr^{q-1}\partial_u\partial_r\phi=-qr^{q-1}\slashed{\Delta}\phi-qDr^{q-1}\partial_r^2\phi-qD'r^{q-1}\partial_r\phi+qD'r^{q-2}\phi\\
&=-qr^{q-1}\slashed{\Delta}\phi-qDr^{-1}\partial_r\Phi+(qDr^{-2}-D'r^{-1})q\Phi+qD'r^{q-2}\phi.
\end{split}
\end{equation*}

We finally obtain,
\begin{equation}
\label{eq:boxPhiinftyq}
\begin{split}
\square_g\Phi&=r^{-1}\left[(2+q)D-D'r\right]\partial_r\Phi-2r^{-1}\partial_u\Phi+(2-q)r^{q-1}\slashed{\Delta}\phi\\
&+r^{-1}[-D''r+(q+1)D'-qDr^{-1}]\Phi+r^{q-1}[D''+(q-1)D'r^{-1}]\phi.
\end{split}
\end{equation}

Since moreover,
\begin{equation*}
\square_g\Phi=-2\partial_u\partial_r\Phi+D\partial_r^2\Phi-2r^{-1}\partial_u\Phi+(D'+2Dr^{-1})\partial_r\Phi+\slashed{\Delta}\Phi,
\end{equation*}
we have that $\Phi$ satisfies the equation
\begin{equation}
\label{equationforPhiinftyq}
\begin{split}
-2\partial_u\partial_r\Phi&=-D\partial_r^2\Phi+r^{-1}\left[qD-2D'r\right]\partial_r\Phi+(2-q)r^{q-1}\slashed{\Delta}\phi-\slashed{\Delta}\Phi\\
&+r^{-1}[-D''r+(q+1)D'-qDr^{-1}]\Phi+r^{q-1}[D''+(q-1)D'r^{-1}]\phi
\end{split}
\end{equation}
We obtain (\ref{equationforPhiinftyv2}) by rearranging the above terms.

Now consider $\widetilde{\Phi}=r(r-M)\partial_r\phi$. Then we can use (\ref{eq:boxPhiinftyq}) with $q=2$ and $q=1$ to obtain
\begin{equation*}
\begin{split}
\square_g\widetilde{\Phi}=&\:\square_g(r^2\partial_r\phi-Mr\partial_r\phi)\\
=&\:r^{-1}(4D-D'r)\partial_r\widetilde{\Phi}+DMr^{-1}\partial_r(r\partial_r\phi)-2r^{-1}\partial_u\widetilde{\Phi}-M\slashed{\Delta}\phi\\
&+r^{-1}[-D''r+3D'-2Dr^{-1}]\widetilde{\Phi}+(D'r^{-1}-Dr^{-2})Mr\partial_r\phi+[rD''+D'-MD'']\phi.
\end{split}
\end{equation*}

We rewrite
\begin{align*}
DMr^{-1}\partial_r(r\partial_r\phi)=&\:DMr^{-1}\partial_r(r(r-M)(r-M)^{-1}\partial_r\phi)\\
=&\:MDr^{-1}(r-M)^{-1}\partial_r\widetilde{\Phi}-DMr^{-1}(r-M)^{-2}\widetilde{\Phi},\\
(D'r^{-1}-Dr^{-2})Mr\partial_r\phi=&\:M(r-M)^{-1}(D'r^{-1}-Dr^{-2})\widetilde{\Phi}.
\end{align*}

Now we obtain the final expression for $\square_g\widetilde{\Phi}$:
\begin{equation*}
\begin{split}
\square_g\widetilde{\Phi}=&\:r^{-1}(4D-D'r+MD(r-M)^{-1})\partial_r\widetilde{\Phi}-2r^{-1}\partial_u\widetilde{\Phi}-M\slashed{\Delta}\phi\\
&+r^{-1}[-D''r+3D'-2Dr^{-1}-MD(r-M)^{-2}+M(r-M)^{-1}(D'-Dr^{-1})]\widetilde{\Phi}\\
&+[(r-M)D''+D']\phi.
\end{split}
\end{equation*}
Similarly, we use (\ref{equationforPhiinftyq}) with $q=2$ and $q=1$ to obtain
\begin{equation*}
\begin{split}
-2\partial_u\partial_r\widetilde{\Phi}=&\:=-2\partial_u\partial_r(r^2\partial_r\phi-Mr\partial_r\phi)\\
=&\:-D\partial_r^2\widetilde{\Phi}+r^{-1}\left[2D-2D'r\right]\partial_r\widetilde{\Phi}+MDr^{-1}\partial_r(r\partial_r\phi)-M\slashed{\Delta}\phi-\slashed{\Delta}\widetilde{\Phi}\\
&+r^{-1}[-D''r+3D'-2Dr^{-1}]\widetilde{\Phi}+(D'r^{-1}-Dr^{-2})Mr\partial_r\phi+[rD''+D'-MD'']\phi\\
=&\:-D\partial_r^2\widetilde{\Phi}+r^{-1}\left[2D-2D'r+MD(r-M)^{-1}\right]\partial_r\widetilde{\Phi}-M\slashed{\Delta}\phi-\slashed{\Delta}\widetilde{\Phi}\\
&+r^{-1}[-D''r+3D'-2Dr^{-1}-MD(r-M)^{-2}+M(r-M)^{-1}(D'-Dr^{-1})]\widetilde{\Phi}\\
&+[(r-M)D''+D']\phi.
\end{split}
\end{equation*}
\end{proof}

\begin{proof}[Proof of Lemma \ref{lm:commute2time}]
Recall from Lemma \ref{lm:commute1time} that,
\begin{equation*}
\begin{split}
\square_g\Phi&=r^{-1}\left[4D-D'r\right]\partial_r\Phi-2r^{-1}\partial_u\Phi\\
&+r^{-1}[-D''r+3D'-2Dr^{-1}]\Phi+r[D''+D'r^{-1}]\phi.
\end{split}
\end{equation*}

We have that
\begin{equation*}
\begin{split}
[\square_g,\partial_r]\Phi&=-D'\partial_r^2\Phi-2r^{-2}\partial_u\Phi-(D''+2D'r^{-1}-2Dr^{-2})\partial_r \Phi+2r^{-1}\slashed{\Delta}\Phi.
\end{split}
\end{equation*}
and
\begin{equation*}
\begin{split}
\partial_r(\square_g \Phi)&=r^{-1}[4D-D'r]\partial_r^2\Phi-2r^{-1}\partial_u\partial_r\Phi\\
&+r^{-1}[-D'' r+3D'-2Dr^{-1}]\partial_r\Phi+r[D''+D'r^{-1}]\partial_r\phi\\
&+r^{-1}[4D'-D''r-D'-r^{-1}4D+D']\partial_r\Phi\\
&+2r^{-2}\partial_u\Phi\\
&+r^{-1}[-D''' r-D''+3D''-2D'r^{-1}+2Dr^{-2}+D''-3r^{-1}D'+2Dr^{-2}]\Phi\\
&+r[D'''+D''r^{-1}-D'r^{-2}+r^{-1}D''+D'r^{-2}]\phi\\
&=r^{-1}[4D-D'r]\partial_r^2\Phi-2r^{-1}\partial_u\partial_r\Phi\\
&+r^{-1}[-2D'' r+7D'-6Dr^{-1}]\partial_r\Phi+2r^{-2}\partial_u\Phi\\
&+r^{-1}[-D'''r+4D''-4r^{-1}D'+4Dr^{-2}]\Phi\\
&+r[D'''+2D''r^{-2}]\phi.
\end{split}
\end{equation*}
Hence,
\begin{equation*}
\begin{split}
\square_g(\partial_r\Phi)&=[\square_g,\partial_r]\Phi+\partial_r(\square_g \Phi)\\
&=r^{-1}[4D-2D'r]\partial_r^2\Phi-2r^{-1}\partial_u\partial_r\Phi+2r^{-1}\slashed{\Delta}\Phi\\
&+r^{-1}[-3D'' r+5D'-4Dr^{-1}]\partial_r\Phi\\
&+r^{-1}[-D'''r+4D''-4r^{-1}D'+4Dr^{-2}]\Phi\\
&+r[D'''+2D''r^{-1}]\phi.
\end{split}
\end{equation*}
We define $\Phi_{(2)}=r^2\partial_r\Phi$ and obtain
\begin{equation*}
\begin{split}
\square_g\Phi_{(2)}&=\square_g(r^2\partial_r\Phi)=-2\partial_u\partial_r(r^2\partial_r\Phi)+D\partial_r^2(r^2\partial_r\Phi)-2r^{-1}\partial_u(r^2\partial_r\Phi)\\
&+(2Dr^{-1}+D')\partial_r(r^2\partial_r\Phi)+\slashed{\Delta}(r^2\partial_r\Phi)\\
&=r^2\square_g(\partial_r\Phi)-4r\partial_u\partial_r\Phi+4Dr\partial_r^2\Phi+2D\partial_r\Phi+2(2Dr^{-1}+D')r\partial_r\Phi\\
&=r[8D-2D'r]\partial_r^2\Phi-6r\partial_u\partial_r\Phi+2r\slashed{\Delta}\Phi+r[2Dr^{-1}-3D'' r+7D']\partial_r\Phi\\
&+r[-D'''r+4D''-4r^{-1}D'+4Dr^{-2}]\Phi+r^3[D'''+2D''r^{-1}]\phi.
\end{split}
\end{equation*}
By (\ref{equationforPhiinfty}) we have that
\begin{equation*}
\begin{split}
-4r\partial_u\partial_r\Phi&=-2rD\partial_r^2\Phi+2\left[2D-2D'r\right]\partial_r\Phi-2r\slashed{\Delta}\Phi\\
&+2[-D''r+3D'-2Dr^{-1}]\Phi+2r^2[D''+D'r^{-1}]\phi.
\end{split}
\end{equation*}
Hence,
\begin{equation*}
\begin{split}
\square_g\Phi_{(2)}&=r[6D-2D'r]\partial_r^2\Phi-2r\partial_u\partial_r\Phi+r[6Dr^{-1}-3D'' r+3D']\partial_r\Phi\\
&+r[-D'''r+2D''+2D'r^{-1}]\Phi+r^3[D'''+4D''r^{-1}+2D'r^{-2}]\phi.
\end{split}
\end{equation*}
Finally, we use that $r\partial_r^2\Phi=r^{-1}\partial_r\Phi_{(2)}-2r^{-2}\Phi_{(2)}$ to obtain
\begin{equation*}
\begin{split}
\square_g\Phi_{(2)}&=r^{-1}[6D-2D'r]\partial_r\Phi_{(2)}-2r^{-1}\partial_u\Phi_{(2)}+r^{-1}[-6Dr^{-1}-3D'' r+7D']\Phi_{(2)}\\
&+r[-D'''r+2D''+2D'r^{-1}]\Phi+r^3[D'''+4D''r^{-1}+2D'r^{-2}]\phi. \qedhere
\end{split}
\end{equation*}
Since moreover, by definition of the wave operator $\square_g$, we have that
\begin{equation*}
\square_g\Phi_{(2)}=-2\partial_u\partial_r\Phi_{(2)}+D\partial_r^2\Phi_{(2)}-2r^{-1}\partial_u\Phi_{(2)}+(D'+2Dr^{-1})\partial_r\Phi_{(2)}+\slashed{\Delta}\Phi_{(2)},
\end{equation*}
we obtain the following equation for $\Phi_{(2)}$:
\begin{equation*}
\begin{split}
2 \partial_u\partial_r\Phi_{(2)}=&\:D\partial_r^2\Phi_{(2)}-[4Dr^{-1}+D']\partial_r\Phi_{(2)}+\slashed{\Delta}\Phi_{(2)}-r^{-1}[-6Dr^{-1}-3D'' r+7D']\Phi_{(2)}\\
&-r[-D'''r+2D''+2D'r^{-1}]\Phi-r^3[D'''+4D''r^{-1}+2D'r^{-2}]\phi.
\end{split}
\end{equation*}
We obtain (\ref{equationforPhi2v2}) by rearranging the above terms.
\end{proof}

\begin{proof}[Proof of Lemma \ref{lm:boxdrkPsi2}]
We will prove the lemma by induction. If $D=1-\frac{2M}{r}+O_2(r^{-1-\beta})$, then \eqref{eq:boxdrkPsi2} holds for $k=0$, by Lemma \ref{lm:commute2time}. Now suppose $D=1-\frac{2M}{r}+O_{n+3}(r^{-1-\beta})$ and moreover \eqref{eq:boxdrkPsi2} holds for $k\leq n$. We have that
\begin{equation*}
\square_g(\partial_r^{n+1} \Phi_{(2)})=\partial_r(\square_g(\partial_r^n  \Phi_{(2)}))+[\square_g,\partial_r](\partial_r^n \Phi_{(2)}).
\end{equation*}
By \eqref{eq:commdr}, we have that
\begin{equation*}
[\square_g,\partial_r](\partial_r^n  \Phi_{(2)})=-D'\partial_r^{n+2} \Phi_{(2)}-2r^{-2}\partial_u\partial_r^n \Phi_{(2)}-(D''+2D'r^{-1}-2Dr^{-2})\partial_r^{n+1}  \Phi_{(2)}+2r^{-1}\slashed{\Delta} \partial_r^n\Phi_{(2)}.
\end{equation*}
Furthermore,
\begin{equation*}
\begin{split}
\partial_r(\square_g(\partial_r^n \Phi_{(2)}))=&\:\left[6r^{-1}+O(r^{-2})\right]\partial_r^{n+2}{\Phi}_{(2)}+O(r^{-2})\partial_r^{n+1} \Phi_{(2)}-2r^{-1}\partial_u\partial_r^{n+1} \Phi_{(2)}+2r^{-2}\partial_u\partial_r^{n} \Phi_{(2)}\\
&+\sum_{j=0}^{n}O(r^{-j-2})\partial_{r}^{n+1-j} \Phi_{(2)}+\sum_{j=0}^{n}O(r^{-j-3})\partial_{r}^{n-j} \Phi_{(2)}\\
&+2nr^{-1}\slashed{\Delta}(\partial_r^{n}\Phi_{(2)})-4nr^{-2}\slashed{\Delta}(\partial_r^{n-1}\Phi_{(2)})\\
&+\sum_{j=0}^{\max\{n-2,0\}} n(n-1)O(r^{-2-j})\slashed{\Delta}\partial_r^{n-j-1}\Phi_{(2)}\\
&+\sum_{j=0}^{\max\{n-2,0\}} n(n-1)O(r^{-3-j})\slashed{\Delta}\partial_r^{n-j-2}\Phi_{(2)}\\
&+O(r^{-n-4})\Phi_{(2)}+O(r^{-n-3}){\Phi}+O(r^{-n-3})\phi.
\end{split}
\end{equation*}
We conclude that
\begin{equation*}
\begin{split}
\square_g(\partial_r^{n+1} \Phi_{(2)})=&\:\left[6r^{-1}+O(r^{-2})\right]\partial_r^{n+2}\Phi_{(2)}-2r^{-1}\partial_u\partial_r^n\Phi_{(2)}+\sum_{j=0}^{n+1}O(r^{-j-2})\partial_{r}^{n+1-j}\Phi_{(2)}\\
&+2(n+1)r^{-1}\slashed{\Delta}(\partial_r^{n}\Phi_{(2)})+\sum_{j=0}^{\max\{n-1,0\}} nO(r^{-2-j})\slashed{\Delta}\partial_r^{n+1-j-2}\Phi_{(2)}\\
&+O(r^{-n-3})\Phi+O(r^{-n-3})\phi.
 \quad \qedhere
\end{split}
\end{equation*}
\end{proof}

\begin{proof}[Proof of Lemma \ref{lm:commboxdrkpsi0}]
We will prove this by induction. For $k=1$ from Appendix \ref{app:commmultp} we have that
$$ \Box_g (\partial_r \phi ) = \left( \frac{2D}{r} - D' \right) \partial_r^2 \phi  + \left( \frac{D'}{r} - D'' \right) \partial_r \phi + \left( \frac{D''}{r} - \frac{D'}{r^2} \right ) \phi , $$
which is of the form \eqref{eq:partialkphi}. Now we assume that \eqref{eq:partialkphi} holds for some $k\in \mathbb{N}$. Then we have that
$$ \Box_g (\partial_r^{k+1} \phi ) = [ \Box_g , \partial_r ] (\partial_r^k \phi ) + \partial_r \left( \Box_g (\partial_r^k \phi ) \right) . $$
From \eqref{eq:commdr} from the appendix \ref{app:commmultp} we have that
$$ [ \Box_g , \partial_r ] (\partial_r^k \phi ) = - D' \partial_r^{k+2} \phi + \left( \frac{2D}{r^2} - \frac{2D'}{r} - D'' \right) \partial_r^{k+1} \phi - \frac{2}{r^2} \partial_u \partial_r^k \phi . $$
On the other hand by using the inductive hypothesis we have that
$$ \partial_r \left( \Box_g (\partial_r^k \phi ) \right) = \left( \frac{2}{r} + O (r^{-2} ) \right) \partial_r^{k+2} \phi + \left( -\frac{2}{r^2} + O (r^{-3} ) \right) \partial_r^{k+1} \phi + $$ $$ + \sum_{m=0}^{k+1} O (r^{-m-3} ) \partial_r^{k+1-m} \phi - \frac{2}{r} \partial_u \partial_r^{k+1} \phi + \frac{2}{r^2} \partial_u \partial_r^k \phi . $$
Now the result follows by adding the two last expression where we notice that the $O(r^{-2})$ of $\partial_r^{k+1} \phi$ cancel out as well as the terms involving $\partial_u \partial_r^k \phi$, and because $D^{(m)} = O (r^{-m} )$.
\end{proof}

Department of Mathematics, University of California, Los Angeles, CA 90095, United States, yannis@math.ucla.edu

\bigskip

Princeton University, Department of Mathematics, Fine Hall, Washington Road, Princeton, NJ 08544, United States, aretakis@math.princeton.edu

\bigskip

Department of Mathematics, University of Toronto Scarborough 1265 Military Trail, Toronto, ON, M1C 1A4, Canada, aretakis@math.toronto.edu

\bigskip

Department of Mathematics, University of Toronto, 40 St George Street, Toronto, ON, Canada, aretakis@math.toronto.edu

\bigskip

Department of Mathematics, Imperial College London, SW7 2AZ, London, United Kingdom, dejan.gajic@imperial.ac.uk

\bigskip

Department of Applied Mathematics and Theoretical Physics, University of Cambridge, Wilberforce Road, Cambridge CB3 0WA, United Kingdom, dg405@cam.ac.uk

\end{document}